\documentclass[11pt]{amsart}
\usepackage[pdftex,margin=1in,rmargin=1in]{geometry}
\usepackage{amsmath,latexsym,amssymb,amsthm}
\usepackage[final]{graphicx}
\usepackage{mathtools}
\usepackage{xfrac}
\usepackage{fix-cm}
\usepackage{color,transparent}

\usepackage{xcolor}
\usepackage{mdframed}

\newmdenv[
  backgroundcolor=yellow!20,
  linecolor=yellow!50!black
]{highlight}

\usepackage{verbatim}
\usepackage[a-3u]{pdfx}
\hypersetup{bookmarksdepth=3,bookmarksopenlevel=2,%
    colorlinks=true,%
    linkcolor=blue,%
    citecolor=blue,%
    filecolor=blue,%
    menucolor=blue,%
    urlcolor=blue}
\usepackage{stmaryrd}
\usepackage{caption}
\usepackage{subcaption}
\usepackage{multirow}
\usepackage{float}
\usepackage{tabularx}
\usepackage[shortlabels,inline]{enumitem}
\usepackage{pdfsync}
\usepackage{tikz}
\usepackage{tikz-cd}
\usepackage{float}
\usepackage[T1]{fontenc}
\usepackage{geometry}
\usepackage{pdflscape}
\usepackage[utf8]{inputenc}
\usepackage[textsize=small]{todonotes}
\usepackage[shortlabels,inline]{enumitem}
\newcommand{\ora}[1]{\overrightarrow{#1}}
\newcommand{\ole}[1]{\overleftarrow{#1}}
\newcommand{\DPTtodo}[2][]{}
\newcommand{\DPTfyi}[2][]{}
\newcommand{\DMGtodo}[2][]{}
\usepackage{hyphenat}
\usepackage{booktabs}
\usepackage{pifont}
\usepackage{pdfpages}
\usepackage[graphicx]{realboxes}
\usepackage{bm} % Bold math symbols
\usepackage{svg}

\newcommand*{\doi}[1]{\href{https://doi.org/#1}{\begingroup\urlstyle{rm}\nolinkurl{doi:#1}\endgroup}}

\def\mathcenter#1{%
  \vcenter{\hbox{$#1$}}%
}

\def\acts{\curvearrowright}

\graphicspath{{mpdraws/}{draws/}}
\def\mathcenter#1{\vcenter{\hbox{$#1$}}}
\def\mfig#1{\mathcenter{\includegraphics{#1}}}

\def\mfigrotate#1{\mathcenter{\rotatebox[origin=c]{-90}{\includegraphics{#1}}}}

\setenumerate[1]{label=(\arabic*)}

\numberwithin{equation}{section}
\numberwithin{figure}{section}
\newtheorem{mainthm}{Theorem}
\newtheorem{maincor}[mainthm]{Corollary}
\renewcommand*{\themainthm}{\Alph{mainthm}}
\renewcommand*{\themainthm}{\Alph{mainthm}}
\newtheorem{mainthmprime}{Theorem}

\newtheorem*{theorem*}{Theorem}
\newtheorem{theorem}{Theorem}[section]
\newtheorem{lemma}[theorem]{Lemma}
\newtheorem{claim}[theorem]{Claim}
\newtheorem{corollary}[theorem]{Corollary}
\newtheorem{proposition}[theorem]{Proposition}

\newtheorem*{ex*}{Exercise}

\newcounter{case}[theorem]
\newenvironment{case}%
  {\smallskip\par\refstepcounter{case}\noindent\textbf{Case \thecase.}\ \begingroup\itshape}%
  {\endgroup}

\newcounter{subcase}[case]
\newenvironment{subcase}%
  {\smallskip\par\refstepcounter{subcase}\noindent\textbf{Subcase \thecase.\thesubcase.}\ \begingroup\itshape}%
  {\endgroup}

\theoremstyle{definition}
\newtheorem{remark}[theorem]{Remark}

\newtheorem{example}[theorem]{Example}
\newtheorem{definition}[theorem]{Definition}
\newtheorem{question}[theorem]{Question}
\newtheorem{conjecture}[theorem]{Conjecture}

\newtheorem*{notation*}{Notation}

\newtheorem{convention}[theorem]{Convention}

\theoremstyle{plain}
\newtheorem{taggedthmx}{Theorem}
\newenvironment{taggedthm}[1]
 {\renewcommand\thetaggedthmx{#1}\taggedthmx}
 {\endtaggedthmx}
\newtheorem{taggedcorx}{Corollary}
\newenvironment{taggedcor}[1]
 {\renewcommand\thetaggedcorx{#1}\taggedcorx}
 {\endtaggedcorx}

\newcommand\bbN{\ensuremath{\mathbb{N}}}
\newcommand\bbZ{\ensuremath{\mathbb{Z}}}

\newcommand\bbR{\ensuremath{\mathbb{R}}}
\newcommand\bbC{\ensuremath{\mathbb{C}}}

\newcommand\bbH{\ensuremath{\mathbb{H}}}
\newcommand\cA{\ensuremath{\mathcal{A}}}
\newcommand\cB{\ensuremath{\mathcal{B}}}

\newcommand\cF{\ensuremath{\mathcal{F}}}
\newcommand\cG{\ensuremath{\mathcal{G}}}

\newcommand\cM{\ensuremath{\mathcal{M}}}
\newcommand\cR{\ensuremath{\mathcal{R}}}
\newcommand\cS{\ensuremath{\mathcal{S}}}

\DeclareMathOperator{\Curr}{Curr}
\DeclareMathOperator{\CAT}{CAT}
\DeclareMathOperator{\Teich}{Teich}
\DeclareMathOperator{\Isom}{Isom}

\DeclareMathOperator{\DG}{\mathcal{D}\mathcal{G}}

\DeclareMathOperator{\tr}{tr}

\DeclareMathOperator{\EL}{EL}
\DeclareMathOperator{\sys}{sys}

\DeclareMathOperator{\bdy}{\partial}

\newcommand{\OneHalf}{{\textstyle \frac{1}{2}}}
\newcommand{\OneFourth}{{\textstyle \frac{1}{4}}}
\newcommand*{\wt}[1]{\widetilde{#1}}
\newcommand*{\wh}[1]{\widehat{#1}}

\renewcommand{\epsilon}{\varepsilon}

\DeclareMathOperator{\SO}{\mathit{SO}}

\DeclareMathOperator{\SL}{\mathit{SL}}
\DeclareMathOperator{\PSL}{\mathit{PSL}}

\DeclarePairedDelimiter\abs{\lvert}{\rvert}
\DeclarePairedDelimiter\norm{\lVert}{\rVert}

\newcommand{\AConvex}{\mathrm{AConvex}}
\newcommand{\Convex}{\mathrm{Convex}}
\newcommand{\QConvex}{\mathrm{QConvex}}
\newcommand{\AQConvex}{\mathrm{AQConvex}}
\newcommand{\Cont}{\mathrm{Cont}}
\newcommand{\ACont}{\mathrm{ACont}}

\newcommand{\Curves}{\mathcal{C}}

\newcommand{\wid}{\operatorname{wid}}
\def\acts{\curvearrowright}

\newcommand{\ML}{\mathcal{ML}}

\newcommand{\Cay}{\mathrm{Cay}}

\newcommand{\reducesto}{\mathrel{\searrow}}

\DeclareMathOperator{\supp}{supp}

\usetikzlibrary{arrows}
\tikzset{
    labl/.style={anchor=south, rotate=90, inner sep=.5mm}
}
\tikzstyle{every picture}=[> = to]
\tikzset{cdlabel/.style={execute at begin node=$\scriptstyle,execute at end node=$}}
\tikzset{implication/.style={double equal sign distance, -implies}}
\tikzset{biimplication/.style={double equal sign distance, implies-implies}}

\begin{document}
\title{Lengths as intersections on surfaces} 
\author[Martínez-Granado]{Dídac Martínez-Granado}
\address{Department of Mathematics\\University of Luxembourg\\Av. de la Fonte 6, Esch-sur-Alzette, L-4364, Luxembourg}       
\email{didac.martinezgranado@uni.lu}

\author[Thurston]{Dylan~P.~Thurston}
\address{Department of Mathematics\\
  Boston College\\
  Chestnut Hill, Mass.\ 02467-3806\\
  USA}
\email{thurst@bc.edu}

\begin{abstract}
Geodesic currents on closed orientable hyperbolic surfaces are measures on the
unit tangent bundle invariant
under geodesic flow and orientation reversal.
Every geodesic current induces a kind of ``length'', namely the dual
functional on
curves defined by the
geometric intersection pairing with the current. It is natural to ask
which other ``lengths'' or curve
functionals arise this way.

In this paper we give a purely axiomatic and combinatorial
characterization of curve functionals on closed orientable hyperbolic
surfaces that are dual to
geodesic currents. This
yields a new definition of geodesic currents as curve functionals or,
equivalently, as functions on surface groups, without reference to
measures or flows. More precisely, we show that a functional on curves
arises as the geometric intersection pairing with a geodesic current
if and only if it is additive under disjoint union and satisfies a
simple \emph{smoothing} property: it is non-increasing under
surgery of essential crossings.

As applications, we obtain new axiomatic characterizations of measured
laminations and hyperbolic length functionals, and new descriptions of surface group actions on real trees dual to measured laminations, including a streamlined proof of a
classical theorem of Skora. We also provide a unified framework for
dual geodesic currents arising from metric structures and generalized
cross-ratios, including those associated with certain Anosov
representations. Our approach subsumes all previously
known constructions of dual geodesic currents for closed surfaces and
yields broad new
families of examples.
\end{abstract}
\maketitle

\section{Introduction}

Geodesic currents on a closed orientable hyperbolic surface $S$ are $\pi_1(S)$-invariant
Radon measures on the
space of unoriented geodesics on the universal cover~$\wt{S}$, forming the weak$^*$
closure of the set
of weighted closed geodesics. Equivalently, from a dynamical
viewpoint, they may be regarded as measures on the unit tangent bundle
that are invariant under the geodesic flow and the orientation flip.
Since their introduction by Bonahon~\cite{Bonahon86:EndsHyperbolicManifolds}
in his study of ends of hyperbolic $3$-manifolds, they have become a
central tool in low-dimensional geometry and topology. They play a
fundamental role in compactifications of spaces of geometric
structures \cite{Bonahon88:GeodesicCurrent, DLR10:DegenerationFlatMetrics,
MZ19:PositivelyRatioed,BIPP24:PositiveCrossratios}, in curve counting
problems~\cite{RS19:CountingProblems, EPS20:CountingCurves,
ES22:GeodesicCount}, in marked length spectrum
rigidity \cite{Otal90:SpectreMarqueNegative, CFF92:RigidityNonPosCurvedRiem,
HP97:RigidityNegCurvedCone, Con18:MarkedNonpos}, and more recently in
geometric group theory \cite{HV24:Biautomaticity, CR25:Manhattan, CMGR26:GreenMetrics} and probability theory~\cite{KS26:Singularity}.

In this paper, we give a third equivalent definition of geodesic
currents on closed surfaces, in a purely topological and combinatorial way.

The geometric intersection number of multi-curves (finite unions of
closed curves on $S$) extends continuously to a bilinear function of
geodesic currents $i(\cdot, \cdot)$ by a result of
Bonahon \cite[Proposition~4.5]{Bonahon86:EndsHyperbolicManifolds}.
Moreover, by Otal's work
\cite[Th\'eor\`eme~2]{Otal90:SpectreMarqueNegative}, a geodesic
current $\mu$ is uniquely determined by the functional $i(\mu,\cdot)
\colon \mathcal{C}(S) \to \mathbb{R}$, where $\mathcal C(S)$ denotes
the space of multi-curves. Thus geodesic currents embed into the space
of real-valued functionals on curves via intersection.

Many natural curve functionals are known to arise in this way. Hyperbolic length
is given by intersection with the Liouville
current \cite[Theorem~10]{Bonahon88:GeodesicCurrent}; the same holds
for length in  negatively curved~\cite{Otal90:SpectreMarqueNegative}
and non-positively
curved~\cite{CFF92:RigidityNonPosCurvedRiem} Riemannian metrics.
Beyond examples arising from metrics on the
surface, word length with respect to simple generating sets admits a
dual current \cite[Theorem~1.2]{Erlandsson:WordLength}, and length
functions associated to certain higher-rank representations of surface
groups (such as Hitchin representations) also arise from
currents~\cite{MZ19:PositivelyRatioed,BIPP24:PositiveCrossratios}.

From this perspective, geodesic
currents are the geometric-intersection dual objects
to a distinguished subclass
of curve functionals, in analogy to Poincar\'e duality (which uses
algebraic intersection). The ambient space of curve functionals is
vast and includes many
natural examples not arising from currents
(Section~\ref{subsec:non-examples}). We give
a necessary and sufficient intrinsic axiomatic characterization of
those curve functionals that arise as intersection with a geodesic
current.

\begin{mainthm}
\label{thm:intersect}
  A curve functional $f \colon \Curves(S)\to\mathbb{R}$ is dual to a
  geodesic current~$\mu_f$  in the sense that
\[
\text{for all curves } C,\quad
i(\mu_f,C)=f(C)
\]
 if and only if it satisfies the
\emph{smoothing property}
\begin{equation}\label{eq:smoothing}
f\!\left(\mfig{curves-0}\right)
\ge
\max\!\left\{
f\!\left(\mfig{curves-1}\right),
f\!\left(\mfig{curves-2}\right)
\right\},
\end{equation}
at every essential crossing of a multi-curve and is \emph{additive} on connected components.
  Furthermore, $\mu_f$ is unique, and the topology on geodesic
  currents agrees with the topology on functionals as a subspace of
  $\bbR^{\Curves(S)}$ with the product topology.
\end{mainthm}

In particular, Theorem~\ref{thm:intersect} gives an equivalent
definition of geodesic currents.

An essential crossing is an intersection
that cannot be removed by homotopy; see
Section~\ref{subsec:crossings}.
Smoothing an essential crossing of a multi-curve $C$ consists of
cutting and regluing its strands in a small neighborhood of the
crossing, thereby producing a new multi-curve $C'$ with fewer
crossings, denoted $C \reducesto C'$.
Such crossings include both
self-intersections of a single component
and intersections between distinct components. For more on this and
the definition of additivity, see Section~\ref{subsec:mainhypotheses}.

The hypotheses of smoothing and additivity are
independent and necessary (see
Section~\ref{subsec:non-examples}).
The axioms in Theorem~\ref{thm:intersect} single out a locally compact
cone in $\mathbb{R}^{\Curves(S)}$. We expect that suitable weakenings
of some of these axioms will impose enough structure to ensure local
compactness of other natural spaces of functionals arising in
compactifications of higher-rank surface groups representations and in
the context of rational maps; see Sections~\ref{subsec:examples}
and~\ref{sec:local-compactness}.

Theorem~\ref{thm:intersect} produces numerous new dual currents and
unifies several constructions.  We show that every symmetric length
pseudo-metric on $S$ has a dual geodesic current recovering its marked
length spectrum; no curvature or dynamical hypotheses are required.
In particular, this associates a canonical generalized Liouville
current to an arbitrary Riemannian metric on~$S$, including those with
conjugate points, or, more generally, any reversible Finsler metric.
In fact this gives another characterization of geodesic
currents.
\begin{maincor}
  \label{maincor:length_dual}
  An additive curve functional satisfies the smoothing
  property~\eqref{eq:smoothing}, and is thus dual to a geodesic
  current, if and only if it is the induced length on curves from a
  length pseudo-metric on~$S$.
\end{maincor}

This extends earlier constructions for negatively curved,
non-positively curved, and locally $\CAT(-1)$ or $\CAT(0)$ metrics \cite{Otal90:SpectreMarqueNegative, CFF92:RigidityNonPosCurvedRiem, DLR10:DegenerationFlatMetrics, HP97:RigidityNegCurvedCone,BL17:RigidityFlat,Con18:MarkedNonpos, PS26:1ntranslation}.
The construction of the length pseudo-metric from a geodesic current
is due to Burger, Iozzi, Parreau, and Pozzetti
\cite[Section~4]{BIPP21:Systoles}; see also \cite{dRMG22:Duals}.
We caution the reader that the length pseudo-metrics in this corollary
do not in general
induce the standard topology on~$S$, and the distance functions may
not even be continuous on $S \times S$ \cite[Propositions 4.7,
4.12]{dRMG22:Duals}.
These terms are defined, and Corollary~\ref{maincor:length_dual} is
proved, in Section~\ref{subsec:length-metrics}.

In short, Theorem~\ref{thm:intersect} and Corollary~\ref{maincor:length_dual} give two new equivalent geometric perspectives on geodesic currents: as additive curve functionals satisfying smoothing and as curve lengths of length pseudo-metrics on $S$. 

Theorem~\ref{thm:intersect} further yields dual currents (in fact, dual multi-curves) for stable
lengths associated to embedded graphs. Let $\iota \colon \mathcal{G}
\hookrightarrow S$ be a filling embedded graph, in the sense that the complementary
regions are disks. 
Endow $\mathcal{G}$ with the edge metric, where each edge has
length~1. Then any closed multi-curve $C$ on~$S$ can be homotoped so
that it factors through $\mathcal{G}$.
 Let $\ell_{\mathcal{G}}(C)$ be the length of the shortest multi-curve
 $D$ on $\mathcal{G}$ so that $\iota(D)$ is homotopic to $C$.
\begin{maincor}
\label{maincor:embedded}
Let $\mathcal{G} \hookrightarrow S$ be a filling embedded graph
endowed with the edge metric. Then
the functional $\ell_{\mathcal{G}}$ is dual to a
geodesic current $\mu_{\mathcal{G}} = \frac12 M$ where $M$ is null-homologous in
$H_1(S; \mathbb{Z}/2\mathbb{Z})$.
\end{maincor}

This generalizes the work of Erlandsson~\cite{Erlandsson:WordLength}.
Corollary~\ref{maincor:embedded} is proved in
Section~\ref{subsec:embedded_graphs}.

All of these examples give lengths extending continuously to $\Curr(S)$
by Bonahon’s bilinear extension of intersection. A result
of Erlandsson--Parlier--Souto \cite[Theorem~1.5]{EPS20:CountingCurves}
establishes continuous
extensions for arbitrary Riemannian lengths and, more generally, for a
broad class of lengths on Gromov hyperbolic spaces. In the surface
setting, Theorem~\ref{thm:intersect} strengthens this by showing that
many such extensions arise canonically as intersection with a current.
    
    For curve functionals arising in algebraic contexts---such as
    translation lengths of group actions on real trees
    (Section~\ref{sec:skora}) or cross-ratios
    (Section~\ref{sec:crossratios})---it is natural to reformulate the
    smoothing condition in purely group-theoretic language, or, equivalently, in terms of
    oriented curves.
    Let $a,b \in \pi_1(S)$ be non-trivial elements that are not common
    powers. We write $A$ for the inverse of $a \in \pi_1(S)$. Fix any
    hyperbolic structure on $S$, and denote by $\ora{a}$ and $\ora{b}$
    the corresponding axes in the universal cover. We say $(\ora{a}, \ora{b})$ are \emph{crossing} if they intersect at a point, and are \emph{parallel} if they are non-intersecting and coherently aligned. (See
Definition~\ref{def:crossing} and
    Figure~\ref{fig:cases-a-b} for a precise definition.)

\begin{mainthmprime}
\label{thm:intersections_group}
Let $f \colon \pi_1(S) \to \mathbb{R}$ be a functional on $\pi_1(S)$
and $\overline{f}(g) \coloneqq \frac{1}{2} \left(f(g) + f(G) \right)$
be its symmetrization.
Suppose $f$ satisfies the following conditions for $a,b \in \pi_1(S)$:
\begin{enumerate}
\item class invariance: $f(a)=f(b)$ whenever $a$ is conjugate to $b$.
\item stability: 
$f(a^n)=n f(a) \mbox{ for all } n \geq 1$.
\item oriented smoothing:
 \begin{align}
    \label{eq:conn_oriented} f(ab) &\geq f(a) + f(b) && \mbox{ if } (\ora{a},\ora{b})\mbox{ are } \mbox{parallel;} \\
    \label{eq:disconn_oriented} f(a) + f(b) &\geq f(ab)  && \mbox{ if }  (\ora{a},\ora{b})\mbox{ are } \mbox{crossing}.
  \end{align}
\end{enumerate}
Then there exists a dual geodesic current $\mu_{\overline{f}}$ on $S$ so that
\begin{equation*}
i(\mu_{\overline{f}}, [g]) = \overline{f}(g)
\quad\text{for every }g \in \pi_1(S).
\end{equation*}
As a partial converse, if $\mu$ is any geodesic current with
$i(\mu,[g]) = \overline{f}(g)$ for every $g$, then $\overline{f}$ satisfies
conditions (1)--(3) above.
\end{mainthmprime}

Condition~(3) in Theorem~\ref{thm:intersections_group} is formally
weaker than Eq.~\eqref{eq:smoothing} of Theorem~\ref{thm:intersect}, as
it captures smoothing only in an oriented sense (see
Section~\ref{sec:smoothing}). Nevertheless, when combined with the
remaining hypotheses, it becomes equivalent to the full smoothing
condition for $\overline{f}$; more precisely, equations~\eqref{eq:conn_oriented}
and~\eqref{eq:disconn_oriented} together recover smoothing in the
sense of Theorem~\ref{thm:intersect}
(Proposition~\ref{prop:oriented_smoothing_unoriented}).

Inequalities~\eqref{eq:conn_oriented} and~\eqref{eq:disconn_oriented} are individually necessary and logically independent (Section~\ref{subsec:mainhypotheses}), but become equivalent for functionals that extend continuously to geodesic currents (Proposition~\ref{prop:cont_conn_disc}). Moreover,
the symmetrization $\overline f(g)=\tfrac12(f(g)+f(G))$ is
essential; without it, the conclusion fails
(Appendix~\ref{sec:asymmetric}).

\subsection{Characterization of measured laminations and hyperbolic metrics}

As an application of Theorem~\ref{thm:intersections_group}, we obtain
an intrinsic
characterization of curve functionals whose dual currents are
measured laminations.
(Measured laminations are currents $\alpha$
so that $i(\alpha,\alpha)=0$
\cite[Proposition~17]{Bonahon88:GeodesicCurrent}.)

\begin{mainthm}
\label{thm:lam-char}
Let $f \colon \pi_1(S) \to \mathbb{R}$ be a functional on~$\pi_1(S)$.
Then there exists a measured lamination~$\mu_{f}$ on $S$ such that
\[
\text{for every } g\in \pi_1(S),\quad
i(\mu_{f},[g])=f(g)
\]
if and only if $f$ satisfies the conditions of symmetry, class invariance,
stability, oriented smoothing,
and, when $(\ora{g}, \ora{h})$ are crossing,
\begin{equation}\label{eq:max-smoothing}
f(g)+f(h)=\max\{f(gh),f(gH)\}.
\end{equation}
\end{mainthm}

Here the conditions, except for the last, come from
Theorem~\ref{thm:intersections_group}.

We can likewise characterize curve functionals giving length with respect to
(some) hyperbolic metric, following Bonahon \cite[Theorem~13]{Bonahon88:GeodesicCurrent}.
\begin{mainthm}
\label{thm:hyp-char}
Let $f \colon \pi_1(S) \to \mathbb{R}$.
Then there exists a  hyperbolic metric~$d$
on $S$, with associated Liouville current $\mu_d$, such that
\[
\text{for every } g\in \pi_1(S),\quad
f(g) = \ell_d([g]) = i(\mu_d,[g])
\]
if and only if $f$ satisfies the conditions of symmetry, class invariance,
stability, oriented smoothing,
positivity ($f(g) > 0$ for every non-identity element~$g$),
and, when $(\ora{g},\ora{h})$ are crossing,
\begin{equation}\label{eq:lambda-smoothing}
\lambda(g)\lambda(h)
  = \lambda(gh)+\lambda(gH)
\end{equation}
where $\lambda(g)\coloneqq 2\cosh\!\bigl(f(g)/2\bigr)$.
\end{mainthm}

These theorems are proved in Sections \ref{subsec:laminations}
(Theorem~\ref{thm:lam-char}) and~\ref{subsec:liouville}
(Theorem~\ref{thm:hyp-char}).

We highlight some relevant consequences of these two results.

\begin{itemize}
\item They show that, among additive curve functionals satisfying smoothing, the spectra of measured laminations and hyperbolic metrics are singled out by identities involving only crossing pairs of group elements. In particular, in the
    hyperbolic case this avoids checking relations coming from
    self-crossings. In the setting of trace identities in the
    $\PSL(2,\mathbb{R})$ character variety, the relations for
    self-crossings are more intricate, as they depend on sign choices.
    \item They extend Luo’s characterizations of those
functions on simple closed curves that arise as either intersection with
measured laminations~\cite[Theorem~1.1, Theorem~1.2]{Luo10:Simple}
or as hyperbolic length
functions~\cite{Luo98:GeodesicLength}.
\end{itemize}
    For classes of metrics with marked length spectrum rigidity, Theorem~\ref{thm:hyp-char} lets us conclude a metric is
    hyperbolic by looking at its marked length spectrum. Indeed, its marked length spectrum satisfies smoothing by
Corollary~\ref{maincor:length_dual}, and thus is isometric to a
hyperbolic metric
if and only if it satisfies Equation~\eqref{eq:lambda-smoothing}. 
     \begin{maincor}\label{maincor:rigidity_hyperbolic} Let $\mathcal{M}$ be a class of symmetric length pseudo-metrics on $S$ containing hyperbolic metrics up to diffeomorphism isotopic to the identity. Suppose marked length
spectrum rigidity holds. Then $d \in \cM$ is a hyperbolic metric if and only if it satisfies Equation~\eqref{eq:lambda-smoothing}.
\end{maincor}

Some classes $\mathcal{M}$ for which this
applies are negatively curved Riemannian metrics
\cite{Otal90:SpectreMarqueNegative,HP97:RigidityNegCurvedCone},
nonpositively curved
Riemannian metrics
\cite{CFF92:RigidityNonPosCurvedRiem,Con18:MarkedNonpos}, and
Riemannian metrics with Anosov geodesic
flow~\cite{GLP25:MarkedAnosov}.

\subsection{Characterizing actions on real trees dual to measured laminations} 

We next apply Theorem~\ref{thm:intersect} to give new equivalent characterizations of $\pi_1(S)$ actions 
on $\mathbb{R}$-trees dual to measured laminations. In
particular, we translate the smoothing axioms into explicit
identities for translation lengths on pairs of group elements and relate them to the configuration of axes in the tree. As a consequence, we obtain a streamlined proof of Skora’s classical theorem
\cite{Sko90:GeometricAction}.

An \emph{$S$-tree} is a minimal isometric action of $\pi_1(S)$ on an $\mathbb{R}$-tree $T$. 
For $g\in\pi_1(S)$, let
\[
\norm{g}=\inf_{x\in T} d(x,gx)
\]
denote its translation length, and let $T_g \subset T$ denote its
\emph{characteristic set}, i.e., those points that are translated the
least. Concretely,
an isometry $g$ of $T$ either has a non-empty set $T_g$ of fixed
points (\emph{elliptic}) or
has no fixed points and translates points on its invariant
axis $T_g$ by a definite distance of $\norm{g}>0$
(\emph{hyperbolic}). See \cite[\S1.3]{CM87:GroupActions}.
We use the standard notion of irreducible action; see Section~\ref{sec:skora}.

We say the action on~$T$ \emph{preserves characteristic intersection} if, whenever
the axes of $a,b \in \pi_1(S)$ in $\wt\Sigma$ intersect, their
characteristic sets $T_a$ and $T_b$ intersect.
If we only require this when $a$ and $b$ act hyperbolically on $T$,
we say the action \emph{preserves axis intersection}.
Two hyperbolic
isometries of $T$ \emph{overlap incoherently} if their axes intersect
in a
non-degenerate arc on which they translate in opposite directions.
Finally, for a measured lamination~$\lambda$, the tree $T_\lambda$ is
$S$-tree dual to~$\lambda$, obtained by collapsing~$\widetilde{\Sigma}$
by the pseudo-metric from the transverse
measure on~$\lambda$.

\begin{mainthm}
\label{mainthm:skora}
Let $T$ be an $S$-tree. The following are equivalent:
\begin{enumerate}[label=(\alph*),ref=\themainthm(\alph*)]
\item\label{item:skora_lamination}
There exists a measured lamination $\lambda$ on $S$ such that
$T$ is $\pi_1(S)$-equivariantly isometric to~$T_\lambda$;
\item\label{item:skora_axis}
The action is irreducible and preserves axis intersection;
\item\label{item:skora_characteristic}
The action is irreducible and preserves characteristic intersection;
\item\label{item:skora_maxsmoothing}
The translation length function
$\norm{\cdot}\colon \pi_1(S)\to\mathbb{R}$ satisfies \emph{max-smoothing}:
\begin{enumerate}[(i)]
\item\label{item:skora_maxsmoothing_crossing}
For every pair $(\ora{a},\ora{b})$ crossing in $\wt{S}$,
\begin{equation}
\norm{a}+\norm{b}
=
\max\{\norm{ab},\norm{aB}\};
\label{eq:small_disconn}
\end{equation}
\item\label{item:skora_maxsmoothing_parallel}
For every pair $(\ora{a},\ora{b})$ parallel in $\wt{S}$,
\begin{equation}
\norm{ab}
=
\max\{\norm{a}+\norm{b},\norm{aB}\}.
\label{eq:small_conn}
\end{equation}
\end{enumerate}
\item\label{item:skora_incoherent}
  No parallel pair $(\ora{a},\ora{b})$ acts hyperbolically on $T$ with incoherently overlapping axes.
\end{enumerate}
\end{mainthm}

The implication $\text{(a)} \Rightarrow \text{(b)}$ is due to
Morgan--Otal~\cite{MO93:RelativeGrowths}, and
$\text{(b)} \Rightarrow \text{(a)}$ recovers
Skora’s theorem~\cite{Sko90:GeometricAction} (which we reprove).
All the remaining equivalences are new. Skora also proved that \ref{item:skora_lamination} is equivalent to the action being \emph{small}, meaning that every non-degenerate arc has cyclic stabilizer~\cite{Sk96:Splittings}.

The max-smoothing identities in \ref{item:skora_maxsmoothing} refine the general characterization of translation lengths of real trees due to Culler--Morgan and Parry~\cite{CM87:GroupActions,Par91:Translation}. Roughly, they characterize when a conjugacy-invariant
function $\norm{\cdot}\colon\Gamma\to\mathbb{R}_{\geq 0}$ on a general finitely generated group $\Gamma$ is the translation length
of a $\Gamma$ action on a tree $T$ in terms of inequalities involving, for every $a,b\in\Gamma$ the quantities
\[
\norm{a}+\norm{b}, \qquad \norm{ab}, \qquad \norm{aB}.
\]

For surface groups, Theorem~\ref{mainthm:skora} refines this
characterization: we single out trees dual to measured laminations
by restricting which equality occurs, according to the geometric
relation between $\ora{a}$ and~$\ora{b}$.

\subsection{Generalized cross-ratios}

Finally, we apply Theorem~\ref{thm:intersect} to generalized
cross-ratios, thereby giving a unified framework for
the constructions of
Martone--Zhang~\cite{MZ19:PositivelyRatioed} and
Burger--Iozzi--Parreau--Pozzetti~\cite{BIPP24:PositiveCrossratios},
both associating geodesic currents to generalized cross-ratios arising
from certain Anosov representations.

A \emph{generalized cross-ratio} is a function
\[
[\cdot,\cdot,\cdot,\cdot]\colon (\partial_\infty S)^{(4)}\to\mathbb{R}
\]
on ordered 4-tuples of distinct boundary points satisfying natural
invariance, symmetry, finite additivity, and non-negativity properties
(see Definition~\ref{def:general_crossratio}). 
For instance, if $\Sigma$ is a hyperbolic metric on $S$, the classical
hyperbolic cross-ratio $[\cdot]_\Sigma$ obtained from the upper
half-plane model provides such an example.

Every generalized cross-ratio induces an additive curve functional,
called its \emph{period}, denoted $\ell_{[\cdot]}$.
In the hyperbolic case, $\ell_{[\cdot]_\Sigma}$ coincides with the
hyperbolic length function.

We show that periods automatically satisfy the
axioms of Theorem~\ref{thm:intersect}.
\begin{mainthm}
Let $[\cdot,\cdot,\cdot,\cdot]$ be a generalized cross-ratio with
associated period $\ell$. Then $\ell$ satisfies the hypotheses of
Theorem~\ref{thm:intersect} and therefore induces a geodesic current
$\mu_\ell$.
\label{mainthm:crossratio}
\end{mainthm}
  
In particular, for closed surfaces, this result
recovers~\cite[Proposition~4.8]{BIPP24:PositiveCrossratios}, which
constructs geodesic currents from generalized cross-ratios; their
result applies more generally to finite-type surfaces. Their
construction uses the values of the cross-ratio on all geodesic boxes.
By contrast, our proof constructs the associated geodesic current solely
from the periods of the cross-ratio. This is useful because these
periods can be accessed directly from the underlying
representation, without considering their limit maps. At the same time, our method provides an alternative
approach to constructing geodesic currents associated with
compactifications of certain spaces of Anosov representations, a topic
that has recently attracted considerable attention.
 (see,
e.g.,~\cite{OT23:SO23,OT24:Blaschke,Reid25:Boundary,PS26:1ntranslation}).

\subsection{Acknowledgments}
\label{sec:acknowledgments}

We would like to thank Viveka Erlandsson, Maxime Fortier Bourque, Meenakshy Jyothis, Giuseppe Martone, Hugo Parlier, Beatrice Pozzetti, Kasra Rafi, Eduardo Reyes, Jenya Sapir, Noelle Sawyer, Juan Souto and Ivan Tel\-pu\-khov\-skiy for useful conversations.

The first author was supported by the Hazel King Thompson fellowship
of the Department of Mathematics at Indiana University Bloomington,
for the initial part of this work, and by the Luxembourg National
Research Fund AFR/Bilateral-ReSurface 22/17145118 and the Marie
Skłodowska-Curie Action CurrGeo grant (101154865) for the latter part.
Part of this work was carried out during the first author’s visit to
the National University of Singapore; he thanks the department for its
hospitality and excellent working conditions. The second author was
supported by the National Science
Foundation under Grant Numbers DMS-1507244, DMS-2110143, and DMS-2508603, and did
most of this work while at Indiana University Bloomington.

\subsection{Use of AI}

In the second revision of this manuscript, after substantial
improvements by the authors alone,
we used the LLM Open Math OS DeepSeek V4.1
for copy-editing, catching typos and argument verification
on a final draft.
It identified notation and typographical errors.  Mathematical
feedback included
small inconsistencies (group identities, cyclic orders,
exponents, and summation bounds) and
checked citations, and suggested improvements to them.
It found a gap in the proof of
Theorem~\ref{thm:char_laminations_plus}, which we fixed with our own arguments.
Its role was limited to
reviewing existing material; no new mathematical content or
simplifications of the original arguments were introduced through its
use.

\subsection{Organization}
\label{sec:organization}

\begin{table}[ht!]
    \begin{tabular}{@{}ll@{}}
      \toprule
      Notation & Meaning \\
      \midrule
         $S$ &  closed, connected, orientable topological surface of genus $\geq 2$ \\
         $\Sigma$ &  hyperbolic structure on $S$ \\
         $\pi_1(S)$ &  fundamental group of $S$ \\
         $G(S)$ &  oriented geodesics of $\wt{S}$ \\
         $\partial_{\infty}S$ &  boundary at infinity of $\wt{S}$ \\        
        $(\partial _\infty S)^{(4)}$ & 4-tuples of ccw distinct ordered points in $\partial_\infty S$ \\
         $\Curves(S)$ & unoriented multi-curves \\
         $\Curves^+(S)$ & oriented multi-curves \\
         $f$ & curve functional \\
         $\Curr(S)$ & flip-invariant geodesic currents \\
         $\Curr^+(S)$ & oriented geodesic currents \\
         $\gamma$ & concrete multi-curve \\
         $C,D$ & multi-curves on $S$\\
         $\delta_C$ & geodesic current associated to a multi-curve~$C$\\
         $g,h,a,b,x$ & elements of $\pi_1(S)$ \\ 
         $\ora{g}$ & axis of hyperbolic isometry of $g \in \pi_1(S)$ \\  
         $\Gamma$ & general group (sometimes Gromov hyperbolic) \\
         $a^{-1}$ or $A$ & the inverse of $a \in \Gamma$ \\ 
         $[a]$ & conjugacy class of $a \in \Gamma$, or the
                 corresponding curve \\
         $[a][b]$ & multi-curve $[a] \cup [b]$ \\
         $\cB$ & family of right-handed boxes \\ 
         $\cR = \cR_{\cB}$ & ring of sets generated by $\cB$ \\
         $T$ & real tree \\
         $T_g$ & axis/fixed point set of isometry $g \in \Gamma$ acting on $T$ \\ 
        $[\cdot,\cdot,\cdot, \cdot ]$ & generalized cross-ratio \\
        $\mathcal{G}$ & graph\\ 
        $\vec{\imath}(\cdot, \cdot)$ & asymmetric intersection number between oriented currents \\
        $i(\cdot, \cdot)$ &  geometric intersection number between oriented currents \\ \bottomrule\addlinespace
    \end{tabular}
  \caption{Notation for the most common objects in the paper.}\label{tab:object-notation}
\end{table} 

In Section~\ref{sec:functionals}, we describe in more detail the
hypotheses of Theorem~\ref{thm:intersect} and
Theorem~\ref{thm:intersections_group}, and place them in the broader
setting of curve functionals. We give examples that satisfy these
hypotheses, as well as natural ones that do not.

In Section~\ref{sec:backround}, we introduce the basic background. In particular, we recall the measure-theoretic background and the existence of limits of certain sequences used in the construction of the measure in Theorem~\ref{thm:intersect}. 

In Section~\ref{sec:smoothing}, we establish basic lemmas on essential
crossings and smoothings. We review left and right actions of
$\pi_1(S)$ on $\wt{S}$. We relate the topological, algebraic, and
geometric notions of smoothing. We prove further results on crossings
and parallelism used throughout the paper. We then derive more
technical lemmas involving tuples of elements of $\pi_1(S)$, which
form a key step in the proof of Theorem~\ref{thm:intersect}.

In Section~\ref{sec:define_measure}, we define the dual measure associated to a curve functional by specifying its values on a dense family of right-handed boxes, which we show form a semi-ring in the measure theoretic sense. The measure is $\pi_1(S)$-invariant on this family. We prove that it is non-negative and finitely additive on the semi-ring.

In Section~\ref{sec:construct}, we finish the construction of the dual
geodesic current. In particular, we show that the dual measure is
countably additive and defines an oriented geodesic current. By
averaging, we obtain a geodesic current. Finally, we show that this
current is dual to the original curve functional. We also introduce
the notion of asymmetric intersection number. This completes the proof
of Theorem~\ref{thm:intersect}.
We then prove Theorem~\ref{thm:intersections_group} and several variants as corollaries.

In Section~\ref{sec:smoothing-types}, we analyze the different notions
of smoothing (oriented vs.\ unoriented, connected vs.\ disconnected). We use
these results to deduce Theorem~\ref{thm:intersections_group} from
Theorem~\ref{thm:intersect}.

In Section~\ref{sec:special-classes}, we discuss subclasses of
geodesic currents, including filling currents, multi-curves, measured
laminations, and hyperbolic lengths/Liouville currents. In particular,
we prove
Corollaries~\ref{maincor:length_dual} and~\ref{maincor:embedded} and
Theorems~\ref{thm:lam-char} and~\ref{thm:hyp-char},
construct currents dual to
embedded graphs, and explain how to obtain dual currents for arbitrary
length metrics.

In Section~\ref{sec:skora}, we prove Theorem~\ref{mainthm:skora},
starting with the necessary background on real trees and group actions.

In Section~\ref{sec:crossratios}, we show that period functionals
arising from generalized cross-ratios satisfy the hypotheses of
Theorem~\ref{thm:intersections_group}. They therefore induce geodesic
currents; this proves Theorem~\ref{mainthm:crossratio}.

Finally, Section~\ref{sec:questions} discusses further questions and conjectures.

The Appendices address topics not central to the main proofs. Appendix~\ref{sec:hyp-par} gives a hyperbolic
parallelogram identity used to characterize hyperbolic lengths.
Appendix~\ref{sec:asymmetric} explains why asymmetric word metrics are
not dual to right-handed intersection numbers, highlighting the role
of symmetry in Theorem~\ref{thm:intersect}.
Appendix~\ref{app:necessary} gives an alternative proof
that $i(\mu,\cdot)$ satisfies smoothing, using hyperbolic geometry on $\wt{S}$.

%%% Local Variables:
%%% mode: latex
%%% TeX-master: "Intersections"
%%% End:

\section{Curve functionals: properties, general context and examples}
\label{sec:functionals}

In this section we state the main technical hypotheses that we require
of our curve functionals, including examples and non-examples for the
hypotheses of Theorem~\ref{thm:intersect}.
In particular, we relate the results of this paper to our previous
work~\cite{MGT21:Smoothings}.
Throughout, $S$ will be a closed, orientable surface of genus at least~$2$.

\subsection{Curve functionals and their properties}
\label{subsec:mainhypotheses}

In our constructions, it will be crucial to work in the setting of curves.

\begin{definition}[multi-curve]\label{def:multi-curve}  
A \emph{concrete multi-curve} $\gamma$ on a surface $S$ is a smooth
oriented 1-manifold without boundary $X(\gamma)$ together with a continuous map
(also called $\gamma$) from $X(\gamma)$
into~$S$. $X(\gamma)$ is not necessarily connected.
We say that $\gamma$ is \emph{trivial} if it is
homotopic to a constant map.
Two concrete multi-curves $\gamma$ and~$\gamma'$ are \emph{equivalent} if
they are related by a sequence of the following moves:
\begin{itemize}
    \item \emph{homotopy} within the space of all maps from $X(\gamma)$ to~$S$;
    \item \emph{orientation-preserving reparametrization} of $X(\gamma)$; and
    \item \emph{dropping} trivial components.
\end{itemize}
The equivalence class of $\gamma$ is denoted by~$\vec{C}=[\gamma]$,
and we will
call it a \emph{multi-curve}. 
If $X(\gamma)$ is connected, we will call $\gamma$ a \emph{concrete curve}.
A concrete multi-curve~$\gamma$ is \emph{simple} if $\gamma$ is injective, and a multi-curve is simple if it has a concrete representative
that is simple.
We write $\Curves^+(S)$ for the space of all multi-curves.
Most of the multi-curves we consider will be oriented, as above. There
is a corresponding unoriented notion, which we call an
\emph{unoriented multi-curve}. This is defined as above, except that
the 1-manifold $X(\gamma)$ is taken without orientation, and the
equivalence relation is generated by reparametrizations that are not
required to preserve orientation. In this case we will drop the arrow
from the notation, writing $C$ instead of $\vec{C}$.
\end{definition}

  \begin{notation*}
    From now on, for $a, b\in \pi_1(S)$, we write $[a]$ for the conjugacy class of~$a$, or, equivalently, the free homotopy class of the curve represented by $a$, and write $[a][b]$ to denote the multi-curve $[a] \cup [b]$.
  \end{notation*} 

At times we will allow multi-curves to have weights.

\begin{definition}[weighted multi-curve]\label{def:weighted-multi-curve}
  A \emph{weighted multi-curve} $\vec{C}=\bigcup_i a_i \ora{C_i}$ is a multi-curve in which each connected component is given a non-negative real coefficient~$a_i$. If coefficients are not specified, they are~$1$. We add further moves to the equivalence relation:
  \begin{itemize}
  \item \emph{merging} two parallel components and adding their weights; and
  \item \emph{nullifying}, deleting a
    component with weight~$0$.
  \end{itemize}
  For instance, $\vec{C} \cup \vec{C}$ is equivalent to $2\vec{C}$.
\end{definition}

Oriented multi-curves are more natural from a group-theoretic
perspective, appearing in
the study of group actions on real trees
(Section~\ref{sec:skora}) and
of cross-ratios
(Section~\ref{sec:crossratios}).
Moreover, to verify smoothing for curve functionals that do not arise from natural length metrics on the surface, it is convenient to work in terms of the relative position of pairs of hyperbolic elements in $\pi_1(S)$, where orientation is essential (see Section~\ref{subsec:triality} and compare Theorem~\ref{thm:intersect} with Theorem~\ref{thm:intersections_group}).

The geometric intersection number will play a central role throughout.
The intersection number $i(C,D)$ of two multi-curves is
defined to be the minimal number of intersections among all
representatives. If $\vec{C}$ and $\ora{D}$ are oriented,
the
intersection number is defined by forgetting orientations. See also
Definition~\ref{def:intersection_currents}.

This gives an important example of curve functionals, $i(\mu, \cdot) \colon
\Curves^+(S) \to \mathbb{R}$, for a given oriented multi-curve or,
more generally, oriented geodesic current $\mu$.

As explained in Appendix~\ref{sec:asymmetric}, there is an asymmetric
variant of intersection number for curves which only records
crossings of a given orientation. For such notions, the analogue of
Theorem~\ref{thm:intersect} fails,
and on the other hand curves are not determined by their asymmetric
intersection number functionals,
so we will not use it in the bulk
of this paper.

A curve functional that does not depend
on orientation may equivalently be viewed as a functional on
unoriented multi-curves. Most of our examples will be of this form.

\begin{definition}\label{def:properties}
Let $S$ be a \emph{closed} surface, i.e., an orientable compact topological surface without boundary, and let
$f \colon \Curves^+(S) \to \bbR$ be a function on the space of oriented multi-curves.
We will refer to such an $f$ as a \emph{curve functional}.
(The term ``functional'' indicates that $f$ takes scalar values; it is not assumed to be linear, although in the applications of Theorem~\ref{thm:intersect} it will be additive.)
We introduce several properties that $f$ may satisfy.

\begin{itemize}

\item \textbf{Symmetry:}
If $\vec{C}$ and $\vec{C'}$ are oriented multi-curves with the same underlying unoriented multi-curve, then
\begin{equation}
f(\vec{C})=f(\vec{C'}).
\label{eq:symmetry}
\end{equation}

\item \textbf{Oriented smoothing:}
Let $\vec{C}$ be a multi-curve on $S$, and let $x$ be an essential crossing of $\vec{C}$.
Let $\vec{C'}$ denote the oriented smoothing of $\vec{C}$ at $x$.
Then
\begin{equation}\label{eq:or-smoothing}
f\!\left(\mfig{curves-3}\right)
\;\ge\;
f\!\left(\mfig{curves-4}\right).
\end{equation}
(An essential crossing is one that cannot be removed by
homotopy; see Section~\ref{subsec:crossings}.)

\item \textbf{Additivity:}
For curves $\ora{C_1}$ and $\ora{C_2}$ on $S$,
\begin{equation}\label{eq:union-additive}
f(\ora{C_1} \cup \ora{C_2}) = f(\ora{C_1}) + f(\ora{C_2}).
\end{equation}

\end{itemize}
\end{definition}

The oriented smoothing property was called just ``smoothing'' in our prior work
\cite{MGT21:Smoothings}. In the present paper,
we distinguish this notion from its unoriented counterpart, and
similarly for quasi-smoothing (see below).
Oriented smoothing condition can be further divided into three
types; see
Definition~\ref{def:smoothing_types}.  Accordingly, we will
say $f$ satisfies a particular smoothing property if $f$ is
non-increasing under such a type of smoothing.

\begin{lemma}
\label{lem:oriented-smoothing-alg}
  A curve functional $f \colon \Curves^+(S) \to \mathbb{R}$ satisfies oriented smoothing iff the following three families of inequalities are satisfied.
\begin{itemize}

\item \textbf{Oriented disconnected smoothing:}
For any pair of non-trivial elements $a,b \in \pi_1(S)$ whose hyperbolic axes cross transversely,
\begin{equation}\label{eq:or-disconn-smoothing}
f([a][b]) \ge f([ab]).
\end{equation}
\item \textbf{Oriented connected smoothing:}
For any pair of non-trivial elements $a,b \in \pi_1(S)$ whose hyperbolic axes are parallel
(i.e., disjoint and coherently oriented),
\begin{equation}\label{eq:or-conn-smoothing}
f([ab]) \ge f([a][b]).
\end{equation}
\item \textbf{Power smoothing:}
For any non-trivial element $a \in \pi_1(S)$ and integers $n,m \ge 1$,
\begin{equation}
f([a^{n+m}]) \ge f([a^m][a^n]).
\label{eq:power_smoothing}
\end{equation}
\end{itemize}
\end{lemma}

Lemma~\ref{lem:oriented-smoothing-alg}
is proved in
Section~\ref{subsec:triality}.

\begin{definition}
  For $\vec{C}$ a multi-curve, $\vec{C}^n$ is the multi-curve with as many components as~$\vec{C}$, in which each component of~$\vec{C}$ is covered by an $n$-fold cover. That is, if $\vec{C} =
  [g]$ for $g \in \pi_1(S,x)$, then $\vec{C}^n = [g^n]$.
  On the other hand, $n\vec{C}$ is the multi-curve that consists of $n$~parallel copies of~$\vec{C}$ (so with $n$ times as many components) or, in the context of weighted multi-curves,
  with weights multiplied by~$n$.
\end{definition}

\begin{definition}\label{def:stable}
  Let $f$ be a curve functional. We define some properties related to powers.
  \begin{itemize}
  \item \textbf{Stability:} For all $n \ge 0$ and all multi-curves~$\vec{C}$,
    \begin{equation}
      \label{eq:stable}
      f(\vec{C}^n) = f(n\vec C).
    \end{equation}%
  \item \textbf{Homogeneity:} for all $n \ge 0$ and all multi-curves~$\vec C$,
    \begin{equation}
      \label{eq:homogeneous}
      f(n\vec C) = nf(\vec C).
    \end{equation}
  \end{itemize}
\end{definition}

It is immediate that if $f$ is additive, then it is homogeneous and
stability implies power smoothing.

We emphasize that oriented connected smoothing does not imply oriented
disconnected smoothing, nor vice versa.
For instance, the curve functional $f$ assigning to a curve its number
of connected components satisfies oriented disconnected smoothing but
not oriented connected smoothing. Its negative satisfies oriented
connected smoothing but not oriented disconnected smoothing.
We can also construct additive and
stable functionals exhibiting only one of these two smoothing
properties;
see Example~\ref{ex:linEL}.
There also exist functionals that satisfy stability and both varieties
of smoothing but
fail additivity; see
Example~\ref{examp:el}.

On the other hand, Proposition~\ref{prop:power_disconnected_imply_stable} shows that
smoothing implies stability. 

Theorem~\ref{thm:intersect} considers curve functions on the space of unoriented multi-curves $\Curves(S)$.
When performing a smoothing for an oriented multi-curve, one could
also choose to reglue in the incoherent orientation. In this case,
there are two possible (incoherent) orientations to equip the resulting
curve with. See also Section~\ref{subsec:crossings}.

\begin{definition}
  Let $f$ be an unoriented curve functional, or equivalently a curve functional invariant under symmetry. We define the most important property for this paper.
  \begin{itemize}
  \item \textbf{Smoothing:} If $C$ is a closed curve, and $C'$ either of the two smoothings, then $f(C) \geq f(C')$,
    Schematically, we have
     \[
      f\left(\mfig{curves-0}\right) \ge \max \left\{ f\left(\mfig{curves-1}\right), f\left(\mfig{curves-2}\right)\right\}.
    \]
  \item \textbf{Max-smoothing:} As above, but the inequality is strengthened to an equality.
  \end{itemize}
\label{def:smoothing_and_max}
\end{definition}

Like oriented smoothing, general smoothing and max-smoothing admit
algebraic reformulations.
For this it is convenient to work with oriented curves.
Given a curve functional on unoriented multi-curves
$f \colon \Curves(S) \to \mathbb{R}$, there is a canonical extension to
oriented multi-curves defined by forgetting orientation:
\begin{equation}
\label{eq:oriented_extension}
f^+(\vec C) \coloneqq f(C),
\end{equation}
where $C$ denotes the underlying unoriented multi-curve.
Conversely, if $f \colon \Curves^+(S) \to \mathbb{R}$ is symmetric, then it induces a functional on unoriented multi-curves
\begin{equation}
\label{eq:unoriented_destriction}
\overline{f} \colon \Curves(S) \to \mathbb{R}, 
\qquad
\overline{f}(C) \coloneqq f(\vec C),
\end{equation}
where $\vec C$ is any oriented representative of the unoriented
multi-curve~$C$.

\begin{definition}
\label{def:smoothing_for_oriented}
  We define properties for an \emph{oriented} curve functional $f \colon \Curves^+(S) \to \mathbb{R}$.
\begin{itemize}
\item \textbf{Disconnected smoothing:} For any pair of elements $a,b \in \pi_1(S)$ whose hyperbolic axes intersect, we have 
  \[ f([a] [b]) \geq \max\,\{\,f([ab]), f([aB])\,\}. \]
\item \textbf{Connected smoothing:} For any pair of elements $a,b \in
  \pi_1(S)$ whose hyperbolic axes are parallel, we have \[ f([ab]) \geq \max\,\{\,f([aB]), f([a] [b])\,\}. \]
 \item \textbf{Power smoothing:} For any element $a \in \pi_1(S)$, and any integers $n,m\geq 1$, we have
\begin{equation} f([a^{n+m}]) \geq f([a^{ m}][a^{ n}]).\label{eq:power_smoothing_unoriented}\end{equation}
\end{itemize}
All of the above have corresponding max-smoothing
properties, strengthening the inequality to an equality.
\end{definition}

\begin{lemma}
\label{lem:smoothing-alg}
  Let $f \colon \Curves^+(S)\to \mathbb{R}$ be a symmetric oriented curve functional. Then $f$ satisfies smoothing iff  it satisfies disconnected smoothing, connected smoothing, and power smoothing.
  Likewise, $f$ satisfies max-smoothing iff it satisfies disconnected max-smoothing, connected max-smoothing, and stability.
\end{lemma}

We refer the reader to Sections~\ref{subsec:crossings}
and~\ref{subsec:triality} for precise
definitions and proofs of
Lemma \ref{lem:smoothing-alg}
and~\ref{lem:oriented-smoothing-alg}. See in particular
Proposition~\ref{prop:crossings-triality}.

\subsection{General context and examples}

To put the present work in the broader context of our previous
work~\cite{MGT21:Smoothings}, we recall some weakenings of
the properties above.
\begin{definition} Let $f$ be a curve functional. We define the following properties.
\begin{itemize}
    \item \textbf{Oriented quasi-smoothing:} There is a constant $K\ge 0$ with the following property. Let $\vec C$~be a curve on~$S$, and let $x$~be an essential crossing of~$\vec C$. Let $\vec C'$ be the oriented smoothing of~$\vec C$ at~$x$. Then $f(\vec C) \ge f(\vec C')-K$. Schematically, we have
    \begin{equation}\label{eq:or-quasismoothing}
      f\left(\mfig{curves-3}\right) \ge f\left(\mfig{curves-4}\right) -K.
    \end{equation}
    If $K=0$, we recover the oriented smoothing property, Eq.~\eqref{eq:or-smoothing}.
\item \textbf{Quasi-smoothing:} There is a constant $K \ge 0$ so that, if $C$ is a closed curve and $C'$ is either of the two crossings, we have $f(C) \ge f(C') - K$. Schematically,
\begin{equation}\label{eq:quasismoothing}
    f\!\left(\mfig{curves-0}\right)
    \ge
    \max\!\left\{
      f\!\left(\mfig{curves-1}\right),
      f\!\left(\mfig{curves-2}\right)
    \right\} - K
\end{equation}
\item \textbf{Convex union:} Let $\vec C_1$ and~$\vec C_2$ be two curves on~$S$. Then
    \begin{equation}\label{eq:union-convex}
      f(\vec C_1 \cup \vec C_2) \le f(\vec C_1) + f(\vec C_2).
    \end{equation}

    If $f$ is symmetric, the induced unoriented curve functional $\overline{f}$ on unoriented multicurves (see Eq.~\eqref{eq:unoriented_destriction}) also satisfies convex union. If Eq.~\eqref{eq:union-convex} is strengthened to an equality, we recover additivity (Eq.~\eqref{eq:union-additive}).
\end{itemize}
\end{definition}

Oriented and unoriented quasi-smoothing also have algebraic
formulations parallel to Lemmas \ref{lem:oriented-smoothing-alg}
and~\ref{lem:smoothing-alg}.

With this setup, we can consider a variety of spaces of
functionals on currents, as summarized in
Figure~\ref{fig:curve-functionals} and defined below.
Our main theorems are about the smallest space $\AConvex(S)$, but
we digress to explore the other spaces to better
delineate the scope of our results.

\begin{figure}
  \centering
  \begin{tikzpicture}[x=3.25cm, y=1.25cm]
    \node (AConvex) at (0,0) {$\strut\AConvex(S)$};
    \node (Convex) at (1,1) {$\strut\Convex(S)$};
    \node (AQConvex) at (1,0) {$\strut\AQConvex(S)$};
    \node (AConvex+) at (1,-1) {$\strut\AConvex^+(S)$};
    \node (QConvex) at (2,1) {$\strut\QConvex(S)$};
    \node (Convex+) at (2,0) {$\strut\Convex^+(S)$};
    \node (AQConvex+) at (2,-1) {$\strut\AQConvex^+(S)$};
    \node (QConvex+) at (3,0) {$\strut\QConvex^+(S)$};
    %
%    \node (ACont) at (2,1.75) {$\strut\ACont(S)$};
%    \node (ACont+) at (3,0.75) {$\strut\ACont^+(S)$};
    \node (Cont) at (2.75,2) {$\strut\Cont(S)$};
    \node (Cont+) at (3.75,1) {$\strut\Cont^+(S)$};
    %%
%    \draw [right hook->] (AQConvex) to (ACont);
%    \draw [right hook->] (AQConvex+) to (ACont+);
%    \draw [right hook->] (ACont) to (ACont+);
%    \draw [right hook->] (ACont) to (Cont);
%    \draw [right hook->] (ACont+) to (Cont+);
%    \draw [white,line width=1.5ex] (QConvex) to (Cont);
    \draw [right hook->] (QConvex) to (Cont);
    \draw [right hook->] (QConvex+) to (Cont+);
    \draw [right hook->] (Cont) to (Cont+);
    \draw [right hook->] (AConvex) to (Convex);
    \draw [right hook->] (AConvex) to (AQConvex);
    \draw [right hook->] (AConvex) to (AConvex+);
    %
%    \draw [white,line width=1.5ex] (Convex) to (QConvex);
    \draw [right hook->] (Convex) to (QConvex);
    \draw [right hook->] (AQConvex) to (QConvex);
    \draw [right hook->] (AQConvex) to (AQConvex+);
    \draw [right hook->] (AConvex+) to (AQConvex+);
    \draw [white,line width=1.5ex] (Convex) to (Convex+);
    \draw [right hook->] (Convex) to (Convex+);
    \draw [white,line width=1.5ex] (AConvex+) to (Convex+);
    \draw [right hook->] (AConvex+) to (Convex+);
    %
%    \draw [white,line width=1.5ex] (QConvex) to (QConvex+);
    \draw [right hook->] (QConvex) to (QConvex+);
%    \draw [white,line width=1.5ex] (Convex+) to (QConvex+);
    \draw [right hook->] (Convex+) to (QConvex+);
    \draw [right hook->] (AQConvex+) to (QConvex+);
  \end{tikzpicture}
\caption{The various families of curve functionals and their
  inclusions. For clarity,
  $\ACont(S)$ and $\ACont^+(S)$ are not shown.}
\label{fig:curve-functionals}
\end{figure}

\begin{definition}[Variations on convex curve functionals]
  \label{def:functional-variations}
  The space of \emph{convex curve functionals}, denoted $\Convex(S)$,
  is the set of functionals on unoriented multi-curves satisfying smoothing,
  convex union, stability, and homogeneity. We may equivalently think
  of these as functionals on oriented multi-curves that satisfy the symmetry
  property, convex union, stability, and homogeneity.
  We use the qualifier \emph{convex} because these functions extend
  continuously to the space of geodesic currents and, under a slight
  strengthening of the smoothing hypothesis, their restrictions to
  train-track charts are convex in the corresponding train-track
  coordinates; see Sections~\ref{subsec:multismoothing}
  and~\ref{sec:convexity_charts}.

  For the qualifier $A$, the \emph{additive} functionals, replace the
  condition of convex union~\eqref{eq:union-convex} with
  additivity~\eqref{eq:union-additive}.

  For the qualifier $Q$, the \emph{quasi-convex} functionals, replace
  the condition of smoothing~\eqref{eq:smoothing} with
  quasi-smoothing~\eqref{eq:quasismoothing}.
  We can add a subscript, like $\QConvex_K(S)$, to
  indicate a fixed value of~$K$.

  For the qualifier ${}^+$, the \emph{oriented} functionals, we
  consider instead functions on oriented multi-curves. The smoothing
  property~\eqref{eq:smoothing} is replaced by the weaker oriented
  smoothing property~\eqref{eq:or-smoothing}, and likewise, if relevant,
  quasi-smoothing~\eqref{eq:quasismoothing} is replaced by oriented
  quasi-smoothing~\eqref{eq:or-quasismoothing}.

  Finally, the space $\Cont(S)$ is the much larger space of functions on the
  space of unoriented currents $\Curr(S)$ that are \emph{continuous}
  and \emph{homogeneous}: for $a \ge 0$, we have
  $f(a\mu) = a f(\mu)$.
  Here, the qualifiers $A$ and ${}^+$ make sense:
  $\Cont^+(S)$ is the larger space of homogeneous continuous functions on oriented
  geodesic currents, and $\ACont(S)$ is the space of functions that
  are linear on segments contained in the convex set $\Curr(S)$.
\end{definition}

All the inclusions in Figure~\ref{fig:curve-functionals} are immediate
from the definitions, with the
exception of the inclusions into $\Cont(S)$ and $\Cont^+(S)$, which
was the main result of our previous work.

\begin{theorem}{\cite[Theorem~A]{MGT21:Smoothings}}\label{thm:convex}
   Let $f \in \QConvex^+(S)$. Then there is a unique continuous homogeneous function
  $\bar{f} \colon \Curr^+(S) \to \bbR$ that extends~$f$.
\end{theorem}

\subsection{Examples}
\label{subsec:examples}
The properties of symmetry,
smoothing, additivity, and stability---defining the space
$\AConvex(S)$---are satisfied by every intersection functional
\[
i(\mu,\cdot) \colon \Curves^+(S) \to \mathbb{R},
\]
where $\mu$ is a geodesic current. This was established
in~\cite[Section~4]{MGT21:Smoothings} by first verifying the
properties for multi-curves and then extending to arbitrary currents
using the density of weighted curves together with the continuity and
bilinearity of the intersection pairing. (Smoothing is the only
non-obvious property.)
For completeness, we provide an alternative proof in
Appendix~\ref{app:necessary}, more closely aligned with the arguments
in this paper
involving lifts to the universal cover.

Many natural notions of ``length'' satisfy the conditions of
$\AConvex(S)$ and were previously known
to extend continuously to geodesic currents, and furthermore have dual
geodesic currents. We list several such examples below, along with the
appropriate references. These are all subsumed by
Theorem~\ref{thm:intersect}; we give pointers to the particular
section discussing the examples in more detail.

\begin{enumerate}
   \item Intersection number with a fixed geodesic current
     \cite[Proposition~4.5]{Bonahon86:EndsHyperbolicManifolds}
     (Appendix~\ref{app:necessary}).
   \item \label{item:metric-hyperbolic}\label{item:metric-first}
     Length with respect to a fixed hyperbolic metric
      \cite[Proposition~14]{Bonahon88:GeodesicCurrent}
      (Section~\ref{subsec:liouville}).
    \item More generally, length w.r.t.\ negatively curved Riemannian
      metrics \cite[Th\'eor\`eme~1]{Otal90:SpectreMarqueNegative}
      (Section~\ref{subsec:length-metrics} for this and the next
      several examples).
    \item More generally, length w.r.t.\ non-positively curved Riemannian metrics
\cite[Theorem~A]{CFF92:RigidityNonPosCurvedRiem}.
\item Length w.r.t.\ a negatively curved Riemannian metric with
  conical singularities of angle $\geq 2\pi$
  \cite[Theorem~A]{HP97:RigidityNegCurvedCone}.
\item Length w.r.t.\ a singular Euclidean metric coming from quadratic
  differentials \cite[Lemma~9]{DLR10:DegenerationFlatMetrics}.
\item \label{item:metric-last}
  More generally, length w.r.t.\ a non-positively curved Euclidean
  cone metric of angle $\geq 2\pi$
  \cite[Proposition~3.3]{BL17:RigidityFlat},
  \cite{Con18:MarkedNonpos}. As an important case, see the recent construction of dual currents for $\frac{1}{n}$-translation structures~\cite[Definition~2.2]{PS26:1ntranslation}.
\item \label{item:simple-gen}
  Word length w.r.t.\ simple generating sets for $\pi_1(S)$
  \cite[Theorem~1.2]{Erlandsson:WordLength}
  (Section~\ref{subsec:embedded_graphs}).
\item \label{item:positively-ratioed}
  Lengths associated to Anosov representations induced by positive
  generalized cross-ratios. Specifically,
  Martone--Zhang \cite[Proposition~2.24]{MZ19:PositivelyRatioed} showed that a family of representations equipped with positive H\"older generalized
cross-ratios and termed \emph{positively ratioed} induce curve
functionals, via their \emph{periods}, dual to geodesic currents. They further established that
these functionals satisfy connected smoothing~\cite[Proposition~4.5]{MZ19:PositivelyRatioed}.
Burger--Iozzi--Parreau--Pozzetti later extended their construction to
generalized cross-ratios that need not be continuous
\cite[Theorem~1.6]{BIPP24:PositiveCrossratios}
(Section~\ref{sec:crossratios}).
\end{enumerate}

We have generalizations of many of these. In
Section~\ref{subsec:length-metrics}, we further generalize items
\ref{item:metric-first}--~\ref{item:metric-last} to show that any curve length arising from a length
(pseudo)metric structure on~$S$ induces a dual geodesic current. In
Section~\ref{subsec:embedded_graphs} we show that
item~\ref{item:simple-gen} admits a further generalization to stable lengths associated
to embedded graphs on the surface.
Theorem~\ref{mainthm:crossratio},
proved in Section~\ref{sec:crossratios}, gives an alternate proof
of item~\ref{item:positively-ratioed} using only the periods and not
the cross-ratio values on general boxes.

\subsection{Non-examples}
\label{subsec:non-examples}

In order to illustrate the complexity of the problem,
we now show examples of many natural functions on curves that do
\emph{not} satisfy the $\AConvex$ properties, even though many of them
still extend continuously to geodesic currents. Unless stated
otherwise, we give
the largest class the example is contained in.

\begin{example}[Directed graphs]
  \label{examp:embedded-dir-graph}
  Generalizing item~\ref{item:simple-gen} in
  Section~\ref{subsec:examples}, we can consider length with
  respect to embedded \emph{directed} graphs, giving examples in
  $\bm{\AConvex^+(S)}$. See Appendix~\ref{sec:asymmetric}, expanding on
  \cite[Section~4.6]{MGT21:Smoothings}.
\end{example}

\begin{example}[Non-embedded graphs]
  \label{examp:non-emb-graph}
  Continuing from the previous example, we can consider length of
  curves with respect to a graph that need not be embedded.
  An important special case is when the graph has a single vertex, in
  which case we are looking at word length with respect to an
  arbitrary (not simple) generating set of $\pi_1(S)$. Such a word length need
  not satisfy smoothing or stability
  (see~\cite[Example~4.10]{MGT21:Smoothings},
  \cite[Theorem~1.5]{EPS20:CountingCurves},
  \cite{Erlandsson:WordLength}). The associated stable length
  satisfies quasi-smoothing but
  not necessarily smoothing
  (see~\cite[Example~4.10]{MGT21:Smoothings}
  and~\cite[Theorem~B]{MGT21:Smoothings}), thus giving examples in
  $\bm{\AQConvex(S)}$ or $\bm{\AQConvex^+(S)}$ depending on
  whether the graph is
  undirected or directed, or correspondingly whether a generating
  set for $\pi_1(S)$ is symmetric
  under inverse or not.
\end{example}

\begin{example}[Extremal length]
  \label{examp:el}
  The square root of extremal length w.r.t.\ a fixed conformal
  structure on~$S$ satisfies symmetry, smoothing, stability, convex
  union, homogeneity, but not additivity
  \cite[Section~4.8]{MGT21:Smoothings} and therefore provides an example in
  $\bm{\Convex(S)}$.
\end{example}

\begin{example}[Extremal length for graphs]
  \label{examp:el-graph}
  Continuing from the previous examples, there is a version of
  extremal length with respect to a graph~$\cG$ generating $\pi_1(S)$;
  see \cite[Section~4.9]{MGT21:Smoothings}, following
  \cite{Thurston19:Elastic}. As explained there, the
  resulting functionals $\sqrt{\EL}(C; \cG)$ give examples in
  $\bm{\Convex(S)}$ or
  $\bm{\QConvex(S)}$
  depending on whether $\cG$ is embedded or not \cite[Propositions 4.22
  and~4.23]{MGT21:Smoothings}. (For the case of non-embedded graphs we
  need to stabilize, as in
  Example~\ref{examp:non-emb-graph}.) We can also allow $\cG$ to be a
  directed graph generating $\pi_1(S)$; the same proofs show the stabilization of $\sqrt{\EL}$ is in $\bm{\Convex^+(S)}$ or $\bm{\QConvex^+(S)}$.

  This example can be further generalized to energies $E_p(C; \cG)$, depending on a parameter $1 \le p \le \infty$, with the same properties for $p < \infty$ \cite[Proposition~4.26]{MGT21:Smoothings}.
\end{example}

Examples~\ref{examp:embedded-dir-graph}, \ref{examp:non-emb-graph},
and~\ref{examp:el-graph} give a variety of examples of different types
of functionals:
\[
  \left\{ \begin{matrix} \mathrm{length} \\ \sqrt{\EL} \end{matrix} \right\}
  \text{ with respect to a(n) }
  \left\{ \begin{matrix} \mathrm{embedded} \\ \mathrm{general} \end{matrix} \right\}
  \left\{ \begin{matrix} \mathrm{undirected} \\ \mathrm{directed} \end{matrix} \right\}
  \text{ graph on $S$ generating $\pi_1(S)$}
\]
gives a curve functional in
\[
  \left\{ \begin{matrix} \mathrm{A} \\ \emptyset \end{matrix} \right\}
  \left\{ \begin{matrix} \emptyset \\ \mathrm{Q} \end{matrix} \right\}
  \Convex
  \left\{ \begin{matrix} \emptyset \\ {}^+ \end{matrix} \right\}
  (S).
\]
In particular this shows that all the inclusions in
Figure~\ref{fig:curve-functionals} are proper inclusions, except for
those involving the space $\Cont(S)$ or a variant thereof.
It is easy to see those inclusions are also proper, and we give some
examples below.

\begin{example}[Quasi-Fuchsian representations]
  \label{examp:qf}
A quasi-Fuchsian representation $\rho \colon \pi_1(S) \to
\PSL(2,\mathbb{C})$ is a
conjugate of Fuchsian group by a (properly normalized) quasi-conformal
self\hyp homeomorphism of the 2-sphere $f_{\rho}$
\cite[Definition~6.12.2]{Hub06:Vol1}.
Stable lengths associated to
quasi-Fuchsian representations into $\PSL(2,\bbC)$ satisfy quasi-smoothing
\cite[Proposition~4.11]{MGT21:Smoothings}, and satisfy strict
smoothing if and only if the representation is Fuchsian
(\cite[Theorem~A]{FF22:QF}, \cite[Proposition~1.8]{BR24:Cubulations}).
These examples are in $\bm{\AQConvex_K(S)}$ for some $K > 0$.
\end{example}

Recall that the
\emph{quasi-Fuchsian} constant of $\rho$ is the smallest quasi-conformal
constant for~$f_{\rho}$.
\begin{question}\label{quest:qf-const}
  Let $\rho$ be a
quasi-Fuchsian representation, and let $K>0$ be the \emph{quasi-smoothing
constant}, i.e., the minimal $K>0$ so that $\ell_{\rho} \in \AQConvex_K(S)$. 
Does $K \to \infty$ as the quasi-Fuchsian constant of~$\rho$ goes to~$\infty$?
\end{question}

\begin{example}[Higher rank length functions]
  \label{examp:higher-rank}
Example~\ref{examp:qf} fits into the broader framework of length
functions arising from higher-rank representations of
$\pi_1(S)$ into a connected real semisimple algebraic group~$\mathbb{G}$ of
non-compact type, such as $\SL(n,\mathbb{R})$.
Let $\mathfrak{a}$ be a Cartan subspace with an associated set of simple restricted
roots $\Pi$, and let $\mathfrak{a}^+ \subset \mathfrak{a}$ be the Weyl chamber cone. Given $\Theta\subset\Pi$, let
$\mathfrak{a}_\Theta \coloneqq \bigcap_{\alpha \in \Pi-\Theta} \ker
\alpha \subset\mathfrak{a}$, let $p_\Theta \colon \mathfrak{a} \to
\mathfrak{a}_\Theta$ be the natural projection, and let
$P_\Theta<\mathbb{G}$ be the
parabolic subgroup. Let $\lambda \colon \mathbb{G} \to \mathfrak{a}^+$ denote the vector of generalized eigenvalues, and $\lambda_\Theta \coloneqq p_\Theta \circ \lambda$.
If
$\rho\colon\pi_1(S)\to G$ is $P_\Theta$-Anosov, we define the
\emph{Benoist limit cone} $\mathcal{L}_\Theta^\rho$ as the smallest closed cone containing all the elements $\lambda_\Theta(\rho(\gamma))$ for $\gamma \in \pi_1(S)$.
For every 
linear functional
$\varphi\in\mathfrak{a}^*$ that is positive on $\mathcal{L}_\Theta^\rho$, the function
\begin{equation}\label{eq:higher-rank-funct}
\varphi\circ\lambda_\Theta\circ\rho\colon\pi_1(S)\to\mathbb{R}
\end{equation}
extends continuously to the space of oriented geodesic currents
by \cite[Appendix~A]{JMG26:Horoboundary}
and \cite[Section~7]{CMGR26:GreenMetrics}, giving natural examples of
curve functionals in $\bm{\ACont^+(S)}$ (and maybe in a subspace).
If $\Theta$ is symmetric with respect to the natural Cartan involution, Martone--Zhang
\cite[Theorem~1.1]{MZ19:PositivelyRatioed} show
that Eq.~\eqref{eq:higher-rank-funct}
satisfies full smoothing for $\varphi=\omega_\alpha$, $\alpha \in
\Theta$
(where $\omega_\alpha$ denotes the fundamental weight of $\alpha$) if
$\rho$ is $\Theta$-positively ratioed
(see
item~\ref{item:positively-ratioed} in Section~\ref{subsec:examples}).
Hence, for non-positively-ratioed representations we get examples that
are in $\ACont(S)$ but not in
$\AConvex(S)$.

These non-positively-ratioed examples generalize
Example~\ref{examp:qf} on
quasi-Fuchsian representations, though in the broader context it is
not clear whether or not these functionals are in the
intermediate space
$\AQConvex(S)$.

When $\mathbb{G}$ is the non-semisimple group $\SO(n+1,n) \ltimes \bbR^{2n+1}$, the signed
Lorentzian length associated to a \emph{Margulis space time} (i.e.,
representations $\rho$ where $\rho(\Gamma)$ acts properly on
$\mathbb{R}^{2n+1}$) \cite{GM00:FlatLorentzian} is a natural example
of an asymmetric curve functional satisfying stability and additivity
and extending continuously to geodesic
currents \cite[Section~6.1]{GML09:ProperAffine}\cite{Lab01:Diffusion}.
\end{example}

\begin{example}[Biautomatic structures]
\label{examp:biautomatic}
Automatic structures provide regular-language normal forms for group
elements and yield efficient solutions to the word problem. Many
geometrically defined groups, including hyperbolic groups, admit such
structures. A biautomatic structure satisfies additional
fellow-traveling properties that ensure compatibility with both left
and right multiplication and imply, in particular, solvability of the
conjugacy problem (see~\cite[Section~2.2]{HV24:Biautomaticity}).

Given a biautomatic structure $\mathcal{L}$ on a surface group
$\pi_1(S)$, one obtains an associated translation length
$\norm{\cdot}_{\mathcal{L}}$ defined via its regular language.
Hughes--Valiunas~\cite[Proposition~4.2]{HV24:Biautomaticity} showed that
these translation lengths satisfy oriented quasi-smoothing, and used
this property to construct a hierarchically hyperbolic group that is
not biautomatic. 

Since this class includes ordinary word length (cf.\
Example~\ref{examp:non-emb-graph}), such translation lengths do not
satisfy smoothing in general and lie in $\bm{\AQConvex(S)}$.
\end{example}

\begin{example}[Self intersection]
  \label{examp:self-intersect}
The square root of the self-intersection number $\mu \mapsto
\sqrt{i(\mu,\mu)}$ satisfies stability,
homogeneity, and smoothing, but fails to satisfy additivity and even
convex union (see~\cite[Example~1.12]{MGT21:Smoothings}). 
Thus it provides an example of a curve functional lying outside
$\QConvex(S)$, despite extending to geodesic currents by
Bonahon’s work~\cite{Bonahon86:EndsHyperbolicManifolds}, so it lies in
$\bm{\ACont(S)}$.
\end{example}

\begin{example}[Exotic boundary metrics]\label{examp:exotic-lengths}
  Another example of functionals that are only continuous are exotic length functions on surface groups which vanish precisely on simple closed geodesics \cite[Section~7]{CMGR26:GreenMetrics}: these extend to geodesic currents, since they satisfy the bounded\hyp backtracking property (see~\cite[Section~6]{CMGR26:GreenMetrics}). However, they cannot satisfy $K$-quasi-smoothing for any $K>0$: indeed, we can take sequences of simple elements $a_n,b_n \in \pi_1(S)$, with $\ora{a_n}, \ora{b_n}$ crossing, so that $[a_nb_n] \reducesto [a_n][b_n]$ and $\lim_{n\to
    \infty}\ell(a_n b_n) = \infty$. This shows they violate disconnected $K$-quasi-smoothing for every~$K$, giving examples in $\bm{\ACont(S)}$.
\end{example}

\begin{example}[Lengths from quotient graphs]
  \label{examp:quotient-graph}
  Consider a $\pi_1$-surjective map $\phi$ from $S$ to  a graph~$\cG$. For instance, take a pair-of-pants decomposition of~$S$ and let $\cG$ be the corresponding trivalent graph. The map $\phi^* \colon \pi_1(S) \to \pi_1(\cG)$ necessarily has a large kernel. Now consider the length function $\ell_{\cG}(C) = \ell(\phi(C))$, assigning to a curve the length of its image, up to homotopy in~$\cG$. In particular $\ell_\cG$ vanishes on $\ker \phi^*$. Note that $\ell_\cG$ can be equivalently defined as an infimum over the translation distances in the universal cover~\cite[Exercise~8.3.13]{BBI01:MetricGeometry}, and, since the universal cover of $\cG$ is a tree, it follows that $\ell_\cG$ is stable \cite[Section~2.6]{Gromov1981HyperbolicManifolds}. Bonahon used the large kernel to show $\ell_\cG$ is \textbf{not in $\bm{\Cont(S)}$} \cite[Proposition~11]{Bonahon91:CurrentsGroups}. In particular, $\ell_\cG$ does not satisfy quasi-smoothing. We now give a direct proof that quasi-smoothing is violated.
\end{example}

\begin{proposition}
Suppose $\Gamma=\pi_1(S)$ acts on a geodesic space $X$ by isometries, and
let $\ell$ be the stable length of the action. Suppose $\ell$ is
non-trivial and there is an infinite normal subgroup~$N < \Gamma$ so
that $\ell(b)=0$ for all $b \in N$. Then $\ell$ does not satisfy
quasi-smoothing.
\label{prop:graph_collapse}
\end{proposition}
\begin{proof}
  We first recall Bonahon's proof that $\ell$ does not extend continuously to currents. Let $a \in \Gamma$, with $\ell(a)>0$, and fix $b\in N$. Then $\frac{1}{2n} b a^n b a^{-n}$ converges to $a$ in the weak$^*$ topology on the space of unoriented currents, but since $a^n b a^{-n} b\in N$, we have $\ell(a^n b a^{-n} b)=0$. Hence, $\ell$ cannot extend continuously.

  To see explicitly that
$\ell$ does not satisfy quasi-smoothing, find crossing
$\ora{a}$ and $\ora{b}$ as above. This can be achieved, for example, since $\ell$ is conjugation invariant, so we can replace $a$ by a conjugate, such that the corresponding translate of $\vec a$ crosses $\vec b$.
Then Proposition~\ref{prop:join-split-geom},
Eqs.~\eqref{eq:cross-split} and~\eqref{eq:cross-cancel} says that, for
some fixed $s > 0$ and all sufficiently large $n$,
\[
  [b a^{-n} b a^n] \reducesto [b a^{-n+s}][b a^{n-s}] \reducesto
     [a^{2n - 2s}].
\]
Thus if $\ell$ satisfied $K$-quasi-smoothing, we would have
\[
0 = \ell(ba^n ba^{-n}) \geq \ell(a^{2n-s})-2K = (2n-s)\ell(a)-2K,
\]
a contradiction.
\end{proof}

\begin{example}[Additive version of $\sqrt{\EL}$]
  \label{ex:linEL}
Let $\sqrt{EL}$ denote the (square root of) extremal length.
In \cite[Section 4.8]{MGT21:Smoothings} the authors proved that
$\sqrt{EL}$ satisfies stability, homogeneity, smoothing, and convex
union. By  \cite[Theorem~A]{MGT21:Smoothings}, $\sqrt{EL}$ extends to
a continuous function on geodesic currents.
$\sqrt{EL}$ does not satisfy additivity, so it is not dual to a
geodesic current. However, we can consider $\sqrt{EL}$ as a function
on \emph{connected} curves and extend it additively to a function
$G$ on (weighted)
multi-curves,
which is thus not equal to $\sqrt{EL}$ on multi-curves. $G$ is
\textbf{not in $\bm{\ACont(S)}$}, since we know $G=\sqrt{EL}$ on
(weighted) connected curves, which are dense in the space of geodesic
currents.
Thus, the purported continuous extension $\overline{G}$ would have to satisfy $\overline{G}=\sqrt{EL}$ on all geodesic currents, hence on all multi-curves, a contradiction.

Since $G$ is additive and stable, it follows from
\cite[Theorem~A]{MGT21:Smoothings} that $G$ cannot satisfy oriented
smoothing or quasi-smoothing. We go further, and show that $G$
satisfies disconnected smoothing
but not connected smoothing.
Indeed, suppose $(\ora{a}, \ora{b})$ are a crossing pair. Then
\begin{align*}
  G([a][b]) &= G([a]) + G([b]) &&\text{additivity of $G$}\\
   &= \sqrt{\EL}([a]) + \sqrt{EL}([b]) &&\text{$G=\sqrt{\EL}$ on connected curves}\\
   &\ge \sqrt{\EL}([a][b]) &&\text{$\sqrt{\EL}$ satisfies convex union}\\
   &\ge \sqrt{\EL}([ab]) &&\text{$\sqrt{\EL}$ satisfies smoothing}\\
   &= G([ab]) &&\text{$G=\sqrt{\EL}$ on connected curves}.
\end{align*}
Since $G$ satisfies disconnected smoothing, it cannot satisfy
connected quasi-smoothing; i.e., for any $K$ there is some parallel
pair $(\ora{a},\ora{b})$ so that $G([ab]) < G([a]) + G([b]) - K$.
\end{example}

\begin{example}[Asymmetric self-intersection]
In Appendix~\ref{sec:asymmetric} we define a continuous asymmetric intersection
number $\vec\imath(\alpha,\beta)$ for $\alpha,\beta \in \Curr^+(S)$.
The function
\[
\mu \mapsto \sqrt{\vec\imath(\mu,\mu)}
\]
is in $\bm{\Cont^+(S)}$ but not in $\QConvex^+(S)$, by the same reasoning as
in Example~\ref{examp:self-intersect}.
\end{example}

\subsection{Non-negativity}
\label{subsec:nonnegative}

Observe that we do not assume in the hypotheses of
Theorem~\ref{thm:intersect} that the curve
functional is non-negative: in fact it follows
from the other conditions. Here we prove, more generally, that any
functional $f \in \QConvex(S)$ is non-negative, in two slightly
different forms. (We only use the first version in this paper.)

\begin{proposition}
  Let $f$ be a curve functional satisfying convex union and disconnected smoothing. Then $f([a]) \geq 0$ for any $a \in \pi_1(S)$.
  \label{prop:curve-positive}
\end{proposition}
\begin{proof}
Fix $a \in \pi_1(S)$ and let $x \in \pi_1(S)$ be so that $\ora{a}$
crosses $\ora{x}$, so
$[a][x] \reducesto [ax]$ (by oriented disconnected smoothing).
By Lemma~\ref{lem:rhs_prod} below,
$(\ora{ax},\ora{a})$ cross, so
$[ax][a] \reducesto [x]$ (by unoriented disconnected smoothing). We thus have
\begin{align*}
    f([a]) + f([x]) &\geq f([a] [x]) \geq f([ax])\\
     f([ax]) + f([a]) &\geq f([ax] [a]) \geq f([x]).
\end{align*}
Adding the inequalities, we get $f([a]) \geq 0$.
\end{proof}

\begin{proposition}
  Let $f$ be a curve functional satisfying convex union, quasi-smoothing, stability, and homogeneity. Then $f([a]) \geq 0$ for any $a \in
  \pi_1(S)$.
  \label{prop:curve-positive2}
\end{proposition}

\begin{proof}
Let $x \in \pi_1(S)$ be so that $\ora{a}$ crosses $\ora{x}$. As in
Proposition~\ref{prop:curve-positive}, we have
\begin{align*}
[a^n][x] &\reducesto [a^nx]\\
[a^nx][a^n]&\reducesto [x].
\end{align*}
Since $f$ satisfies convex union and $K$-quasi-smoothing (oriented and
unoriented) for
some $K \ge 0$, we get
\begin{align*}
    f([a^n]) + f([x]) &\geq f([a^n] [x]) \geq f([a^nx]) - K\\
     f([a^nx]) + f([a^n]) &\geq f([a^nx] [a^n]) \geq f([x]) - K
\end{align*}
for some $K>0$.
Adding the inequalities, we get $f([a^n]) \geq -2K$. By stability and
homogeneity, this yields $n f([a]) \geq -2K$, and taking $n\to\infty$
we see that $f([a]) \geq 0$.
\end{proof}

\begin{remark}
The use of both the oriented and the unoriented (quasi-)smoothing hypothesis
is essential in both of the above results. In fact, there exist
natural examples of functionals in $\AConvex^+(S)$ (satisfying
oriented but not unoriented smoothing, as well as the other
hypotheses above) that take negative values. For instance, curve
functionals from an element of $H^1(S; \bbZ)$ have this property. More
generally, the much larger space of
homogeneous quasi-morphisms is in $\AQConvex^+(S)$. All of these take negative
values; see
Section~\ref{sub:quasimorphism}.
\end{remark}

%%% Local Variables:
%%% mode: latex
%%% TeX-master: "Intersections"
%%% End:

\section{Background}
\label{sec:backround}

\subsection{General measure theory}
\label{subsec:backround}

To ultimately construct the geodesic current in
Theorem~\ref{thm:intersect}, we review some measure theory, following
Bogachev~\cite{Bo07:MeasureTheory}.
Let $X$ be a set.

\begin{definition}[Semi-ring, ring, $\sigma$-algebra]
  A family $\mathcal{S}$ of subsets of a set~$X$ is called a
  \emph{semiring} (of sets) if it contains the empty set and, for
  every $A,B \in \mathcal{S}$, we have
  $A \cap B \in \mathcal{S}$ and if $A \subset B$ the difference
  $B \setminus A$ is a disjoint union of finitely many sets
  in~$\mathcal{S}$. A family~$\mathcal{R}$ of subsets of a set $X$ is
  called a \emph{ring} if it contains the empty set and for all
  $A,B \in \mathcal{R}$, the sets $A \cap B$, $A \cup B$ and
  $A \setminus B$ belong to~$\mathcal{R}$.  A ring $\mathcal{R}$ is an \emph{algebra} if $X \in \mathcal{R}$. An algebra $\cA$ is called a \emph{$\sigma$-algebra} if for any sequence of sets $A_i$ in $\cA$ one has $\bigcup_{i=1}^{\infty} A_i \in \cA$.
If $\cF$ is a family of subsets of $X$, let $\sigma(\cF)$, the \emph{$\sigma$-algebra generated by $\mathcal{F}$}, be the smallest $\sigma$-algebra containing $\mathcal{F}$.
  \label{def:semiring}
\end{definition}

\begin{lemma}
  For any semiring~$\mathcal{S}$, the collections of all finite unions
  of sets in $\mathcal{S}$ forms a ring $\mathcal{R}_{\mathcal{S}}$,
  the \emph{ring generated by $\mathcal{S}$}. Every set in
  $\mathcal{R}_{\mathcal{S}}$ is a finite union of pairwise disjoint
  sets in~$\mathcal{S}$.
\end{lemma}
\begin{proof}
See \cite[1.2.14~Lemma]{Bo07:MeasureTheory}.
\end{proof}

\begin{definition}

A function $\mu$ defined on a semiring $\cS$ and valued in the
extended reals $[0,\infty]$ is called
\emph{finitely additive}, resp.\ \emph{countably sub-additive}, resp.\
\emph{countably additive}, if
\[
\mu\Bigl(\bigcup\nolimits_{i=1}^n A_i \Bigr) = \sum_{i=1}^n \mu(A_i),\qquad
\text{resp.\ }\mu\Bigl(\bigcup\nolimits_{i=1}^\infty A_i \Bigr) \le \sum_{i=1}^\infty \mu(A_i),\qquad
\text{resp.\ }\mu\Bigl(\bigcup\nolimits_{i=1}^\infty A_i \Bigr) = \sum_{i=1}^\infty \mu(A_i),
\]
for all pairwise
disjoint sets $A_i \in \cS$
such that $\bigcup_{i=1}^{n}
A_i \in \cS$ (resp.\ $\bigcup_{i=1}^{\infty}
A_i \in \cS$).
\end{definition}

\begin{proposition}[{\cite[1.3.9~Proposition]{Bo07:MeasureTheory}}]
Let $\mu$ be a non-negative finitely additive set function on a ring $\cR$.
If $\mu$ is countably sub-additive then $\mu$ is countably additive.
\label{prop:subadd_to_add}
\end{proposition}

\begin{proposition}[{\cite[1.3.10~Proposition]{Bo07:MeasureTheory}}]
\label{prop:addsemitoring}
Let $\mathcal S$ be a semiring and $\mu:\mathcal S\to [0,\infty]$ a finitely additive function. Then there exists a unique finitely additive extension
 $\hat\mu:\mathcal R_{\mathcal S} \to [0,\infty]$. If, in addition, $\mu$ is countably sub-additive on $\mathcal S$, then $\hat\mu$ is countably additive on $\mathcal R_{\mathcal S}$.
\end{proposition}

Finally, a finitely and countably additive function on a ring can be extended to the $\sigma$-algebra. A function $\mu$ on a ring~$\cR$ is called \emph{$\sigma$-finite} if
there exists a sequence of $X_n \in \cR$ with $\mu(X_n)<\infty$ and
$X=\bigcup_{n=1}^{\infty} X_n$.

\begin{theorem}[Carathéodory Extension Theorem, {\cite[1.11.8~Theorem]{Bo07:MeasureTheory}}]
\label{thm:caratheodory}
Given a ring~$\mathcal{R}$ on a set~$X$ and an additive and countably additive non-negative function~$\mu$ on~$\mathcal{R}$.
Then $\mu$ extends to a measure on the $\sigma$-algebra
$\sigma(\mathcal{R})$ generated by $\mathcal{R}$. If, moreover, $\mu$
is $\sigma$-finite, this extension is unique.
\end{theorem}

For any non-negative function $\mu$ defined on a certain class $\mathcal{A}$ of subsets of $X$ containing $X$, the formula
\[
\mu^*(A) = \inf \left\{ \sum_{n=1}^\infty \mu(A_n) \mid A_n \in \mathcal{A}, A \subset \bigcup_{n=1}^{\infty} A_n \right\}
\]
defines a new function on \emph{every} $A \subset X$, called the \emph{outer measure} of $\mu$.

Let $\mu$ be a finitely additive and countably additive function on a
ring $\mathcal{R}$. A set $A \subset X$ is called
\emph{$\mu$-measurable} if for every $E \subset X$ one has
\[
\mu^*(E \cap A) + \mu^*(E \setminus A) = \mu^*(E).
\]

\begin{definition}[{Atom of a measure~\cite[Definition~1.12.7]{Bo07:MeasureTheory}}]
\label{def:atommeasure}
Given a topological space $X$, a nonnegative measure $\mu$ and $A\subset X$ a measurable subset, we say $A$ is an \emph{atom} of $\mu$ if the following two conditions are satisfied:
\begin{enumerate}
    \item $\mu(A) > 0$, and
    \item for any measurable set $B$ properly contained in~$A$, $\mu(B)=0$.
\end{enumerate}
\end{definition}

In the following definitions, let $X$ be a topological space.

\begin{definition}[Borel $\sigma$-algebra, Borel measure]
The \emph{Borel $\sigma$-algebra} is the smallest $\sigma$-algebra containing all open sets in $X$. A \emph{Borel measure} is a measure on $\mathcal{B}(X)$.
\end{definition}

\begin{definition}[Support of a Borel measure]
Let $\mu$ be a Borel measure on $X$. The \emph{support} of $\mu$ is defined as the set of all points $x \in X$ for which every open neighborhood $N_x$ of $x$ has positive measure. We will denote it $\supp(\mu)$.
\label{def:support}
\end{definition}

If $X$ is a separable metric space, and $\mu$ is a Borel measure, then
every atom of $\mu$ is the union of a singleton set and a $\mu$-measurable set of
measure zero~\cite[2.1.6, Theorem~2]{Kad18:MeasureTheory}.

\subsection{Radon and Borel measures}
\label{subsec:RadonvsBorel}

\begin{definition}
Let $X$ be a locally compact Hausdorff space. A Borel measure $\mu$ on $X$ is \emph{locally finite} if every point has a neighborhood $U$ with $\mu(U)< \infty$.
The measure $\mu$ is \emph{inner regular} if, for every open set $U$,
\[
  \mu(U) = \sup\, \{ \mu(K) \mid K \subset U,\ \text{$K$ compact} \}.
\]
We say that $\mu$ is \emph{outer regular} if, for every Borel set $B$,
\[
  \mu(B) = \inf\, \{ \mu(U) \mid B \subset U,\ \text{$U$ open} \}.
\]
The measure $\mu$ is \emph{Radon} if it is Borel, locally finite,
inner regular, and outer regular.
\label{def:radon}
\end{definition}

\begin{definition}
  The \emph{space of measures $\mathcal{M}(X)$} on a topological
  space~$X$ is the set of all Radon measures on~$X$.
\end{definition}

In a general topological space $X$, not all locally-finite Borel
measures are Radon. However, in a locally compact Hausdorff space with
a countable basis, locally-finite Borel measures and Radon measures
are the same~\cite[Theorem~29.12]{Bau01:Measure}.
Furthermore, under these same topological conditions on $X$, locally
finite Borel measures are automatically $\sigma$-finite.
\begin{lemma}
\label{lem:sigma_finite}
If $X$ is locally compact, Hausdorff with a countable basis, any locally finite Borel measure on $X$ is $\sigma$-finite.
\end{lemma}
\begin{proof}
The claim follows from~\cite[Section~411G(e)]{Fre13:SpacesOfMeasures}.
\end{proof}

We define the following topology in the space of measures.

\begin{definition} We endow $\cM(X)$ with the weak$^*$-topology, the
  weakest topology with respect to which for all $\psi\in C_c(X)$ the map
\begin{align*}
  \mathcal{M}(X) &\to \mathbb{R}\\
  \mu &\mapsto \int \psi \mu
\end{align*}
is continuous.
\end{definition}

\subsection{Geodesic currents}
\label{sec:currents}
We recall the definition of geodesic currents. For a more complete
treatment, we refer the reader
to~\cite[Chapter~3]{ES22:GeodesicCount}.

\begin{definition}[Boundary at infinity]
\label{def:boundaryinfty}
Endow $S$ with a complete
hyperbolic  metric~$g$; we denote the pair $(S,g)$ by~$\Sigma$. Then
we can consider the metric
universal covering $p \colon \wt{\Sigma} \to \Sigma$, with
$\wt{\Sigma}$ isometric to the hyperbolic plane.
Two quasi-geodesic rays $c,c' \colon [0,\infty) \to \wt{\Sigma}$ are
said to be \emph{asymptotic} if the Hausdorff distance between their images is finite.
We define
$\partial_{\infty} S$, the \emph{boundary at infinity}
of~$S$, to be the set of equivalence classes of
asymptotic quasi-geodesic rays. This boundary at infinity, which is homeomorphic to $S^1$, is
independent of the hyperbolic structure on~$S$ up to canonical
homeomorphism.
\end{definition}

\begin{definition}[Space of oriented geodesics]

\label{def:spacegeodesics}
Let $G(\Sigma)$ denote the \emph{space of oriented geodesics} in~$\wt{\Sigma}$, i.e.,
\[
G(\Sigma) \coloneqq\bdy_\infty \Sigma  \times\partial_\infty\Sigma - \Delta.
\]
Since this is independent of the hyperbolic structure, we will also
write $G(S)$.
\end{definition}

There is a natural involution map $\sigma \colon G(S) \to G(S)$, defined by $(a,b) \mapsto (b,a)$, i.e., reversing the orientation of the geodesic. We will sometimes refer to $\sigma$ as the \emph{flip map}.

\begin{definition}\label{def:currents}
We define $\Curr^+(S)$, the space of \emph{oriented geodesic currents} on $S$,
to be the
space of $\pi_1(S)$-invariant, Radon measures on
$G(S)$.
We define the space of 
\emph{geodesic currents} to be the space of flip-invariant oriented
geodesic currents on $S$, and denote it by $\Curr(S)$. We endow both
spaces with their natural subspace topologies as subsets of
$\mathcal{M}(G(S))$.
\end{definition}

Any flip-invariant oriented geodesic current is determined by its
geometric intersection number with all closed
curves~\cite[Th\'eor\`eme~2]{Otal90:SpectreMarqueNegative}, a fact
which fails to be true for general oriented geodesic currents
(Appendix~\ref{sec:asymmetric}).

As a consequence of the discussion in
Section~\ref{subsec:RadonvsBorel}, we can replace ``Radon'' with
``locally-finite Borel'' in the above definitions.
The space of geodesic currents enjoys the following nice topological properties.

\begin{proposition}[{\cite[Theorem~3.3]{ES22:GeodesicCount}}]
$\Curr^+(S)$ and $\Curr(S)$  are locally compact, metrizable, second-countable, and Hausdorff spaces. In particular, $\mathbb{P}\Curr(S)$ is compact.
\label{prop:local_compactness}
\end{proposition}

The $\pi_1(S)$-invariance of geodesic currents highly constrains their atoms.
\begin{proposition}[{\cite[Lemma~8.2.8]{Martelli16:IntroGeoTop}}]
  For $l \in G(S)$, the orbit $\pi_1(S)\cdot l$ is
  discrete in $G(S)$ if and only if its projection $p(l)$ to~$S$ is a closed
  geodesic.
\end{proposition}

For $g \in \pi_1(S)$, let $\ora{g}$ denote the \emph{hyperbolic axis}
of $g$, i.e., the oriented geodesic that is invariant by $g$, acting
on the left, along which
$g$ translates the least positive hyperbolic distance. (See also
Section~\ref{sec:actions}.)
The orbit $\pi_1(S)\cdot \ora{g}$ is discrete and countable, and does
not depend on $g$ within its conjugacy class. One can therefore define
the geodesic current naturally associated to the oriented primitive closed curve
$\vec C=[\gamma]=[g]$ by
\begin{equation}\label{eq:delta-curve}
\delta_{\vec C}
= \sum_{h \in \pi_1(S)/\operatorname{Stab}(\ora{\gamma})}
\delta_{h\cdot \ora{\gamma}},
\end{equation}
where $\operatorname{Stab}(\ora{g}) = \langle g \rangle$ is the
stabilizer and $\delta_l$ is the Dirac measure supported on
the oriented geodesic $l\in G(S)$. For non-primitive curves $\vec C=[h^n]$,
where $[h]$ is primitive, we scale as $\delta_{\vec C} = n \delta_{[h]}$.
For a curve $C$, we define the corresponding flip-invariant geodesic current (recall \eqref{eq:unoriented_flipinvariant}) by
\[
\delta_C \coloneqq \frac12 \left( \delta_{\vec{C}}+ \delta_{\sigma(\vec C)} \right).
\]
These definitions extend naturally to associate a geodesic current to
any weighted curve (oriented or unoriented) and, by linearity, to any
weighted (oriented or unoriented) multi-curve. Moreover, geodesic
currents arising from weighted multi-curves (or from
weighted curves) are dense in the space of geodesic
currents~\cite[Proposition~4.4]{Bonahon86:EndsHyperbolicManifolds}.

It is easy to see that a geodesic $l \in G(S)$ can only be an atom of a geodesic current when it is the lift of a closed geodesic~\cite[Proposition~8.2.7]{Martelli16:IntroGeoTop}.
One can also look at 1-dimensional sets.
For $x \in \partial_{\infty}\Sigma$, the
\emph{pencil} $P(x)$ is the subset of $G(S)$ of geodesics with either
endpoint at~$x$.

\begin{proposition}
  \label{prop:spikes}
  Let $\mu$ be a geodesic current. If, for some~$x$, $\mu(P(x)) \ne
  0$, then $P(x)$ contains a unique lift $\ora{\gamma}$ of a closed
  geodesic and $\mu$ has an atom at $\ora{\gamma}$.
\end{proposition}
\begin{proof}
  For existence of $\ora{\gamma}$, see, e.g.,
  \cite[Lemma~2.6]{dRMG22:Duals}.
For uniqueness, see~\cite[Corollary~4.2.3]{Martelli16:IntroGeoTop}.
\end{proof}

Some authors, such as
Bonahon~\cite{Bonahon86:EndsHyperbolicManifolds,Bonahon88:GeodesicCurrent}, use the space of geodesic currents $\Curr^u(S)$ as
$\pi_1$-invariant
Radon measures on the space of unoriented geodesics of
  $\wt{\Sigma}$, i.e.,
\[
G^u(\Sigma) \coloneqq G(\Sigma)/(\bbZ/2)
\]
where the action of $\bbZ/2$ is by the flip involution~$\sigma$. Let
$q \colon G(\Sigma) \to G^u(\Sigma)$ be the projection map.

\begin{lemma}
\label{lem:unoriented_flipinvariant}
The restriction to  flip-invariant currents of the push-forward of $q$ gives a natural homeomorphism
\begin{equation}
q_* \colon \Curr(S) \to \Curr^u(S).
\label{eq:unoriented_flipinvariant}
\end{equation}
\end{lemma}
\begin{proof}
We claim that the map $q_*$ has inverse
\[
L \colon \Curr^u(S) \to \Curr(S)
\]
defined by sending the unoriented geodesic current $\nu \in \Curr^u(S)$ to the flip-invariant geodesic current $L(\nu)$ defined as follows. For any measurable $A \subset G(\Sigma)$, define
\[
L(\nu)(A) \coloneqq \frac12 \int_{G^u(\Sigma)} \# (q^{-1}(\gamma) \cap A) d\nu(\gamma).
\]

In particular, if $\delta_C^u \in \Curr^u(S)$ denotes the geodesic current representing the unoriented curve $C$, we have, applying this definition,
\[
L(\delta_C^u) = \frac12 \left(\delta_{\vec C} + \delta_{\sigma(\vec C)} \right)
\]
which implies
\begin{equation}
q_* (L(\delta_C^u)) = \delta_C^u
\label{eq:first_inverse}
\end{equation}
as well as
\begin{equation}
L\left(q_* \delta_C \right) =\delta_C.
\label{eq:second_inverse}
\end{equation}
Note that $q_*$ is continuous with respect to the weak$^*$-topology.
It remains to see $L$ is also continuous. Indeed, for every compactly supported continuous function $\varphi \in C_c(G(\Sigma))$, consider
\[
\int \varphi dL(\nu) = \frac12 \int_{G^u(\Sigma)} \left(\sum_{\ell \in q^{-1}(\gamma)} \varphi(\ell)\right)d\nu(\gamma),
\]
and note that $\gamma \mapsto \sum_{\ell \in q^{-1}(\gamma)} \varphi(\ell)=\varphi(\ell) + \varphi(\sigma(\ell))$ is a continuous and compactly supported function on $G^u(\Sigma)$.
Hence, $q_*$ and $L$ are continuous and linear, and since weighted flip-invariant oriented closed curves are dense in $\Curr(S)$ and weighted unoriented closed curves are dense in $\Curr^u(S)$, by Equations~\ref{eq:first_inverse} and~\ref{eq:second_inverse}, it follows that $q_*$ and $L$ are inverses of each other, and $q_*$ is a homeomorphism, as claimed.
\end{proof}
\begin{remark}
  The difference between thinking of (unoriented)
  geodesic currents as flip-invariant measures on oriented geodesics or
  as measures on unoriented geodesics
  is the difference between thinking of them as invariants or
  coinvariants of the $\bbZ/2$ action. In some cases (as above)
  invariants and coinvariants are isomorphic~\cite[Lemma~2.20]{AM10:Coinvariant}. Furthermore, the above statement holds more generally whenever $q$ is an arbitrary finite cover.
\end{remark}

Oriented currents are
more
general and appear naturally in geometric group theory~\cite{Bonahon91:CurrentsGroups,EPS20:CountingCurves,
  CR25:Manhattan, CMGR26:GreenMetrics} as well as
higher Teichm\"uller theory and Anosov
representations~\cite{BridgemanCanaryLabourieSambarino18:SimpleRoots}.

Given $p,q \in \wt{\Sigma}$ we will write $G[p,q)\subset G(\Sigma)$
for the set of geodesics intersecting the geodesic segment
$[p,q) \subset \wt{\Sigma}$ transversally. Note this depends on the
hyperbolic structure on~$\Sigma$.

\begin{definition}[Intersection form of geodesic currents]
\label{def:intersection_currents}
Let $\DG(S) \subset \{ (\gamma_1,\gamma_2) \mid \gamma_1 \cap \gamma_2 = \{ \mbox{pt} \} \}$ be the set of pairs of oriented geodesics of $\wt{S}$ intersecting transversely.
There is a continuous equivariant map $(\gamma_1,\gamma_2) \mapsto
\gamma_1 \cap \gamma_2$ from $\DG(S)$ to $\wt{\Sigma}$,
so the action of $\pi_1(S)$ on $\DG(S)$ is properly discontinuous and free. Let
$\mathcal{F}$ be a fundamental domain of this action. Following
Bonahon~\cite{Bonahon86:EndsHyperbolicManifolds}, for any two oriented
currents $\mu, \nu$ we define
\begin{equation}
i(\mu, \nu) \coloneqq \int_{\mathcal{F}} \mu \times \nu.
\label{eq:intersection}
\end{equation}
\end{definition}
In fact, Bonahon defined intersection number $i^u$ for unoriented
geodesic currents $\Curr^u(S)$, but we have the following relation
\begin{lemma}
For $q \colon G(\Sigma) \to G^u(\Sigma)$ the natural quotient map
and any $\mu_1, \mu_2 \in \Curr(S)$, we have
\[
i(\mu_1,\mu_2) = i^u(q_*(\mu_1), q_*(\mu_2))
\]
\end{lemma}
\begin{proof}
To define his intersection number $i^u$, Bonahon considers pairs of transversely intersecting unoriented geodesics $\mathcal{D}\mathcal{G}^u(S)$.
We have the relation
\[
\mathcal{D}\mathcal{G}(S)=(q\times q)^{-1}(\mathcal{D}\mathcal{G}^u(S)).
\]
Since $q$ commutes with the $\pi_1$ action, there is the same
relationship for fundamental domains.
Hence, for flip-invariant oriented geodesic currents $\mu_1, \mu_2$,
by the change-of-variables principle we have
\[
\int_{\mathcal{F}}d(\mu_1 \otimes \mu_2) = \int_{\mathcal{F}^u}d(q_*(\mu_1) \otimes q_*(\mu_2)) = i^u(q_*(\mu_1),q_*(\mu_2)).\qedhere
\]
\end{proof}
With this definition, intersection number remains unchanged upon flipping the
orientation of any of its factors.
It enjoys all the other properties of Bonahon's intersection: it is continuous and extends the geometric intersection of closed curves.

We will write $i(\mu, C)$ for $i(\mu, \delta_C)$.
In this case, we have
\begin{equation}
i(\mu, \delta_C)=\mu(G[q,y \cdot q))
\label{eq:intersection_curve}
\end{equation}
where $y \in \pi_1(S)$ is an element representing~$C$ and $q$ is an
arbitrary point in
$\ora{y}$.
In the setting of unoriented geodesic currents, Otal
\cite[Th\'eor\`eme~2]{Otal90:SpectreMarqueNegative} shows that two
geodesic currents are completely determined by their intersection
numbers with closed curves.
This is no longer true for general oriented geodesic currents; for example,
any oriented geodesic current $\mu$ and its orientation flip $\sigma(\mu)$ have the same intersection numbers. 
See also Definition~\ref{def:asymmetric} and
Appendix~\ref{sec:asymmetric}.

Finally, we highlight a special subset of geodesic currents that can be regarded as the interior of the cone of geodesic currents.

\begin{definition}
  Let $\mu \in \Curr(S)$. We say $\mu$ is \emph{filling}, and write
  $\mu \in \Curr_{\mathrm{fill}}(S)$, if for every non-zero geodesic
  current $\nu \in \Curr(S)$, we have $i(\mu, \nu)>0$.
\label{def:filling_current}
\end{definition}

For example, $\delta_{C}$ is filling if and only if $C$ is a \emph{filling multi-curve}, in the sense that the connected components of $S\backslash C$ are topological disks.

 \subsection{Convex functions on the integers}
 \label{subsec:convexity}

We will need some results guaranteeing the existence of limits of sequences.
 We recall some basic results on convexity and existence of limits of
 functions on the integers, starting with an old result and then
 generalizing to two variables and to less-standard notions of
 convexity. We also recall Fekete's Lemma.
 
 Let $\lim_{n^{+}}f(n)$ denote $\lim_{n \to +\infty}f(n)$, and
 similarly let $\lim_{n^+,m^+}$ denote the two variable limit
 $\lim_{n,m \to +\infty}$

 We say $g \colon \mathbb{\mathbb{Z}} \to \mathbb{R}$ is \emph{convex}
 if $2g(n) \leq g(n+1) + g(n-1)$ for all $n \in \mathbb{Z}$.
 Similarly, $g \colon \mathbb{N} \to \mathbb{R}$ is convex if
 $2g(n) \leq g(n+1) + g(n-1)$ for all $n \geq 1$. (For us $\bbN$
 includes $0$.)

 \begin{lemma}
 Let $g \colon \bbN \to \bbR$ be a convex and bounded function.
 Then $\lim_{n^+} g(n)$ exists.
 \label{lem:onevarconvexbounded}
 \end{lemma}
 \begin{proof}
   Consider $h(n) \coloneqq g(n+1) - g(n)$. Convexity of $g$ implies
   that $h$ is monotone increasing. Thus if $h(n_0) > 0$,
   then $g(n)$ increases monotonically without bound for $n > n_0$,
   contradicting boundedness. Therefore $h(n) \le 0$ for all $n$, so $g$ is
   monotonically decreasing, so has a limit.
\end{proof}

  \begin{lemma}\label{lem:monotone-double-limit}
  Let $(a_{n,m})_{n,m \in \mathbb{N}}$ be a double sequence that is
  bounded and monotone, in the sense that if $n' \leq n$ and
  $m' \leq m$ then
  $a_{n',m'} \leq a_{n,m}$. Then the following limits
  exist and are equal:
  \[
  \lim_{n^+}\lim_{m^+} a_{n,m} = \lim_{m^+}\lim_{n^+} a_{n,m} =
    \lim_{n^+, m^+} a_{n,m} = \sup_{n,m \ge 0} a_{n,m}.
  \]
  \end{lemma}
This lemma appears in Habil \cite[Thm.\ 4.2]{Hab16:DoubleSeq}, but we
include a proof for reference.

  \begin{proof}
    Let $V = \sup_{n,m} a_{n,m}$. For any $\epsilon > 0$, let
    $A_\epsilon \coloneqq \{(n,m) \mid a_{n,m} > V - \epsilon\}$. By
    definition of $\sup$, this set is non-empty; say $(n_0,m_0) \in
    A_\epsilon$. By monotonicity, $A_\epsilon$ contains all $(n,m)$
    with $n\ge n_0, m \ge m_0$. Since $\epsilon$ is arbitrary, it
    follows that $\lim_{n^+,m^+} a_{n,m} = V$. It also follows that
    for fixed $n_1 \ge n_0$, we have $\lim_{m^+} a_{n_1,m} >
    V-\epsilon$. (This limit
    exists by single-variable monotonicity.) We now see that the
    double limits exist and are equal to~$V$.
  \end{proof}

\begin{definition}
  A function $g \colon \bbN \to \bbR$ is \emph{$R$-convex} if for all
  $r \ge R$ and $n \in \bbN$ with $n \ge r$, we have
  $g(n-r) + g(n+r) \geq 2 g(n)$. Similarly a function
  $g \colon \bbN \times \bbN \to \bbR$ is
  \emph{$(R,S)$-axis-convex} if, for all $n,m$, $g(\cdot, m)$ is
  $R$-convex and $g(n, \cdot)$ is $S$-convex. We say that $g$ is
  \emph{axis-convex} if it is $(1,1)$-axis-convex.
\end{definition}
 
 \begin{lemma}
   Let $g \colon \mathbb{N} \times \mathbb{N} \to \mathbb{R}$ be axis-convex and bounded. Then the following limits exist and are equal:
   \[\lim_{n^+,m^+} g(n,m) = \lim_{n^+} \lim_{m^+} g(n,m)
     = \lim_{m^+} \lim_{n^+} g(n,m) = \inf_{n,m \ge 0} g(n,m).\]
 \end{lemma}
\begin{proof}
  As in Lemma~\ref{lem:onevarconvexbounded}, the function $h_1(n,m) =
  g(n+1,m) - g(n,m)$ is increasing as a function of $n$ by convexity
  and hence non-positive by boundedness of $g$. Similarly for
  $h_2(n,m) = g(n,m+1) - g(n,m)$. Thus $g$ is monotonically decreasing
  and the result follows by Lemma~\ref{lem:monotone-double-limit}
  applied to $-g(n,m)$.
\end{proof}

\begin{lemma}
  Let $g$ be a bounded, $R$-convex function
  $g \colon \bbN \to \bbR$. Then $\lim_{n^+} g(n)$ exists.
 \label{lem:limexistrmidconvex}
\end{lemma}
\begin{proof}
  For $0 \le k < R$, define $g_k(n) \coloneqq g(k + nR)$. Then each
  $g_k$ is convex and bounded, and thus by
  Lemma~\ref{lem:onevarconvexbounded} has a well-defined limit as
  $n \to \infty$. We must show these limits agree for different~$k$.
  By $R$-convexity,
  $g(k + nR) \le \frac{1}{2} \bigl(g(k-1 + (n-1)R) + g(k+1 +
  (n+1)R\bigr)$. Taking the limit, we see that

  $\lim_{n^+} g_k \le \frac{1}{2} (\lim_{n^+} g_{k-1} + \lim_{n^+} g_{k+1})$.
  The limits of the~$g_k$ are therefore a convex function on
  $\bbZ/R\bbZ$ and thus constant.
\end{proof}
 
\begin{lemma}
  \label{lem:limexistrmidconvextwovar}
  Let $g \colon \mathbb{N} \times \mathbb{N} \to \mathbb{R}$ be a
  bounded, $(R,S)$-axis-convex function. Then the
  following limits exist and are equal:
  \[
    \lim_{n^+,m^+} g(n,m) = \lim_{n^+} \lim_{m^+} g(n,m)
    = \lim_{m^+} \lim_{n^+} g(n,m) = \inf_{n,m \ge 0} g(n,m).
  \]
\end{lemma}

\begin{proof}
  As in Lemma~\ref{lem:limexistrmidconvex}, for $0 \le k < R$ and $0
  \le \ell < S$, let $g_{k,\ell}(n,m) \coloneqq g(k+nR, \ell+mS)$.
  Then each $g_{k,\ell}$ is axis-convex and thus has a well-defined limit
  as $n,m \to \infty$ (agreeing with the iterated limits), while the
  same averaging argument shows that the limits don't depend on
  $k,\ell$ and thus the overall limits exist and are equal.
\end{proof}

 \begin{lemma}
 Let $f_m \colon \mathbb{N} \to \mathbb{R}$ be a sequence of convex functions converging to a  function
 $f \colon \mathbb{N} \to \mathbb{R}$. Then $f$ is also convex.
 \label{lem:limitconvex}
 \end{lemma}

 \begin{proof}
   Straightforward.
 \end{proof}
 
Finally, we will use the following classical lemma on sub-additive sequences.
\begin{lemma}[Fekete's Lemma]\label{lem:fekete}
Let $(a_n)$ be a sequence of real numbers such that, for all $m,n\ge 0$,
\[
a_{m+n} \le a_m + a_n 
\qquad \text{(respectively, } a_{m+n} \ge a_m + a_n\text{).}
\]
Then
\[
\lim_{n\to\infty} \frac{a_n}{n}
= \inf_{n} \frac{a_n}{n} \le a_1
\qquad \text{(respectively, } 
\lim_{n\to\infty} \frac{a_n}{n}
= \sup_{n} \frac{a_n}{n} \ge a_1\text{).}
\]
\end{lemma}

\begin{remark}
  We use the term \emph{axis-convex} for the 2-variable functions
  because the condition is not strong enough to deserve the term
  \emph{convex}. For instance, the function $g(n,m) = -mn$ is
  axis-convex by the above definition, but
  $g(n+1,m+1) + g(n-1,m-1) - 2g(n,m) = -2 < 0$.
\end{remark}

%%% Local Variables:
%%% mode: latex
%%% TeX-master: "Intersections"
%%% End:

\section{Crossings and smoothings}
\label{sec:smoothing}

\subsection{Left vs right actions of \texorpdfstring{$\pi_1(S)$}{π1(S)}}
\label{sec:actions}

To fix conventions, we briefly recall elementary facts about actions of $\pi_1(S,\ast)$ on the universal cover~$\wt S$ of a surface (or indeed any connected length space). We
follow the convention that path composition works in the natural
order, where a product $\gamma_1 \cdot \gamma_2$ of paths in~$S$ is
defined if $\gamma_1(1) = \gamma_2(0)$.
Recall that $a \in \pi_1(S,\ast)$ is the homotopy class of a
curve relative to the basepoint $*$, which we will write $a =
(\gamma)$.
(Recall that for us $[\gamma]$ means the unbased homotopy class of
$\gamma$, which is equivalent to the conjugacy class $[a]$.)
 We will follow the
 convention that a capital letter means the inverse of a
 group element.

\begin{definition}[Right action]
For a covering map $p \colon \wt S \to S$ of~$S$ and base-point
$\ast \in S$, we have a \emph{right} action of $\pi_1(S,\ast)$ on
the fiber $p^{-1}(\ast)$, defined by, for $x \in p^{-1}(\ast)$, setting
$x\cdot (\gamma) = \wt \gamma(1)$, where $\wt\gamma$ is the
unique lift of~$\gamma$ with $\wt\gamma(0) = x$.
\end{definition}
We observe
the following.
\begin{itemize}
\item 
  This right action only depends on $\ast \in S$ and not on a lift
  to~$\wt S$, and works for any covering space.%
  \footnote{One succinct statement of the classification theorem for
    covering spaces is that coverings of $S$ are in bijection with
    sets with a right action of $\pi_1(S)$. If we restrict to
    connected covering spaces, they
    correspond to transitive actions of $\pi_1(S)$, which
    in turn correspond to conjugacy classes of subgroups. Thus this
    right action is just the classification of covering spaces.}
\item In the induced length metric on~$\wt S$, the distance between
  $x$ and $x\cdot (\gamma)$ is bounded by the length of~$\gamma$.
\end{itemize}

\begin{definition}[Left action]
When $\wt S$ is the universal cover of~$S$ (or, more generally, any
normal cover of~$S$), there is a \emph{left} action of $\pi_1(S,\ast)$
on all of $\wt S$ by deck transformations. Take
the standard model of $\wt S$, as follows. Lift $\ast \in S$ to a basepoint
$\wt \ast \in \wt S$, which we think about as the constant path. We
write other
points in $\wt S$ as $\wt{\ast} \cdot (\gamma)$, where $\gamma$ is a
path in~$S$ with $\gamma(0) = \ast$, considered up to homotopy fixing
the endpoints. The left action is then
defined by
$(\gamma_1) \cdot (\wt \ast \cdot (\gamma)) \coloneqq
\wt \ast \cdot (\gamma_1 \cdot \gamma)$.
\end{definition}

We observe the following.
\begin{itemize}
\item This left action is an isometry of all of~$\wt S$, and extends
  to the circle at infinity $\bdy_\infty S$.
\item The left action depends on a lift of
  $\ast \in S$ to a basepoint  $\wt \ast \in \wt S$.
\item The distance between $x \in \wt S$ and $(\gamma) \cdot x$
  is unbounded, although since $\wt \ast$ is moved by the
  length of~$\gamma$ we do have a bound on
  the translation length.
\item While the left and right actions are generally unrelated, we do
  have $\wt \ast \cdot (\gamma) = (\gamma) \cdot \wt \ast$.
\end{itemize}

For any non-trivial $g \in \pi_1(S,*)$, there is a corresponding
oriented closed curve $[g]$ on $S$, corresponding to the conjugacy
class of~$g$. We also need to specify different lifts of $[g]$ to
$\wt S$. This is conveniently specified by $g$ itself: $g$ determines
a sequence of points $\{\wt \ast \cdot g^k \mid k \in \bbZ\}$,
which is a
quasi-geodesic in $\wt S$. After choosing a hyperbolic structure
$\Sigma$ on $S$, this quasi-geodesic determines an oriented hyperbolic
geodesic
denoted $\ora{g}$ with endpoints $g^+ = \lim_{k^+} \wt \ast \cdot g^k$ 
and $g^- = \lim_{k^-} \wt \ast \cdot g^k$.
The element $g$ acts on $\wt S$ by a hyperbolic isometry
fixing $\ora{g}$ set-wise, extending continuously to
$\overline{\wt S} \coloneqq \wt S \cup \partial_{\infty} S$ with \emph{north-south (NS) dynamics}: for any point
$x \in \overline{\wt S}$ other
than $g^-$, we have $\lim _{k^+} g^k \cdot x = g^+$.

The other lifts of $[g]$ to $\wt S$ are translates $q \cdot \ora{g}$
of $\ora{g}$ by a
group element~$q$; we can either think of this as the geodesic
$\ora{qgQ}$ passing within a bounded distance of the points $(qg^kQ) \cdot \wt \ast$; 
or we can think of it as the geodesic corresponding to the coset
$q\langle g\rangle$, i.e., passing within a bounded distance of the
points $(qg^k) \cdot \wt \ast$ in the coset. Note these two sequences of
points are within a
bounded distance of each other, since the right action of $\pi_1$
moves points a bounded distance.

We will frequently draw pictures of group elements $g\in\pi_1(S,\ast)$
acting on~$\bbH^2$. In these pictures, there will be two types of
arrows (see, e.g., Figure~\ref{fig:newmodel3}): arrows
connecting two endpoints of $\partial_{\infty}S$, which
correspond to hyperbolic axes $\ora{g}$ of an element
$g \in \pi_1(S)$, the axis of the \emph{left} action; and arrows
representing paths connecting points on these axes that correspond to
the \emph{right} action on the fiber. As explained above, the
relevant lifts of the basepoints form a quasi-geodesic that are a
bounded distance from the axis; the pictures are clearer if we imagine
the basepoints are actually \emph{on} the axis.

\begin{figure}
\centering{\resizebox{55mm}{!}{\Huge{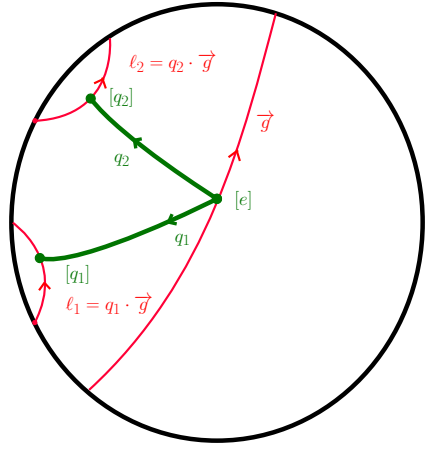}}}
\caption{Two different translates of $\protect\ora{g}$ captured by the double coset $\protect\langle g \protect\rangle q_1^{-1} q_2 \protect\langle g \protect\rangle.$}\label{fig:doublecoset}
\end{figure}

With these preliminaries, we can see that two different lifts of
$[g]$ (see Figure~\ref{fig:doublecoset}) 
\begin{align*}
  \ora{\ell_1} &= q_1 \cdot \ora{g} & \ora{\ell_2} &= q_2 \cdot \ora{g},
\end{align*}
are related by the isometry $q_2 q_1^{-1}$,
\begin{equation}\label{eq:line-translate}
  \ora{\ell_2} = q_2 q_1^{-1} \cdot \ora{\ell_1},
\end{equation}
while the geometric relation between the two lifts, meaning up to the left
$\pi_1(S,*)$ action, is captured instead by the double coset
\begin{equation}\label{eq:double_coset}
  \langle g \rangle q_1^{-1} q_2 \langle g \rangle.
\end{equation}
Note that the order of the product is reversed relative to
\eqref{eq:line-translate}. This double coset is what is relevant for,
e.g., determining if $\ora{\ell_1}$ and $\ora{\ell_2}$ intersect.

We can express many of our constructions as working with either
elements of $\pi_1(S)$ (acting by isometries on the universal cover)
or with general hyperbolic elements of $\Isom^+(\bbH^2)$.
(Every element of $\pi_1(S)$ necessarily acts hyperbolically on~$\wt S$.)
When there is a choice, we prefer to state lemmas in terms
of $\Isom^+(\bbH^2)$: in principle it is more general, as we do
not need to assume the group elements generate a discrete subgroup.

\subsection{Crossings and smoothings}
\label{subsec:crossings}

Loosely speaking, an essential crossing is a crossing of a
multi-curve that can't be homotoped away. Recall that a multi-curve is
considered up to homotopy (Definition~\ref{def:multi-curve}).

\begin{definition}\label{def:essential-crossing}
  A \emph{crossing} of a concrete multi-curve~$\gamma$ is a pair
  $(p,q) \in X(\gamma)^2$ with $p \ne q$ so that
  $\gamma(p) = \gamma(q) \in S$. We say $\gamma$~is \emph{minimal} if
  it has only transverse crossings and a minimum number of
  them among all representatives of $\vec C = [\gamma]$.
  A crossing that appears in some minimal representative is said
  to be an \emph{essential crossing}.
\end{definition}

Given an essential crossing of a multi-curve, we can do an operation
of ``smoothing'' in two different ways, each reducing the original
multi-curve to a different one.
We follow~\cite{MGT25:SmoothingsErratum} closely.

\begin{definition}[smoothings]
  \label{def:smoothing}
  Let $(p,q) \in X(\gamma)^2$ be an essential crossing of~$\gamma$
  on~$S$. To make a \emph{smoothing}~$\gamma'$ of $(p,q)$, cut
  $X(\gamma)$ at $p$ and~$q$ and reglue the resulting four endpoints
  in one of the two other possible ways, getting a new $1$-manifold
  $X(\gamma')$. The map $\gamma'$ agrees with~$\gamma$; this is
  well-defined since $\gamma(p) = \gamma(q)$. In pictures we will
  homotop $\gamma'$ slightly to round out the resulting
  corners, and write $\gamma \reducesto \gamma'$.
\end{definition}

It turns out that ``some'' in the definition of essential crossing can be replaced by ``every'': if there is a crossing $x$ in some minimal representative $\gamma$ of~$\vec C$, then for any other minimal representative $\gamma'$ there is a corresponding crossing $x'$ yielding the same smoothings.

\begin{lemma}[{\cite[Lemma~2]{MGT25:SmoothingsErratum}}]\label{lem:essential-cross}
  If $\gamma$ and $\gamma'$ are homotopic concrete
  multi-curves, with $\gamma$ minimal and $(p,q) \in X(\gamma)^2$ a
  crossing of~$\gamma$ (essential by definition),
  then there is a crossing $(p',q') \in X(\gamma')^2$ so
  that the smoothings of $(p,q)$ and of $(p',q')$ are homotopic.
\end{lemma}

As a result of Lemma~\ref{lem:essential-cross}, we can speak of
essential crossings for concrete curves $\gamma$ that are not minimal:
a crossing is essential if its two smoothings also arise from a
crossing in some minimal diagram.

We distinguish between two types of smoothing.
\begin{definition}\label{def:smoothing_types}
If $\gamma$ is oriented, the \emph{oriented smoothing} is the smoothing that respects the orientation on $X(\gamma)$:
\begin{equation*}
\mfig{curves-3} \reducesto \mfig{curves-4}.
\end{equation*}
One may also perform the smoothing that does not respect the orientation; we refer to this as the \emph{unoriented smoothing}:
\begin{equation*}
\mfig{curves-3} \reducesto \mfigrotate{curves-4}.
\end{equation*}
In this case, the resulting curve does not carry a preferred orientation. Accordingly, unoriented smoothing will only be used in the context of symmetric curve functionals, for which all choices of orientation yield the same value.
  In addition, we call a smoothing \emph{connected} or \emph{disconnected} if the
  original curve is connected or disconnected.

  We thus have four types 
  of smoothings between non-trivial pairs of elements which are not
  common powers. We use group notation that we will expand on and justify
  in the next section.
  \begin{itemize}
  \item \emph{Oriented connected smoothing:} As shown in
    Figure~\ref{fig:sm_oriented_connected},
    parallel oriented geodesics
    $\ora{a}$ and $\ora{b}$ (on the bottom) give a self-crossing oriented
    curve $[ab]$, with $[ab] \reducesto [a][b]$. (The result is
    disconnected.)
  \item \emph{Unoriented connected smoothing:}
    As shown in
    Figure~\ref{fig:sm_unoriented_connected}, parallel oriented geodesics
    $\ora{a}$ and $\ora{b}$ (from the previous case) give a
    self-crossing curve $[ab]$ with
    $[ab] \reducesto [aB]$ (The result is
    connected.)
  \item \emph{Oriented disconnected smoothing:} As shown in
    Figure~\ref{fig:sm_oriented_disconnected}
 crossing geodesics $\ora{a}$ and $\ora{b}$ give a crossing
 multi-curve $[a][b]$, with $[a][b] \reducesto [ab]$. (The result is connected.)
  \item \emph{Unoriented disconnected smoothing:} As shown in
    Figure~\ref{fig:sm_unoriented_disconnected}, crossing geodesics
    $\ora{a}$ and $\ora{b}$ give a crossing multi-curve
    $[a][b]$, with $[a][b] \reducesto [aB]$. (The result is again connected.)
  \end{itemize}
  Finally, there is a special case of oriented connected smoothing,
  self-crossings of a non-primitive element, that will need to be treated
  specially:
  \begin{itemize}
  \item \emph{Power smoothing:}
    $[a^{n+m}] \reducesto [a^m][a^n]$ for every $a \in \pi_1(S)$ and every
    $m,n \geq 1$, as shown in Figure~\ref{fig:sm_power} for $m=1,n=2$.
  \end{itemize}
\end{definition}

\begin{figure}[ht]
\centering
\begin{subfigure}{0.18\textwidth}
\centering
\resizebox{\linewidth}{!}{\fontsize{26pt}{14pt}\selectfont%% Creator: Inkscape 1.3 (0e150ed6c4, 2023-07-21), www.inkscape.org
%% PDF/EPS/PS + LaTeX output extension by Johan Engelen, 2010
%% Accompanies image file 'sm_oriented_connected.pdf' (pdf, eps, ps)
%%
%% To include the image in your LaTeX document, write
%%   \input{<filename>.pdf_tex}
%%  instead of
%%   \includegraphics{<filename>.pdf}
%% To scale the image, write
%%   \def\svgwidth{<desired width>}
%%   \input{<filename>.pdf_tex}
%%  instead of
%%   \includegraphics[width=<desired width>]{<filename>.pdf}
%%
%% Images with a different path to the parent latex file can
%% be accessed with the `import' package (which may need to be
%% installed) using
%%   \usepackage{import}
%% in the preamble, and then including the image with
%%   \import{<path to file>}{<filename>.pdf_tex}
%% Alternatively, one can specify
%%   \graphicspath{{<path to file>/}}
%% 
%% For more information, please see info/svg-inkscape on CTAN:
%%   http://tug.ctan.org/tex-archive/info/svg-inkscape
%%
\begingroup%
  \makeatletter%
  \providecommand\color[2][]{%
    \errmessage{(Inkscape) Color is used for the text in Inkscape, but the package 'color.sty' is not loaded}%
    \renewcommand\color[2][]{}%
  }%
  \providecommand\transparent[1]{%
    \errmessage{(Inkscape) Transparency is used (non-zero) for the text in Inkscape, but the package 'transparent.sty' is not loaded}%
    \renewcommand\transparent[1]{}%
  }%
  \providecommand\rotatebox[2]{#2}%
  \newcommand*\fsize{\dimexpr\f@size pt\relax}%
  \newcommand*\lineheight[1]{\fontsize{\fsize}{#1\fsize}\selectfont}%
  \ifx\svgwidth\undefined%
    \setlength{\unitlength}{254.24450251bp}%
    \ifx\svgscale\undefined%
      \relax%
    \else%
      \setlength{\unitlength}{\unitlength * \real{\svgscale}}%
    \fi%
  \else%
    \setlength{\unitlength}{\svgwidth}%
  \fi%
  \global\let\svgwidth\undefined%
  \global\let\svgscale\undefined%
  \makeatother%
  \begin{picture}(1,1.03979173)%
    \lineheight{1}%
    \setlength\tabcolsep{0pt}%
    \put(0,0){\includegraphics[width=\unitlength,page=1]{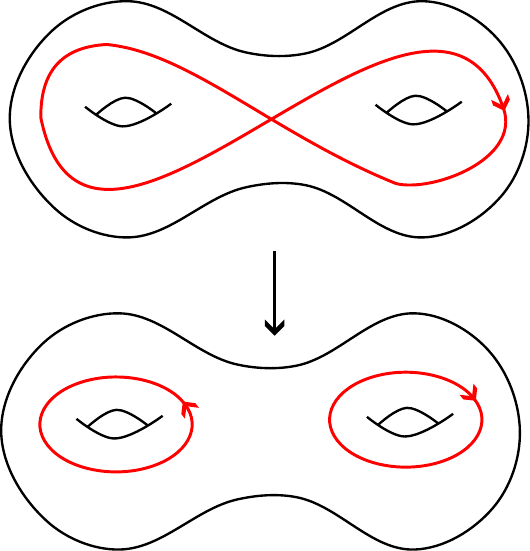}}%
    \put(0.52508517,0.59877903){\color[rgb]{1,0,0}\makebox(0,0)[t]{\lineheight{1.25}\smash{\begin{tabular}[t]{c}$[ab]$\end{tabular}}}}%
    \put(0.22211968,0.06759421){\color[rgb]{1,0,0}\makebox(0,0)[t]{\lineheight{1.25}\smash{\begin{tabular}[t]{c}$[a]$\end{tabular}}}}%
    \put(0.78997858,0.06906935){\color[rgb]{1,0,0}\makebox(0,0)[t]{\lineheight{1.25}\smash{\begin{tabular}[t]{c}$[b]$\end{tabular}}}}%
  \end{picture}%
\endgroup%
}
\caption{Oriented connected smoothing: $[ab] \reducesto [a][b]$}
\label{fig:sm_oriented_connected}
\end{subfigure}
\hfill
\begin{subfigure}{0.18\textwidth}
\centering
\resizebox{\linewidth}{!}{\fontsize{26pt}{14pt}\selectfont%% Creator: Inkscape 1.3 (0e150ed6c4, 2023-07-21), www.inkscape.org
%% PDF/EPS/PS + LaTeX output extension by Johan Engelen, 2010
%% Accompanies image file 'sm_unoriented_connected.pdf' (pdf, eps, ps)
%%
%% To include the image in your LaTeX document, write
%%   \input{<filename>.pdf_tex}
%%  instead of
%%   \includegraphics{<filename>.pdf}
%% To scale the image, write
%%   \def\svgwidth{<desired width>}
%%   \input{<filename>.pdf_tex}
%%  instead of
%%   \includegraphics[width=<desired width>]{<filename>.pdf}
%%
%% Images with a different path to the parent latex file can
%% be accessed with the `import' package (which may need to be
%% installed) using
%%   \usepackage{import}
%% in the preamble, and then including the image with
%%   \import{<path to file>}{<filename>.pdf_tex}
%% Alternatively, one can specify
%%   \graphicspath{{<path to file>/}}
%% 
%% For more information, please see info/svg-inkscape on CTAN:
%%   http://tug.ctan.org/tex-archive/info/svg-inkscape
%%
\begingroup%
  \makeatletter%
  \providecommand\color[2][]{%
    \errmessage{(Inkscape) Color is used for the text in Inkscape, but the package 'color.sty' is not loaded}%
    \renewcommand\color[2][]{}%
  }%
  \providecommand\transparent[1]{%
    \errmessage{(Inkscape) Transparency is used (non-zero) for the text in Inkscape, but the package 'transparent.sty' is not loaded}%
    \renewcommand\transparent[1]{}%
  }%
  \providecommand\rotatebox[2]{#2}%
  \newcommand*\fsize{\dimexpr\f@size pt\relax}%
  \newcommand*\lineheight[1]{\fontsize{\fsize}{#1\fsize}\selectfont}%
  \ifx\svgwidth\undefined%
    \setlength{\unitlength}{254.24450251bp}%
    \ifx\svgscale\undefined%
      \relax%
    \else%
      \setlength{\unitlength}{\unitlength * \real{\svgscale}}%
    \fi%
  \else%
    \setlength{\unitlength}{\svgwidth}%
  \fi%
  \global\let\svgwidth\undefined%
  \global\let\svgscale\undefined%
  \makeatother%
  \begin{picture}(1,1.03979173)%
    \lineheight{1}%
    \setlength\tabcolsep{0pt}%
    \put(0,0){\includegraphics[width=\unitlength,page=1]{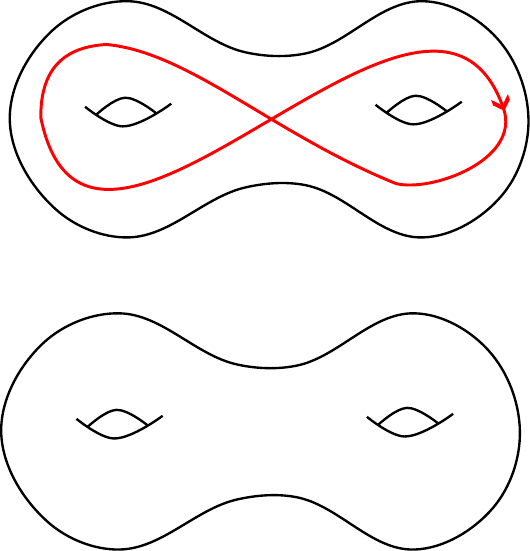}}%
    \put(0.50251992,0.20557551){\color[rgb]{1,0,0}\makebox(0,0)[t]{\lineheight{1.25}\smash{\begin{tabular}[t]{c}$[aB]$\end{tabular}}}}%
    \put(0,0){\includegraphics[width=\unitlength,page=2]{sm_unoriented_connected.pdf}}%
    \put(0.51717522,0.59687245){\color[rgb]{1,0,0}\makebox(0,0)[t]{\lineheight{1.25}\smash{\begin{tabular}[t]{c}$[ab]$\end{tabular}}}}%
  \end{picture}%
\endgroup%
}
\caption{Unoriented connected smoothing: $[ab]\reducesto[aB]$.}
\label{fig:sm_unoriented_connected}
\end{subfigure}
\hfill
\begin{subfigure}{0.18\textwidth}
\centering
\resizebox{\linewidth}{!}{\fontsize{26pt}{14pt}\selectfont%% Creator: Inkscape 1.3 (0e150ed6c4, 2023-07-21), www.inkscape.org
%% PDF/EPS/PS + LaTeX output extension by Johan Engelen, 2010
%% Accompanies image file 'sm_oriented_disconnected.pdf' (pdf, eps, ps)
%%
%% To include the image in your LaTeX document, write
%%   \input{<filename>.pdf_tex}
%%  instead of
%%   \includegraphics{<filename>.pdf}
%% To scale the image, write
%%   \def\svgwidth{<desired width>}
%%   \input{<filename>.pdf_tex}
%%  instead of
%%   \includegraphics[width=<desired width>]{<filename>.pdf}
%%
%% Images with a different path to the parent latex file can
%% be accessed with the `import' package (which may need to be
%% installed) using
%%   \usepackage{import}
%% in the preamble, and then including the image with
%%   \import{<path to file>}{<filename>.pdf_tex}
%% Alternatively, one can specify
%%   \graphicspath{{<path to file>/}}
%% 
%% For more information, please see info/svg-inkscape on CTAN:
%%   http://tug.ctan.org/tex-archive/info/svg-inkscape
%%
\begingroup%
  \makeatletter%
  \providecommand\color[2][]{%
    \errmessage{(Inkscape) Color is used for the text in Inkscape, but the package 'color.sty' is not loaded}%
    \renewcommand\color[2][]{}%
  }%
  \providecommand\transparent[1]{%
    \errmessage{(Inkscape) Transparency is used (non-zero) for the text in Inkscape, but the package 'transparent.sty' is not loaded}%
    \renewcommand\transparent[1]{}%
  }%
  \providecommand\rotatebox[2]{#2}%
  \newcommand*\fsize{\dimexpr\f@size pt\relax}%
  \newcommand*\lineheight[1]{\fontsize{\fsize}{#1\fsize}\selectfont}%
  \ifx\svgwidth\undefined%
    \setlength{\unitlength}{254.24450251bp}%
    \ifx\svgscale\undefined%
      \relax%
    \else%
      \setlength{\unitlength}{\unitlength * \real{\svgscale}}%
    \fi%
  \else%
    \setlength{\unitlength}{\svgwidth}%
  \fi%
  \global\let\svgwidth\undefined%
  \global\let\svgscale\undefined%
  \makeatother%
  \begin{picture}(1,1.03979173)%
    \lineheight{1}%
    \setlength\tabcolsep{0pt}%
    \put(0,0){\includegraphics[width=\unitlength,page=1]{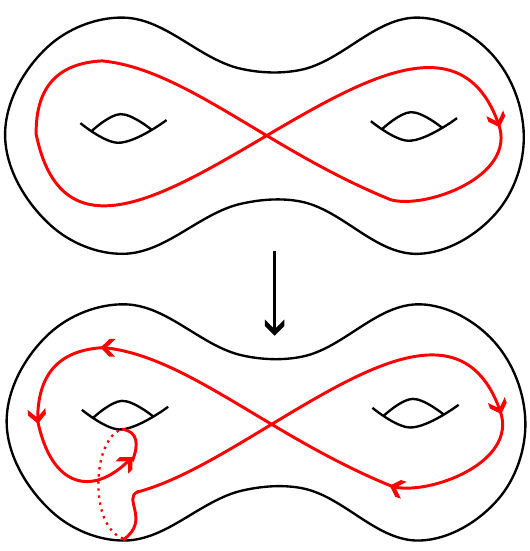}}%
    \put(0.51593594,0.58235469){\color[rgb]{1,0,0}\makebox(0,0)[t]{\lineheight{1.25}\smash{\begin{tabular}[t]{c}$[b]$\end{tabular}}}}%
    \put(0,0){\includegraphics[width=\unitlength,page=2]{sm_oriented_disconnected.pdf}}%
    \put(0.23186059,0.49238222){\color[rgb]{0,0,1}\makebox(0,0)[t]{\lineheight{1.25}\smash{\begin{tabular}[t]{c}$[a]$\end{tabular}}}}%
    \put(0.51890888,0.03662029){\color[rgb]{1,0,0}\makebox(0,0)[t]{\lineheight{1.25}\smash{\begin{tabular}[t]{c}$[ab]$\end{tabular}}}}%
    \put(0,0){\includegraphics[width=\unitlength,page=3]{sm_oriented_disconnected.pdf}}%
  \end{picture}%
\endgroup%
}
\caption{Oriented disconnected smoothing: $[a][b]\reducesto [ab]$}
\label{fig:sm_oriented_disconnected}
\end{subfigure}
\hfill
\begin{subfigure}{0.18\textwidth}
\centering
\resizebox{\linewidth}{!}{\fontsize{26pt}{14pt}\selectfont%% Creator: Inkscape 1.3 (0e150ed6c4, 2023-07-21), www.inkscape.org
%% PDF/EPS/PS + LaTeX output extension by Johan Engelen, 2010
%% Accompanies image file 'sm_unoriented_disconnected.pdf' (pdf, eps, ps)
%%
%% To include the image in your LaTeX document, write
%%   \input{<filename>.pdf_tex}
%%  instead of
%%   \includegraphics{<filename>.pdf}
%% To scale the image, write
%%   \def\svgwidth{<desired width>}
%%   \input{<filename>.pdf_tex}
%%  instead of
%%   \includegraphics[width=<desired width>]{<filename>.pdf}
%%
%% Images with a different path to the parent latex file can
%% be accessed with the `import' package (which may need to be
%% installed) using
%%   \usepackage{import}
%% in the preamble, and then including the image with
%%   \import{<path to file>}{<filename>.pdf_tex}
%% Alternatively, one can specify
%%   \graphicspath{{<path to file>/}}
%% 
%% For more information, please see info/svg-inkscape on CTAN:
%%   http://tug.ctan.org/tex-archive/info/svg-inkscape
%%
\begingroup%
  \makeatletter%
  \providecommand\color[2][]{%
    \errmessage{(Inkscape) Color is used for the text in Inkscape, but the package 'color.sty' is not loaded}%
    \renewcommand\color[2][]{}%
  }%
  \providecommand\transparent[1]{%
    \errmessage{(Inkscape) Transparency is used (non-zero) for the text in Inkscape, but the package 'transparent.sty' is not loaded}%
    \renewcommand\transparent[1]{}%
  }%
  \providecommand\rotatebox[2]{#2}%
  \newcommand*\fsize{\dimexpr\f@size pt\relax}%
  \newcommand*\lineheight[1]{\fontsize{\fsize}{#1\fsize}\selectfont}%
  \ifx\svgwidth\undefined%
    \setlength{\unitlength}{254.24450251bp}%
    \ifx\svgscale\undefined%
      \relax%
    \else%
      \setlength{\unitlength}{\unitlength * \real{\svgscale}}%
    \fi%
  \else%
    \setlength{\unitlength}{\svgwidth}%
  \fi%
  \global\let\svgwidth\undefined%
  \global\let\svgscale\undefined%
  \makeatother%
  \begin{picture}(1,1.03979173)%
    \lineheight{1}%
    \setlength\tabcolsep{0pt}%
    \put(0,0){\includegraphics[width=\unitlength,page=1]{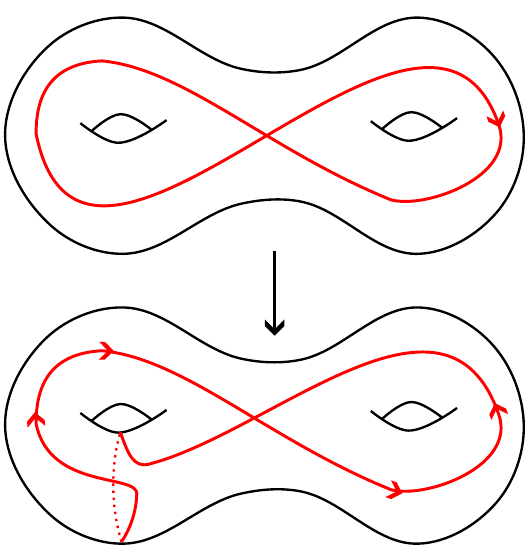}}%
    \put(0.52153619,0.58235469){\color[rgb]{1,0,0}\makebox(0,0)[t]{\lineheight{1.25}\smash{\begin{tabular}[t]{c}$[b]$\end{tabular}}}}%
    \put(0.51446106,0.02629559){\color[rgb]{1,0,0}\makebox(0,0)[t]{\lineheight{1.25}\smash{\begin{tabular}[t]{c}$[aB]$\end{tabular}}}}%
    \put(0.23421895,0.48648247){\color[rgb]{0,0,1}\makebox(0,0)[t]{\lineheight{1.25}\smash{\begin{tabular}[t]{c}$[a]$\end{tabular}}}}%
    \put(0,0){\includegraphics[width=\unitlength,page=2]{sm_unoriented_disconnected.pdf}}%
  \end{picture}%
\endgroup%
}
\caption{Unoriented disconnected smoothing: $[a][b]\reducesto [aB]$}
\label{fig:sm_unoriented_disconnected}
\end{subfigure}
\hfill
\begin{subfigure}{0.18\textwidth}
\centering
\resizebox{\linewidth}{!}{\fontsize{26pt}{14pt}\selectfont%% Creator: Inkscape 1.3 (0e150ed6c4, 2023-07-21), www.inkscape.org
%% PDF/EPS/PS + LaTeX output extension by Johan Engelen, 2010
%% Accompanies image file 'sm_power.pdf' (pdf, eps, ps)
%%
%% To include the image in your LaTeX document, write
%%   \input{<filename>.pdf_tex}
%%  instead of
%%   \includegraphics{<filename>.pdf}
%% To scale the image, write
%%   \def\svgwidth{<desired width>}
%%   \input{<filename>.pdf_tex}
%%  instead of
%%   \includegraphics[width=<desired width>]{<filename>.pdf}
%%
%% Images with a different path to the parent latex file can
%% be accessed with the `import' package (which may need to be
%% installed) using
%%   \usepackage{import}
%% in the preamble, and then including the image with
%%   \import{<path to file>}{<filename>.pdf_tex}
%% Alternatively, one can specify
%%   \graphicspath{{<path to file>/}}
%% 
%% For more information, please see info/svg-inkscape on CTAN:
%%   http://tug.ctan.org/tex-archive/info/svg-inkscape
%%
\begingroup%
  \makeatletter%
  \providecommand\color[2][]{%
    \errmessage{(Inkscape) Color is used for the text in Inkscape, but the package 'color.sty' is not loaded}%
    \renewcommand\color[2][]{}%
  }%
  \providecommand\transparent[1]{%
    \errmessage{(Inkscape) Transparency is used (non-zero) for the text in Inkscape, but the package 'transparent.sty' is not loaded}%
    \renewcommand\transparent[1]{}%
  }%
  \providecommand\rotatebox[2]{#2}%
  \newcommand*\fsize{\dimexpr\f@size pt\relax}%
  \newcommand*\lineheight[1]{\fontsize{\fsize}{#1\fsize}\selectfont}%
  \ifx\svgwidth\undefined%
    \setlength{\unitlength}{254.24450251bp}%
    \ifx\svgscale\undefined%
      \relax%
    \else%
      \setlength{\unitlength}{\unitlength * \real{\svgscale}}%
    \fi%
  \else%
    \setlength{\unitlength}{\svgwidth}%
  \fi%
  \global\let\svgwidth\undefined%
  \global\let\svgscale\undefined%
  \makeatother%
  \begin{picture}(1,1.03979173)%
    \lineheight{1}%
    \setlength\tabcolsep{0pt}%
    \put(0,0){\includegraphics[width=\unitlength,page=1]{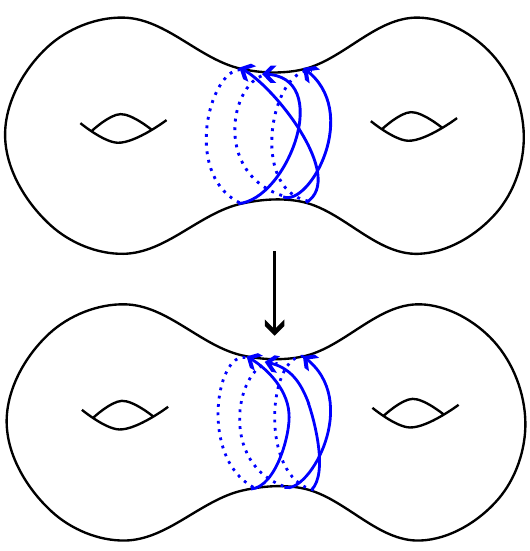}}%
    \put(0.69450304,0.68766594){\color[rgb]{0,0,1}\makebox(0,0)[t]{\lineheight{1.25}\smash{\begin{tabular}[t]{c}$[a^3]$\end{tabular}}}}%
    \put(0.7255798,0.13322098){\color[rgb]{0,0,1}\makebox(0,0)[t]{\lineheight{1.25}\smash{\begin{tabular}[t]{c}$[a^2]$\end{tabular}}}}%
    \put(0.35475661,0.13329516){\color[rgb]{0,0,1}\makebox(0,0)[t]{\lineheight{1.25}\smash{\begin{tabular}[t]{c}$[a]$\end{tabular}}}}%
  \end{picture}%
\endgroup%
}
\caption{Power smoothing: $[a^{m+n}]\reducesto[a^m][a^n]$.}
\label{fig:sm_power}
\end{subfigure}
\caption{Cases of smoothing.}
\label{fig:cases--sm}
\end{figure}

\subsection{Triality of crossings: topological, geometric and algebraic}
\label{subsec:triality}
Following~\cite{MGT25:SmoothingsErratum}, we present a group-theoretic point of view for essential crossings. For $a,b \in \pi_1(S,*)$, say that a
factorization $[ab]$ \emph{reduces} if some (equiv.\ every) minimal
representative of $[ab]$ has a crossing so that one smoothing gives
$[a][b]$ and the other smoothing gives $[aB]$. Similarly say that
$[a][b]$ \emph{reduces} if some/every minimal representative has a crossing
whose smoothings give $[ab]$ and $[aB]$. (This is an abuse of
notation, since the reduction is not invariant under individual
conjugation of $a$ and $b$.)

For $a,b \in \pi_1(S,*)$, let $a * b$ be a 4-valent graph mapped into
$S$ obtained by taking the union of the paths $a,b$ joined at the
basepoint~$*$, and let $[a * b]$
be its free homotopy class. Algebraically, free homotopy
amounts to invariance under simultaneous
conjugation: $[a * b] = [caC * cbC]$. A \emph{minimal} representative of $[a * b]$
is one with a minimum number of transverse crossings. (The 4-valent
vertex is not a crossing.)

\begin{definition}\label{def:crossing}
  Depending on the (counterclockwise) cyclic order of the endpoints,
  we say that a pair $(\ora{a},\ora{b})$ of geodesics are
  \begin{itemize}
  \item \emph{R-parallel} if they appear in order $a^+,a^-,b^-,b^+$
    (below left);
  \item \emph{R-anti-parallel} if they appear in order
    $a^+,a^-,b^+,b^-$ (below middle) and
  \item
    \emph{R-cross} or \emph{$\ora{a}$ crosses $\ora{b}$ to the right} if they appear in
    order $a^+,b^+,a^-,b^-$ (below right).
  \end{itemize}
  \[
    \underset{\text{\normalsize R-parallel}}{\mfig{curves-14}}\qquad\qquad
    \underset{\text{\normalsize R-anti-parallel}}{\mfig{curves-15}}\qquad\qquad
    \underset{\text{\normalsize R-cross}}{\mfig{curves-13}}
  \]
  In each case we say that they are
  L-parallel/L-anti-parallel/L-cross if they appear in the reverse
  cyclic order, and are parallel/anti-parallel/cross if they appear in
  either the above order or its reverse.

  We will furthermore say a triple $(\ora{a},\ora{b},\ora{c})$ is
  R-parallel if both $(\ora{a},\ora{b})$ and
  $(\ora{b},\ora{c})$ are R-parallel, and
  similarly for L-parallel and for larger tuples. The other relations
  are not transitive so we will not use the extended notation.
\end{definition}

\begin{definition}
  Non-trivial group elements $a,b\in \pi_1(S)$ are \emph{common powers}
  if there are non-zero $k,\ell$ so that $a^k = b^\ell$. Since $\pi_1(S)$
  has no torsion, we may assume that $k,\ell$ are relatively prime and
  then we see this is equivalent to the existence of a $c$ so that
  $a = c^\ell$ and $b = c^k$.
\end{definition}

The following justifies the power smoothing
hypothesis~\eqref{eq:power_smoothing_unoriented}.

\begin{proposition}[{\cite[Lemma 1.12]{HS94:ShorteningCurves}}]
  \label{prop:crossings-powers}
  If $a = c^k$ and $b = c^\ell$ are common powers of a primitive
  element~$c$ with $k,\ell \ge 1$, then the factorization $[ab]$
  reduces to $[a][b]$.
\end{proposition}

The following result justifies the oriented/unoriented connected/disconnected smoothing variations of smoothing.

\begin{proposition}[{\cite[Proposition 3]{MGT25:SmoothingsErratum}}]
  \label{prop:crossings-triality}
  Let $a,b \in \pi_1(S,*)$ be non-trivial elements that are not
  common powers. Then in the following three trialities the same
  possibility holds in each case.
  \begin{enumerate}
  \item (Reduction) One of the following reduction happens:\\
    \begin{enumerate*}[itemjoin=\qquad]
    \item $[ab] \searrow [aB],[a][b]$;
    \item $[aB] \searrow [ab],[a][B]$; or
    \item $[a][b] \searrow [ab], [aB]$.
    \end{enumerate*}
  \item\label{item:triality-geom}
    (Geometric) For every minimal embedding of $a * b$, around the
    4-valent vertex we see in cyclic
    order, up to orientation reversal:\\
    \begin{enumerate*}[itemjoin=\qquad]
    \item $\mfig{curves-30}$;
    \item $\mfig{curves-31}$; or
    \item $\mfig{curves-32}$.
    \end{enumerate*}
  \item\label{item:triality-alg} (Algebraic) The geodesics $\ora{a}$ and $\ora{b}$ are:\\
    \begin{enumerate*}[itemjoin=\qquad]
    \item parallel;
    \item anti-parallel; or
    \item cross.
    \end{enumerate*}
  \end{enumerate}
\end{proposition}

We will use Proposition~\ref{prop:crossings-triality} without explicit
reference. In particular if $\ora{a}$ crosses $\ora{b}$, then
\[
  [a][b] \reducesto [ab]\qquad\text{and}\qquad[a][b] \reducesto [aB]
\]
and if $\ora{a}$ is parallel to $\ora{b}$, then
\[
  [ab] \reducesto [a][b]\qquad\text{and}\qquad[ab] \reducesto [aB].
\] 

Now we give the proofs of the algebraic rephrasings of various types
of smoothing we gave in Section~\ref{sec:functionals}.

\begin{proof}[Proof of Lemma~\ref{lem:oriented-smoothing-alg}]
The first two items follow by applying
Equation~\eqref{eq:or-smoothing} to the essential crossing guaranteed
by Proposition~\ref{prop:crossings-triality}.
The last item follows from Proposition~\ref{prop:crossings-powers}.
\end{proof}

\begin{proof}[Proof of Lemma~\ref{lem:smoothing-alg}]
The first two items follow by apply Equation~\eqref{eq:smoothing} and
symmetry of~$f$ to
the essential crossing guaranteed by Proposition~\ref{prop:crossings-triality}.
The last item follows from Proposition~\ref{prop:crossings-powers}.
\end{proof}

If we swap and/or invert $\ora{a}$ and $\ora{b}$, the type of L/R crossing/parallelism
can change, as recorded below for future reference.

\begin{lemma} Let $a,b\in \Isom^+(\bbH^2)$ be hyperbolic elements.
\begin{enumerate}
  \item If $(\ora{a},\ora{b})$ R-cross, then $(\ora{B},\ora{A})$
    R-cross and $(\ora{b},\ora{a})$ L-cross.
  \item If $(\ora{a},\ora{b})$ are R-parallel, then
    $(\ora{B},\ora{A})$ are R-parallel, $(\ora{b},\ora{a})$ are
    L-parallel, and $(\ora{a},\ora{B})$ are
    R-anti-parallel.
  \item If $(\ora{a},\ora{b})$ are R-anti-parallel, then
    $(\ora{b},\ora{a})$ are R-anti-parallel, $(\ora{A},\ora{B})$ are
    L-anti-parallel,
    and $(\ora{a},\ora{B})$
    are R-parallel.
\end{enumerate}
\end{lemma}

\begin{proof}
  Immediate from the definitions.
\end{proof}

\subsection{Parallelism and crossing of hyperbolic axes}
\label{subsec:parallel_crossing}
We now turn to some elementary lemmas about crossing.

\begin{convention}\label{conv:circle-intervals}
  Intervals on the circle are always written with endpoints in counterclockwise
  order. Thus if $x,y \in \partial_{\infty} \Sigma$ are distinct points,
  $\{x\} \cup (x,y) \cup \{y\}\cup (y,x)$ is the entire circle.
\end{convention}

\begin{lemma}\label{lem:NS-dynamics}
  Let $a \in \Isom^+(\bbH^2)$ and $(x,y) \subset \partial\bbH^2$. Then
  the following are equivalent:
  \begin{itemize}
  \item $a$ is hyperbolic with attracting and repelling
    endpoint $a^+ \in(x,y)$ and $a^- \in [y,x]$; and
\item The closed interval $[x,y]$ is mapped strictly into its
interior:
\[
a \cdot [x,y]\subset(x,y).
\]
  \end{itemize}
In this case,
\[
\bigcap_{n\geq 0}a^n([x,y])=\{a^+\},
\qquad
\bigcap_{n\geq 0}a^{-n}([y,x])=\{a^-\}.
\]
If $a$ is already known to be non-parabolic, then the hypothesis can be replaced by
\[
a\cdot(x,y)\subsetneq(x,y).
\]
\end{lemma}
\begin{proof}
  This follows from the classification of isometries of the hyperbolic
  plane and the north-south dynamics of hyperbolic isometries.
\end{proof}

\begin{remark}
If $a$ is already known to be hyperbolic, then
\[
a(x,y)\subsetneq(x,y)
\]
is equivalent to
\[
a^+\in(x,y)
\qquad\text{and}\qquad
a^-\notin(x,y).
\]
This weaker open-interval formulation cannot be used to prove that
$a$ is hyperbolic, since it also permits parabolic isometries.
\end{remark}

\begin{lemma}
  Let $a,x \in \Isom^+(\bbH^2)$ be hyperbolic elements that do not share any endpoints.
  Then $(\ora{a},\ora{x})$ R-cross (resp.\ L-cross) iff
  $(\ora{x}, a \cdot \ora{x})$ are R-parallel (resp.\ L-parallel).
 \label{lem:rhs}
\end{lemma}
\begin{proof}
  If $(\ora{x}, a\cdot\ora{x})$ are R-parallel, then (by
  definition) $a\cdot(x^-,x^+) \subsetneq (x^-,x^+)$,
  which is equivalent, by Lemma~\ref{lem:NS-dynamics}, to $a^+ \in (x^-,x^+)$ and
  $a^- \in (x^+,x^-)$, which is equivalent to
  $(\ora{a},\ora{x})$ R-crossing.
  For the reverse implication, if $(\ora{a},\ora{x})$ R-cross, then by
  NS dynamics of $a$ we have $a \cdot x^+ \in (x^+,a^+)$ and
  $a \cdot x^- \in (a^+,x^-)$, which implies that
  $(\ora{x},a \cdot \ora{x})$ are R-parallel.
\end{proof}

\begin{lemma}
  Fix $a,b,x \in \Isom^+(\bbH^2)$ hyperbolic elements with distinct
  endpoints so that $(\ora{a},\ora{x})$ and $(\ora{b},\ora{x})$
  R-cross. Then:
  \begin{enumerate}[label=(\Roman*)]
  \item All four pairs $(\ora{xa},\ora{x})$, $(\ora{Xa},\ora{x})$, $(\ora{ax},\ora{x})$, and
    $(\ora{aX},\ora{x})$ R-cross.
  \item $(\ora{ab},\ora{x})$ R-cross.
  \item More generally, for any $n,k \in \bbZ$, $(\ora{a x^n},\ora{x})$
    and $(\ora{ax^nbx^k},\ora{x})$ R-cross.
  \end{enumerate}
    \label{lem:rhs_prod}
\end{lemma}
\begin{proof}
  For the first claim, consider the action of the group
  elements on $(x^-,x^+)$: we have
  $x^{\pm1} \cdot (x^-,x^+) = (x^-,x^+)$ and
  $a \cdot (x^-,x^+) \subsetneq (x^-,x^+)$, so all four of $ax^{\pm1}$ and
  $x^{\pm1}a$ map this interval strictly inside itself. By Lemma~\ref{lem:NS-dynamics},
  all these elements are hyperbolic and R-cross~$\ora{x}$.

  For the second claim, $a$ and $b$ both map
  $(x^-,x^+)$ strictly inside itself, so their product does too.

  The last claim is an immediate consequence of the first two.
\end{proof}

\begin{lemma} 
  Let $a,b\in \Isom^+(\bbH^2)$ be hyperbolic elements. Then:
  \begin{enumerate}[label=(\Roman*)]
  \item\label{item:ab-cross} If $(\ora{a}, \ora{b})$ R-cross, then $ba$ and $bA$ are also
    hyperbolic, and the pairs

    $(\ora{ba},\ora{ab})$ and $(\ora{aB},\ora{Ba})$ are both
    R-parallel, with endpoints as in Figure~\ref{fig:cross}.
  \item\label{item:ab-parallel} If $(\ora{a},\ora{b})$ are R-parallel, then $ab$ is hyperbolic
    and $(\ora{ba},\ora{ab})$ R-cross, with endpoints as in
    Figure~\ref{fig:parallel}.
    In particular, the triples
    $(\ora{a},\ora{ab}, \ora{b})$ and $(\ora{a},\ora{ba},\ora{b})$ are R-parallel.
  \item If $(\ora{a},\ora{b})$ are R-anti-parallel, then $aB$ is
    hyperbolic and $(\ora{Ba},\ora{aB})$ R-cross, with endpoints as in Figure~\ref{fig:anti}.
  \end{enumerate}
  \label{lem:righthandedfoundation}
\end{lemma}

\begin{proof}
  Note that the actions of $a$ and $b$ on geodesics act by conjugation
  on the underlying word. Specifically, this means that we have
  \[
    \begin{tikzpicture}[x=1.5cm]
      \node (ba) at (0,0) {$\ora{ba}$};
      \node (ab) at (1,0) {$\ora{ab}$};
      \node (Ba) at (3,0) {$\ora{Ba}$};
      \node (aB) at (4,0) {$\ora{aB}$.};
      \draw [->,bend left=15] (ba) to node[above,cdlabel]{a} (ab);
      \draw [->,bend left=15] (ab) to node[below,cdlabel]{b} (ba);
      \draw [->,bend left=15] (Ba) to node[above,cdlabel]{a} (aB);
      \draw [->,bend right=15] (Ba) to node[below,cdlabel]{b} (aB);
    \end{tikzpicture}
  \]
  The hyperbolic elements $a$ and~$b$ have north-south dynamics. Using
  this in the three cases depending on the relation of $\ora{a}$ and
  $\ora{b}$, we have the
  following.
  \begin{enumerate}[label=(\Roman*)]
  \item If $(\ora{a},\ora{b})$ R-cross, then we claim that the
    endpoints of the various geodesics must be laid out as in
    Figure~\ref{fig:cross}.

\begin{figure}
    \begin{subfigure}[b]{0.3\textwidth}
 \resizebox{\linewidth}{!}{\fontsize{9pt}{9pt}\selectfont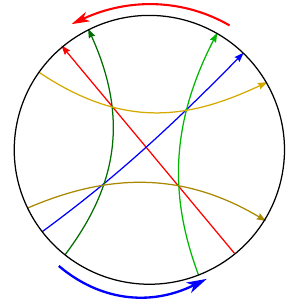}
    \caption{$\protect\ora{a}$ and $\protect\ora{b}$ R-cross}
    \label{fig:cross}
     \end{subfigure}
     \hfill
     \begin{subfigure}[b]{0.3\textwidth}
 \resizebox{\linewidth}{!}{\fontsize{9pt}{9pt}\selectfont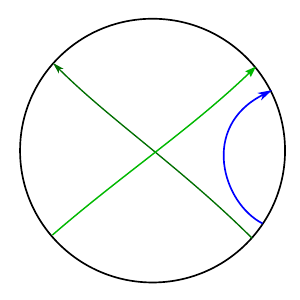}
    \caption{$\protect\ora{a}$ and $\protect\ora{b}$ are R-parallel }
    \label{fig:parallel}
\end{subfigure} 
\hfill
\begin{subfigure}[b]{0.3\textwidth}
 \resizebox{\linewidth}{!}{\fontsize{9pt}{9pt}\selectfont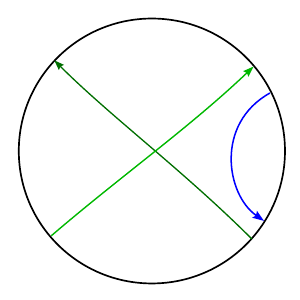}
    \caption{$\protect\ora{a}$ and $\protect\ora{b}$ are R-anti-parallel }
    \label{fig:anti}
   \end{subfigure}
   \caption{Cases for relative position 
     of $\protect\ora{a}$ with $\protect\ora{b}$, and corresponding
     endpoints of products.  On $\bdy_\infty S$, blue arrows indicate
     the action of~$a$
     and red arrows indicate the action of~$b$.}
   \label{fig:cases-a-b}
 \end{figure}
    To see this, by the dynamics of $a$ and $b$ on the
    circle, $ba$ maps the interval
    $(a^+,b^+)$ strictly inside itself, so $ba$ is
    hyperbolic with $(ba)^+\in(a^+,b^+)$;
    similarly $(ba)^- \in (a^-,b^-)$. The actions of
    $b$ and $a$ interchange $(ba)^- \leftrightarrow (ab)^-$ and
    $(ba)^+ \leftrightarrow (ab)^+$, forcing the configuration shown.
    (Note that all four endpoints $(ba)^\pm,(ab)^\pm$ are on the
    portions of the circle
    where $b$ and $a$ move in opposite directions.)

    Since $(\ora{B},\ora{a})$ also R-cross, the
    other conclusion follows.
  \item If $(\ora{a},\ora{b})$ are R-parallel, then $ab \cdot (b^+,a^+) \subsetneq (b^+,a^+)$, so $ab$ is
    hyperbolic and $(ab)^+ \in (b^+,a^+)$. The same interchange
    property then forces
    $(ab)^+, (ba)^+ \in (b^+,a^+)$, as well as
    $(ab)^-, (ba)^- \in (a^-,b^-)$. From this we see that
    $(\ora{ba},\ora{ab})$ R-cross, as in Figure~\ref{fig:parallel}.
  \item If $\ora{a}$ and $\ora{b}$ are anti-parallel, we can reduce to
    the parallel case by replacing $b$ with $B$, arranged as in
    Figure~\ref{fig:anti}.\qedhere
  \end{enumerate}
\end{proof}

\subsection{Position of axes 
\texorpdfstring{$\protect\ora{aB}$}{aB} and \texorpdfstring{$\protect\ora{Ba}$}{Ba} in parallel  case}
\label{subsec:antiparallel}
If $\ora{a}$ and $\ora{b}$ are R-parallel,
Lemma~\ref{lem:righthandedfoundation} makes no claim about the
relationship of $\ora{aB}$ and $\ora{Ba}$. Indeed this is not
determined by the hypotheses: they can be crossing or anti-parallel
(Lemma~\ref{lem:aB_bA_antiparallel_cross}).
Nevertheless, we can extract some finer information about composite
pairs, useful in subsidiary results.

\begin{convention}\label{conv:aB-hyperbolic}
  For this subsection, $a,b \in \Isom^+(\bbH^2)$ are elements so that
  $\ora{a}$ and $\ora{b}$ are R-parallel. We know that $ab$ and $ba$
  are hyperbolic (with crossing axes). Suppose also that $aB$ is
  hyperbolic. (This will be the case, for example, if
  $a,b \in \pi_1(S)$.)
\end{convention}

\begin{lemma}
If $(\ora{a},\ora{b})$ are parallel and $aB$ is hyperbolic, $\ora{aB}$
does not cross $\ora{a}$ or $\ora{b}$.
\label{lem:aB_a_nocross}
\end{lemma}
\begin{proof}
If $\ora{aB}$ crossed $\ora{a}$, then by Lemma~\ref{lem:rhs_prod}, $\ora{(bA)a}=\ora{b}$ would
cross $\ora{a}$, a contradiction. Similarly $\ora{aB}$ cannot cross $\ora{b}$.
\end{proof}

\begin{lemma}
If $(\ora{a},\ora{b})$ are parallel and $aB$ is hyperbolic, for $g \in \{ a, b \}$, the
endpoints of $\ora{aB}$ and the endpoints of $\ora{Ba}$ are in the
same connected component of $\partial_\infty S - \{ g^+,g^-\}$.
\label{lem:aB_Ba_together}
\end{lemma}
\begin{proof}
  By the previous lemma, the endpoints of $\ora{Ba}$ are in one
  connected component.
  Since $a \cdot \ora{Ba} = b \cdot \ora{Ba} = \ora{aB}$, by NS dynamics of $a$ and $b$,
the endpoints of $\ora{aB}$ are in the same connected component as the
endpoints of $\ora{Ba}$.
\end{proof}

In light of the above lemmas, if $(\ora{a},\ora{b})$ are R-parallel,
  we say that $\ora{aB}$ \emph{splits}
$\ora{a}, \ora{b}$ when
\[
\bigl\{(aB)^-,(aB)^+,(Ba)^-,(Ba)^+\bigr\} \subset (b^+,a^+) \cup (a^-,b^-),
\]
as in Figure~\ref{fig:r3}.
Similarly, we say that $\ora{aB}$ is \emph{on the $b$-side} of
$\ora{a},\ora{b}$ when
\[
\bigl\{(aB)^-,(aB)^+,(Ba)^-,(Ba)^+\bigr\} \subset (b^-,b^+),
\]
as in Figure~\ref{fig:r1} or Figure~\ref{fig:c2}.
Similarly $\ora{aB}$ is \emph{on the $a$-side} of $\ora{a},\ora{b}$ if
all endpoints are in $(a^+,a^-)$. When $\ora{a},\ora{b}$ are L-parallel
the terminology is adjusted accordingly. We will also use this
terminology for powers $\ora{a^n B^m}$ and $\ora{B^m a^n}$ with
$n,m \ge 1$; these will
have the same division of cases since $\ora{a^n}$ and $\ora{a}$ have
the same axis.

\begin{figure}
  \centering
  \setlength{\abovecaptionskip}{3pt}
  \setlength{\belowcaptionskip}{0pt}

  \begin{subfigure}[t]{0.42\textwidth}
    \centering
    \resizebox{\linewidth}{!}{\fontsize{9pt}{9pt}\selectfont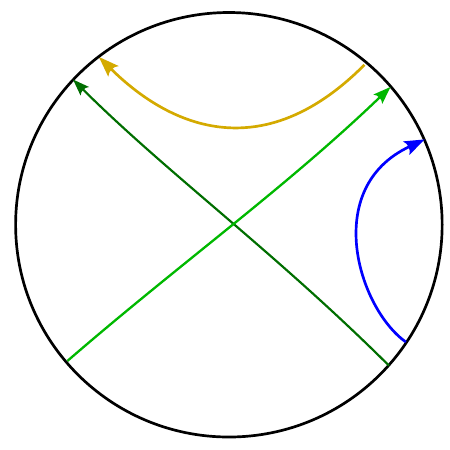}
    \caption{$\protect\ora{Ba},\protect\ora{aB}$ splitting}
    \label{fig:r3}
  \end{subfigure}
  \hfill
 \begin{subfigure}[t]{0.42\textwidth}
    \centering
    \resizebox{\linewidth}{!}{\fontsize{9pt}{9pt}\selectfont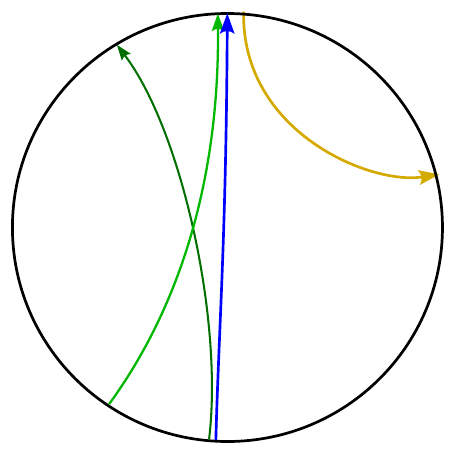}
    \caption{$\protect\ora{aB}$, $\protect\ora{Ba}$ are $b$-sided and antiparallel to each other}
    \label{fig:r1}
  \end{subfigure}
  \vspace{10pt}\par
  \begin{subfigure}[t]{0.42\textwidth}
    \centering
    \resizebox{\linewidth}{!}{\fontsize{9pt}{9pt}\selectfont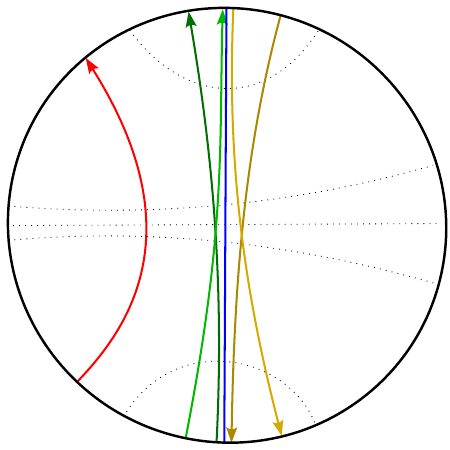}
    \caption{$\protect\ora{aB}$, $\protect\ora{Ba}$ are $b$-sided and crossing}
    \label{fig:c2}
  \end{subfigure}

  \caption{Topological cases for the location of $\protect\ora{aB}$ and
    $\protect\ora{Ba}$ when $(\protect\ora{a},\protect\ora{b})$ are parallel.}
  \label{fig:cases-conn-smoothing}
\end{figure}

\begin{lemma}
\label{lem:trichotomy_ab}
If $(\ora{a},\ora{b})$ are parallel, then $\ora{aB},\ora{Ba}$ are
exactly one of $a$-sided, splitting, or $b$-sided.
\end{lemma}
\begin{proof}
This follows from Lemma~\ref{lem:aB_Ba_together}.
\end{proof}

\begin{lemma}
\label{lem:aB_bA_antiparallel_cross}
If $(\ora{a},\ora{b})$ are parallel, then $(\ora{aB},\ora{Ba})$  are anti-parallel or cross.
\end{lemma}
\begin{proof}
By Lemma~\ref{lem:aB_Ba_together} and the fact that $a \cdot \ora{Ba} = \ora{aB}$,
it follows that $\ora{aB},\ora{Ba}$ cannot be nested with respect to $\ora{a}$.
Then since $a \cdot \ora{Ba} = \ora{aB}$, they must either be both
parallel or both anti-parallel to $\ora{a}$, so cannot be parallel to each other.
\end{proof}

See Figure~\ref{fig:c2} and Example~\ref{ex:ab_parallel_aB_crossing} for examples where $(\ora{Ba},\ora{aB})$ cross.

\begin{lemma}
\label{lem:splitting_specific}
If $(\ora{a},\ora{b})$ are R-parallel and $\ora{aB},\ora{Ba}$ split
$\ora{a},\ora{b}$,
then $(aB)^{\pm} \in (b^+, a^+)$ and $(Ba)^{\pm} \in (a^-,b^-)$.
Furthermore $\ora{aB}$ and $\ora{Ba}$ are both parallel to~$\ora{a}$
and anti-parallel to~$\ora{b}$.
\end{lemma}
\begin{proof}
  Note that $a \cdot (Ba)^+ = b \cdot (Ba)^+ = (aB)^+$ and $a,b$ move
  in opposite directions (cw/ccw) on the intervals $(b^+,a^+)$ and
  $(a^-,b^-)$. The only way the $a$ and $b$ actions on can end up in
  the same place is if they move from one interval to the other.
  Examination of the NS dynamics shows specifically that $(Ba)^+ \in
  (a^-,b^-)$ and $(aB)^+ \in (b^+, a^+)$. The same argument works for
  the repelling endpoints.

  To see that $\ora{aB}$ is parallel to $\ora{a}$, note that $a,B$
  move the same direction on $(b^+,a^+)$, away from $b^+$ and
  towards~$a^+$. Thus the attracting endpoint of $\ora{aB}$ must be
  closer to $a^+$ than to $b^+$ within this interval. Similarly for
  $\ora{Ba}$ and for the relation to $\ora{b}$.
\end{proof}

To draw specific examples precisely, consider the common hyperbolic perpendicular of $\ora{a}$
and~$\ora{b}$, which we will denote~$R_1$. We will also use the
same notation for the reflection through~$R_1$.
Let $R_2, R_3, R_4, R_5$ be the lines and corresponding reflections so
that $a=R_1 R_2 = R_4 R_1$, and $b=R_1 R_3 = R_5 R_1$. Precisely,
$R_2$ and $R_4$ are perpendicular to $\ora{a}$, and $R_3$ and $R_5$
are perpendicular to~$\ora{b}$.

In fact every pair of products of the $R_i$ has a name:
\begin{align*}
  R_1 R_2 &= a  &  R_1 R_3 &= b  &  R_4 R_1 &= a  &  R_5 R_1 &= b\\
               &&  R_3 R_2 &= Ba &  R_4 R_2 &= a^2&  R_5 R_2 &= ba\\
                              &&&&  R_4 R_3 &= ab &  R_5 R_3 &= b^2\\
                                             &&&&&& R_4 R_5 &= aB.
\end{align*}
Since all
these elements are hyperbolic by Convention~\ref{conv:aB-hyperbolic},
we see
that the $R_i$ are pairwise non-intersecting reflection
lines. Some realized configurations are shown
in Figure~\ref{fig:cases-conn-smoothing}.

We next show that by taking powers we can reduce to the splitting
case, starting with a general convergence lemma.

\begin{lemma}
\label{lem:convergence_action}
Let $g,x \in \PSL(2,\mathbb{R})$, with $x$ hyperbolic and
$g \cdot x^+ \neq x^-$. Then, for all sufficiently large $n$,
the element $x^n g$ is hyperbolic and
\[
\lim_{n \to \infty} (x^n g)^+ = x^+ .
\]
More generally, if $g,h,x \in \PSL(2,\bbR)$ with $x$ hyperbolic and
$hg \cdot x^+ \ne x^-$, then for all sufficiently large $n$, the
element $g x^n h$ is hyperbolic and
\begin{align*}
  \lim_{n \to \infty} (g x^n h)^+ &= g \cdot x^+, &
  \lim_{n \to \infty} (g x^n h)^- &= H \cdot x^-.
\end{align*}
If $x,g,h \in \pi_1(S)$, the hypotheses that $g \cdot x^+ \ne x^-$,
resp.\ $hg \cdot x^+ \ne x^-$, are automatic.
\end{lemma}

\begin{proof}
This follows from NS dynamics on $\partial_\infty S$.
Indeed, choose a neighborhood $V$ of $x^+$ such that
$x^- \notin g \cdot V$. Then, by NS, there is
$n_0>0$ such that, for every $n \ge n_0$,
\[
x^n g \cdot V \subsetneq V.
\]
Thus, by Lemma~\ref{lem:NS-dynamics}, $x^n g$ is hyperbolic, with
\[
(x^n g)^+ \in x^n g \cdot V \subsetneq V.
\]
As $n \to \infty$, $x^n g \cdot V \to x^+$ and so $(x^n g)^+ \to x^+$.

For the more general statement, apply the first part of the statement
to $x^n hg$ and to $X^n GH$. Since
$g \cdot \ora{x^n hg} = \ora{g x^n h}$ and $H \cdot \ora{X^n GH} =
\ora{H X^n G}$, we get the claimed result.

If $g,x \in \pi_1(S)$, since each point in $\bdy_\infty S$ is the
endpoint of at most one geodesic, if $g \cdot x^+ = x^-$ then also $g
\cdot x^- = x^+$. But this implies that $g$ is elliptic of order~$2$,
which is impossible in a surface group. Similarly for the more general
statement.
\end{proof}

\begin{lemma}
\label{lem:once_splits}
Let $\ora{a},\ora{b}$ be parallel. If, for some $n,m \ge 1$,
$\ora{a^nB^m}$ splits $\ora{a},\ora{b}$, then so does
$\ora{a^NB^M}$ for every $N \geq n$ and $M \geq m$.
\end{lemma}
\begin{proof}
It suffices to prove the claim for $\ora{a^{n+1}B^m}$ and $\ora{a^nB^{m+1}}$.
By Lemma~\ref{lem:splitting_specific}, $(\ora{a}, \ora{a^n B^m})$ are
parallel. By Lemma~\ref{lem:righthandedfoundation}, we have
$(a^{n+1}B^m)^+ \in ((a^nB^m)^+, a^+) \subset (b^+, a^+)$, proving the
claim for $\ora{a^{n+1} B^m}$.

Similarly, $\ora{a^nB^m},\ora{B}$ are parallel, so $(a^nB^{m+1})^- \in
(B^-, (a^nB^m)^-) \subset (b^+,a^+)$, completing the proof.
\end{proof}

\begin{proposition}
If $\ora{a}$ and $\ora{b}$ are parallel, then there exists $s > 0$ so
that for every $n,m \ge s$, 
$\ora{a^n B^m}$ is splitting.
\label{prop:eventual_splitting}
\end{proposition}

\begin{proof}
By Lemma~\ref{lem:once_splits}, it suffices to show that there exists
some $n,m \geq 1$
such that $\ora{a^{n} B^{m}}$ is splitting.
Suppose for definiteness that $(\ora{a},\ora{b})$ are R-parallel.
Choose $m>0$ so that
\[
b^m \cdot a^- \in (b^+, a^+).
\]
In particular, $b^m \cdot a^- \ne a^+$ so $B^m \cdot a^+ \ne a^-$.
Now consider $a^n B^m$ for varying~$n$.
Lemma~\ref{lem:convergence_action} gives
\begin{align*}
\lim_{n \to \infty} (a^n B^m)^+ &= a^+, &
\lim_{n \to \infty} (a^n B^m)^- &= b^m \cdot a^-.
\end{align*}
In particular, for $n$ sufficiently large, we have
$(a^n B^m)^- \in (b^+, a^+)$,
so $\ora{a^n B^m}$ is splitting.
\end{proof}

\begin{remark}
A ping-pong argument~\cite[\S 8.2.E]{Gromov1987HyperbolicGroups}
shows that, in the splitting configuration, the subgroup generated by $a$ and $b$
is free.
Accordingly, Proposition~\ref{prop:eventual_splitting} is related to the existence
of an exponent $s>0$ (depending on $a,b$) such that, for all $m,n \ge s$,
the subgroup generated by $a^m$ and $b^n$ is free.

Results of this type hold in much greater generality, for instance for
Gromov hyperbolic groups and for mapping class groups; see, e.g.,
\cite{Fujiwara2015SubgroupsPseudoAnosovII}.
\end{remark}

From the above analysis, it follows that when $(\ora{a},\ora{b})$ are
parallel, many composite pairs are parallel. We can extract a
particular result we will use later.
\begin{lemma}\label{lem:a_parallel_Ba}
Let $a,b \in \pi_1(S)$.
If $\ora{a},\ora{b}$ are parallel and $(\ora{aB}, \ora{Ba})$ cross,
let $\mathcal{P} \coloneqq \{ \ora{a},\ora{b}, \ora{ab},\ora{ba} \}$
and $\mathcal{Q} \coloneqq \{ \ora{aB}, \ora{Ba} \}$. Then either
\begin{itemize}
\item $(\ora{g},\ora{h})$ are parallel for every $\ora{g} \in \mathcal{P}, \ora{h} \in \mathcal{Q}$; or
\item $(\ora{g},\ora{h})$ are anti-parallel for every $\ora{g} \in \mathcal{P}, \ora{h} \in \mathcal{Q}$.
\end{itemize}
\end{lemma}
\begin{proof}
Since
$(\ora{aB},\ora{Ba})$ cross, they must be either $a$-sided or
$b$-sided.
If they are $a$-sided, then they are either both parallel or both
anti-parallel to $\ora{a}$.
If they are parallel, since $\ora{ab}$, $\ora{ba}$, and $\ora{b}$ are parallel to
$\ora{a}$ and on the opposite side of $\ora{ab}$, we are in the first
case of the statement.
Otherwise we are in the second case.

If they are $b$-sided, the analysis is similar: just swap $a \leftrightarrow B$ and $b \leftrightarrow A$, since $(\ora{B},\ora{A})$ are parallel.
\end{proof}

Suppose that $\ora{Ba},\ora{aB}$ are $b$-sided and cross, as in Figure~\ref{fig:c2}. If we replace $b$ by $B$, we see that $\ora{ab}$ and
$\ora{ba}$ can cross even though $\ora{a}$ and $\ora{b}$ are not
parallel (in fact, they are anti-parallel). Let us show a concrete
example where this situation arises. See also the discussion
in~\cite{MGT25:SmoothingsErratum}.

\begin{example}
  To see that the case where $(\ora{a},\ora{b})$ are parallel and
  $(\ora{aB},\ora{Ba})$ cross actually occurs (thus
  violating a naive smoothing condition), consider any two parallel $\ora{g}$
  and $\ora{h}$, for instance coming from an appropriately-oriented
  pair of disjoint simple closed curves sitting
  as the cuffs of an embedded pair of pants.
  Define $a \coloneqq hghgh$ and $b
  \coloneqq hgh$. We claim that $\ora{a}$ and $\ora{b}$ are
  parallel. First, by
  Lemma~\ref{lem:righthandedfoundation}\ref{item:ab-parallel},
  $(\ora{hg},\ora{gh})$ cross and $(\ora{h},\ora{gh},\ora{g})$ are
  parallel. Applying the lemma again to $(\ora{h},\ora{gh})$, we see
  that $(\ora{hgh},\ora{gh})$ are parallel, and applying the lemma
  once more we see $\ora{b} = \ora{hgh}$ and $\ora{a} = \ora{hghgh}$
  are parallel. See Figure~\ref{fig:crossing_aB}.
  Finally, $aB=(hghgh)(HGH)=hg$ and $Ba=(HGH)(hghgh)=gh$, which cross,
  as desired.
  \label{ex:ab_parallel_aB_crossing}
\end{example}

\begin{figure}
\centering{
\fontsize{9pt}{9pt}\selectfont%% Creator: Inkscape 1.3 (0e150ed6c4, 2023-07-21), www.inkscape.org
%% PDF/EPS/PS + LaTeX output extension by Johan Engelen, 2010
%% Accompanies image file 'crossing_aB.pdf' (pdf, eps, ps)
%%
%% To include the image in your LaTeX document, write
%%   \input{<filename>.pdf_tex}
%%  instead of
%%   \includegraphics{<filename>.pdf}
%% To scale the image, write
%%   \def\svgwidth{<desired width>}
%%   \input{<filename>.pdf_tex}
%%  instead of
%%   \includegraphics[width=<desired width>]{<filename>.pdf}
%%
%% Images with a different path to the parent latex file can
%% be accessed with the `import' package (which may need to be
%% installed) using
%%   \usepackage{import}
%% in the preamble, and then including the image with
%%   \import{<path to file>}{<filename>.pdf_tex}
%% Alternatively, one can specify
%%   \graphicspath{{<path to file>/}}
%% 
%% For more information, please see info/svg-inkscape on CTAN:
%%   http://tug.ctan.org/tex-archive/info/svg-inkscape
%%
\begingroup%
  \makeatletter%
  \providecommand\color[2][]{%
    \errmessage{(Inkscape) Color is used for the text in Inkscape, but the package 'color.sty' is not loaded}%
    \renewcommand\color[2][]{}%
  }%
  \providecommand\transparent[1]{%
    \errmessage{(Inkscape) Transparency is used (non-zero) for the text in Inkscape, but the package 'transparent.sty' is not loaded}%
    \renewcommand\transparent[1]{}%
  }%
  \providecommand\rotatebox[2]{#2}%
  \newcommand*\fsize{\dimexpr\f@size pt\relax}%
  \newcommand*\lineheight[1]{\fontsize{\fsize}{#1\fsize}\selectfont}%
  \ifx\svgwidth\undefined%
    \setlength{\unitlength}{144bp}%
    \ifx\svgscale\undefined%
      \relax%
    \else%
      \setlength{\unitlength}{\unitlength * \real{\svgscale}}%
    \fi%
  \else%
    \setlength{\unitlength}{\svgwidth}%
  \fi%
  \global\let\svgwidth\undefined%
  \global\let\svgscale\undefined%
  \makeatother%
  \begin{picture}(1,1)%
    \lineheight{1}%
    \setlength\tabcolsep{0pt}%
    \put(0,0){\includegraphics[width=\unitlength,page=1]{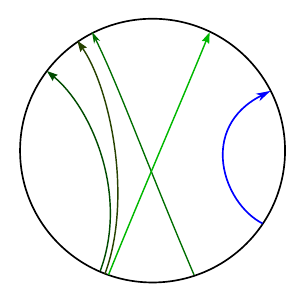}}%
    \put(0.2997907,0.91081972){\color[rgb]{0,0,0}\makebox(0,0)[t]{\lineheight{1.25}\smash{\begin{tabular}[t]{c}$hg$\end{tabular}}}}%
    \put(0.10150957,0.77005206){\color[rgb]{0,0,0}\makebox(0,0)[t]{\lineheight{1.25}\smash{\begin{tabular}[t]{c}$hgh$\end{tabular}}}}%
    \put(0.15689326,0.8603916){\color[rgb]{0,0,0}\makebox(0,0)[t]{\lineheight{1.25}\smash{\begin{tabular}[t]{c}$hghgh$\end{tabular}}}}%
    \put(0.93589672,0.70405374){\color[rgb]{0,0,0}\makebox(0,0)[t]{\lineheight{1.25}\smash{\begin{tabular}[t]{c}$g$\end{tabular}}}}%
    \put(0.72030356,0.92026102){\color[rgb]{0,0,0}\makebox(0,0)[t]{\lineheight{1.25}\smash{\begin{tabular}[t]{c}$gh$\end{tabular}}}}%
    \put(0,0){\includegraphics[width=\unitlength,page=2]{crossing_aB.pdf}}%
    \put(0.07669038,0.67827698){\color[rgb]{0,0,0}\makebox(0,0)[t]{\lineheight{1.25}\smash{\begin{tabular}[t]{c}$h$\end{tabular}}}}%
  \end{picture}%
\endgroup%

\caption{The proof that $\protect\ora{a} \coloneqq \protect\ora{hghgh}$ and
  $\protect\ora{b} \coloneqq \protect\ora{hgh}$ are parallel if
  $\protect\ora{g},\protect\ora{h}$ are parallel.}
\label{fig:crossing_aB}
}
\end{figure}

\subsection{Fundamental smoothing results}
\label{sec:fund-smoothing}

As part of our construction, we will use large powers of some fixed
elements. We collect some propositions guaranteeing essential
crossings in this context.
\begin{proposition}\label{prop:join-split-geom}
Let $a,b,x \in \pi_1(S)$ be curves with $\ora{a}$, $\ora{b}$,
crossing $\ora{x}$ to the right.
Then there is $s \ge 0$ (depending on $a,b,x$)
  so that, for all
 $n, m, r \in \mathbb{Z}$ with either $s < r < m-n-s$ or $m-n+s < r <
 -s$, we have
  \begin{align}
    \label{eq:cross-join}   [a x^n] [b x^m] &\reducesto [ax^{n+r}bx^{m-r}]\\
    \label{eq:cross-split} [a x^nb x^m] &\reducesto [ax^{n+r}][bx^{m-r}].
  \end{align}

  There exists $s' \ge 0$ (depending on $a,x$) so that for all
  $\abs{n-m} > s'$,
  \begin{align}
    \label{eq:cross-cancel}[ax^n][ax^m] &\reducesto [x^{\abs{n-m}}].
  \end{align}

 If $x$ is simple and $a=b$, we may take $s=s'=0$. In particular, for
 all $n\in\bbZ$,
 \begin{align}
    [a x^{n+1}] [a x^{n-1}] &\reducesto [(a x^n)^2] \reducesto 2[ax^{n}].\label{eq:cross-simple}
  \end{align}
\end{proposition}

For the allowed values of $r$, Eqs.\ \eqref{eq:cross-join}
and~\eqref{eq:cross-split} move the exponents closer to the
diagonal
$m=n$.

We will make heavy use of the \emph{band model centered on
$\ora{x}$}, meaning the conformal map from $\bbH^2$ to a horizontal
strip $(-\infty,\infty) \times [-v,v]$ mapping $x^-$ to infinity along
the negative $x$-axis and $x^+$ to infinity on the positive $x$-axis.
We choose the \emph{normalized} band model, picking $v$ so that the
action of $x$ is horizontal translation by $z \mapsto z+1$.

For any geodesic $\ora{y}$ that is parallel to
$\ora{x}$, there is a \emph{window} $w(\ora{y})$ cut out by the
endpoints on the band model centered on~$\ora{x}$, the interval between $y^\pm$ that
misses~$x^\pm$. For instance, if $(\ora{x},\ora{y})$ are R-parallel
(so $\ora{y}$ is on the bottom of the band model),
then $w(\ora{y}) = (y^-,y^+)$, while if $(\ora{y},\ora{x})$ are R-parallel than
$w(\ora{y}) = (y^+,y^-)$.
Observe that $w(x^k\cdot\ora{y})$ is the horizontal
translate of $w(\ora{y})$ by $k$.
Also define $\overline{w}(\ora{y})$ to be the (open)
complementary window.
See Figure~\ref{fig:abcrossings} for
examples.

\begin{lemma}
\label{lem:windows}
  Given $a,b, x \in \pi_1(S)$ so that $\ora{a}$ and $\ora{b}$ cross $\ora{x}$ to
  the right and $g_1 = x^k ax^\ell$ with $k,\ell \in \bbZ$, we have
  $g_1^+ \in w(x^k a\cdot\ora{x})$ and
  $g_1^- \in w(x^{-\ell}A\cdot\ora{x})$. Similarly for
  $g_2 = x^k a x^\ell b x^n$ with $k,\ell,n\in\bbZ$, we have
  $g_2^+ \in w(x^ka \cdot \ora{x})$ and $g_2^- \in w(x^{-n} B
  \cdot\ora{x})$.
\end{lemma}

\begin{proof}
  For the first part, consider the action of $(g_1)^2$ on
  $\overline{w}(x^{-\ell}A \cdot \ora{x})$.We have
  \begin{align*}
    (g_1)^2 \cdot x^{-\ell}A \cdot (x^-, x^+)
      &= g_1 \cdot (x^-,x^+)\\
      &= x^k a \cdot (x^-,x^+) = w(x^k a \cdot \ora{x}).
  \end{align*}
  This implies the statement by Lemma~\ref{lem:NS-dynamics}.

  Similarly for the second part, consider the action of $g_2$
  on $\overline{w}(x^{-n} B \cdot \ora{x})$.
\end{proof}

\begin{figure}
\centering{\resizebox{150mm}{!}{\fontsize{14pt}{14pt}\selectfont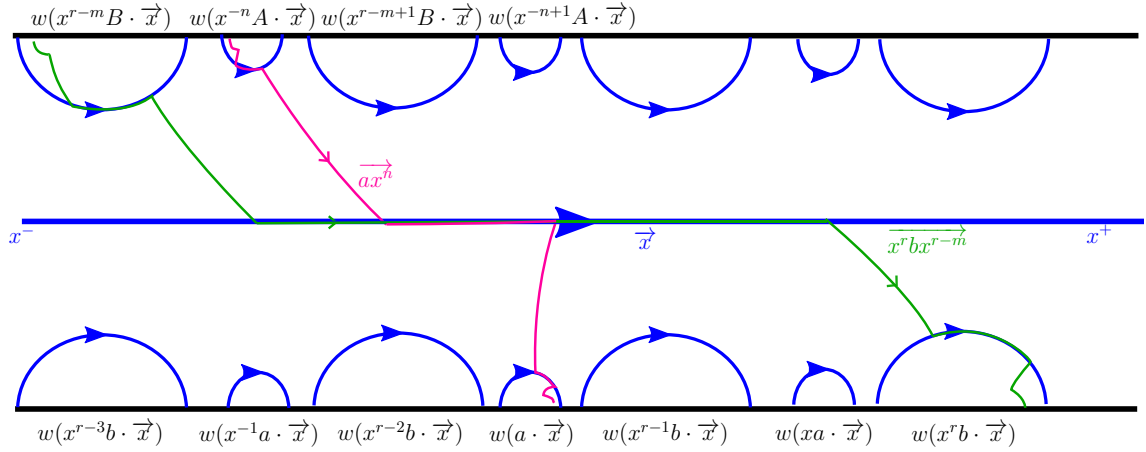}}
\caption{The band model centered on $\protect\ora{x}$ used in
  Proposition~\ref{prop:join-split-geom}, in a case where $\protect [x]$ is simple.
  Some lifts of $\protect [x]$ are depicted in blue. The central blue lift
  is $\protect\ora{x}$, and we also show translates by $ \protect x^k a$, $\protect x^l
  A$, $\protect x^m b$ and $\protect x^n B$, together with their windows.
  If $\protect[x]$ is not simple, the blue lifts can intersect each other.}\label{fig:abcrossings}
\end{figure}

\begin{proof}[Proof of Proposition~\ref{prop:join-split-geom}]
  In each case, we will show that we can choose lifts of the curves to
  the normalized band model for which a
  relevant crossing is essential. We start with
  Eq.~\eqref{eq:cross-join} looking first at the windows $w(a\cdot
  \ora{x})$ and $w(x^rb\cdot\ora{x})$.
  Since each window is a
  finite size, there is a finite range $[s_1,s_2]$ so that $w(x^rb\cdot\ora{x})$
  is to the left of (resp.~to the right of) $w(a\cdot\ora{x})$ if $r < s_1$
  (resp.\ $r > s_2$).
  Similarly there
  is a range $[s_1',s_2']$ so that $w(A\cdot\ora{x})$ is disjoint from
  $w(x^rB\cdot\ora{x})$ if $r \notin [s_1',s_2']$, and the exponents determine the
  ordering of the windows. Set
  $s \coloneqq \max\{s_2,s_2',-s_1,-s_1'\}$.

  Now, for Eq.~\eqref{eq:cross-join}, consider
  $g_1 = ax^n$ and $g_2 = x^rbx^{m-r}$, with $s < r < m-n-s$. We have
  \[
    r-m < -n-s \le -n+s_1',
  \]
  so, by Lemma~\ref{lem:windows}, $g_2^- \in w(x^{r-m}B\cdot\ora{x})$
  is to the left of $g_1^- \in w(x^{-n}A\cdot\ora{x})$, and
  \[
    s_2 \le s < r,
  \]
  so $g_2^+ \in w(x^rb\cdot\ora{x})$ is to the right of $g_1^+ \in
  w(a\cdot\ora{x})$. Thus
  $\ora{x^rbx^{m-r}}$ crosses $\ora{ax^n}$ to the left.

  Similarly if
  $m-n+s < r < -s$, then $\ora{x^rbx^{m-r}}$ crosses $\ora{ax^n}$ to the
  right. In either case,
    \[
[ax^n][bx^m] = [ax^n][x^r bx^{m-r}] \reducesto [ax^{n+r}bx^{m-r}]
  \]
  by oriented disconnected smoothing.

  For Eq.~\eqref{eq:cross-split},  write   $[ax^n bx^m]=[gh]$, with
  $g=x^rax^n$, $h=bx^{m-r}$ and $s < r < m-n-s$
  Since $s_1 < r$, then $h^+ \in w(b\cdot\ora{x})$ is to the left of
  $g^+ \in w(x^r a\cdot\ora{x})$.
  Similarly, $r-m < -n+s_1'$, so $h^-\in w(x^{r-m}B \cdot\ora{x})$ is to the left
  of $g^- \in w(x^{-n}A\cdot\ora{x})$. 
  Thus $\ora{h}$ and $\ora{g}$
  are R-parallel, giving Eq.~\eqref{eq:cross-split} by oriented
  connected smoothing.
  We similarly get $(\ora{g},\ora{h})$  are parallel if $m-n+s < r < -s$.
  To see Eq.~\eqref{eq:cross-cancel}, suppose without loss of generality
  that $n < m$. Let $s_0$ be the ceiling of the width of $w(a\cdot\ora{x})$
  (measuring width in the normalized band model),
  and
  let $s' = 2s_0 - 1$. Since $m-n \ge 2s_0$, we can pick $r$ so that
  $s_0 \le r \le m-n-s_0$. By the above argument, $\ora{ax^n}$ and
  $\ora{x^rax^{m-r}}$ R-cross. (For instance, $(ax^n)^+ \in w(a\cdot\ora{x})$ and
  $(x^rax^{m-r})^+ \in w(x^r a\cdot\ora{x})$, which are disjoint.) We thus get the
  desired smoothing from unoriented
  disconnected smoothing.

  In the cases when $x$ is simple, two windows are disjoint, nested,
  or identical; this implies that $s_1 = s_2$ and $s_1' = s_2'$.
  If in
  addition $a=b$, we also see that $s = s_1=s_1'=0$, as claimed.

  To see that with $x$ simple we get Eq.~\eqref{eq:cross-cancel} with
  $s'=0$ (i.e., $[ax^n][ax^{n+1}] \reducesto [x]$), we look slightly deeper
  in the words.%
  Consider $(a x^n)^3$ acting on
  $\overline{w}(x^{-n} A \cdot \ora{x}) = x^{-n} A \cdot (x^-,x^+)$.
  Computations as in Lemma~\ref{lem:windows} show that
  \[
    (a x^n)^3 \cdot \overline{w}(x^{-n} A \cdot \ora{x}) = w(a x^n a
    \cdot \ora{x}),
  \]
  so $(a x^n)^- \in w(x^{-n} A \cdot \ora{x})$ and
  $(a x^n)^+ \in w(a x^n a \cdot \ora{x})$. Similarly
  $(a x^{n+1})^- \in w(x^{-n-1} A \cdot \ora{x})$ and
  $(a x^{n+1})^+ \in w(a x^{n+1} a \cdot \ora{x})$. It follows that
  $\ora{ax^n}$ crosses $\ora{ax^{n+1}}$ as desired. (To see the
  relative positions of $w(a x^n a \cdot \ora{x})$ and
  $w(a x^{n+1} a \cdot \ora{x})$, look at the band model centered on
  $a \cdot \ora{x}$, or alternately conjugate by $A$.)
  Eq.~\eqref{eq:cross-simple} follows by applying
  Eq.~\eqref{eq:cross-join} with $r=1$.
\end{proof}

\begin{remark}
\label{rmk:triple_limits}
By similar arguments as those used to prove Equations~   \ref{eq:cross-join},   \ref{eq:cross-split}, we can show that if $a,b,c$ are R-crossing $x$, then there is $s \ge 0$ (depending on $a,b,c,x$)
  so that, for all
 $n, m, r \in \mathbb{Z}$ with either $s < r < m-n-s$ or $m-n+s < r <
 -s$, we have
  \begin{align}
    \label{eq:cross-join-triple} [cx^{-n}][a x^n] [b x^m] &\reducesto [cx^{-n}ax^{n+r}bx^{m-r}]\\
    \label{eq:cross-split-triple} [cx^{-n}a x^nb x^m] &\reducesto [cx^{-n}ax^{n+r}][bx^{m-r}].
  \end{align}
\end{remark}

We also need some results about smoothings where we have high powers of
two different elements. We work with three elements $a,x,y\in\pi_1(S)$
so that $(\ora{y},\ora{x},a\cdot\ora{y})$ are R-parallel; this
implies that $\ora{a}$ R-crosses both $\ora{x}$ and
$\ora{y}$, and that $(a\cdot\ora{y},a\cdot\ora{x})$ are also
R-parallel.
We again consider the normalized band model
centered on~$\ora{x}$, starting with another lemma about endpoints.

\begin{lemma}\label{lem:ayx-window}
  For any triple $a,x,y\in\pi_1(S)$ so that
  $(\ora{y},\ora{x},a\cdot\ora{y})$ are R-parallel, and for
  $g=x^k ay^n x^\ell$ with $k,\ell,n\in\bbZ$, we have
  $g^+ \in w(x^k a \cdot\ora{y})$ and
  $g^- \in w(x^{-\ell}\cdot\ora{y})$.
\end{lemma}

\begin{proof}
  We have
  $g \cdot \overline{w}(x^{-\ell} \cdot \ora{y}) = x^k ay^n
  \cdot(y^-,y^+) = x^k a\cdot (y^-,y^+) = w(x^k a \cdot \ora{y})$.
  Lemma~\ref{lem:NS-dynamics} then proves the result.
\end{proof}

\begin{proposition}\label{prop:fundamental-ayx}
  For any triple $a,x,y\in\pi_1(S)$ so that $(\ora{y},\ora{x},a\cdot\ora{y})$ are R-parallel, there exists
  $s \ge 0$ so that for all $r > s$ and $m,n \in \bbZ$,
  \begin{align}
    \label{eq:crossuniform_n}
       [a y^m x^{n+r}][a y^m x^{n-r}] &\reducesto [(a y^m x^n)^2]\\
    \label{eq:crossuniform_m}
       [a y^{m+r} x^n][a y^{m-r} x^n] &\reducesto [(a y^m x^n)^2].
  \end{align}
  Furthermore, for $n,m \ge s$,
  \begin{align}\label{eq:cancel-ayx}
    [a y^m x^n][A] &\reducesto [y^m][x^n],\\
    \intertext{and for $n,m > 2s$,}
    \label{eq:crossbelowy}
    [a Y^m x^n][A] &\reducesto [x^s Y^s] [A x^s a Y^s] [x^{n-2s}] [Y^{m-2s}].
  \end{align}
\end{proposition}

\begin{proof}
  We first prove Eq.~\eqref{eq:crossuniform_n}. We will have a number
  of windows $w(x^m \cdot \ora{y})$ on the top, all of equal width
  in the normalized band model, and $w(x^m a \cdot \ora{y})$ on the
  bottom, also of equal widths (but not necessarily equal to widths on
  the top). Take
  \[
    s \ge \max(\wid(w(\ora{y})),
    \wid(w(a\cdot\ora{y}))).
  \]
  For $r > s$, consider the elements $g_1 = x^n ay^m x^r$ and $g_2 =
  x^{n-r} a y^m$. By Lemma~\ref{lem:ayx-window}, we have
  \begin{align*}
    g_1^+ &\in w(x^n a\cdot\ora{y}) & g_1^- &\in w(x^{-r}\cdot\ora{y})\\
    g_2^+ &\in w(x^{n-r} a \cdot\ora{y}) & g_2^- &\in w(\ora{y}).
  \end{align*}
  Since $r>s$, the endpoints appear in the cyclic order
  $(x^-,g_2^+,g_1^+,x^+,g_2^-,g_1^-)$ on $\partial_\infty S$. In
  particular $[g_1][g_2] \reducesto [g_1 g_2]$ by oriented disconnected smoothing, as desired.
  
  To see Eq.~\eqref{eq:crossuniform_m}, apply
  Eq.~\eqref{eq:crossuniform_n} with the change of variables
  $a' = a$, $y' = x$, and $x' = ayA$.
  This gives the desired result, after possibly increasing~$s$.

  For Eq.~\eqref{eq:cancel-ayx}, we consider $g = a y ^m x^n$ and
  compare the endpoints of $\ora{g}$ to those of $\ora{a}$. By
  Lemma~\ref{lem:ayx-window}, we have
  $g^- \in w(x^{-n} \cdot \ora{y})$, which for large enough $n$ is to
  the left of $a^- \in w(\ora{y})$. For $g^+$,
  we again have to look deeper at the endpoints: since
  $g \cdot (x^+,x^-) = a y^m \dot (x^+, x^-) = w(a y^m \cdot
  \ora{x})$, we have
  $g^+ \in w(a y^m \cdot \ora{x}) \subset w(a \cdot \ora{y})$.
  Similarly $a^+ \in w(a \cdot \ora{x}) \subset w(a \cdot \ora{y})$.
  Considering the band model centered on $a \cdot \ora{y}$ shows that,
  for $m$ large enough, $g^+$ is to the right of $a^+$, and thus
  $\ora{g}$ crosses $\ora{a}$ to the right. Thus if $s$ is large
  enough and $n,m
  \ge s$ we therefore have
  \[
    [a y^m x^n][A] = [g][A] \reducesto [x^n y^m] \reducesto [x^n][y^m]
  \]
  by oriented disconnected smoothing and then (by
  Lemma~\ref{lem:righthandedfoundation}) oriented connected
  smoothing.

  Finally, we see Eq.~\eqref{eq:crossbelowy}. By
  Proposition~\ref{prop:eventual_splitting}, there is an $s$ so that
  for every $k,\ell \geq s$ we have
  $\ora{x^k Y^\ell},\ora{Y^\ell x^k}$ split $\ora{y},\ora{x}$, and
  furthermore $\ora{A x^k a Y^\ell},\ora{Y^\ell A x^k a}$ split $A \cdot
  \ora{x}, \ora{y}$; see
  Figure~\ref{fig:crossing_yx_A}.
  Additionally ensure that $s$ is large enough so that for $t, n-t
  \ge s$ and
  $g = x^t a Y^m x^{n-t}$, we have $\ora{a},\ora{g}$ crossing. (This
  is possible from window analysis as above since by
  Lemma~\ref{lem:ayx-window}, $g^- \in w(x^{-(n-t)} \cdot \ora{y})$ and
  $g^+ \in w(x^t a \cdot \ora{y})$.)
  For $n,m > 2s$, we
  then have
  \begin{align*}
    [a Y^m x^n] [a] &= [x^s a Y^m x^{n-s}] [a] \\
    &\reducesto [A x^s a Y^m x^{n-s}]
      &&\text{$(\ora{a}, \ora{x^s a Y^m x^{n-s}})$ cross}\\
    &\reducesto [A x^s a Y^s] [Y^{m-s} x^{n-s}]
      &&\text{$(\ora{A x^s a Y^s}, \ora{Y^{m-s} x^{n-s}})$ parallel}\\
    &\reducesto [A x^s a Y^s] [Y^{m-s} x^s] [x^{n-2s}]
      &&\text{$(\ora{Y^{m-s} x^s}, \ora{x^{n-2s}})$ parallel}\\
    &\reducesto [A x^s a Y^s] [Y^{m-2s}] [Y^s x^s] [x^{n-2s}]
      &&\text{$(\ora{Y^{m-2s}}, \ora{Y^s x^s})$ parallel.} \qedhere
  \end{align*}
\end{proof}

\begin{figure}
\centering{
\resizebox{100mm}{!}{\fontsize{14pt}{14pt}\selectfont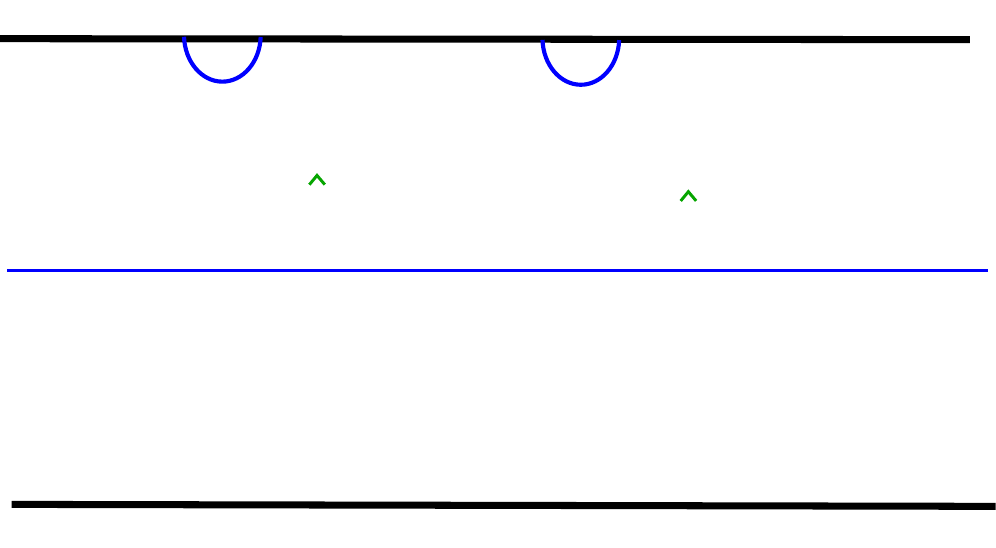}
\caption{Configuration  $\protect\ora{Y^s x^s}$ and $\protect\ora{x^s Y^s}$ splitting $\protect\ora{y},\protect\ora{x}$, and
   $\protect\ora{A x^s a Y^s}$ and $\protect\ora{Y^s A x^s a}$ splitting $A \cdot
  \protect\ora{x}, \protect\ora{y}$, used in
  Proposition~\ref{prop:fundamental-ayx}. Note this is shown in the band model
  centered on $\protect\ora{a}$.}
\label{fig:crossing_yx_A}
}
\end{figure}

In practice we will apply Proposition~\ref{prop:fundamental-ayx} with
four group elements.
\begin{corollary}\label{cor:fund-axby}
  Let $a,b,x,y \in \pi_1(S)$ be chosen so that $(B\cdot \ora{x},
  \ora{y}, a\cdot\ora{x})$ are R-parallel.
  Then there exists
  $s \ge 0$ so that for all $r > s$ and $m,n \in \bbZ$,
  \begin{align}
    \label{eq:axby-cross-n}
       [a x^{n+r} by^m][ax^{n-r} by^m] &\reducesto [(ax^n by^m)^2]\\
    \label{eq:axby-cross-m}
       [ax^n by^{m+r}][ax^n by^{m-r}] &\reducesto [(ax^n by^m)^2].
  \end{align}
  Furthermore, for $n,m \ge s$,
\begin{align}\label{eq:cancel-axby}
  [a x^n b y^m][ab] &\reducesto [x^n][y^m]\\
  \intertext{and for $n,m > 2s$,}
  \label{eq:crossbelowy-ab}
   [a x^n b Y^m][ab] &\reducesto [a x^s A Y^s] [B x^s b Y^s] [x^{n-2s}] [Y^{m-2s}]
\end{align}
\end{corollary}

\begin{proof}
  The conditions imply that
  $(\ora{y}, a \cdot \ora{x}, ab \cdot \ora{y})$ are R-parallel.
  Substitute $x' = axA$, $y'=y$, and $a' = ab$ into
  Eqs.~\eqref{eq:crossuniform_n},
  \eqref{eq:crossuniform_m}, \eqref{eq:cancel-ayx},
  and~\eqref{eq:crossbelowy} in order to get, respectively,
  Eqs.~\eqref{eq:axby-cross-n},
    \eqref{eq:axby-cross-m},
    \eqref{eq:cancel-axby}, and~\eqref{eq:crossbelowy-ab}.
\end{proof}

We now check this configuration is easy to find. The following result
is standard.

\begin{lemma}
The action of $\pi_1(S)$ on $\partial_{\infty} S$ is minimal,
i.e., for any $p \in \partial_{\infty} S$, the orbit
$\pi_1(S)\cdot p$ is dense.
\label{lem:minimal}
\end{lemma}

We care about the orientations, so we need a slightly more precise result.

\begin{lemma}
\label{lem:diagonaldense}
Let $x \in \pi_1(S)$.
Given any $z \in \partial_\infty  S$
and any $\epsilon>0$, there exists a lift $g \cdot \ora{x}$ whose
endpoints $g\cdot x^+$ and $g\cdot x^-$ are in $(z-\epsilon,z+\epsilon)$, and are ordered counter-clockwise (or clockwise)
with respect to the $\partial_\infty S$ order.
\end{lemma}
\begin{proof}
By Lemma~\ref{lem:minimal}, endpoints of axes of elements of
$\pi_1(S)$ are dense in $\partial_\infty S$; choose any $g$ so that
$g^+ \in (z-\epsilon,z+\epsilon)$. Next choose any $h$ so that
$\ora{h}$ crosses~$\ora{g}$. Then for sufficiently large~$k$, $h^{\pm k}\cdot
\ora{x}$ will be in a neighborhood of $h^{\pm}$ and in particular will
not cross~$\ora{g}$.
Finally, for large enough $n$, we have
that $g^n h^{\pm k} \cdot x^{\pm} \in (z-\epsilon,z+\epsilon)$.
Tracing through orientations shows that one of these translates will
be clockwise and the other will be counterclockwise.
\end{proof}

\begin{lemma}\label{lem:initial}
  Given $x, y \in \pi_1(S)$, there exist $a,b \in \pi_1(S)$ so that
  $(B \cdot \ora{x}, \ora{y}, a \cdot \ora{x})$ are R-parallel.
\end{lemma}
\begin{proof}
  Immediate from Lemma~\ref{lem:diagonaldense}.
\end{proof}

%%% Local Variables:
%%% mode: latex
%%% TeX-master: "Intersections"
%%% End:

\section{Boxes and their measures}
\label{sec:define_measure}
Our main goal in the paper is to define a geodesic current, a measure
on the space of geodesics, associated to a curve functional~$f$. 
Roughly speaking, our strategy is to associate a positive number, in a
countably additive way, to a collection of boxes of geodesics that is 
sufficiently rich to generate the $\sigma$-algebra of open sets of
$G(\Sigma)$.

To define the measure of our elementary boxes, we look at the
functional~$f$ applied to longer and longer curves~$\vec C_i$. The
measure of a box is a linear combination of $f(\vec C_i)$ chosen so that
the highest growth term vanishes, and the limit of interest picks out
the next asymptotics. This is reminiscent of the construction of Otal
\cite{Otal90:SpectreMarqueNegative}.

\subsection{Right-handed boxes}
\label{sec:boxes}

We now outline the construction of the ``sufficiently rich family of boxes'' on which our measure will be defined.

Let $f\colon \Curves^+(S) \to \mathbb{R}$ be a
functional satisfying symmetry, additive union, smoothing, and stability.
Fix
(arbitrarily) a simple closed curve~$[x]$ which for convenience we
take to pass through the basepoint. The endpoints of the lifts of
$[x]$ are dense in
$\partial_\infty S$ (Lemmas~\ref{lem:minimal} and~\ref{lem:diagonaldense}).

The prototype for
the kind of boxes to which we assign measures is shown in
Figure~\ref{fig:newmodel3}: we have four geodesics, all lifts of~$[x]$, with
two lifts $\ora{\ell_1},\ora{\ell_2}$ on the left and two $\ora{r_1},\ora{r_2}$ on the right,
and consider the window
$[\ell_1^-,\ell_2^-) \times [r_2^+,r_1^+) \subset
\partial_\infty S \times \partial_\infty  S$ as a basic domain
for defining the measure.
(Recall that we give the endpoints of intervals in counterclockwise
order.)

The lifts are defined by a 4-tuple of
elements $[p_1,q_1,p_2,q_2]$ of $\pi_1(S)$ by
$\ell_i = P_i \cdot \ora{x}$ and $r_i = q_i \cdot \ora{x}$, with
conditions explained below.

\begin{figure}
\centering{
\resizebox{100mm}{!}{\Huge{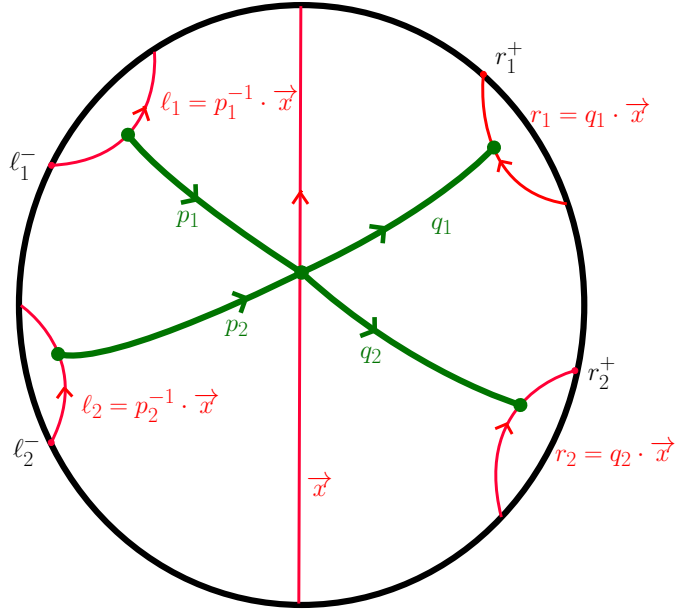}}
\caption{Schematically determining an RH box from group theory data with different lifts of a single curve $\protect[x]$. There are no constraints on how the base lift $\protect\ora{x}$ sits with respect to the other lifts. The other four lifts have constraints in how they sit; some other possibilities are shown in Figure~\ref{fig:rhbox}. Recall from Section~\ref{sec:actions} that the arrows between lifts of $x$ denote
the \emph{right} action of $\pi_1$.}
\label{fig:newmodel3}
}
\end{figure}
The measure that we ultimately assign to the associated windows will
be invariant under
\begin{subequations}\label{eq:changes}
\begin{align}
\text{(changing $p_i$ preserving $\ell_i$)} \quad & p_i \mapsto xp_i, \label{eq:changes-p}\\
\text{(changing $q_i$ preserving $r_i$)} \quad & q_i \mapsto q_ix. \label{eq:changes-q}
\end{align}
\end{subequations}
Invariance under these two transformations is clearly necessary since
they do not change the endpoints of~$\ell_i$ and~$r_i$.
In fact, ultimately we will take a limit of these
transformations: we look at the limit as $n\to \infty$ of $p_i \mapsto x^np_i$
and $q_i \mapsto q_i x^n$.

The measure $\mu$ will also be $\pi_1(S)$ invariant, which translates to
invariance under
(left-)translating all the lifts by $g \in \pi_1(S)$:
  \begin{equation}\label{eq:box-translate}
    (p_1,q_1,p_2,q_2) \mapsto g \cdot (p_1,q_1,p_2,q_2) \coloneqq
      (p_1 G, gq_1, p_2 G,  gq_2).
  \end{equation}

We will put off quotienting by this invariance for the moment.

We now turn to the conditions on the lifts; for
instance, the window
$[\ell_1^-,\ell_2^-) \times [r_2^+,r_1^+)$ must avoid the
diagonal.

\begin{definition}\label{def:RH-lift-system}
  A \emph{right-handed lift system} (RH lift system) is a collection of elements $(p_1,q_1,p_2,q_2)$ in $\pi_1(S)$ so that
  \begin{enumerate}
  \item\label{item:box-nontriv-lift} $p_1P_2$ and $Q_1q_2$ are not powers of~$x$;
  \item\label{item:box-edges-cross-lift} for $i,j \in \{1,2\}$,
    $(\ora{p_iq_j},\ora{x})$ R-cross; and
  \item\label{item:box-diags-cross-lift} for sufficiently large~$n$, $(\ora{q_2x^np_1},\ora{q_1x^np_2})$ R-cross.
  \end{enumerate}
\end{definition}
See Figure~\ref{fig:rhbox} for examples of
RH systems, and Figure~\ref{fig:box4} for a non-example.
 We
distinguish three types of RH systems: \emph{type 1} has
no nested lifts of $[x]$; \emph{type 2} has one parallel pair of nested lifts; and
\emph{type 3} has two parallel pairs of nested lifts.
Note that since $[x]$ is simple, lifts
cannot intersect each other, and nested lifts cannot be anti-parallel
by condition \ref{item:box-edges-cross-lift}.

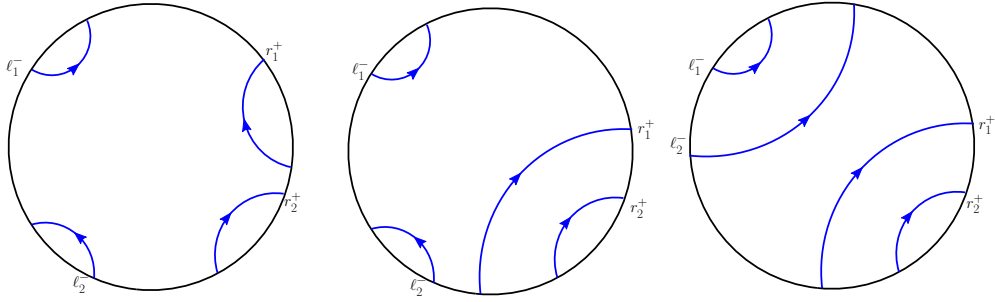
\begin{figure}
     \centering
\resizebox{45mm}{!}{\Huge{%% Creator: Inkscape 1.3 (0e150ed6c4, 2023-07-21), www.inkscape.org
%% PDF/EPS/PS + LaTeX output extension by Johan Engelen, 2010
%% Accompanies image file 'box1.pdf' (pdf, eps, ps)
%%
%% To include the image in your LaTeX document, write
%%   \input{<filename>.pdf_tex}
%%  instead of
%%   \includegraphics{<filename>.pdf}
%% To scale the image, write
%%   \def\svgwidth{<desired width>}
%%   \input{<filename>.pdf_tex}
%%  instead of
%%   \includegraphics[width=<desired width>]{<filename>.pdf}
%%
%% Images with a different path to the parent latex file can
%% be accessed with the `import' package (which may need to be
%% installed) using
%%   \usepackage{import}
%% in the preamble, and then including the image with
%%   \import{<path to file>}{<filename>.pdf_tex}
%% Alternatively, one can specify
%%   \graphicspath{{<path to file>/}}
%% 
%% For more information, please see info/svg-inkscape on CTAN:
%%   http://tug.ctan.org/tex-archive/info/svg-inkscape
%%
\begingroup%
  \makeatletter%
  \providecommand\color[2][]{%
    \errmessage{(Inkscape) Color is used for the text in Inkscape, but the package 'color.sty' is not loaded}%
    \renewcommand\color[2][]{}%
  }%
  \providecommand\transparent[1]{%
    \errmessage{(Inkscape) Transparency is used (non-zero) for the text in Inkscape, but the package 'transparent.sty' is not loaded}%
    \renewcommand\transparent[1]{}%
  }%
  \providecommand\rotatebox[2]{#2}%
  \newcommand*\fsize{\dimexpr\f@size pt\relax}%
  \newcommand*\lineheight[1]{\fontsize{\fsize}{#1\fsize}\selectfont}%
  \ifx\svgwidth\undefined%
    \setlength{\unitlength}{481.88976378bp}%
    \ifx\svgscale\undefined%
      \relax%
    \else%
      \setlength{\unitlength}{\unitlength * \real{\svgscale}}%
    \fi%
  \else%
    \setlength{\unitlength}{\svgwidth}%
  \fi%
  \global\let\svgwidth\undefined%
  \global\let\svgscale\undefined%
  \makeatother%
  \begin{picture}(1,1)%
    \lineheight{1}%
    \setlength\tabcolsep{0pt}%
    \put(0,0){\includegraphics[width=\unitlength,page=1]{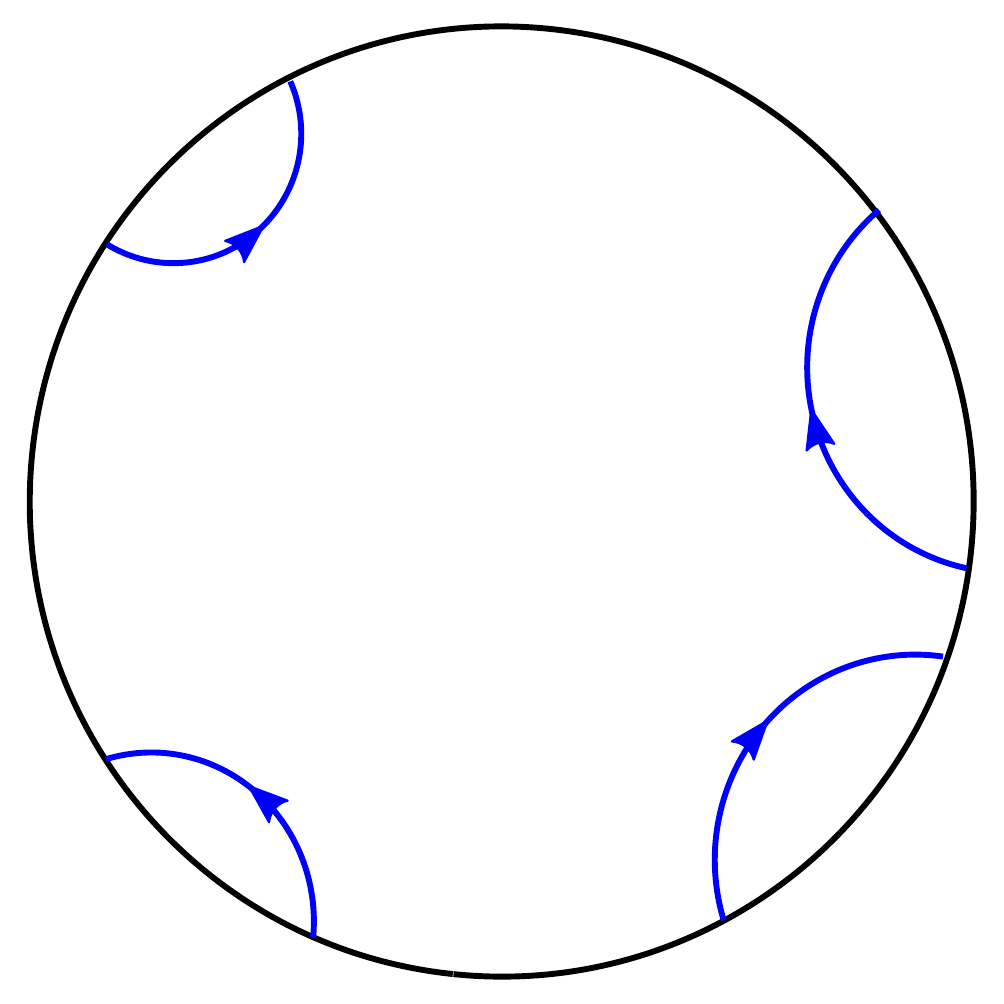}}%
    \put(0.88001856,0.80355029){\color[rgb]{0,0,0}\makebox(0,0)[lt]{\lineheight{1.25}\smash{\begin{tabular}[t]{l}$r_1^+$\end{tabular}}}}%
    \put(0.29728627,0.03719072){\color[rgb]{0,0,0}\makebox(0,0)[rt]{\lineheight{1.25}\smash{\begin{tabular}[t]{r}$\ell_2^-$\end{tabular}}}}%
    \put(0.07747119,0.75287979){\color[rgb]{0,0,0}\makebox(0,0)[rt]{\lineheight{1.25}\smash{\begin{tabular}[t]{r}$\ell_1^-$\end{tabular}}}}%
    \put(0.94091897,0.30420326){\color[rgb]{0,0,0}\makebox(0,0)[lt]{\lineheight{1.25}\smash{\begin{tabular}[t]{l}$r_2^+$\end{tabular}}}}%
  \end{picture}%
\endgroup%
}}
\quad
     \resizebox{45mm}{!}{\Huge{%% Creator: Inkscape 1.3 (0e150ed6c4, 2023-07-21), www.inkscape.org
%% PDF/EPS/PS + LaTeX output extension by Johan Engelen, 2010
%% Accompanies image file 'box.pdf' (pdf, eps, ps)
%%
%% To include the image in your LaTeX document, write
%%   \input{<filename>.pdf_tex}
%%  instead of
%%   \includegraphics{<filename>.pdf}
%% To scale the image, write
%%   \def\svgwidth{<desired width>}
%%   \input{<filename>.pdf_tex}
%%  instead of
%%   \includegraphics[width=<desired width>]{<filename>.pdf}
%%
%% Images with a different path to the parent latex file can
%% be accessed with the `import' package (which may need to be
%% installed) using
%%   \usepackage{import}
%% in the preamble, and then including the image with
%%   \import{<path to file>}{<filename>.pdf_tex}
%% Alternatively, one can specify
%%   \graphicspath{{<path to file>/}}
%% 
%% For more information, please see info/svg-inkscape on CTAN:
%%   http://tug.ctan.org/tex-archive/info/svg-inkscape
%%
\begingroup%
  \makeatletter%
  \providecommand\color[2][]{%
    \errmessage{(Inkscape) Color is used for the text in Inkscape, but the package 'color.sty' is not loaded}%
    \renewcommand\color[2][]{}%
  }%
  \providecommand\transparent[1]{%
    \errmessage{(Inkscape) Transparency is used (non-zero) for the text in Inkscape, but the package 'transparent.sty' is not loaded}%
    \renewcommand\transparent[1]{}%
  }%
  \providecommand\rotatebox[2]{#2}%
  \newcommand*\fsize{\dimexpr\f@size pt\relax}%
  \newcommand*\lineheight[1]{\fontsize{\fsize}{#1\fsize}\selectfont}%
  \ifx\svgwidth\undefined%
    \setlength{\unitlength}{481.88976378bp}%
    \ifx\svgscale\undefined%
      \relax%
    \else%
      \setlength{\unitlength}{\unitlength * \real{\svgscale}}%
    \fi%
  \else%
    \setlength{\unitlength}{\svgwidth}%
  \fi%
  \global\let\svgwidth\undefined%
  \global\let\svgscale\undefined%
  \makeatother%
  \begin{picture}(1,1)%
    \lineheight{1}%
    \setlength\tabcolsep{0pt}%
    \put(0,0){\includegraphics[width=\unitlength,page=1]{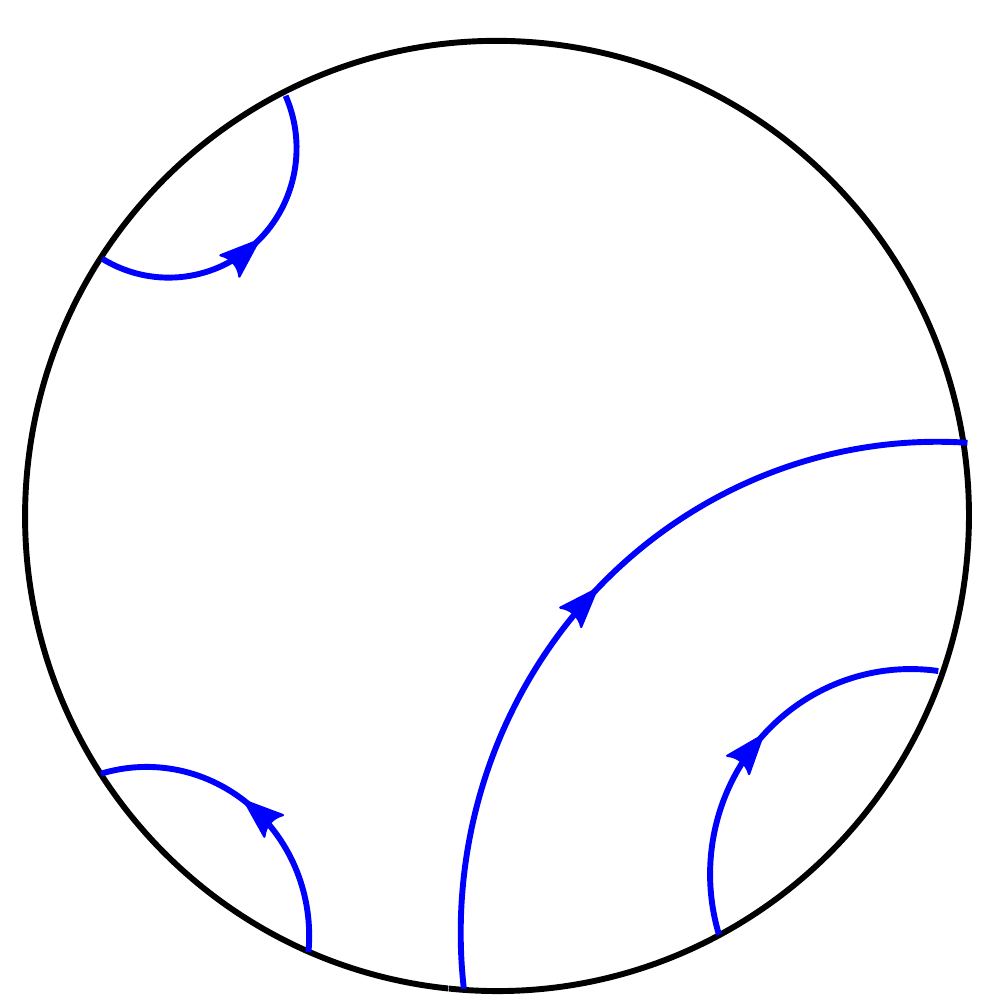}}%
    \put(0.98248639,0.54648341){\color[rgb]{0,0,0}\makebox(0,0)[lt]{\lineheight{1.25}\smash{\begin{tabular}[t]{l}$r_1^+$\end{tabular}}}}%
    \put(0.28460059,0.02678444){\color[rgb]{0,0,0}\makebox(0,0)[rt]{\lineheight{1.25}\smash{\begin{tabular}[t]{r}$\ell_2^-$\end{tabular}}}}%
    \put(0.0933848,0.74149741){\color[rgb]{0,0,0}\makebox(0,0)[rt]{\lineheight{1.25}\smash{\begin{tabular}[t]{r}$\ell_1^-$\end{tabular}}}}%
    \put(0.9579443,0.28503899){\color[rgb]{0,0,0}\makebox(0,0)[lt]{\lineheight{1.25}\smash{\begin{tabular}[t]{l}$r_2^+$\end{tabular}}}}%
  \end{picture}%
\endgroup%
}}
\quad
\resizebox{45mm}{!}{\Huge{%% Creator: Inkscape 1.3 (0e150ed6c4, 2023-07-21), www.inkscape.org
%% PDF/EPS/PS + LaTeX output extension by Johan Engelen, 2010
%% Accompanies image file 'box4.pdf' (pdf, eps, ps)
%%
%% To include the image in your LaTeX document, write
%%   \input{<filename>.pdf_tex}
%%  instead of
%%   \includegraphics{<filename>.pdf}
%% To scale the image, write
%%   \def\svgwidth{<desired width>}
%%   \input{<filename>.pdf_tex}
%%  instead of
%%   \includegraphics[width=<desired width>]{<filename>.pdf}
%%
%% Images with a different path to the parent latex file can
%% be accessed with the `import' package (which may need to be
%% installed) using
%%   \usepackage{import}
%% in the preamble, and then including the image with
%%   \import{<path to file>}{<filename>.pdf_tex}
%% Alternatively, one can specify
%%   \graphicspath{{<path to file>/}}
%% 
%% For more information, please see info/svg-inkscape on CTAN:
%%   http://tug.ctan.org/tex-archive/info/svg-inkscape
%%
\begingroup%
  \makeatletter%
  \providecommand\color[2][]{%
    \errmessage{(Inkscape) Color is used for the text in Inkscape, but the package 'color.sty' is not loaded}%
    \renewcommand\color[2][]{}%
  }%
  \providecommand\transparent[1]{%
    \errmessage{(Inkscape) Transparency is used (non-zero) for the text in Inkscape, but the package 'transparent.sty' is not loaded}%
    \renewcommand\transparent[1]{}%
  }%
  \providecommand\rotatebox[2]{#2}%
  \newcommand*\fsize{\dimexpr\f@size pt\relax}%
  \newcommand*\lineheight[1]{\fontsize{\fsize}{#1\fsize}\selectfont}%
  \ifx\svgwidth\undefined%
    \setlength{\unitlength}{481.88976378bp}%
    \ifx\svgscale\undefined%
      \relax%
    \else%
      \setlength{\unitlength}{\unitlength * \real{\svgscale}}%
    \fi%
  \else%
    \setlength{\unitlength}{\svgwidth}%
  \fi%
  \global\let\svgwidth\undefined%
  \global\let\svgscale\undefined%
  \makeatother%
  \begin{picture}(1,1)%
    \lineheight{1}%
    \setlength\tabcolsep{0pt}%
    \put(0,0){\includegraphics[width=\unitlength,page=1]{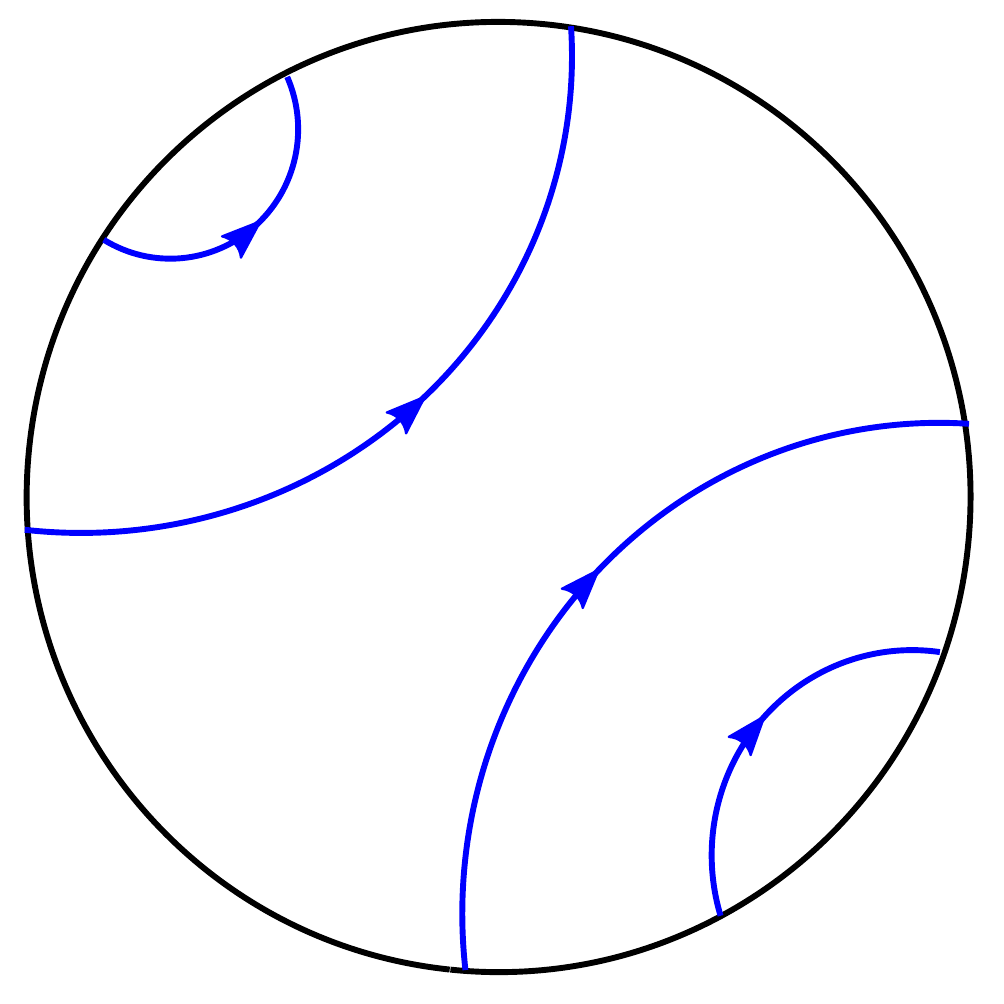}}%
    \put(0.98143101,0.56078439){\color[rgb]{0,0,0}\makebox(0,0)[lt]{\lineheight{1.25}\smash{\begin{tabular}[t]{l}$r_1^+$\end{tabular}}}}%
    \put(0.01627567,0.49701579){\color[rgb]{0,0,0}\makebox(0,0)[rt]{\lineheight{1.25}\smash{\begin{tabular}[t]{r}$\ell_2^-$\end{tabular}}}}%
    \put(0.07660767,0.75579839){\color[rgb]{0,0,0}\makebox(0,0)[rt]{\lineheight{1.25}\smash{\begin{tabular}[t]{r}$\ell_1^-$\end{tabular}}}}%
    \put(0.94116717,0.29933997){\color[rgb]{0,0,0}\makebox(0,0)[lt]{\lineheight{1.25}\smash{\begin{tabular}[t]{l}$r_2^+$\end{tabular}}}}%
  \end{picture}%
\endgroup%
}}
     \caption{Examples of right-handed boxes, of Type 1 (left), Type 2
     (middle), and Type 3 (right).}
     \label{fig:rhbox}
\end{figure}

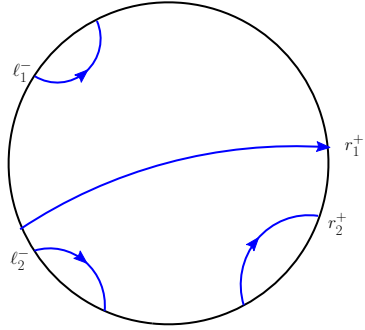
\begin{figure}
  \centering{
\resizebox{45mm}{!}{\Huge{%% Creator: Inkscape 1.3 (0e150ed6c4, 2023-07-21), www.inkscape.org
%% PDF/EPS/PS + LaTeX output extension by Johan Engelen, 2010
%% Accompanies image file 'box3.pdf' (pdf, eps, ps)
%%
%% To include the image in your LaTeX document, write
%%   \input{<filename>.pdf_tex}
%%  instead of
%%   \includegraphics{<filename>.pdf}
%% To scale the image, write
%%   \def\svgwidth{<desired width>}
%%   \input{<filename>.pdf_tex}
%%  instead of
%%   \includegraphics[width=<desired width>]{<filename>.pdf}
%%
%% Images with a different path to the parent latex file can
%% be accessed with the `import' package (which may need to be
%% installed) using
%%   \usepackage{import}
%% in the preamble, and then including the image with
%%   \import{<path to file>}{<filename>.pdf_tex}
%% Alternatively, one can specify
%%   \graphicspath{{<path to file>/}}
%% 
%% For more information, please see info/svg-inkscape on CTAN:
%%   http://tug.ctan.org/tex-archive/info/svg-inkscape
%%
\begingroup%
  \makeatletter%
  \providecommand\color[2][]{%
    \errmessage{(Inkscape) Color is used for the text in Inkscape, but the package 'color.sty' is not loaded}%
    \renewcommand\color[2][]{}%
  }%
  \providecommand\transparent[1]{%
    \errmessage{(Inkscape) Transparency is used (non-zero) for the text in Inkscape, but the package 'transparent.sty' is not loaded}%
    \renewcommand\transparent[1]{}%
  }%
  \providecommand\rotatebox[2]{#2}%
  \newcommand*\fsize{\dimexpr\f@size pt\relax}%
  \newcommand*\lineheight[1]{\fontsize{\fsize}{#1\fsize}\selectfont}%
  \ifx\svgwidth\undefined%
    \setlength{\unitlength}{481.88976378bp}%
    \ifx\svgscale\undefined%
      \relax%
    \else%
      \setlength{\unitlength}{\unitlength * \real{\svgscale}}%
    \fi%
  \else%
    \setlength{\unitlength}{\svgwidth}%
  \fi%
  \global\let\svgwidth\undefined%
  \global\let\svgscale\undefined%
  \makeatother%
  \begin{picture}(1,1)%
    \lineheight{1}%
    \setlength\tabcolsep{0pt}%
    \put(0,0){\includegraphics[width=\unitlength,page=1]{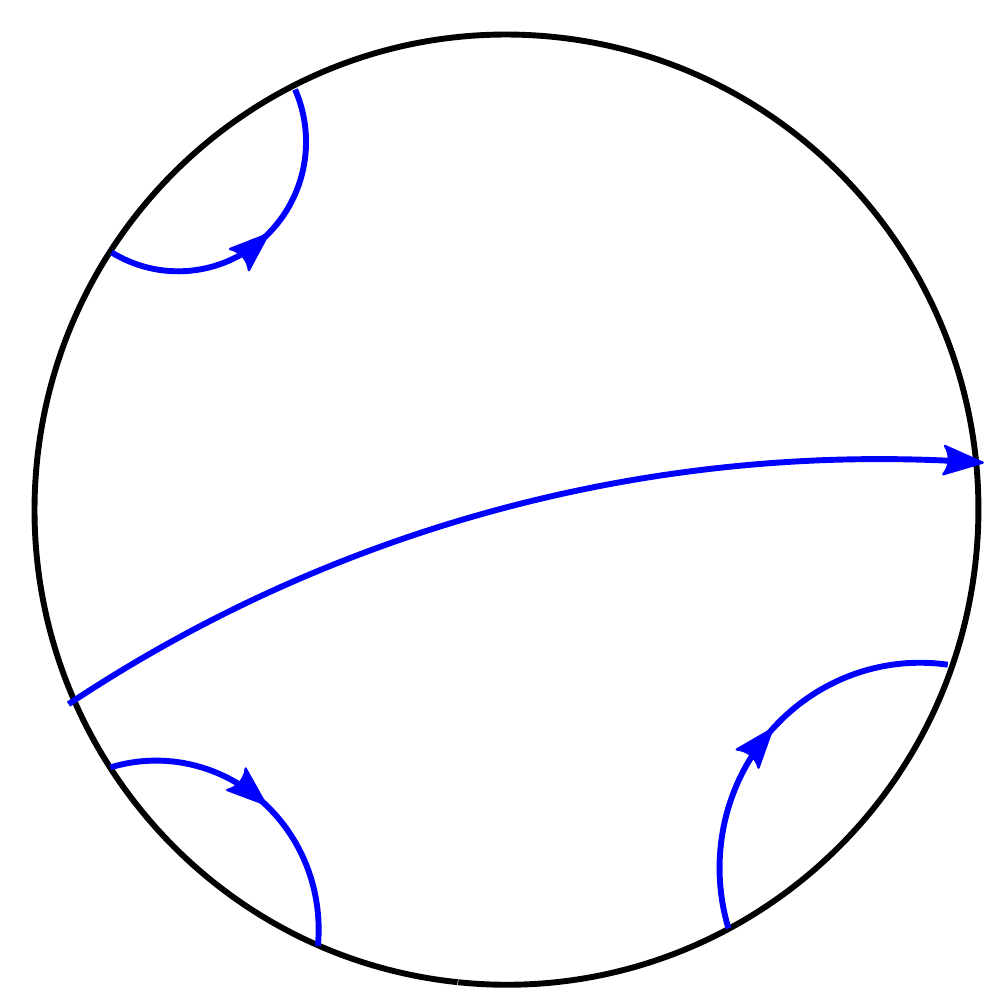}}%
    \put(1.02311016,0.53110157){\color[rgb]{0,0,0}\makebox(0,0)[lt]{\lineheight{1.25}\smash{\begin{tabular}[t]{l}$r_1^+$\end{tabular}}}}%
    \put(0.09568692,0.19691912){\color[rgb]{0,0,0}\makebox(0,0)[rt]{\lineheight{1.25}\smash{\begin{tabular}[t]{r}$\ell_2^-$\end{tabular}}}}%
    \put(0.10256505,0.74812599){\color[rgb]{0,0,0}\makebox(0,0)[rt]{\lineheight{1.25}\smash{\begin{tabular}[t]{r}$\ell_1^-$\end{tabular}}}}%
    \put(0.97341325,0.30110056){\color[rgb]{0,0,0}\makebox(0,0)[lt]{\lineheight{1.25}\smash{\begin{tabular}[t]{l}$r_2^+$\end{tabular}}}}%
  \end{picture}%
\endgroup%
}}
\caption{Non-example of right-handed box. It is a non-example because $\protect\ora{p_2q_1}$ does not cross $\protect\ora{x}$
  essentially (to the right or to the left)}
\label{fig:box4}
}
\end{figure}

\begin{definition}\label{def:RH-box}
  Every right handed lift system $(p_1,q_1,p_2,q_2)$ determines a \emph{right-handed box} of oriented geodesics in $\partial_{\infty} \Sigma \times \partial_{\infty} \Sigma - \Delta$,  given by \[
  B(p_1,q_1,p_2,q_2)=[\ell_1^-,\ell_2^-) \times [r_2^+,r_1^+),\] where $\ora{\ell_i} = P_i \cdot \ora{x}$ and $\ora{r_i} = q_i \cdot \ora{x}$.
The set of all right-handed boxes and the empty set is denoted by $\mathcal{B}$.
\end{definition}

\begin{remark}
\label{rmk:box_from_system}
Modifying the right-handed lift system by changing $p_i$ by
Eq.~\eqref{eq:changes-p} or changing $q_i$ by Eq.~\eqref{eq:changes-q} determines
the same right-handed box. Conversely, choices of the $p_i$/$q_i$
determining the same RH box are related in this way.
\end{remark}
The conditions in Definition \ref{def:RH-lift-system} are chosen to
make the following proposition work.
\begin{proposition}\label{prop:rh-boxes}
  A tuple $(p_1,q_1,p_2,q_2)$ determines an RH lift system iff, with $\ora{\ell_i} =
  P_i \cdot \ora{x}$ and $\ora{r_i} = q_i \cdot \ora{x}$, the endpoints $\ell_1^-, \ell_2^-, r_2^+, r_1^+$ appear in that counter-clockwise order. In particular RH boxes in $\cB$ avoid the diagonal.
\end{proposition}
\begin{proof}
Condition~\ref{item:box-nontriv-lift} guarantees that $\ora{\ell_1} \ne \ora{\ell_2}$ and $\ora{r_1} \ne \ora{r_2}$. Condition~\ref{item:box-edges-cross-lift} implies that for any $i,j \in \{1,2\}$, $\ora{x}$ and $p_iq_j\cdot \ora{x}$ are R-parallel (Lemma~\ref{lem:rhs}) and thus (by translation by $P_i$) that $(\ora{\ell_i}, \ora{r_j})$ are R-parallel, so all endpoints are unique.

  To exploit condition~\ref{item:box-diags-cross-lift}, for $i,j \in \{1,2\}$ and $n \ge 0$, set $g_{i,j,n} \coloneqq q_j x^n p_i$. By Lemma~\ref{lem:convergence_action}, \begin{align}\label{eq:box-endpoints}
    \lim_{n^+} g_{i,j,n}^- &= P_i \cdot x^- = \ell_i^- &
    \lim_{n^+} g_{i,j,n}^+ &= q_j \cdot x^+ = r_j^+.
  \end{align}
  Condition~\ref{item:box-diags-cross-lift} says that for sufficiently large~$n$ the endpoints $g_{1,2,n}^-, g_{2,1,n}^-, g_{1,2,n}^+, g_{2,1,n}^+$ appear in that ccw order, which by Eq.~\eqref{eq:box-endpoints} implies that $\ell_1^-,\ell_2^-,r_2^+,r_1^+$ appear in that order, as desired.

  The converse direction follows the same way.
\end{proof}

\begin{remark}
\label{rmk:largeenough}
  For boxes of type~1, the group elements in condition~\ref{item:box-diags-cross-lift} cross for all~$n$. For boxes of type~2 or type~3, we need the ``sufficiently large'' condition.
\end{remark}

\begin{remark}
  The reader may wonder why we choose to make the boxes half-open in the specified way. It turns out \textit{post facto} we could have made a different choice and it would not matter for the resulting current, as by Proposition~\ref{prop:spikes} the weight of the boundary of these boxes under any geodesic current is necessarily zero: the only lift of closed geodesics with an endpoint starting at one of the endpoints of the box is a single lift of~$[x]$, and by definition, these lifts are not contained in the RH box. The half-open choice we made is more convenient the proof of Proposition~\ref{prop:countablyadditive}.
  \label{rmk:no_atom_boundary_box}
\end{remark}

\begin{proposition}
The family of right handed boxes~$\mathcal{B}$ is dense in the space
of rectangles in $G(S)$, in the following sense. For any
$v_1,v_2,w_2,w_1\in \bdy_\infty S$ in that cyclic order,
and any $\epsilon>0$ so the respective $\epsilon$-neighborhoods are
disjoint, there exist $p_1,p_2,q_1,q_2 \in \pi_1(S)$ forming a
type~1 RH box $B(p_1,q_1,p_2,q_2)$ so that $P_i\cdot x^- \in B_\epsilon(v_i)$
and $q_j\cdot x^+ \in B_\epsilon(w_j)$.

In particular, the interiors of right-handed boxes
form a basis of $G(\widetilde S)$.  
\label{prop:density}
\end{proposition}
\begin{proof}
  By Lemma~\ref{lem:diagonaldense}, we can find translates $\ora{\ell_i} = P_i\cdot \ora{x}$ and $\ora{r_j} = q_j \cdot \ora{x}$ with both endpoints in the respective $\epsilon$-neighborhoods of the $v_i$ and $w_j$. Furthermore, that same lemma lets us arrange the orientations so that $\ora{\ell_i}$ and $\ora{r_j}$ are R-parallel, thus forming an RH lift system of type~$1$.

  To prove the statement about a basis, consider $(v,w) \in G(\widetilde S) \subset (\bdy_\infty S)^2$ and a product neighborhood $(v,w) \subset V \times W \subset G(\widetilde S)$. Choose target endpoints in $V$ on the two sides of~$v$, and target endpoints in $W$ on the two sides of~$w$. Applying the previously proven statement with $\varepsilon$ sufficiently small gives a right-handed box containing $(v,w)$ and contained in $V\times W$.
\end{proof}

We now turn to an alternate parametrization that builds in invariance
under translation by $\pi_1(S)$, as in
Eq.~\eqref{eq:box-translate}. Consider group elements $(a,b,c,d)$ related to the $p_i$ and $q_j$ by
\begin{equation}
\begin{aligned}
  a &= p_1 q_1\qquad&\qquad b &= p_1 q_2\\[2pt]
  c &= p_2 q_1 & d &= p_2 q_2.
\end{aligned}
  \label{eq:relationdatum}
\end{equation}

\begin{figure}
\centering{
\resizebox{100mm}{!}{\Huge{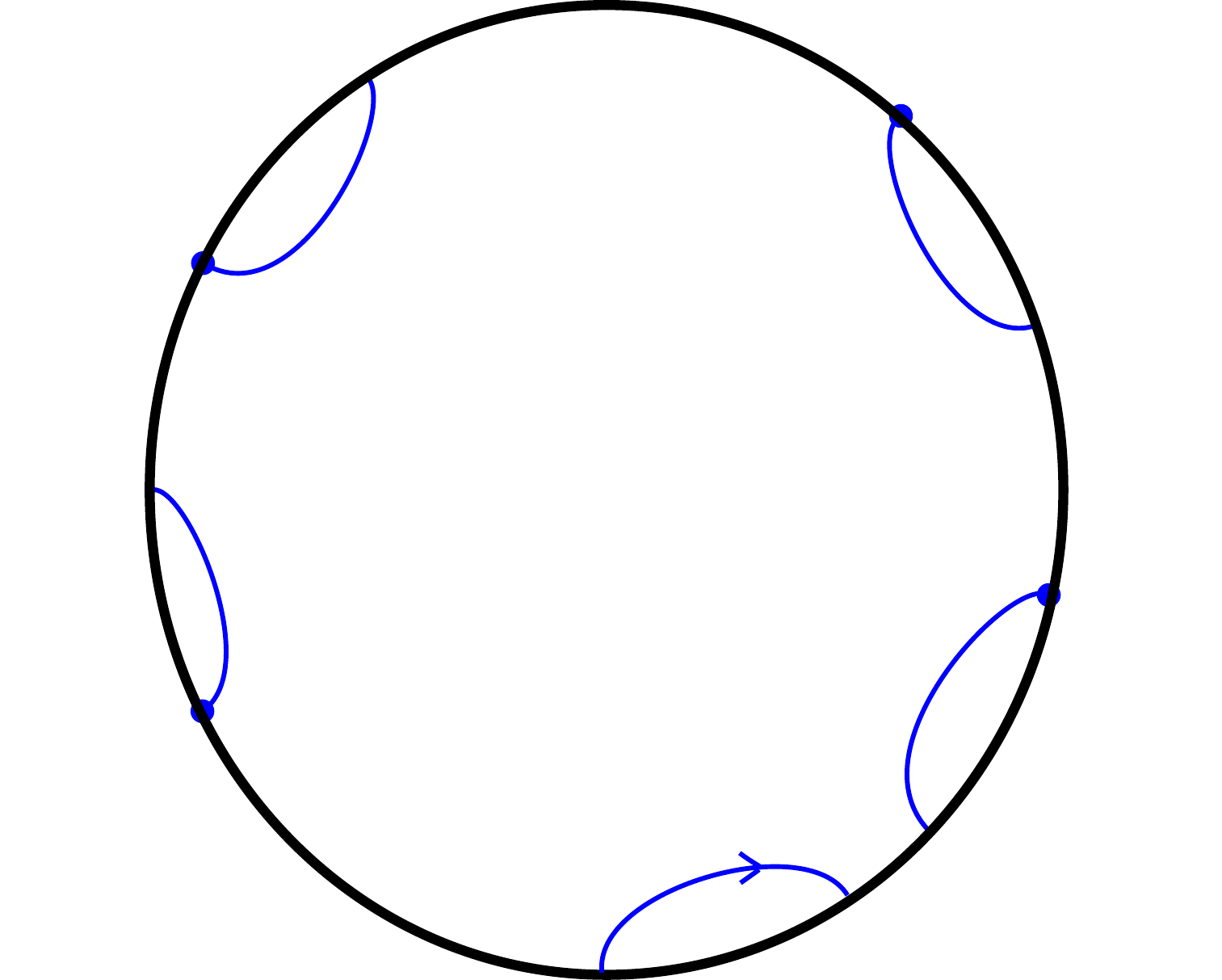}}
\caption{Parametrizing an RH orbit via elements $a,b,c,d$; compare Figure~\ref{fig:newmodel3}. The original position of $\protect\ora{x}$ is arbitrary. This is an orbit of type~1.}
\label{fig:newmodel3_orbit}
}
\end{figure}

Compare Figure~\ref{fig:newmodel3_orbit}. In this parametrization, we necessarily have $aCdB = 1$.
Conversely, given elements $a,b,c,d \in \pi_1(S)$ with
$aCdB = 1$, we can
construct four lifts of $[x]$ related by Equation~\eqref{eq:relationdatum}, for instance
setting $p_1 = 1$, $q_1 = a$, $q_2 = b$, and $p_2 = cA$.

Note that, for instance, the points $\wt \ast \cdot P_1$ and $\wt \ast
\cdot q_1$ are related via $a$ by the right action: $(\wt \ast \cdot
P_1) \cdot a = \wt \ast \cdot q_1$. This is related to the fact that,
by Eq.~\eqref{eq:double_coset}, the relative position
of (say) $\ora{\ell_1}$ and $\ora{r_1}$ is given by the double coset
$\langle x \rangle a \langle x \rangle$.
On the other hand, the lifts of $[x]$ give by $\ora{\ell_1}$ and
$\ora{r_1}$ are related via the element $q_1 p_1$ by the left action:
$q_1 p_1 \cdot \ora{\ell_1} = \ora{r_1}$. Furthermore,
$(\ora{\ell_1},\ora{r_1})$ are R-parallel, which is equivalent (by
Lemma~\ref{lem:rhs}) to $\ora{q_1p_1}$ R-crossing $\ora{p_1}$, or
equivalently $p_1 \cdot \ora{q_1 p_1} = \ora{p_1 q_1} = \ora{a}$
R-crossing~$\ora{x}$.

The action of translating all lifts by $\pi_1(S)$ has been baked in to
this alternate parametrization,
but the measure corresponding to the associated window should be
invariant under the following four transformations:
\begin{subequations}\label{eq:translations}
\begin{align}
\text{(translating along $\ell_1$)} \quad & a \mapsto xa,\quad b \mapsto xb \label{eq:translations-l1}\\
\text{(translating along $\ell_2$)} \quad & c \mapsto xc,\quad d \mapsto xd \label{eq:translations-l2}\\
\text{(translating along $r_1$)} \quad & a \mapsto ax,\quad c \mapsto cx \label{eq:translations-r1}\\
\text{(translating along $r_2$)} \quad & b \mapsto bx,\quad d \mapsto dx \label{eq:translations-r2}
\end{align}
\end{subequations}

\begin{definition}[RH orbit]
\label{def:RH-lift-sys}
Let $a,b,c,d \in \pi_1(S,\ast)$ be such that:
\begin{enumerate}[label=(\arabic*${}'$),start=0]
\item\label{item:box-close} $a C d B = 1$;
\item\label{item:box-nontriv} $aC\mathrel{(=}bD)$ and $Ab\mathrel{(=}Cd)$ are not powers of~$x$;
\item\label{item:box-edges-cross} $\ora{a},\ora{b},\ora{c},\ora{d}$ cross $\ora{x}$ to the right; and
\item\label{item:box-diags-cross} for sufficiently large~$n$, $\ora{dx^na C}
  \mathrel{(=} \ora{cAbx^na C})$ crosses $\ora{cx^n}$ to the right.
\end{enumerate}

Then we call the datum $[a,b,c,d]$ a \emph{right-handed lift orbit} (or \emph{RH orbit}, for short).
\end{definition}

As explained above, this determines a $\pi_1(S)$-orbit of RH lift systems.
The one non-obvious point to check is that
condition~\ref{item:box-diags-cross} corresponds
to condition~\ref{item:box-diags-cross-lift}. The two crossing
elements in condition~\ref{item:box-diags-cross} are
\begin{align*}
   dx^naC &= p_2 q_2 x^n p_1 P_2 & cx^n &= p_2 q_1 x^n\\
  \intertext{which after conjugating by $P_2$ yields}
  q_2 x^n p_1 &= g_{1,2,n} & q_1 x^n p_2 &= g_{2,1,n}.
\end{align*}
Thus the two crossing conditions are equivalent.

\begin{definition}
  Let $B = B(p_1,q_1,p_2,q_2)$ be an RH box. The \emph{$\pi_1(S)$-orbit of $B$} is
  \[
    [B] \coloneqq \{ g \cdot B \colon g \in \pi_1(S) \},
  \]
  where we consider the left action. We write $[B(p_1,q_1,p_2,q_2)]$ to emphasize a chosen representative, or $[B]_{a,b,c,d}$ to emphasize the RH orbit.
\end{definition}

\begin{remark}\label{rmk:flip_rh_box}
Even though in the sequel we will fix our base element to be $x$ once and for all, we observe that flipping an RH orbit with respect to $x$ gives a new RH orbit with respect to $x^{-1}$, in the following sense.
Let $[a,b,c,d]$ define an RH orbit with respect to the fixed element $x \in \pi_1(S)$.
Then the datum $[a', b', c', d']$ with $a' \coloneqq D$, $b' \coloneqq B$, $c'\coloneqq C, d'\coloneqq A$ is an RH orbit with respect to the base element $x' \coloneqq X$. Similarly, if $[p_1,q_1,p_2,q_2]$ is an RH lift system in the RH orbit $[a,b,c,d]$ with base element $x$, the datum $[p_1',q_1',p_2',q_2']$
where $p_1' \coloneqq Q_2$, $p_2' \coloneqq Q_1, q_1' \coloneqq P_2, q_2' \coloneqq P_1$
is an RH system with RH orbit $[a',b',c',d']$ and basepoint $x'$.
From this, one immediately has
\[B_x(p_1,q_1,p_2,q_2)=\sigma\bigl(B_{x'}(p_1',q_1',p_2',q_2')\bigr),\]
where $\sigma$ is the flip map, and we introduce notation $B_x$ to be
explicit about the
base element~$x$ in the construction of the RH box.
\end{remark}

\subsection{Measures of boxes}

We will now start defining the measure of an RH orbit $[B]_{a,b,c,d}$:
we will define a function
$\wh\mu_f \colon \mathcal{B} \to \mathbb{R}_{\geq 0}$ in terms of a limit
of values of~$f$.

\begin{proposition}
\label{prop:limitexists}
Let $B(p_1,q_1,p_2,q_2)$ be a right-handed box so that $[B(p_1,q_1,p_2,q_2)]=[B]_{a,b,c,d}$.
The following limit exists and only depends on the $\pi_1(S)$-orbit
of the box:
\begin{equation}
\wh\mu_{f}(B(p_1,q_1,p_2,q_2))\coloneqq\frac{1}{4}\lim_{n \to \infty}f([bx^n]) + f([cx^n])-f([ax^n]) - f([dx^n]).
\label{eq:def_measure}
\end{equation}
\end{proposition}
This function on boxes (which will eventually become a pre-measure and
then a measure) is automatically (left) $\pi_1(S)$-invariant.

To prove the limit exists, we
will use
the properties
of~$f$ from Section~\ref{subsec:mainhypotheses} and the
results from
Section~\ref{sec:fund-smoothing}.
We
first look at a key ingredient limit.

\begin{lemma}\label{lem:additivity-limit}
Let $f$ be a curve functional satisfying symmetry, additivity, smoothing, and
stability.
Let $a,x \in \pi_1(S)$, with $\ora{a}$ crossing $\ora{x}$ to the right. Then
  \begin{equation}\label{eq:additivity-limit}
    \lim_{n \to \pm\infty} f([a x^{n}]) - f([x^n])
  \end{equation}
  exists.
\end{lemma}

\begin{proof}
  By Eq.~\eqref{eq:cross-join}, with $b=a$, there exists $r>0$ so that, for all $n \in \mathbb{Z}$,
  \[
    f([ax^{n-r}]) + f([ax^{n+r}]) \geq 2 f([ax^n]).
  \]
  (If $x$ is simple, we can take $r=1$.) Here we used oriented smoothing and convex union.

  Hence, $n \mapsto f([a x^{n}])$ is $r$-convex. For $n > 0$, so is $n \mapsto f([a x^n]) - n f([x]) = f([a x^n]) - f([x^n])$. (Here we used stability of~$f$.)
 
  By Lemma~\ref{lem:rhs_prod}, additivity, and 
 disconnected smoothing of~$f$, we get, for $n \in \bbZ$,
 \[
 f([a]) + f([x^{n}]) \geq f([ax^{n}]).
 \]
 By Eq.~\eqref{eq:cross-cancel} with $m=0$,
additivity, symmetry, and unoriented disconnected
smoothing of $f$, there is $N \in \mathbb{Z}$ (depending on $a,x$) so that, for $\abs{n} > N$,
  \[
    f([a x^{n}]) \ge -f([a]) + f([x^{ n}]).
  \] 
  This shows that for $\abs{n} > N$, the function $n \mapsto f([ax^n]) - f([x^n])$ is $r$-convex and bounded, so by Lemma~\ref{lem:limexistrmidconvex}, the limit as $n \to \infty$ in \eqref{eq:additivity-limit} exists. The limit as $n \to -\infty$ is similar.
\end{proof}

\begin{remark}
The limit in Lemma~\ref{lem:additivity-limit} need not exist if one
only assumes $f$ satisfies quasi-smoothing.
We give an example on the once-holed torus.
Let $S$ be the once-punctured torus with generating set $\{a,b,b^2\}$,
where $a$ and $b$ are the standard generators of $\pi_1(S)$, chosen so
$\ora{a}$ crosses $\ora{b}$ to the right.
Let $f$ be the stable word-length with respect to this generating set.
Note that $f$ satisfies quasi-smoothing instead of smoothing (see \cite[Example~4.10]{MGT21:Smoothings} for details).
By Lemma~\ref{lem:rhs_prod}, $\ora{ab}$ also crosses $\ora{b}$ to the
right.
The expression in the limit in \eqref{eq:additivity-limit} for the
pair $ab$ and $b$ is
\begin{equation*}
  f([ab^{n+1}])-f([b^n])=\begin{cases} 1+ \frac{n+2}{2} - \frac{n}{2}=2& \mbox{ if } n \mbox{ is even } \\
    1+ \frac{n+1}{2} - \frac{n+1}{2}=1 & \mbox{ if } n \mbox{ is odd }
  \end{cases}
\end{equation*}
which does not approach a limit.
This example can be promoted to one on the genus-two surface.
\label{rmk:nonconvex}
\end{remark}

We also need a two-variable version of
Lemma~\ref{lem:additivity-limit}. For later use, we allow two
different elements $x,y \in \pi_1(S)$.
\begin{lemma}\label{lem:additivity-limit-2}
  Let $a,b,x,y \in \pi_1(S)$ be so that $(B \cdot \ora{x}, \ora{y}, a \cdot \ora{x})$ are R-parallel, and let $f \colon \Curves^+(S) \to \bbR$ satisfy symmetry, additivity, smoothing, stability, and homogeneity. Then the limits
  \begin{gather*}
    \lim_{n,m \to \infty} f([ax^nby^m]) - nf([x]) - mf([y])\\
    \lim_{n,m \to \infty} f([ax^nby^{-m}]) - nf([x]) - mf([y])
  \end{gather*}
  exist.
\end{lemma}
\begin{proof}
  For $n,m \ge 0$, define $g_
  {\pm}(n,m) \coloneqq f([ax^nby^{\pm m}]) - nf([x]) - mf([y])$. By Corollary~\ref{cor:fund-axby}, there is some $s$ so that both $g_+$ and $g_-$ are $(s,s)$-axis-convex: apply $f$ to the left hand side of Eqs.~\eqref{eq:axby-cross-n} and~\eqref{eq:axby-cross-m}, and use convex union, oriented smoothing, stability, and homogeneity. (The extra linear terms in $g_{\pm}$ are irrelevant for convexity.)

  Next we check that $g_{\pm}$ are bounded. On the one hand,
  \[
    [ab][x^n][y^{\pm m}] \reducesto [ax^nb][y^{\pm m}] \reducesto [ax^nby^{\pm m}].
  \]
  For the first reduction, the hypotheses imply that
  $(\ora{x}, ba\cdot\ora{x})$ are R-parallel, and thus $(\ora{ba},\ora{x})$ R-cross by Lemma~\ref{lem:rhs}. For the second one, note that $ax^n b\cdot(B\cdot x^+, B\cdot x^-) = (a\cdot x^+, a\cdot x^-)$; then a look at the bands model centered on $\ora{y}$ shows that $\ora{a x^n b}$ crosses $\ora{y}$.   Applying $f$ and using additivity and oriented smoothing for $g_+$ and unoriented smoothing for $g_-$ shows that
  \[
    g_{\pm}(n,m) \le f([ab]).
  \]

  To get a lower bound for $g_+$, applying $f$ to both sides of Eq.~\eqref{eq:cancel-axby} and using additivity and unoriented smoothing shows that, for $n, m$ sufficiently large,
  \[
    g_+(n,m) \ge -f([ab]).
  \]
 To get a (more complicated) lower bound for $g_-$, similarly apply $f$ to both sides of
 Eq.~\eqref{eq:crossbelowy-ab}.

   Thus by Lemma~\ref{lem:limexistrmidconvextwovar} the limit exists.
\end{proof}

\begin{proposition}\label{prop:join-split-alg}
Let $a,b,x\in \pi_1(S)$ and let $f \colon \Curves^+(S) \to \bbR$ satisfy symmetry, additivity, smoothing, stability, and homogeneity. 
If $a$ and $b$ cross $x$ to the right, then
  \[
    \lim_{n,m\to\infty} f([a x^n b x^m]) - f([ax^n]) - f([bx^m]) = 0.
  \]
\end{proposition}
\begin{proof}
  We must now show that the limits as $n,m \to \infty$ of
  \begin{align*}
    g(n,m) &= f([ax^nbx^m]) - (n+m)f([x])\\
    \shortintertext{and}
    h(n,m) &= f([ax^n]) + f([bx^m]) - (n+m)f([x])
  \end{align*}
  exist and are equal.
  First, Lemmas~\ref{lem:additivity-limit} and~\ref{lem:additivity-limit-2} imply that the limits exist. Equation \eqref{eq:cross-join} says that there are arbitrarily large $n,m$ so that $g(n',m') \le h(n,m)$ with $n'+m' = n+m$, and $n',m'$ closer to the diagonal. This immediately implies $\lim_{n,m \to\infty} g(n,m) \le \lim_{n,m\to\infty} h(n,m)$. (Take $n,m$ large enough so that both $h(n,m)$ and $g(n',m')$ are within $\epsilon$ of their limiting value.) Conversely, Equation~\eqref{eq:cross-split} implies the opposite inequality.
\end{proof}

\begin{proof}[Proof of Proposition~\ref{prop:limitexists}]
The existence of the limit in Equation~\eqref{eq:def_measure} follows from Lemma~\ref{lem:additivity-limit}.
Since this definition depends on an RH orbit $(a,b,c,d)$---which is invariant under the left action of $\pi_1(S)$---it suffices to show that $\wh\mu_f([B(p_1,q_1,p_2,q_2)])$ depends only on the box $B(p_1,q_1,p_2,q_2)$.
To this end, by Remark~\ref{rmk:box_from_system}, it is enough to prove that the limit is invariant under the transformations described in Equations~\eqref{eq:translations}.
We prove this for Equation~\eqref{eq:translations-l1}; the other cases
are analogous:

\begin{align*}
\wh\mu_f(B(xp_1,q_1,p_2,q_2))
&= \frac{1}{4}\lim_{n \to \infty}
\Big( f([bx^{n+1}]) + f([cx^n]) - f([ax^{n+1}]) - f([dx^n]) \Big)
&& \text{(Definition)} \\
&= \frac{1}{4}\lim_{n,m \to \infty}
\Big( f([bx^{m}]) + f([cx^n]) - f([ax^{m}]) - f([dx^n]) \Big)
&& \text{(Lemma~\ref{lem:additivity-limit})} \\
&= \frac{1}{4}\lim_{n \to \infty}
\Big( f([bx^{n}]) + f([cx^n]) - f([ax^{n}]) - f([dx^n]) \Big)
&&\\
&= \wh\mu_f(B(p_1,q_1,p_2,q_2)).
\end{align*}
Hence, $\mu_f$ only depends on the RH orbit.
\end{proof}

\begin{remark} 
 The existence of the double limits in Lemma~\ref{prop:join-split-alg} (and the other limits in this section) is reminiscent of the existence of
 similar limits in the context of the eigenvalue projection function
 on reductive Lie groups (see~\cite[Lemma~7.10]{BQ16:Random}).
 See also Example~\ref{examp:higher-rank}.
 \end{remark}
 
\subsection{Elementary properties of box measure}

We check the hypotheses to apply
Theorem~\ref{thm:caratheodory} to construct a measure. In this
subsection we will check elementary properties, namely positivity,
side-by-side additivity, and finite additivity. (This includes showing
that $\cB$ is a semi-ring.)

In the next section we
will move on to more involved properties such as countable
(sub)additivity.

\begin{proposition}
  \label{prop:box-positive}
  For a right-handed box $B\in \mathcal{B}$, we have $\wh\mu_f(B) \geq 0$.
\end{proposition}

\begin{proof}
  Set $[B] = [B]_{a,b,c,d}$. By condition~\ref{item:box-diags-cross}, for sufficiently large~$n$,
  \[
    [bx^n][cx^n] = [dx^naC][cx^n] \reducesto [dx^nax^n].
  \]
  and thus
\[
f([bx^n]) + f([cx^n]) -f([dx^nax^n]) \geq 0.
\]
  Then by Proposition~\ref{prop:join-split-alg} and taking the limit as $n \to \infty$, we obtain $\wh\mu_f(B) \geq 0$.
\end{proof}

\begin{proposition}
  \label{prop:sidebyside}
  If two RH boxes $B_1$ and $B_2$ fit together side-by-side as in Figure~\ref{fig:finiteadditivity} (or the symmetric version changing the $p_i$), then $B_1 \cup B_2$ is an RH box and $\wh\mu_f(B_1 \cup B_2) = \wh\mu_f(B_1) + \wh\mu_f(B_2)$.
\end{proposition}

\begin{proof}
  Set WLOG $B_1 =  B(p_1,q_1,p_2,q_2)$ and $B_2 = B(p_1,q_2,p_2,q_3)$, and set $\ora{\ell_i} = P_i \cdot \ora x$ and $\ora{r_j} = q_j \cdot \ora{x}$. The hypotheses imply that $\ell_1^-,\ell_2^-,r_3^+, r_2^+, r_1^+$ occur in that cyclic order, so in particular $B_1 \cup B_2 = B(p_1,q_1,p_2,q_3)$ is also an RH box. The definition of $\wh\mu_f$ immediately yields the additivity.
\end{proof}

\begin{figure}
\centering{
\resizebox{60mm}{!}{\Huge{%% Creator: Inkscape 1.3 (0e150ed6c4, 2023-07-21), www.inkscape.org
%% PDF/EPS/PS + LaTeX output extension by Johan Engelen, 2010
%% Accompanies image file 'finiteadditivity.pdf' (pdf, eps, ps)
%%
%% To include the image in your LaTeX document, write
%%   \input{<filename>.pdf_tex}
%%  instead of
%%   \includegraphics{<filename>.pdf}
%% To scale the image, write
%%   \def\svgwidth{<desired width>}
%%   \input{<filename>.pdf_tex}
%%  instead of
%%   \includegraphics[width=<desired width>]{<filename>.pdf}
%%
%% Images with a different path to the parent latex file can
%% be accessed with the `import' package (which may need to be
%% installed) using
%%   \usepackage{import}
%% in the preamble, and then including the image with
%%   \import{<path to file>}{<filename>.pdf_tex}
%% Alternatively, one can specify
%%   \graphicspath{{<path to file>/}}
%% 
%% For more information, please see info/svg-inkscape on CTAN:
%%   http://tug.ctan.org/tex-archive/info/svg-inkscape
%%
\begingroup%
  \makeatletter%
  \providecommand\color[2][]{%
    \errmessage{(Inkscape) Color is used for the text in Inkscape, but the package 'color.sty' is not loaded}%
    \renewcommand\color[2][]{}%
  }%
  \providecommand\transparent[1]{%
    \errmessage{(Inkscape) Transparency is used (non-zero) for the text in Inkscape, but the package 'transparent.sty' is not loaded}%
    \renewcommand\transparent[1]{}%
  }%
  \providecommand\rotatebox[2]{#2}%
  \newcommand*\fsize{\dimexpr\f@size pt\relax}%
  \newcommand*\lineheight[1]{\fontsize{\fsize}{#1\fsize}\selectfont}%
  \ifx\svgwidth\undefined%
    \setlength{\unitlength}{455.92301496bp}%
    \ifx\svgscale\undefined%
      \relax%
    \else%
      \setlength{\unitlength}{\unitlength * \real{\svgscale}}%
    \fi%
  \else%
    \setlength{\unitlength}{\svgwidth}%
  \fi%
  \global\let\svgwidth\undefined%
  \global\let\svgscale\undefined%
  \makeatother%
  \begin{picture}(1,1.00664704)%
    \lineheight{1}%
    \setlength\tabcolsep{0pt}%
    \put(0,0){\includegraphics[width=\unitlength,page=1]{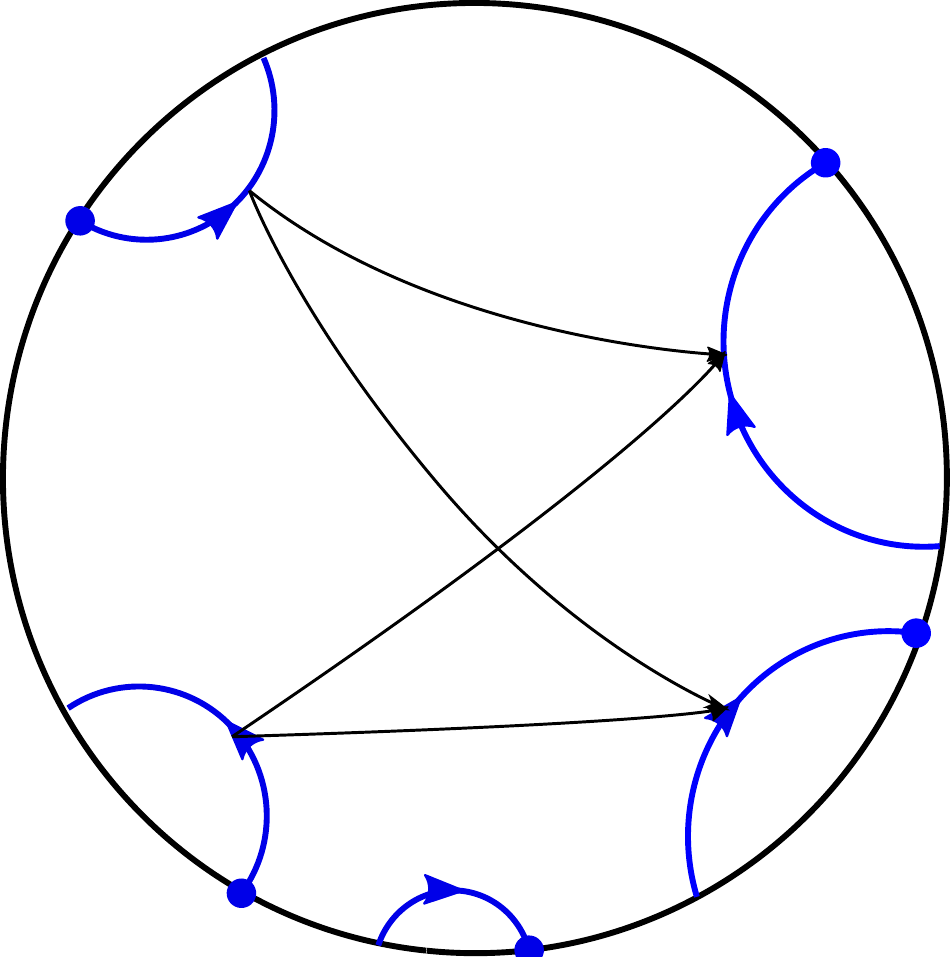}}%
    \put(0.4652804,0.7078909){\color[rgb]{0,0,0}\makebox(0,0)[lt]{\lineheight{1.25}\smash{\begin{tabular}[t]{l}$q_1p_1$\end{tabular}}}}%
    \put(0.41851952,0.5697204){\color[rgb]{0,0,0}\makebox(0,0)[lt]{\lineheight{1.25}\smash{\begin{tabular}[t]{l}$q_2p_1$\end{tabular}}}}%
    \put(0.40688751,0.31547788){\color[rgb]{0,0,0}\makebox(0,0)[lt]{\lineheight{1.25}\smash{\begin{tabular}[t]{l}$q_1p_2$\end{tabular}}}}%
    \put(0.4351367,0.20248129){\color[rgb]{0,0,0}\makebox(0,0)[lt]{\lineheight{1.25}\smash{\begin{tabular}[t]{l}$q_2p_2$\end{tabular}}}}%
    \put(0,0){\includegraphics[width=\unitlength,page=2]{finiteadditivity.pdf}}%
    \put(0.33709546,0.44675348){\color[rgb]{0,0,0}\makebox(0,0)[lt]{\lineheight{1.25}\smash{\begin{tabular}[t]{l}$q_3p_1$\end{tabular}}}}%
    \put(0.2974471,0.12629042){\color[rgb]{0,0,0}\makebox(0,0)[lt]{\lineheight{1.25}\smash{\begin{tabular}[t]{l}$q_3p_2$\end{tabular}}}}%
  \end{picture}%
\endgroup%
}}
\caption{Side-by-side additivity for boxes}%
\label{fig:finiteadditivity}%
}
\end{figure}

To show that this side-by-side additivity shows that $\wh\mu_f$ is finitely
additive, and furthermore that the allowable boxes in $\mathcal{B}$ form a semi-ring,
we examine exactly which boxes are allowed, making concrete the
constraints in Definition~\ref{def:RH-lift-sys}. Boxes
are bounded at positive or negative ends of translates of $\ora{x}$,
i.e., at $P_i \cdot x^-$ (also referred to as $\ell_i^-$) or at
$q_j \cdot x^+$ (also referred to as $r_j^+$).
Conditions~\ref{item:box-nontriv-lift} and \ref{item:box-diags-cross-lift}
together say that $\ell_1^-,\ell_2^-,r_2^+,r_1^+$ are distinct points
in that cyclic order. It remains to understand
Condition~\ref{item:box-edges-cross-lift}, which is a
condition on relative positions of the $\ora{\ell_i}$ and $\ora{r_j}$
for each pair separately. 

Consider a single pair
$(\ora{\ell},\ora{r}) = (P \cdot \ora{x}, q \cdot \ora{x})$. We claim
this pair is R-parallel iff $r^+ \in (\ell^-,\ell^+)$ and
$\ell^- \in (r^+, r^-)$. Indeed, $\ora{\ell}$ and $\ora{r}$ cannot
cross (since they are both lifts of $[x]$, which is simple), and then
checking the cases shows that $r^+ \in (\ell^-,\ell^+)$ iff
$(\ora{\ell},\ora{r})$ are R-parallel or R-anti-parallel, and
$\ell^- \in (r^+,r^-)$ iff $(\ora{\ell},\ora{r})$ are R-parallel or
L-anti-parallel.

To visualize these conditions better, draw the space of oriented
geodesics 
$G(S) = \partial_\infty \Sigma \times \partial_\infty \Sigma \setminus \Delta$ (Definition~\ref{def:spacegeodesics})
in the
usual way as a square with opposite sides identified. We represent
$(r^+,r^-)$ by drawing a ``whisker'': a vertical segment at
$x$-coordinate $r^+$ from $y$-coordinate $r^+$ (on the diagonal) to $r^-$.
Similarly represent $(\ell^-,\ell^+)$ by a
horizontal segment at
$y$-coordinate $\ell^-$ from $x$-coordinate $\ell^-$ to $\ell^+$.
Then $(\ora{\ell}, \ora{r})$
are R-parallel iff these two segments intersect, as in
Figure~\ref{fig:whiskers}.
Lemma~\ref{lem:diagonaldense} says that there are dense sets of these
horizontal and vertical whiskers that extend from the diagonal almost
all the way back to the diagonal.

Conversely, whiskers allow us to characterize RH boxes.
\begin{lemma} \label{lem:whiskertobox} A half-open rectangle whose four boundary sides lie on two horizontal and two vertical
whiskers is an RH box exactly when the four chosen endpoints occur in ccw order and each horizontal whisker
meets each vertical whisker.
\end{lemma}
\begin{proof}
  This follows from the previous discussion.
\end{proof}

  \begin{figure}
\centering{
\resizebox{120mm}{!}{\fontsize{9pt}{9pt}\selectfont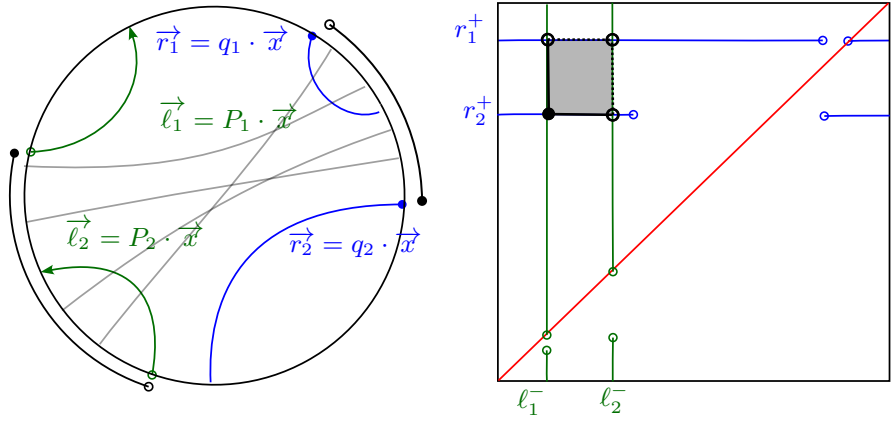}
\caption{Example of the correspondence between right-handed boxes and collections of four whiskers.}
\label{fig:whiskers}
}
\end{figure}

\begin{proposition}
  \label{prop:RHsemiring}
  The set of RH boxes $\mathcal{B}$ is a semi-ring, in the sense of Definition~\ref{def:semiring}.
\end{proposition}

\begin{proof}
  First, the intersection of boxes is a box. Indeed, all sides of a box are necessarily in the respective whisker, and any new vertices of the intersection are thus intersections of whiskers. Thus, Lemma~\ref{lem:whiskertobox} gives an RH box.

  Second, we must show that if $A,B$ are boxes, then $A \setminus B$ is a finite union of boxes. In the interesting case, $A\setminus B$ is an L-shaped region, as on the left of Figure~\ref{fig:box-difference}. (This shows the hardest case of the orientation of the region.) If necessary, add a horizontal whisker slightly above the horizontal inner edge of the L, making sure that it intersects the vertical whisker forming the other inner edge of the~L; this is possible by Lemma~\ref{lem:initial}. The resulting whiskers allow us to write the L as the union of three boxes, as shown on the right of Figure~\ref{fig:box-difference}.
\end{proof}

\begin{proposition}
  \label{prop:finitely-additive}
  The function $\wh\mu_f$ is finitely additive on boxes.
\end{proposition}

  \begin{figure}
\centering{
\resizebox{120mm}{!}{\fontsize{9pt}{9pt}\selectfont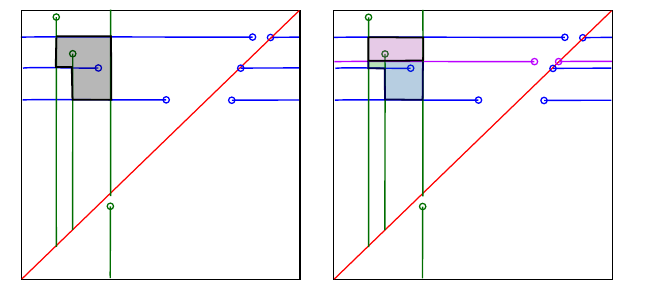}
\caption{Decomposing a difference of boxes. Left: A possible shape for $A \setminus B$. Right: Adding a horizontal whisker allows us to write the region as a union of boxes.
  }
\label{fig:box-difference}
}
\end{figure}

\begin{proof}
  Given a division $B_0 = \bigsqcup_i B_i$, we will show that the $B_i$ can be further divided so that the result can be joined by side-by-side additivity to form $B_0$. Let $W$ be the set of all whiskers involved in any side of any of the~$B_i$, excluding the whiskers supporting the boundary of $B_0$. We proceed by induction on~$\abs{W}$.

 Whiskers in $W$ can be ``long'', passing through the interior of the upper or right edge of~$B_0$, or not.  (Such long whiskers necessarily also pass through the bottom or left edge as well.)
  Suppose first that there are no long whiskers in~$W$. Then the partition is in fact trivial, since the only possible box with no long whiskers that contains a point near the upper-right corner of~$B_0$ is $B_0$ itself.

  Otherwise, there is at least one long whisker. Pick one such long whisker and use it to write $B_0 = B_1' \sqcup B_2'$. Subdivide any of the $B_i$ that intersect both $B_1'$ and $B_2'$, in a side-by-side way, getting sets of boxes $B_{i,1}$ and $B_{i,2}$, which are partitions of $B_1'$ and $B_2'$ respectively. These give subdivisions of $B_1'$ and $B_2'$ that each involve fewer total whiskers: indeed, the chosen whisker is now part of the boundary of each $B_j'$ and hence is not counted in the corresponding induction parameter. Thus by induction and Proposition~\ref{prop:sidebyside},
  \[
    \wh\mu_f(B_0) = \wh\mu_f(B_1') + \wh\mu_f(B_2') = \sum_j \wh\mu_f(B_{j,1}) +
    \sum_k \wh\mu_f(B_{k,2}) = \sum_i \wh\mu_f(B_i). \qedhere
  \]
\end{proof}

 \begin{figure}
\centering{
\resizebox{120mm}{!}{\fontsize{9pt}{9pt}\selectfont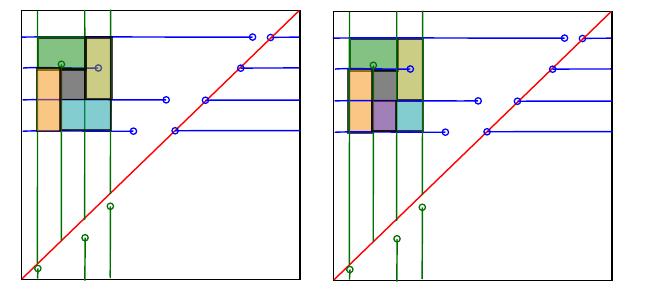}
\caption{Finite additivity of $\wh\mu_f$. We can turn a arbitrary division of a large box into smaller boxes (left) into one that is side-by-side decomposable (right). Note that we subdivided the cyan box in the initial division to make this possible. }
\label{fig:finite-additivity}
}
\end{figure}

Let $\cR = \cR_{\cB}$ be the ring generated by the
semiring~$\mathcal{B}$. By
Proposition~\ref{prop:addsemitoring}, $\wh\mu_f$ is finitely additive
on~$\mathcal{R}$.

%%% Local Variables:
%%% mode: latex
%%% TeX-master: "Intersections"
%%% End:

\section{Construction of the geodesic current}
\label{sec:construct}

\subsection{Countable additivity}
\label{sec:countable-additivity}

We will prove countable sub-additivity of the function $\wh\mu_{f}$ by
proving approximation results. We first find controlled sequences of
decreasing nested RH boxes that we can show have measure going to zero.

\begin{definition}
\label{def:accordion}
Let the RH lift system $(p_1,q_1,p_2,q_2)$ determine a right-handed
box
$B = [s_1,s_2) \times [t_2,t_1)$, where $s_i = \ell_i^-$ and $t_i =
r_i^+$. Further assume that $\ora{r_2}$ is nested inside $\ora{r_1}$,
so the box is of type~2 or~3.

As in Figure~\ref{fig:vertical-accordion},
define the \emph{vertical accordion} $F(p_1, q_1, p_2,\vec{q_2})$ of
boxes based at $B$ to be
the sequence of boxes, indexed by $k$, with $p_1, p_2, q_1$ constant
and $q_2$ translated by powers of $r_1 = q_1 x Q_1$:
\[
  p_1^{(k)}\coloneqq p_1,\qquad
  p_2^{(k)}\coloneqq p_2,\qquad
  q_1^{(k)}\coloneqq q_1,\qquad
  q_2^{(k)} \coloneqq q_1 x^k Q_1 q_2.
\]
This defines boxes whose height becomes smaller as $k$
grows, with the top endpoint fixed.

\begin{figure}
\centering{
\resizebox{120mm}{!}{\fontsize{9pt}{9pt}\selectfont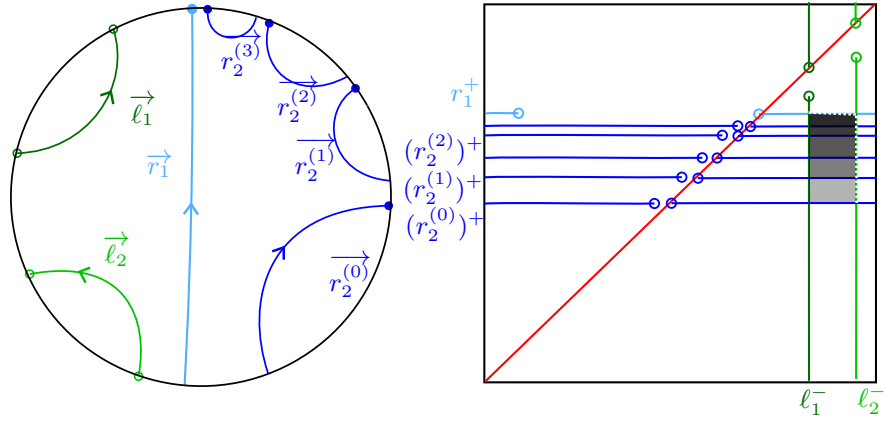}
\caption{A vertical accordion $F(p_1,q_1, p_2,\vec{q_2})$, shown in the disk (left) and in the whiskers picture (right)}
  \label{fig:vertical-accordion}
}
\end{figure}

Similarly, if $(\ora{\ell_1},\ora{\ell_2})$ are R-parallel, that is, $\ora{\ell_2}$ is nested inside $\ora{\ell_1}$ (so we again
have a box of type 2 or 3), as in
Figure~\ref{fig:horizontal-accordion} define the
\emph{horizontal accordion} $F(\vec{p_1},q_1, p_2,q_2)$ by
\[
  p_1^{(k)}\coloneqq p_1 P_2x^k p_2,\qquad
  p_2^{(k)}\coloneqq p_2,\qquad
  q_1^{(k)}\coloneqq q_1,\qquad
  q_2^{(k)} \coloneqq q_2.
\]
\end{definition}

\begin{proposition}
\label{prop:finalfamily}
  Let $F = F(p_1,q_1,p_2,\vec{q_2})$ or $F(\vec{p_1},q_1, p_2,q_2)$ be a vertical or horizontal accordion as above. Then
  \begin{itemize}
  \item each $F(k)$ is an RH box;
  \item $F(k+1) \subset F(k)$; and
  \item $\displaystyle \lim_{k \to \infty} \wh\mu_f(F(k)) = 0$.
  \end{itemize}
\end{proposition}

\begin{proof}
  We first prove the results in the case of vertical accordions, starting with checking that $F(k)$ is a box. Since $\ora{r_2}$ is nested inside $\ora{r_1}$, the NS dynamics of the group element $r_1 = q_1 x Q_1$ translate $\ora{r_2}$ inside $(r_1^-,r_1^+)$, as depicted in Figure~\ref{fig:vertical-accordion}. In particular the cyclic order of the endpoints doesn't change, and we also automatically have $F(k+1) \subset F(k)$.

  To show the measures go to zero, we first reformulate the nesting condition algebraically. Note that $q_2Q_1 \cdot \ora{r_1} = \ora{r_2}$. Then $\ora{r_2}$ and $\ora{r_1}$ being nested means $(\ora{r_1},\ora{r_2})$ are R-parallel which is equivalent, by Lemma~\ref{lem:rhs}, to $(\ora{q_2Q_1}, \ora{r_1})$ R-crossing. This is in turn equivalent to $Q_1 \cdot \ora{q_2Q_1} = \ora{Q_1q_2}$ R-crossing $Q_1 \cdot \ora{r_1} = \ora{x}$; write $e \coloneqq Q_1q_2$.

  In the RH orbit notation, we then have
  \begin{align*}
    a^{(k)} &= p_1 q_1 = a &
    b^{(k)} &= p_1 q_2^{(k)} = p_1 q_1 x^k Q_1q_2 = a x^k e\\
    c^{(k)} &= p_2 q_1 = c &
    d^{(k)} &= p_2 q_2^{(k)} = p_2 q_1 x^k Q_1 q_2 = c x^k e.
  \end{align*}

We thus get the family of RH orbits given by
\[
\bigl([B]_{a,a x^k e,c,c x^k e}\bigr)_{k=0}^\infty.
\]

To show this family of RH orbits have measure going to zero, we
evaluate:
\begin{align*}
  \lim_{k \to \infty} \bigl(\wh\mu_f([B]_{a,ax^ke,c,cx^ke})\bigr)
   &= \frac{1}{4}\lim_{k \to \infty} \lim_{n\to \infty} f([ax^ke x^n]) +
     f([cx^n])-f([ax^n])-f([cx^kex^n]) \\
  &= \frac{1}{4}\lim_{k,n \to \infty} \bigl(f([x^n a]) + f([x^ke])\bigr) + f([cx^n])\\
  &\qquad\qquad\qquad-f([ax^n])-\bigl(f([x^nc]) + f([x^ke])\bigr)\\
  &= 0.
\end{align*}
Here, after applying the definition of $\wh\mu_f$, we apply
Lemma~\ref{lem:additivity-limit-2}
(adding and subtracting suitable copies of $f([x^n])$ and $f([x^k])$)
to see the double limit
exists. Then by applying Proposition~\ref{prop:join-split-alg} twice we
can break up the limit into canceling terms as indicated.

The case of the horizontal accordion is similar, but there are some
confusing inverses so we spell it out briefly.
Since $(\ora{\ell_1},\ora{\ell_2})$ are R-parallel, that is, $\ora{\ell_2}$ is nested inside $\ora{\ell_1}$, the NS dynamics of $\ell_2 = P_2 x p_2$ translate
$\ell_1$ inside $(\ell_1^-,\ell_2^-)$, as in
Figure~\ref{fig:horizontal-accordion}, and the $F(k)$ are boxes with
$F(k+1) \subset F(k)$.
For the limiting computation, since
$P_2 p_1 \cdot \ora{\ell_1} = \ora{\ell_2}$ we have
$(\ora{P_2p_1}, \ora{\ell_2})$ R-cross and thus
$(\ora{p_1 P_2}, \ora{x})$ R-cross.

Set $h \coloneqq p_1 P_2$, so that
\begin{align*}
  a^{(k)} &= p_1^{(k)} q_1 = h x^k c &
  b^{(k)} &= p_1^{(k)} q_2 = h x^k d\\
  c^{(k)} &= p_2 q_1 = c &
  d^{(k)} &= p_2 q_2 = d.
\end{align*}
Lemma~\ref{lem:additivity-limit-2} and
Proposition~\ref{prop:join-split-alg} apply as before.
\end{proof}

  \begin{figure}
\centering{
\resizebox{120mm}{!}{\fontsize{9pt}{9pt}\selectfont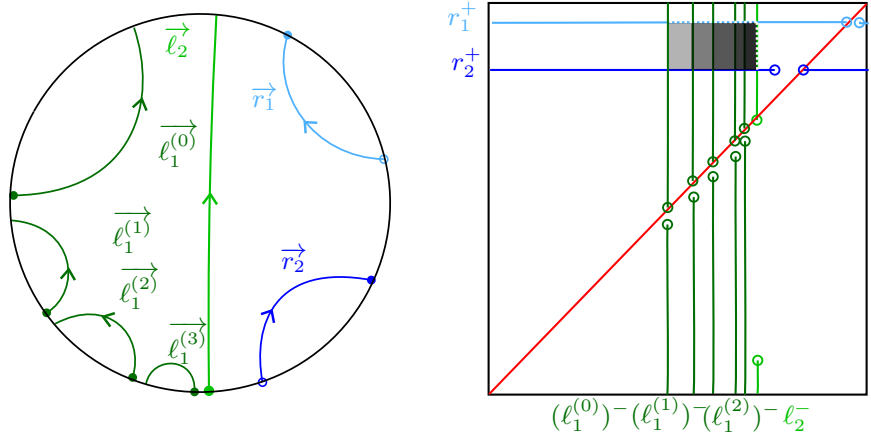}
\caption{Horizontal accordion $F(\vec{p_1},q_1,p_2,q_2)$}
  \label{fig:horizontal-accordion}
}
\end{figure}

We now use Proposition~\ref{prop:finalfamily} to prove that $\wh\mu_f$
is countably additive.
Recall that an RH box~$B$ is by definition half-open, of the form
$[s_1,s_2) \times [t_2,t_1)$. Let
$\overline{B} \coloneqq [s_1,s_2] \times [t_2,t_1]$ and
$B^\circ \coloneqq (s_1,s_2) \times (t_2,t_1)$ be the corresponding
closed and open boxes.

\begin{figure}
     \centering
     \begin{subfigure}[b]{0.47\textwidth}
        \fontsize{9pt}{9pt}\selectfont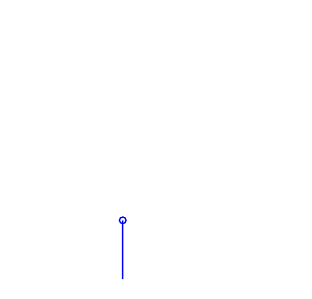
        \caption{$V(B,\epsilon)$ is a down and left extension of an RH
          box $B$ with mass less than $\epsilon$, as in Lemma~\ref{lem:extend1}}\label{fig:ca1}
     \end{subfigure}
     \hfill
     \begin{subfigure}[b]{0.47\textwidth}
        \fontsize{9pt}{9pt}\selectfont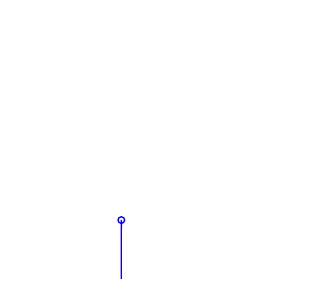
        \caption{$W(B,\epsilon)$ is an $\epsilon$ neighborhood of the upper and right edges of an RH box $B$, as in Lemma~\ref{lem:extend2}}\label{fig:ca2}
     \end{subfigure}
     \caption{Extensions used in the proof of countable sub-additivity}
     \label{fig:cases}
\end{figure}

\begin{lemma}\label{lem:extend1}
  For $B = [s_1,s_2)\times [t_2,t_1)$ an RH box and any
  $\epsilon > 0$, there is an RH box
  $V(B,\epsilon) = [s_1',s_2) \times [t_2',t_1)$ so that $s_1' < s_1$
  and $t_2' < t_2$ and
  $\wh\mu_f(V(B,\epsilon) \setminus B) < \epsilon$.
\end{lemma}
See Figure~\ref{fig:ca1}. That is, $V(B,\epsilon)$ extends $B$ slightly down and to the left,
without increasing the mass much. We write $V^\circ(B,\epsilon)\coloneqq (s_1',s_2)\times(t_2',t_1)$ for the interior of $V(B,\epsilon)$.

\begin{proof}
  Let $B = B(p_1,q_1,p_2,q_2)$. By Lemma~\ref{lem:diagonaldense}, there is some $q_3$ so that $q_3 \cdot \ora{x}$ is nested inside $q_2 \cdot \ora{x}$; in particular $F_1 = F(p_1,q_2,p_2,\vec{q_3})$ is a vertical accordion. Furthermore, for each $k$, $B \cup F_1(k)$ is an RH box.  By Proposition~\ref{prop:finalfamily}, there is $k_1$ large enough so that $\wh\mu_f(F_1(k_1))<\epsilon/2$. Repeat the procedure to find a horizontal accordion $F_2$ abutting on $B \cup F_1(k_1)$. Again there is $k_2$ large enough so that $\wh\mu_f(F_2(k_2)) < \epsilon/2$, and set $V(B,\epsilon) \coloneqq B \cup F_1(k_1) \cup F_2(k_2)$.
\end{proof}

\begin{lemma}\label{lem:extend2}
  For $B = [s_1,s_2) \times [t_2,t_1)$ an RH box and any $\epsilon > 0$, there is some $s_2',t_1'$ with $s_1 < s_2' < s_2$ and $t_2 < t_1' < t_1$ so that the region $W(B,\epsilon) = \bigl([s_1,s_2) \times [t_1',t_1)\bigr) \cup \bigl([s_2',s_2) \times [t_2,t_1)\bigr)$ is in $\mathcal{R}$ and satisfies $\wh\mu_f(W(B,\epsilon)) < \epsilon$.
\end{lemma}
See Figure~\ref{fig:ca2}. That is, $W(B,\epsilon)$ is a neighborhood of the upper and right
edges of~$B$ with small mass.

\begin{proof}
  Find $q_3$ so that $q_3 \cdot \ora{x}$ is nested inside $q_1 \cdot \ora{x}$ and so that $F_1 = F(p_1,q_1,p_2,\vec{q_3})$ is a vertical accordion. Similarly find $p_3$ so that $P_2 \cdot \ora{x}$ is nested inside $P_3 \cdot \ora{x}$ and $F_2 = F(\vec{p_3},q_1,p_2,q_2)$ is a horizontal accordion. Then for sufficiently large $k$, we have $\wh\mu_f(F_1(k)) < \epsilon/2$ and $\wh\mu_f(F_2(k)) < \epsilon/2$. Then set $W(B,\epsilon) = F_1(k) \cup F_2(k)$.
\end{proof}
In the setting of Lemma~\ref{lem:extend2}, also set
\[
W^\circ(B,\varepsilon):=
\bigl([s_1,s_2]\times(t'_1,t_1]\bigr)
\cup
\bigl((s'_2,s_2]\times[t_2,t_1]\bigr);
\]
this is an open neighborhood of the upper and right edges of
$\overline{B}$ in the subspace topology of $\overline{B}$. Note that
neither $W(B,\epsilon)$ nor $W^\circ(B,\epsilon)$ is a subset of the
other.

\begin{proposition}\label{prop:countablyadditive}
  The finitely additive measure $\wh\mu_f$ is countably sub-additive on $\cB$.
\end{proposition}

\begin{proof}
  Let $B_0$ be a right-handed box, and suppose we are given a countable covering of it: $B_0 \subset \bigcup_{i=1}^\infty B_i$. We must show that \[
    \wh\mu_f(B_0) \le \sum_{i=1}^\infty \wh\mu_f(B_i).
  \]
  Fix $\epsilon > 0$. For $i \ge 1$, set $V_i \coloneqq V(B_i, \epsilon/2^i)$. Add one more set by setting $V_0 \coloneqq W(B_0,\epsilon)$, $V_0^\circ \coloneqq W^\circ(B_0,\epsilon)$ and $V_i^\circ \coloneqq V^\circ(B_i,\varepsilon/2^i)$.

We claim $\{V_i^\circ\}_{i=0}^\infty$ cover $\overline{B_0}$. Indeed,
the sets $V_i^\circ$ for $i\geq1$
cover $B_0$, since
$B_i\subset V_i^\circ$;
in addition
$\overline{B}_0\setminus B_0\subset V_0^\circ$.
Moreover, each $V_i^\circ\cap\overline{B}_0$
is open in the subspace topology of $\overline{B}_0$. Since
$\overline{B}_0$ is compact, there is a finite set
$I\subset\mathbb N\cup\{0\}$ such that
\[
\overline{B}_0\subset\bigcup_{i\in I}V_i^\circ.
\]
For simplicity assume that $0\in I$, adding it in if necessary. It
follows that $B_0\subset \bigcup_{i\in I} V_i$; here we use that
$V_0^\circ \cap B_0 \subset V_0 \cap B_0$.

  By finite additivity, we then have
  \begin{align*}
    \wh\mu_f(B_0) &\le \sum_{i \in I} \wh\mu_f(V_i)\\
      &\le \sum_{i \in I, i \ne 0} \wh\mu_f(B_i) + \sum_{i \in I, i \ne 0}
        \wh\mu_f(V_i \setminus B_i) + \wh\mu_f(V_0)\\
      &< \sum_{i=1}^\infty \wh\mu_f(B_i) + 2\epsilon.
  \end{align*}
  Since $\epsilon$ was arbitrary, we are done.
\end{proof}

We now put everything together to obtain an oriented geodesic current,
following the standard argument that measure on boxes extends to the
Lebesgue measure

\begin{proposition}
Let $f \colon \Curves^+(S) \to \mathbb{R}$ be a functional on curves satisfying symmetry, additivity, smoothing and stability.
Then $\wh\mu_f$ extends to an oriented geodesic current $\overline\mu_f$.
\end{proposition}
\begin{proof}
By Propositions \ref{prop:limitexists}, \ref{prop:RHsemiring},
\ref{prop:finitely-additive}, and~\ref{prop:countablyadditive}, $\wh\mu_f$
is a finitely
additive and countably sub-additive pre-measure on $\mathcal{B}$, so
by Proposition~\ref{prop:addsemitoring},
$\wh \mu_f$ is countably additive on~$\cR$.
 Finally, $\wh\mu_f$ is non-negative by
Proposition~\ref{prop:box-positive}. Since every box has finite $\wh\mu_f$-measure and the countably many
boxes in $\mathcal B$ cover $G(S)$ (Proposition~\ref{prop:density}),
the pre-measure $\wh\mu_f$ is $\sigma$-finite on $\cR$ in the sense of
Theorem~\ref{thm:caratheodory}.
By the Carath\'eodory
Extension Theorem~\ref{thm:caratheodory}, $\wh\mu_f$ has a unique
extension to $\sigma(\cR)$, which we denote~$\overline\mu_f$.

By Proposition~\ref{prop:density}, for every
$\xi\in G(S)$ and every open neighborhood $U$ of $\xi$, there
is an RH box $B\in\mathcal B$ such that
$\xi\in B^\circ\subset B\subset U.$
Since $\pi_1(S)$ is countable, $\mathcal B$ is countable. Consequently,
every open set is a countable union of RH boxes, and hence
$\sigma(\mathcal B)=\operatorname{Bor}(G(S))$, the Borel
$\sigma$-algebra of $G(S)$.
Thus $\overline\mu_f$ is a Borel measure.
Moreover, every point has an open neighborhood contained in an RH box
of finite $\overline\mu_f$-mass. A compact set therefore admits a finite
cover by such neighborhoods, so $\overline\mu_f$ is locally finite. 

To check invariance, note that for every $\gamma\in\pi_1(S)$, the measures
$\gamma_*\overline\mu_f$ and $\overline\mu_f$ agree on $\mathcal B$, by the
$\pi_1(S)$-invariance of $\wh\mu_f$. Since both are
$\sigma$-finite extensions of that premeasure, uniqueness in the
Carath\'eodory extension theorem gives
$\gamma_*\overline\mu_f=\overline\mu_f.$
\end{proof}

It is not \emph{a priori} clear from the definition of
$\overline\mu_f$ that it is a flip-invariant geodesic current
(recall
Definition~\ref{def:currents}). A sketch of a proof is given in Lemma~\ref{lem:RH_flip_invariance}, showing that $\overline \mu_f$ is flip-invariant on RH boxes, from where it follows that $\overline \mu_f$ is flip-invariant.

Rather than relying on that sketch, we instead define the
flip-invariant geodesic current
\begin{equation}
\label{eq:flip-invariant}
\mu_f\coloneqq\OneHalf\bigl(\overline\mu_f+\sigma_*\overline\mu_f\bigr),
\end{equation}
where $\sigma_*$ is the push-forward map on oriented geodesic currents
induced from the flip map.

\subsection{Recovering \texorpdfstring{$f$}{f} on curves}
\label{sec:recovery}

It remains to show that $\mu_f$ recovers $f$ as an intersection number on curves.
Before proceeding, we detour to define an asymmetric version of intersection
number between oriented geodesic currents.
\begin{definition}\label{def:asymmetric}
  The \emph{asymmetric intersection} number between oriented geodesic currents is
\begin{equation}
\vec{\imath}(\mu,\nu) \coloneqq \int_{\mathcal{F}} \mathbf{1}_R \, d(\mu \otimes \nu),
\label{eq:asymmetric}
\end{equation}
where $\cF$ is a fundamental domain of the action of $\pi_1(S)$ on
$\DG(S)$, as in Definition~\ref{def:intersection_currents}, and
the function $\mathbf{1}_R \colon \DG(S) \to \mathbb{R}$ is
\[
\mathbf{1}_R(\gamma_1,\gamma_2) \coloneqq
\begin{cases}
1 & \text{if } (\gamma_1,\gamma_2)\text{ are R-crossing},\\
0 & \text{otherwise}.
\end{cases}
\]
\end{definition}
From the definitions, we have a relation to the standard intersection
number~\eqref{eq:intersection}:
\begin{equation}
i(\mu, \nu) = \vec{\imath}(\mu, \nu) + \vec{\imath}(\mu, \sigma(\nu)).
\label{eq:relation_intersections}
\end{equation}
For $p,q \in \wt{S}$, partition $G[p,q)$ by
\begin{equation}\label{eq:partition_oriented}
G[p,q) = \ora{G[p,q)} \sqcup \ole{G[p,q)} 
\end{equation}
where $\ora{G[p,q)}\subset G(S)$ for the \emph{right-oriented} transversals, i.e., the subset of
oriented geodesics in $G[p,q)$ which right-cross the bi-infinite
oriented geodesic~$\ora{l}$ obtained by extending $[p,q) \subset
\wt{S}$. Similarly, we will denote by $\ole{G[p,q)}$ the
\emph{left-oriented} transversals, intersecting the inverted geodesic
$\ora{L}$ to the right. By definition, $\sigma(\ora{G[p,q)})= \ole{G[p,q)}$.

We can compute intersections with curves as follows.

\begin{lemma}
\label{lem:asymmetric_intersection_current}
  For $\vec C = [y]$ an oriented closed curve, choose
$z \in \partial_{\infty} S \setminus \{ y_-, y_+\}$ so that
$y^+,\, y \cdot z,\,z,\,y^-$ are in cyclic order.
The standard intersection and asymmetric intersection
of an oriented closed curve $\vec C=[y]$ with an oriented geodesic current $\mu$ can be computed as 
\begin{align*}
i(\mu, \delta_C) &= \mu([y \cdot z,z) \times (y^-,y^+)) +
                   \mu(\sigma([y\cdot z,z) \times (y^-,y^+))\\
\vec{\imath}(\mu, \delta_{\vec C}) &= \mu([y \cdot z,z) \times (y^-,y^+)).
\end{align*}
\end{lemma}
Recall that $\delta_C$ is the geodesic current associated to the unoriented curve (or flip-invariant)~$C$, while $\delta_{\vec C}$ denotes the geodesic current representing the oriented curve $\vec C$.

Note that
$[y\cdot z,z)\times (y^-,y^+)$ is a Borel fundamental domain for the action of $\langle Y \rangle$ on the set of geodesics that intersect $\ora{y}$ transversely.
\begin{proof}
Choose $z\in(y^+,y^-)$ so that $y^+,\,y\cdot z,\,z,\,y^-$
occur in cyclic order, and set
\[
A_y\coloneqq [y\cdot z,z)\times(y^-,y^+).
\]
Then $A_y$ is a Borel fundamental domain for the action of
$\langle y\rangle$ on the oriented geodesics which cross
$\ora y$ transversely (See~\cite[Lemma~4.8]{HP97:RigidityNegCurvedCone}): the proof only uses
the $\pi_1(S)$-invariance of $\mu$.) Hence
$\vec i(\mu,[y])=\mu(A_y)$,
and, using the opposite crossing orientation,
$i(\mu,[y])=\mu(A_y)+\mu(\sigma A_y)$. The same claim follows for the intersection number using Equation~\ref{eq:relation_intersections}.
\end{proof} 

Analogously, an asymmetric intersection number between curves can be defined as follows.

\begin{definition}\label{def:asym-intersect}
  Given two oriented curves $\vec C$, $\vec D$, the asymmetric intersection number $\vec{\imath}(\vec C, \vec D)$ is the minimal number of \emph{right-handed} crossings between any two homotopy class representatives of $\vec C$ and~$\vec D$, i.e., the number of points on~$S$ where $\vec C$ crosses $\vec D$ to the right in the sense of Definition~\ref{def:crossing}.
\end{definition}

It is straightforward to see that the asymmetric intersection number from Definition~\ref{def:asymmetric} satisfies $\vec{\imath} (\delta_{\vec C}, \delta_{\vec D}) = \vec\imath(\vec C, \vec D)$ and, for a flip-invariant curve $C$ and any oriented geodesic current $\mu$, we have
  \[
  \vec{\imath}(\mu, \delta_{C}) = \OneHalf i(\mu, \delta_C).
  \]
In particular,
\[
\vec{\imath}(\delta_C, \delta_D) = \OneHalf i(C,D)
\]
We finish by proving Theorem~\ref{thm:intersect}, first in the following
form.

\begin{theorem}
  \label{thm:recoverintersection}
Let $f \colon \Curves^+(S) \to \mathbb{R}$ be a curve functional
satisfying symmetry, smoothing, additivity and stability, and let
$\mu_f$ be as above.
For $\vec C$ an oriented closed curve, we have
\[
i(\mu_f, \delta_C) =
  2\vec\imath(\mu_f, \delta_{\vec C}) = f(\vec C).
\]
\end{theorem}
\begin{proof}
  We first prove that $\vec{\imath}(\mu_f,\delta_{\vec C})=\OneHalf f(\vec C)$. Given a representative $y \in \pi_1(S)$ of the conjugacy class~$\vec C$, we will find a sequence of RH boxes to approximating the fundamental domain for $\langle
y\rangle$ above, and then show that computing the measure of that
fundamental domain yields $f([y])$.

Choose $p,q \in \pi_1(S)$ so that $(P\cdot \ora{x}, \ora{y},
q\cdot\ora{x})$ are R-parallel, as in Figure~\ref{fig:recover-f}.
The NS dynamics of the action of $y$ imply that, for $m_1
> m_2$ and $n_1 > n_2$, the box
$B(py^{-m_1},py^{-m_2},y^{n_1}q,y^{n_2}q)$ form an RH box. We will look
at the following particular family for $n \ge 0$:
\begin{align*}
  p_1^{(n)} &\coloneqq p & q_1^{(n)}&\coloneqq y^n q\\
  p_2^{(n)} &\coloneqq p y &q_2^{(n)} &\coloneqq y^{-(n+1)} q.
\end{align*}

\begin{figure}
\centering{
\fontsize{9pt}{9pt}\selectfont%% Creator: Inkscape 1.3 (0e150ed6c4, 2023-07-21), www.inkscape.org
%% PDF/EPS/PS + LaTeX output extension by Johan Engelen, 2010
%% Accompanies image file 'recover-f.pdf' (pdf, eps, ps)
%%
%% To include the image in your LaTeX document, write
%%   \input{<filename>.pdf_tex}
%%  instead of
%%   \includegraphics{<filename>.pdf}
%% To scale the image, write
%%   \def\svgwidth{<desired width>}
%%   \input{<filename>.pdf_tex}
%%  instead of
%%   \includegraphics[width=<desired width>]{<filename>.pdf}
%%
%% Images with a different path to the parent latex file can
%% be accessed with the `import' package (which may need to be
%% installed) using
%%   \usepackage{import}
%% in the preamble, and then including the image with
%%   \import{<path to file>}{<filename>.pdf_tex}
%% Alternatively, one can specify
%%   \graphicspath{{<path to file>/}}
%% 
%% For more information, please see info/svg-inkscape on CTAN:
%%   http://tug.ctan.org/tex-archive/info/svg-inkscape
%%
\begingroup%
  \makeatletter%
  \providecommand\color[2][]{%
    \errmessage{(Inkscape) Color is used for the text in Inkscape, but the package 'color.sty' is not loaded}%
    \renewcommand\color[2][]{}%
  }%
  \providecommand\transparent[1]{%
    \errmessage{(Inkscape) Transparency is used (non-zero) for the text in Inkscape, but the package 'transparent.sty' is not loaded}%
    \renewcommand\transparent[1]{}%
  }%
  \providecommand\rotatebox[2]{#2}%
  \newcommand*\fsize{\dimexpr\f@size pt\relax}%
  \newcommand*\lineheight[1]{\fontsize{\fsize}{#1\fsize}\selectfont}%
  \ifx\svgwidth\undefined%
    \setlength{\unitlength}{144bp}%
    \ifx\svgscale\undefined%
      \relax%
    \else%
      \setlength{\unitlength}{\unitlength * \real{\svgscale}}%
    \fi%
  \else%
    \setlength{\unitlength}{\svgwidth}%
  \fi%
  \global\let\svgwidth\undefined%
  \global\let\svgscale\undefined%
  \makeatother%
  \begin{picture}(1,1)%
    \lineheight{1}%
    \setlength\tabcolsep{0pt}%
    \put(0,0){\includegraphics[width=\unitlength,page=1]{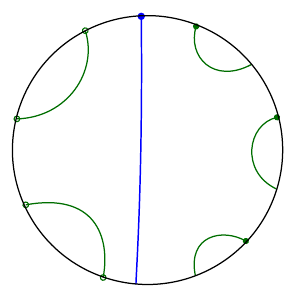}}%
    \put(0.42427753,0.53970789){\color[rgb]{0,0,1}\makebox(0,0)[t]{\lineheight{1.25}\smash{\begin{tabular}[t]{c}$\ora{y}$\end{tabular}}}}%
    \put(0.88727147,0.08248818){\color[rgb]{0,0.43137255,0}\makebox(0,0)[t]{\lineheight{1.25}\smash{\begin{tabular}[t]{c}$y^{-1-n}q \cdot \ora{x}$\end{tabular}}}}%
    \put(1.05210363,0.50138343){\color[rgb]{0,0.43137255,0}\makebox(0,0)[t]{\lineheight{1.25}\smash{\begin{tabular}[t]{c}$q \cdot \ora{x}$\end{tabular}}}}%
    \put(0.94721477,0.86987304){\color[rgb]{0,0.43137255,0}\makebox(0,0)[t]{\lineheight{1.25}\smash{\begin{tabular}[t]{c}$y^n q \cdot \ora{x}$\end{tabular}}}}%
    \put(0.4030187,0.34490451){\color[rgb]{0,0.43137255,0}\makebox(0,0)[t]{\lineheight{1.25}\smash{\begin{tabular}[t]{c}$y^{-1}P \cdot \ora{x}$\end{tabular}}}}%
    \put(0.34416577,0.64972841){\color[rgb]{0,0.43137255,0}\makebox(0,0)[t]{\lineheight{1.25}\smash{\begin{tabular}[t]{c}$P \cdot \ora{x}$\end{tabular}}}}%
    \put(0,0){\includegraphics[width=\unitlength,page=2]{recover-f.pdf}}%
  \end{picture}%
\endgroup%
}
\caption{The boxes used in proving that $\mu_f$ recovers $f$.}
  \label{fig:recover-f}
\end{figure}

In the RH orbit notation, we have
\begin{align*}
  a^{(n)} &= p y^{n}   q &  b^{(n)} &= p y^{-(n+1)} q \\
  c^{(n)} &= p y^{n+1} q &  d^{(n)} &= p y^{-n} q.
\end{align*}

If we let $z=(P \cdot \ora{x})^-$, and $B_n \coloneqq B_{a^{(n)},
  b^{(n)}, c^{(n)}, d^{(n)}}$ we then have
    \begin{align*}
      \vec{\imath}(\overline\mu_f, \delta_{[y]})&=\overline\mu_f([z,Yz)\times(y^-,y^+)) \\
    &= \lim_{n\to\infty} \wh\mu_f(B_n)  \\
    &= \frac{1}{4}\lim_{n\to\infty}\lim_{m\to\infty}f([py^{n+1}qx^{m}]) + f([py^{-(n+1)}qx^{m}])-f([py^{n}qx^{m}]) - f([py^{-n}qx^{m}]) \\
    &= \OneFourth\bigl(f([y]) + f([y^{-1}])\bigr) =\OneHalf f([y]).
    \end{align*}
    The first equality follows by the fact that the boxes defined above converge from below to the  $[z,Y\cdot z) \times (y^-,y^+)$ (the fundamental domain for the action of $\langle Y \rangle$) as $n \to \infty$, since $B_{n-1} \subset B_n$.
    At the second equality, we use definition of $\wh\mu_f$ on RH boxes.
    For the third, Lemma~\ref{lem:additivity-limit-2} lets us pass from the iterated limit to a double limit, and gives a number of canceling terms of $mf([x])$ and $nf([y^{\pm 1}])$; only two terms survive the cancellation. Then symmetry of $f$ gives the result.

We then get $\vec\imath(\sigma_*\overline\mu_f,\delta_{\vec C})
=\vec\imath(\overline\mu_f,\delta_{\vec C^{-1}})=\frac12 f(\vec C)$, using the definition of $\sigma_*$, and the symmetry of~$f$.
Since $\mu_f$ is flip-invariant, Lemma~\ref{lem:asymmetric_intersection_current}
gives
$i(\mu_f,\delta_C)=f(\vec C)$, as desired.
\end{proof}

\begin{remark}
In \cite[Proposition~4.3]{BIPP24:PositiveCrossratios} a geodesic current is induced
from a cross-ratio.
Our construction is more general (in particular, we do not rely on a cross-ratio).
We discuss here why our construction does not follow from their
arguments once we have defined $\mu_f$ on boxes.
Given a dense subset $X \subset \partial_{\infty} S$, they consider the set of $\mathcal{R}(X)$ of all proper rectangles of geodesics with endpoints on $X$.
They start from a finitely additive function $r \colon \mathcal{R}(X) \to \mathbb{R}$ (defined suitably from a cross-ratio)
that is (1) side-by-side additive: if $R,R_1,R_2 \in \mathcal{R}(X)$
with $R=R_1 \cup R_2$, with $R_1,R_2$ disjoint, $r(R)=r(R_1)+r(R_2)$;
and (2) monotonic: if $R,R' \in \mathcal{R}(X)$ with $R \subset R'$, we have $r(R) \leq r(R')$.

Our function $\mu_f$ on RH boxes satisfies these properties by Proposition~\ref{prop:sidebyside}
and Proposition~\ref{prop:finitely-additive}, however, if we take $X$ to
be the set of
endpoints of translates of~$x$, then $\mu_f$ is only
defined on a proper subset of $\mathcal{R}(X)$.
Therefore, we cannot directly apply their result.

Moreover, we need the extra control on the family of RH
boxes to recover the curve functional~$f$ from the geodesic current.
In Section~\ref{sec:crossratios}, we prove that every cross-ratio
induces a curve functional (i.e., period) satisfying the hypotheses of
our Theorem~\ref{thm:intersect}, and thus our result reproves
\cite[Proposition~4.3]{BIPP24:PositiveCrossratios} for closed surfaces.
\end{remark}

\subsection{Proof of Main Theorems~\ref{thm:intersect} and~\texorpdfstring{\ref{thm:intersections_group}}{A'}}

Now we prove the two versions of the main theorem as stated in the introduction. The topological version (Theorem~\ref{thm:intersect}) is stated in terms of functionals on unoriented multi-curves. The algebraic version of the
theorem, namely, Theorem~\ref{thm:intersections_group}, is phrased in terms of functions on the surface group.

\begin{proof}[Proof of Theorem~\ref{thm:intersect}]
  Given a functional~$f$ as in the statement, by Definition~\ref{def:smoothing} and Proposition~\ref{prop:power_disconnected_imply_stable}, the induced additive and symmetric functional on oriented multi-curves satisfies stability. Hence the existence of $\mu_f$ follows by Theorem~\ref{thm:recoverintersection}.

  Uniqueness of the geodesic current $\mu_f$ is a result of Otal, geodesic currents are determined by their marked length spectrum \cite[Th\'eor\`eme~2]{Otal90:SpectreMarqueNegative}. The claim about the topology is due to Duchin--Leininger--Rafi, who show that the map $\Curr(S) \to \bbR^{\Curves(S)}$ defined by $\mu \mapsto i(\mu, \cdot)$ is a homeomorphism onto its image \cite[Theorem~12]{DLR10:DegenerationFlatMetrics}.
\end{proof}

\begin{remark}
  Uniqueness is not true if we don't assume $\mu_f$ is flip-invariant. See Appendix~\ref{sec:asymmetric}.
\end{remark}

\begin{proof}[Proof of Theorem~\ref{thm:intersections_group}]
As in the statement, let $\overline{f}(a) \coloneqq
\frac{1}{2}\left(f(a) + f(A) \right)$, and extend it linearly to be an
additive
function on multi-curves. Since $\overline{f}$ is
symmetric, its oriented smoothing inequalities imply the corresponding
unoriented smoothing inequalities by Proposition~\ref{prop:oriented_smoothing_unoriented}. Stability gives
power smoothing:
\[
\overline{f}(a^{n+m})=\overline{f}(a^n)+\overline{f}(a^m).
\]
Thus $\overline{f}$ satisfies the hypotheses of Theorem~\ref{thm:intersect}, which produces the
required geodesic current.

Conversely, if $\overline{f}(g)=i(\mu,[g])$, then class invariance, stability,
and the smoothing inequalities follow from the corresponding
properties of intersection with a geodesic current; see
Proposition~\ref{prop:necessary}.

To prove the sufficient criterion, we will show that, assuming the oriented smoothing hypotheses on $f$, the extension of $\overline{f}$ satisfies the hypothesis of Theorem~\ref{thm:intersect}.

We will first verify smoothing of $\overline{f}$.
We start
with oriented disconnected smoothing, between crossing elements
$(\ora{a},\ora{b})$, which we can assume R-cross by switching $a,b$ if
necessary. We have
\begin{align*}
  \overline{f}(a) + \overline{f}(b)
    &= \frac{1}{2}\bigl(f(a) + f(A) + f(b) + f(B)\bigr)\\[2pt]
    &\ge \frac{1}{2}\bigl(f(ab) + f(AB)\bigr) = \overline{f}(ab).
\end{align*}
In the inequality, we used the assumption on $f$ and the fact that
$(\ora{A},\ora{B})$ also
R-cross, as well as conjugation invariance for $f(AB) = f((ba)^{-1})$.
By grouping the four terms differently, and using the fact that
$(\ora{b},\ora{A})$ R-cross, we also get $\overline{f}(a) +
\overline{f}(b) \ge \overline{f}(bA)$.

Next we turn to oriented connected smoothing of a curve at an
essential self-crossing.
Using Proposition~\ref{prop:crossings-triality}, assume the essential
self-crossing is realized by the self-intersection of a curve
$[\gamma]=[ab]$, where $(\ora{a},\ora{b})$ are parallel; by switching
them if necessary, we can assume they are R-parallel, and thus so are
$(\ora{B},\ora{A})$. Then
\begin{align*}
  \overline{f}(ab)
    &= \frac{1}{2}\bigl(f(ab) + f(BA)\bigr) \\[2pt]
    &\ge \frac12\bigl(f(a) + f(b) + f(B) + f(A)\bigr)
     = \overline{f}(a) + \overline{f}(b).
\end{align*}

We have thus deduced $\overline{f}$ satisfies oriented connected and disconnected smoothing. Therefore, by Proposition~\ref{prop:oriented_smoothing_unoriented} below,
$\overline{f}$~satisfies unoriented connected and disconnected smoothing as well. 
Finally, stability implies oriented power-smoothing, i.e., $f(a^{n+m}) \geq f(a^n) + f(a^m)$ for every $n, m \geq 1$. This establishes smoothing of $\overline{f}$. Hence, smoothing, stability and additivity are satisfied for the function on unoriented curves $\overline{f}$.
We are thus under the
hypotheses of Theorem~\ref{thm:intersect}.
\end{proof}

\begin{remark}
In fact, we can replace the hypotheses in~Theorem~\ref{thm:intersections_group} by the following slightly weaker ones:
\begin{itemize}
 \item Stability can be replaced by power smoothing (by Proposition~\ref{prop:power_disconnected_imply_stable});
\item oriented smoothing can be replaced by R-oriented smoothing:
 \begin{align}
f(ab) &\ge f(a) + f(b)
&& \text{if } (\ora{a},\ora{b}) \text{ are R-parallel}, \\
f(a) + f(b) &\ge f(ab)
&& \text{if } (\ora{a},\ora{b}) \text{ R-cross}.
\end{align}
\end{itemize}
\end{remark}

%%% Local Variables:
%%% mode: latex
%%% TeX-master: "Intersections"
%%% End:

\section{Relations between types of smoothings}
\label{sec:smoothing-types}

In this section we examine the relationships between different types of smoothing. In particular, we will consider the distinctions between oriented and unoriented smoothing, as well as between connected and disconnected smoothing, and explore how these two attributes interact.

A priori the hypothesis of oriented smoothing on $f$ is weaker than
smoothing, since it does not include the hypotheses on unoriented smoothings.
However, it follows from the other hypotheses; the key ingredient is
our previous work in~\cite{MGT21:Smoothings}. We prove the following,
which was crucial to prove the simpler algebraic statement in
Theorem~\ref{thm:intersections_group}; it is not necessary for the
proof of Theorem~\ref{thm:intersect}.
\begin{proposition}
If $f$ satisfies oriented smoothing
(connected and disconnected), additivity,
stability, and symmetry, then $f$ satisfies smoothing.
\label{prop:oriented_smoothing_unoriented}
\end{proposition}

The proof uses the following result, which can be seen as a weak
converse of the main result in our previous paper. 
\begin{proposition}
Suppose $f$ is a symmetric curve functional satisfying additivity and stability that extends continuously to geodesic currents.
If $f$ satisfies disconnected smoothing (resp.\ connected smoothing) then $f$ satisfies connected smoothing (resp.\ disconnected smoothing).
\label{prop:cont_conn_disc}
\end{proposition}

\begin{proof}[Proof of
  Proposition~\ref{prop:oriented_smoothing_unoriented}, assuming Proposition~\ref{prop:cont_conn_disc}]
  Since $f$ satisfies disconnected oriented smoothing and is
  symmetric, it also
  satisfies disconnected unoriented smoothing. Indeed, let
  $(\ora{a},\ora{b})$ R-cross. Then $(\ora{B},\ora{a})$ R-cross, so by
  disconnected oriented smoothing
\[
f([B][a]) \geq f([Ba]).
\]
Since $f$ is symmetric, we have $f([b][a])=f([B][a]) \geq
f([Ba])$, as desired.

Since $f$ satisfies oriented smoothing, additivity and stability it extends continuously and homogeneously
to oriented geodesic currents, by~\cite[Theorem~A]{MGT21:Smoothings}. Hence by
Proposition~\ref{prop:cont_conn_disc}, it satisfies connected
smoothing, and thus all cases of smoothing.
\end{proof}

We now turn to the proof of Proposition~\ref{prop:cont_conn_disc},
starting with some lemmas. The first lets us approximate
disconnected curves by connected curves and vice versa.
\begin{lemma}
  Given any $g,h \in \pi_1(S)$ with distinct axes, we have
\[
\lim_{n\to\infty} \frac{\delta_{[g^n h^n]}}{n} = \delta_{[g]} + \delta_{[h]}
\]
in the weak$^*$-topology of geodesic currents.
\label{lem:gnhn_to_gsumh}
\end{lemma}

Recall that here $\delta_{[g]}$ denotes the geodesic current representing the curve $[g]$ (Equation~\eqref{eq:delta-curve}).

\begin{proof}
  This is a straightforward fact of the weak$^*$-topology convergence
  of geodesic currents and follows from computations similar
  to~\cite[Proposition~5.1]{Sas22:Currents}.
\end{proof}

Next we give some extensions of Lemma~\ref{lem:righthandedfoundation}
to powers.
\begin{lemma}
  Let  $a,b \in \pi_1(S)$ so that $\bigl(\ora{a},\ora{b}\bigr)$ are
  R-parallel. Then for all
  $n \ge 1$, the triple $(\ora{a},\ora{a^n b^n}, \ora{b})$ is R-parallel.
  \label{lem:powers_parallel}
\end{lemma}
\begin{proof}
The proof is by induction on~$n$, repeatedly applying
Lemma~\ref{lem:righthandedfoundation}\ref{item:ab-parallel}. The case
$n=1$ is directly that lemma.
For the inductive step, we have:
\begin{align*}
  &\bigl(\ora{a},\ora{a^{n-1}b^{n-1}},\ora{b}\bigr)\text{ R-parallel}
    &&\text{(Induction)}\\
  &\quad\Rightarrow
    \bigl(\ora{a},\ora{a^n b^{n-1}},\ora{a^{n-1}b^{n-1}},\ora{b}\bigr)\text{ R-parallel}
    &&\text{(Lemma~\ref{lem:righthandedfoundation}\ref{item:ab-parallel} on
$\bigl(\ora{a},\ora{a^{n-1}b^{n-1}}\bigr)$)}\\
  &\quad\Rightarrow
    \bigl(\ora{a},\ora{a^n b^{n-1}},\ora{a^n b^n},\ora{b}\bigr)\text{ R-parallel}
    &&\text{(Lemma~\ref{lem:righthandedfoundation}\ref{item:ab-parallel} on
$\bigl(\ora{a^n b^{n-1}},\ora{b}\bigr)$)}.
\qedhere
\end{align*}
\end{proof}

\begin{lemma}
  Let  $a,b \in \pi_1(S)$ so that $\bigl(\ora{a},\ora{b}\bigr)$ are
  R-parallel. Then for all
  $n \ge 1$, $m \geq 1$ the pair $(\ora{b^ma^m},\ora{a^n
  b^n})$ R-cross.
\label{lem:anbn_vs_bkak}
\end{lemma}
\begin{proof}
  Assume first $m=1$.
  We proceed by induction on~ $n$.
  The base case
 is clear by
 Lemma~\ref{lem:righthandedfoundation}\ref{item:ab-parallel}. In
 general we have
 \begin{align*}
   &\bigl(\ora{ba},\ora{a^{n-1}b^{n-1}}\bigr)\text{ R-cross}
     &&\text{(Induction)}\\
   &\quad\Rightarrow
     \bigl(\ora{ba},\ora{a^{n-1}b^n a}\bigr)\text{ R-cross}
     &&\text{(Lemma~\ref{lem:rhs_prod})}\\
   &\quad\Rightarrow
     \bigl(\ora{ab}, \ora{a^nb^n}\bigr)\text{ R-cross}
     &&\text{(Conjugate by~$a$)}.
 \end{align*}
Finally,
Lemma~\ref{lem:powers_parallel} and the R-crossing of 
$\bigl(\ora{ab}, \ora{a^nb^n}\bigr)$ restrict the
endpoints of $\ora{a^n b^n}$ to the indicated intervals in
Figure~\ref{fig:parallel_lemma}, which implies that $\ora{a^n b^n}$ crosses
$\ora{ba}$.

For $m \geq 1$, we apply the symmetric argument to show that 
$\bigl(\ora{b^ma^m}, \ora{ba}\bigr)$ R-cross with the endpoints of $\ora{b^ma^m}$ restricted to the intervals in Figure~\ref{fig:parallel_lemma}.
\end{proof}
\begin{figure}
     \begin{subfigure}[t]{0.48\textwidth}
\centerline{\fontsize{9pt}{9pt}\selectfont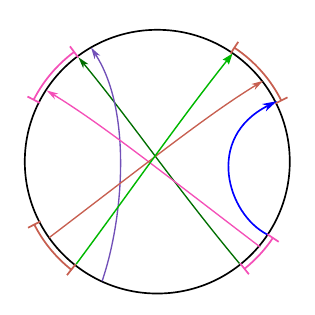}
    \caption{$\protect\ora{a}$ and $\protect\ora{b}$ are R-parallel
      implies $(\protect\ora{b^ma^m}, \protect\ora{a^n b^n})$ R-cross
      and $(\ora{a},\ora{a^n b^n}, \ora{b})$ are R-parallel. Here the endpoints of the geodesic $\protect\ora{a^nb^n}$ satisfy $(a^nb^n)^+ \in (ab^+, a^+)$ and $(a^nb^n)^- \in (ab^-, b^-)$, and the endpoints of the geodesic $\protect\ora{b^ma^m}$ satisfy $(b^ma^m)^+ \in (b^+, ba^+)$ and $(b^ma^m)^- \in (a^-, ba^-)$. }
    \label{fig:parallel_lemma}
  \end{subfigure}
  \hfill
    \begin{subfigure}[t]{0.48\textwidth}
 \centerline{\fontsize{9pt}{9pt}\selectfont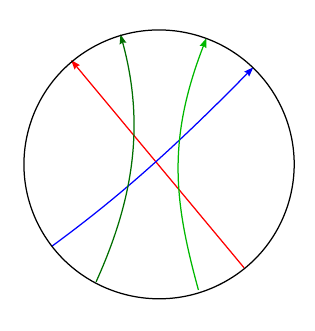}
    \caption{$\protect\ora{b}$ and $\protect\ora{a}$ R-cross implies $\protect\ora{a^n b^n}$ is R-parallel to $\protect\ora{ba}$. Note that the geodesic $\protect\ora{a^{n-1}b^na}$ and $\ora{ba}$ are parallel, since $\protect\ora{a^{n-1}b^{n-1}}$ and $\protect\ora{ba}$ are parallel by induction. After conjugation by $a$ the result follows.}
    \label{fig:cross_lemma}
     \end{subfigure}
  \caption{Crossings of powers}
\end{figure}

\begin{lemma} 
  Let $a,b \in \pi_1(S)$ so that $(\ora{b},\ora{a})$ R-cross. Then for all
  $n \ge 1$, $(\ora{a^n
  b^n},\ora{ba})$ are R-parallel, and for every $n \geq 2$, the pair $(\ora{a^n
  b^n},\ora{ab})$ are R-parallel.
\label{lem:anbn_vs_bkak_parallel}
\end{lemma}
\begin{proof}
We proceed by induction on $n$. The base case is clear by
Lemma~\ref{lem:righthandedfoundation}\ref{item:ab-cross}.
In general we have
\begin{align*}
  \bigl(\ora{a^{n-1}b^{n-1}},\ora{ba}\bigr)\text{ R-parallel}
  &\Rightarrow
    \bigl(\ora{a^{n-1}b^n a}, \ora{ba}\bigr)\text{ R-parallel}\\
  &\Rightarrow
    \bigl(\ora{a^nb^n}, \ora{ab}\bigr)\text{ R-parallel}\\
  &\Rightarrow
    \bigl(\ora{a^nb^n}, \ora{ba}\bigr)\text{ R-parallel}
\end{align*}
where we use the inductive hypothesis,
Lemma~\ref{lem:righthandedfoundation}\ref{item:ab-parallel},
conjugation, and transitivity of R\hyp parallelism. See
Figure~\ref{fig:cross_lemma}.
\end{proof}

The following result allows us to extract inequalities for parallel pairs assuming disconnected smoothing but not connected smoothing.
\begin{lemma}
\label{lem:fekete_upper}
Suppose $f \colon \Curves^+(S) \to \mathbb{R}$ is symmetric and satisfies disconnected smoothing and convex union. Let $(\ora{a},\ora{b})$ be parallel.
Then, the sequence $a_n \coloneqq f(a^nb^n)$ is sub-additive $(a_{n+m}
\leq a_n + a_m, \text{ for } n,m \in \mathbb{N})$ and, for every $n \in \mathbb{N}$,
\begin{equation}
\label{eq:fekete1}
f(ab) \geq \lim_{n\to\infty} \frac{1}{n} f(a^nb^n).
\end{equation}
More specifically, for $n \in \bbN$,
\begin{equation}
\label{eq:fekete2}
f(ab) \geq \frac{1}{2^n} f(a^{2^n}b^{2^n}) \geq \lim_{n \to \infty} \frac{1}{2^n} f(a^{2^n}b^{2^n}).
\end{equation}
\end{lemma}
\begin{proof}
WLOG, assume $(a,b)$ are R-parallel. By Lemma~\ref{lem:anbn_vs_bkak}, $(\ora{b^m a^m}, \ora{a^nb^n})$ R-cross for every $m,n \geq 1$.
Since $f$ satisfies disconnected smoothing and convex union by assumption, it follows that for every $n,m \geq 1$, we have
\[
f(a^n b^n) + f(b^m a^m) \geq f((b^m a^m)(a^n b^n)) = f(a^{n+m}b^{n+m}),
\]
i.e., the sequence $a_n \coloneqq f(a^n b^n)$ is sub-additive, as claimed, and the desired limit in Eq.~\eqref{eq:fekete1} exists and satisfies the inequality by Lemma~\ref{lem:fekete}.
An induction using sub-additivity gives, for every $n \in \bbN$,
\[
f(ab) \geq \frac{1}{2^n} f(a^{2^n}b^{2^n}).
\]
Finally, the last inequality in Eq.~\eqref{eq:fekete2} follows by Eq~~\eqref{eq:fekete1}.
\end{proof}

\begin{proof}[Proof of Direction 1 of Proposition~\ref{prop:cont_conn_disc}]
		
Suppose first $f \colon \Curves^+(S) \to \mathbb{R}$ satisfies additivity, disconnected smoothing, homogeneity and it extends continuously to geodesic currents.

			Given a pair $a, b \in \pi_1(S)$ with $(\ora{a},\ora{b})$ parallel (so
			$[ab] \reducesto [a][b]$ and $[ab]\reducesto[aB]$), we show first $f$
			satisfies oriented connected
			smoothing, i.e., $f([ab]) \geq f([a][b])$.
            Since $f$ satisfies disconnected smoothing, by Lemma~\ref{lem:fekete_upper} we get, 
\begin{equation}\label{eq:power_ineq}
				f(ab)  \geq \lim_n \frac{1}{n} f([a^{n}b^{n}]).
			\end{equation}

			Let $n$ go to infinity, and
			use Lemma~\ref{lem:gnhn_to_gsumh} and the continuity of the extension of $f$ to geodesic currents to obtain
			\[
			f([ab]) \geq f([a] [b])
			\]
			which proves oriented connected smoothing.
			
Second, we prove unoriented connected smoothing: with $a,b$ as above,
			$f([ab]) \ge f([aB])$.

			We split into two cases, depending on the relative position of $\ora{aB}$
			and $\ora{Ba}$. By
			Lemma~\ref{lem:aB_bA_antiparallel_cross}, $(\ora{aB},\ora{Ba})$ are either anti-parallel
			or cross.

                We first consider the case when $(\ora{aB},\ora{Ba})$ are anti-parallel.

                We will prove by induction that
				\begin{equation}
					2nf([ab]) \geq f([(aB)^n (Ab)^n]).
					\label{eq:antiparallel_induction}
				\end{equation}

                				The base case, $n=1$, follows since $(\ora{ab},\ora{ba})$ cross, hence
				\[
				[ba][ab] \reducesto [ba(ab)^{-1}] = [(aB)(Ab)].
				\]
				By homogeneity and disconnected smoothing, we get the base case.
           Now we prove the induction step.
\begin{align*}
2nf([ab])
&= 2(n-1)f([ab]) + 2f([ab]) \\
&\geq f([(aB)^{n-1} (Ab)^{n-1}]) + f([(Ab)(aB)]) && \text{(Induction)} \\
&\geq f([(aB)^{n-1} (Ab)^{n-1} (Ab)(aB)])
&& \text{(Convex union and disconnected smoothing)} \\
&= f([(aB)^n (Ab)^n])
\end{align*}
				where the disconnected smoothing is applied to
				\[[(aB)^{n-1}(Ab)^{n-1}][(Ab)(aB)] \reducesto [(aB)^{n-1} (Ab)^{n-1}
				(Ab)(aB)].
				\]
				 For this reduction, note that $\ora{aB}$ and $\ora{Ba}$ are
				anti-parallel, so $\ora{aB}$ and $\ora{Ab}$ are parallel. Hence, we
				can apply Lemma~\ref{lem:anbn_vs_bkak} to conclude that
				$\ora{(aB)^{n-1} (Ab)^{n-1}}$ and $\ora{(Ab)(aB)}$ cross.

                				Divide Eq.~\eqref{eq:antiparallel_induction} by $2n$, let $n$ go to
				infinity, and use Lemma~\ref{lem:gnhn_to_gsumh} to get
				\[
				f([ab]) \geq \frac{f([aB] [Ab])}{2} = f([aB])
				\]
				where we used symmetry and homogeneity in the last equality.

				Now, still in the unoriented connected smoothing case, suppose
				$(\ora{aB},\ora{Ba})$ cross.
				If $(\ora{a},\ora{b})$ are parallel it follows by Lemma~\ref{lem:a_parallel_Ba} that either $(\ora{a},\ora{bA})$ are parallel or $(\ora{b}, \ora{aB})$ are parallel.
				In the first case, we have 
\begin{align*}
f([ba]) 
& \geq f([bA][a^2]) && \text{(Oriented connected smoothing)}\\
&= f([bA]) + f(a^2)  && \text{(Additivity)} \\
&\geq f([bA]). && \text{(Positivity)}
\end{align*}
				In the second case,
				\[
				f([ab]) \geq f([aB][b^2]) =f([aB]) + f([b^2]) \geq f([aB])
				\]
				similarly as before. This proves unoriented connected smoothing and thus finishes the proof that disconnected smoothing implies connected smoothing.
\end{proof}	
		
\begin{proof}[Proof of Direction 2 of Proposition~\ref{prop:cont_conn_disc}]
		
		We now prove the second part, namely, if $f \colon \Curves^+(S) \to \mathbb{R}$ satisfies symmetry, connected smoothing, additivity and continuity, then $f$ satisfies disconnected smoothing. The proof is very similar to the first part.
			
			 Namely, given $a, b \in \pi_1(S)$ with $(\ora{a},\ora{b})$ crossing, we prove, by induction on $n$, the following inequality
			\begin{equation}
				f([a^{2n}b^{2n}]) \geq 2nf([ab]).
				\label{eq:inequality_smoothing}
			\end{equation}
			
			The base case follows by noting that if $(\ora{a},\ora{b})$ R-cross, then $(\ora{ba},\ora{ab})$ are R-parallel, and then by connected smoothing
			\[
			f([a^2b^2])=f([ab^2a]) \geq 2f([ab]).
			\]
			For the induction step, note that
			\begin{align*}
				f([a^{2n}b^{2n}]) &= f([a^2(a^{2(n-1)}b^{2(n-1)})b^2])\\
				&= f([b^2a^2(a^{2(n-1)}b^{2(n-1)})])&& \text{(Conjugation invariance)} \\
				&\geq f([b^2a^2]) + f([a^{2(n-1)}b^{2(n-1)})]) && \text{(Oriented connected smoothing)} \\
				&\geq 2f([ab]) + 2(n-1)f([ab]) && \text{(Induction)} \\
                &= 2nf([ab]).
			\end{align*}
			To apply oriented connected smoothing, we used that $b^2a^2$ and $a^{2(n-1)}b^{2(n-1)}$ are parallel, which follows from the fact that $a$ and $b$ are crossing, and Lemma~\ref{lem:anbn_vs_bkak_parallel}. 
			Finally, using Lemma~\ref{lem:gnhn_to_gsumh} as above, we get
			\[
			f([a] [b]) \geq f([ab])
			\]
			which proves oriented disconnected smoothing.
	
			Unoriented disconnected smoothing is proved in the same way, replacing
			$b$ by~$B$.
\end{proof}

Finally, we show that power smoothing and oriented disconnected smoothing imply stability.

\begin{proposition}
Let $f \colon \Curves^+(S) \to \mathbb{R}$ be a curve functional satisfying:
\begin{enumerate}
\item additivity;
\item power smoothing; and
\item oriented disconnected smoothing.
\end{enumerate}
Then $f$ satisfies stability: $f(\vec C^n) = f(n\vec C) = nf(\vec C)$.
\label{prop:power_disconnected_imply_stable}
\end{proposition}
\begin{proof}
  Let $a \in \pi_1(S)$ be an arbitrary element.
  By the power smoothing hypothesis and additivity, we have
  $f(a^{n+m}) \geq f(a^n) + f(a^m)$.
By Fekete's Lemma~\ref{lem:fekete}, the limit
\[
F(a)\coloneqq\lim_{n^+} \frac{f(a^n)}{n}
\]
exists, and $f(a) \leq F(a)$. We must show $F(a) \leq f(a)$.
  
  Find $b \in
\pi_1(S)$ so that $(\ora{a},\ora{b})$ R-cross, which we can find by
density of endpoints of axes of hyperbolic elements in $G(S)$.
Then, applying oriented disconnected smoothing to the R-crossing pair $(\ora{a^n}, \ora{b})$, we get
\begin{equation}
f(a^n) + f(b) \geq f(a^nb).
\label{eq:cross1_eq}
\end{equation}

We next show that
\begin{equation}
nf(a) + f(b) \geq f(a^nb)
\label{eq:cross2_eq}
\end{equation}
for $n\geq 1$, by induction on $n$.
Indeed, the base case is true by oriented disconnected smoothing, and the
induction follows from
\[
nf(a) + f(b) \geq f(a^{n-1}b) + f(a) \geq f(a^nb)
\]
where in the first step we used the induction step and in the second we used the fact that the pair $(\ora{a},\ora{a^{n-1}b})$ is R-crossing by Lemma~\ref{lem:rhs_prod}.

Finally, by Lemma~\ref{lem:rhs_prod}, the pair $(\ora{B},\ora{a^nb})$
is R-crossing for all $n \geq 1$, and hence, by oriented disconnected
smoothing,
\begin{equation}
f(B) + f(a^nb) \geq f(a^n).
\label{eq:cross3_eq}
\end{equation}

Combining Eqs.~\eqref{eq:cross1_eq}, and~\eqref{eq:cross3_eq},
dividing by $n$ and taking the limit when $n$ goes to infinity, we get
$F(a)=\lim_{n^+} \frac{f(a^nb)}{n}$. Using Eq.~\eqref{eq:cross2_eq}
we get $f(a) \geq F(a)$. Thus, $F(a)=f(a)$, and hence $f$ satisfies stability since $F$ clearly does.
\end{proof}

\begin{remark}
If we modify the statements of
Propositions~\ref{prop:oriented_smoothing_unoriented}
and~\ref{prop:cont_conn_disc} to refer to $K$-quasi-smoothing (for fixed
$K>0$), then the proofs \emph{do not} go through. For example, in the
induction from connected to disconnected
oriented smoothing, inequality Eq.~\eqref{eq:inequality_smoothing} would
have to be replaced by
\[
f([a^{2n}b^{2n}]) \geq 2nf([ab]) - K n.
\]
Following the same chain of inequalities that prove the induction does
not yield the desired inequality, since one accumulates one extra $-K$
factor.
\end{remark}

%%% Local Variables:
%%% mode: LaTeX
%%% TeX-master: "Intersections"
%%% End:

\section{Special classes of geodesic currents}
\label{sec:special-classes}

In this section we characterize special classes of geodesic currents
as curve functionals, such as filling currents, hyperbolic metrics,
measured laminations, integral multi-curves, and give a criterion for
currents dual to embedded graphs.

\subsection{Filling currents}
\label{subsec:filling}
Using a result of Burger, Iozzi, Parreau, and Pozzetti, we can give a
characterization of filling geodesic currents.

\begin{corollary}\label{cor:systole}
  Let $S$ be a closed surface,
  $\mu$ be a geodesic current and $f$ be the functional on curves associated
  to $\mu$, i.e., so that $\mu_f=\mu$. The following subsets are in
  bijection.
\begin{enumerate}
    \item The set of filling currents (Definition~\ref{def:filling_current}).
    \item The set of functionals $f \colon \Curves^+(S) \to
      \mathbb{R}$ which satisfy symmetry, oriented smoothing,
      additivity, and stability, and so that their continuous extension to geodesic currents is positive.
    \item The set of curve functionals $f \colon \Curves^+(S) \to
      \mathbb{R}$ satisfying symmetry, oriented smoothing, additivity,
      and stability for which their systole
      \[\sys(f)\coloneqq \inf \{ f(\vec C) \mid \vec C \in \Curves^+(S),\ \vec C \mbox{ non-trivial} \}\]
      is
      positive.
\end{enumerate}
\end{corollary}
\begin{proof}
These equivalences follow from Theorem~\ref{thm:intersections_group} and \cite[Theorem~1.3]{BIPP21:Systoles}.
\end{proof}

\subsection{Multi-curves}

We say that a weighted multi-curve $M$ is \emph{integral}
if it is a weighted multi-curve whose weights are integers. We say that
$M$ is a \emph{half-integral} multi-curve if its weights are in
$\frac{1}{2}\mathbb{Z}$. Note that, by definition, all integral
multi-curves are also half-integral.

Also define $\AConvex_{\bbZ}(S)$ to be the functionals in
$f \in \AConvex(S)$ so that $f(C) \in \bbZ$ for all closed curves~$C$.

For $M$ a multicurve and $\vec M$ any orientation, let $h(\vec M) \in
H_1(S; \bbZ)$ and
$h_2(M) \in H_1(S;\mathbb{Z}/2\mathbb{Z})$
denote the respective homology classes.
\begin{corollary}
A curve functional $f$ is in $\AConvex_{\bbZ}(S)$ iff it is dual to a
half-integral multi-curve $\mu=\frac{1}{2}M$ with $h_2(M) = 0$.
\label{cor:int_multicurve}
\end{corollary}
\begin{proof}
  Suppose $f$ as in the statement is $\bbZ$-valued.
By Theorem~\ref{thm:intersect} there exists a geodesic current $\mu$ so that
\[
i(\mu, C) = f(C)
\]
for every closed curve $C$.
We must show $\mu$ is a half-integral multi-curve. Fix a crossing pair $a,b\in\pi_1(S)$, and let $R_n^\pm$ and $B_{a,b}^\pm$
be as in Lemma~\ref{lem:measures_parallelogram}. By parts (1), (2), (3) of
that lemma, together with Lemma~\ref{lem:onepoint},
\[
2\,\mu(B_{a,b}^+)
= \lim_{n\to\infty}\bigl(\mu(R_n^+)+\mu(R_n^-)\bigr)
= \lim_{n\to\infty}\bigl(f(a^n)+f(b^n)-f(a^nb^{-n})\bigr).
\]
Since $f$ is $\mathbb{Z}$-valued and stable, every term on the right is
an integer, and a convergent sequence in the discrete set $\mathbb{Z}$ is
eventually constant; hence $\mu(B_{a,b}^+)$ is one-half of an integer. As
$(a,b)$ ranges over crossing pairs the quadruples $(a^+,b^-,a^-,b^+)$
range over all quadruples in cyclic order, so these boxes are dense;
choosing the quadruples nested and using continuity from below and above,
$\mu(B)$ is one-half of an integer for every box $B$. (For the
discreteness argument itself, see
\cite[Appendix~A]{JMG26:CoarseNegativeCurvature}.)
By \cite[Proposition~4.12]{BIPP24:PositiveCrossratios}, it follows
that $2\mu$ must be an integral multi-curve~$M$, and hence
$\mu=\frac{1}{2} M$ is half-integral.

We furthermore claim that $h_2(M)=0$.
Indeed, if $h_2(M) \ne 0$, there exists a simple closed curve $C$ so
that
\[
  1 = \langle h_2(M), h_2(C) \rangle_{\bbZ/2\bbZ}
  \equiv \langle h(\vec M), h(\vec C)\rangle_{\bbZ}
  \equiv i(M, C) \pmod 2,
\]
where $\langle \cdot, \cdot \rangle$ denotes the algebraic
intersection number (with $\bbZ/2\bbZ$ or $\bbZ$ coefficients, as
indicated),

$\vec M$ and $\vec C$ are arbitrary orientations of the respective multicurves,
and equivalences are modulo~2. It follows that $i(\mu,C)
\notin \bbZ$, a contradiction.

For the other direction, if $h_2(M) = 0$ then
for every $C \in \Curves(S)$, we have
$i(M,C) \equiv \langle
h(\vec M), h(\vec C)\rangle_{\bbZ} \equiv 0 \pmod 2$ so 
$i(\frac12 M, C) \in \mathbb{Z}$.
\end{proof}
It is easy to see that if $M = \sum_{i=1}^n \lambda_i C_i$, where $\lambda_i \geq 0$ are non-negative real numbers, then for every closed curve $C$, there exist integer numbers $n_i$ so that
\begin{equation}
i(M, C) = \sum_{i=1}^n \lambda_i n_i.
\label{eq:weighted_multicurve}
\end{equation}
i.e., $\{i(\mu,C)\}_{\Curves(S)}$ is contained in a finite rank subgroup of $\mathbb{R}$. In this case, we say $\mu$ has \emph{finite rank spectrum}.
This property does not characterize weighted multi-curves, since certain negatively curved metrics that satisfy this condition can be constructed (see~\cite{Agol25:Spectrum}).

Hence, it is natural to ask the following two questions, from
conversations of the first author with, respectively, Meenakshy
Jyothis and Viveka Erlandsson.

\begin{question}
    What is the subspace of geodesic currents with finite rank spectrum?
\end{question}

\begin{question}
    What conditions on the (marked) length spectrum characterize weighted multi-curves?
\end{question}

\subsection{Embedded graphs}
\label{subsec:embedded_graphs}
Related to multi-curves, we also obtain geodesic currents dual to stable lengths associated to embedded graphs in surfaces.

 Let $\iota \colon \mathcal{G} \hookrightarrow S$ be an \emph{embedded
   graph}, i.e., an embedding of a finite graph in $S$ that is
 filling, in the sense that the complementary regions are disks, or
 equivalently $i_*$ is surjective on~$\pi_1$.
Endow~$\mathcal{G}$ with a length metric $g$. Then any closed
multi-curve $C$ on $S$ can be homotoped so that it factors through
$\mathcal{G}$, in many different ways. Let $\ell_{\mathcal{G}}(C)$ be
the length of the smallest
 multi-curve $D$ on $\mathcal{G} $ so that $\iota(D)$ is homotopic to $C$. It is easy to see that this length is realized and is positive. 

\begin{theorem}
Let $\iota \colon \mathcal{G} \to S$ be a filling embedded graph. Then $\ell_{\mathcal{G}}$ is dual to a geodesic current $\mu_{\mathcal{G} }$.
\label{thm:dual_embeddedgraph}
\end{theorem}
\begin{proof}
In \cite[Sec.\ 4.4]{MGT21:Smoothings} we show that such lengths are in $\AConvex(S)$, hence by Theorem~\ref{thm:intersect} they have a dual geodesic current. 
\end{proof}

If furthermore $\mathcal{G}$ is an embedded graph with the edge metric
on $\pi_1(S)$ (i.e., each edge has length~$1$), then as stated in the
Introduction, $\mu_{\mathcal{G}}$ must
be $1/2$ of an integral multi-curve. This generalizes the main result
in~\cite{Erlandsson:WordLength} in the case of closed surfaces, where
Erlandsson finds that
word metric arising from \emph{simple generating set} are dual to 1/2
of an integral multi-curve. Such generating sets correspond to
choosing $\mathcal{G}$ to be a rose graph with only one vertex.

\begin{taggedcor}{\ref{maincor:embedded}}
Let $\mathcal{G} \hookrightarrow S$ be a filling embedded graph
endowed with the edge metric. Then
the functional $\ell_{\mathcal{G}}$ is dual to a
geodesic current $\mu_{\mathcal{G}} = \frac12 M$ where $M$ is a
multi-curve so that $h_2(M)=0$ in
$H_1(S; \mathbb{Z}/2\mathbb{Z})$.
\end{taggedcor}
\begin{proof}
Since $\mathcal{G}$ is equipped with the edge-metric, $\ell_{\mathcal{G}}(C) \in \mathbb{Z}$.
Hence, Corollary~\ref{cor:int_multicurve} implies the result.
\end{proof}

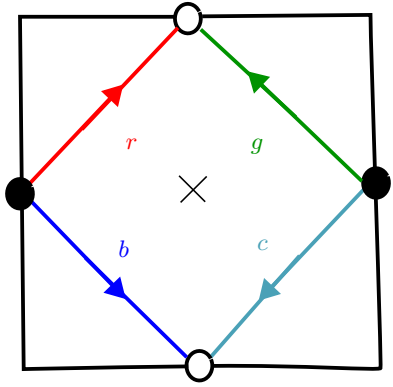
\begin{figure}
\centering{
\fontsize{9pt}{9pt}\selectfont%% Creator: Inkscape 1.3 (0e150ed6c4, 2023-07-21), www.inkscape.org
%% PDF/EPS/PS + LaTeX output extension by Johan Engelen, 2010
%% Accompanies image file 'embedded.pdf' (pdf, eps, ps)
%%
%% To include the image in your LaTeX document, write
%%   \input{<filename>.pdf_tex}
%%  instead of
%%   \includegraphics{<filename>.pdf}
%% To scale the image, write
%%   \def\svgwidth{<desired width>}
%%   \input{<filename>.pdf_tex}
%%  instead of
%%   \includegraphics[width=<desired width>]{<filename>.pdf}
%%
%% Images with a different path to the parent latex file can
%% be accessed with the `import' package (which may need to be
%% installed) using
%%   \usepackage{import}
%% in the preamble, and then including the image with
%%   \import{<path to file>}{<filename>.pdf_tex}
%% Alternatively, one can specify
%%   \graphicspath{{<path to file>/}}
%% 
%% For more information, please see info/svg-inkscape on CTAN:
%%   http://tug.ctan.org/tex-archive/info/svg-inkscape
%%
\begingroup%
  \makeatletter%
  \providecommand\color[2][]{%
    \errmessage{(Inkscape) Color is used for the text in Inkscape, but the package 'color.sty' is not loaded}%
    \renewcommand\color[2][]{}%
  }%
  \providecommand\transparent[1]{%
    \errmessage{(Inkscape) Transparency is used (non-zero) for the text in Inkscape, but the package 'transparent.sty' is not loaded}%
    \renewcommand\transparent[1]{}%
  }%
  \providecommand\rotatebox[2]{#2}%
  \newcommand*\fsize{\dimexpr\f@size pt\relax}%
  \newcommand*\lineheight[1]{\fontsize{\fsize}{#1\fsize}\selectfont}%
  \ifx\svgwidth\undefined%
    \setlength{\unitlength}{144.63555043bp}%
    \ifx\svgscale\undefined%
      \relax%
    \else%
      \setlength{\unitlength}{\unitlength * \real{\svgscale}}%
    \fi%
  \else%
    \setlength{\unitlength}{\svgwidth}%
  \fi%
  \global\let\svgwidth\undefined%
  \global\let\svgscale\undefined%
  \makeatother%
  \begin{picture}(1,0.98614068)%
    \lineheight{1}%
    \setlength\tabcolsep{0pt}%
    \put(0,0){\includegraphics[width=\unitlength,page=1]{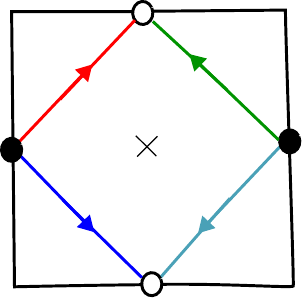}}%
    \put(0.30713501,0.66172668){\color[rgb]{1,0,0}\makebox(0,0)[t]{\lineheight{1.25}\smash{\begin{tabular}[t]{c}$r$\end{tabular}}}}%
    \put(0.64159199,0.67413052){\color[rgb]{0,0.57647059,0}\makebox(0,0)[t]{\lineheight{1.25}\smash{\begin{tabular}[t]{c}$g$\end{tabular}}}}%
    \put(0.69468421,0.30426714){\color[rgb]{0.29019608,0.63137255,0.71764706}\makebox(0,0)[t]{\lineheight{1.25}\smash{\begin{tabular}[t]{c}$c$\end{tabular}}}}%
    \put(0.3350182,0.26255609){\color[rgb]{0,0,1}\makebox(0,0)[t]{\lineheight{1.25}\smash{\begin{tabular}[t]{c}$b$\end{tabular}}}}%
  \end{picture}%
\endgroup%
}
\caption{Example of embedded graph on the torus which does not arise as a simple generating set.}
  \label{fig:embedded-graph}
\end{figure}

\begin{example}[Embedded graph but not simple generating set]
As an example of an embedded graph which does not arise as a simple generating set, consider a square punctured torus $S'$ with an embedded two-vertex graph with four edges, $r, g, c, b$ (for red, green, cyan, blue), see Figure~\ref{fig:embedded-graph}. Since it has more than one vertex, it does not arise from a simple generating set in the sense of Erlandsson. 
Indeed, we have
\[
[x]=[rB]=[gC] , \qquad [y]=[rG]=[bC]
\]
and thus
\[
\ell([x])=2, \qquad 
\ell([y])=2.
\]
Moreover,
\[
\ell([xy])=\ell([(rB)(bC)])=\ell([rC])=2,
\qquad
\ell([xY])=\ell([(gC)(cB)])=\ell([gB])=2.
\]

This phenomenon does not occur for any of the simple generating sets of $\pi_1(S') \simeq F_2$, namely
\[
\{ x^{\pm1},y^{\pm1} \}, \quad
\{ x^{\pm1},(xy)^{\pm1} \},\quad
\{ y^{\pm1},(xy)^{\pm1} \},\quad
\{ x^{\pm1},(xY)^{\pm1} \},\quad
\{ y^{\pm1},(xY)^{\pm1} \},
\]
nor for the analogous generating sets obtained by acting diagonally by the mapping class group. 
Indeed, in none of these cases are $x$, $y$, $xy$, $xY$ the shortest
elements in $\pi_1$ with equal length.

Even though this example corresponds to a punctured surface and hence does not fall in the scope of our theorem, one can produce similar examples in genus 2 from this one.
\end{example}

We note that our result says nothing concrete about the dual multi-curve, whereas Erlandsson's result explicitly constructs it in the specific case of simple generating sets. It would be interesting to describe these multi-curves in general.

\begin{question}\label{quest:embedded-converse}
  Is the converse of Corollary~\ref{maincor:embedded} true? That is,
let $C=\frac{1}{2}M$, where $M$ is an integral
multi-curve with $h(M) = 0$ in $\mathbb{Z}/2\mathbb{Z}$-homology.
Does there exist a filling embedded graph $\mathcal{G} \hookrightarrow S$ endowed with the edge metric so that
\[
\ell_{\mathcal{G}}(D)=i(C,D)
\]
for every closed curve $D \in \Curves(S)$?
\end{question}

A positive answer to Question~\ref{quest:embedded-converse} would give
an equivalence between
\begin{itemize}
\item $\AConvex_{\mathbb{Z}}(S)$;
\item $1/2$-integral multi-curves $M$ which are zero in
  $\mathbb{Z}/2\mathbb{Z}$-homology; and
\item filling embedded graphs in~$S$.
\end{itemize}

\subsection{Length metrics}
\label{subsec:length-metrics}
We now prove Corollary~\ref{maincor:length_dual}: any length induced
by a metric (in fact, by a pseudometric) on $S$ has a dual geodesic
current. We will then use that result to derive some corollaries about
domination.

We first recall the definition of a length induced by a pseudometric on $S$. See~\cite[Chapter~I.1]{BH11:NonPosCurvature} for a definition in the context of metrics. A \emph{pseudometric space} is a pair $(X,d)$, where $X$ is a set and $d \colon X \times X \to \mathbb{R}_{\ge 0}$ is a function satisfying, for all $x,y,z \in X$,
\[
d(x,x)=0,\quad d(x,y)=d(y,x),\quad d(x,z)\le d(x,y)+d(y,z).
\]
 The \emph{$d$-length} $\ell(c)$ of a continuous path $c \colon [a,b] \to X$
is defined as
\begin{equation}
\ell(c) \coloneqq \sup_{a=t_0 \leq t_1 \leq \cdots \leq t_n=b} \sum_{i=0}^{n-1} d(c(t_i),c(t_{i+1})),
\label{eq:length_of_path}
\end{equation}
where the supremum is taken over all possible partitions of $[a,b]$ (no bound on $n$), with
$a=t_0 \leq t_1 \leq \cdots t_n = b$.
The length of a path $c$ is either a non-negative number or it is infinite.

We say
a concrete curve $c$ is \emph{$d$-rectifiable} if its $d$-length is finite.
Consider now a pseudometric $(X,d)$ on $S$.
For a curve $C$ on $S$, we define the \emph{$d$-length} of $C$ as
\[
\ell_d(C) \coloneqq \inf_{\gamma \in C} \ell_d(\gamma)
\]
where $\gamma$ runs over all the concrete curves in the equivalence class of $C$.
We extend this definition of multi-curves linearly.
We call this the \emph{length induced by $d$}.
If there exists at least one $d$-rectifiable concrete curve
representative for each curve $C$, then $\ell_d$ induces a curve
functional. This is the case, for example, when  $(X,d)$ is a
\emph{length pseudometric} on $S$, i.e.,
$d$ is given by the infimum of lengths of paths~\cite[Chapter~I.2, Definition~3.1]{BH11:NonPosCurvature}. Precisely, this means
\begin{equation}\label{eq:length-metric}
d(x,y) = \inf \{ \ell_d(\gamma) \mid \gamma\colon [0,1] \to X, \gamma(0)=x, \gamma(1)=y\}.
\end{equation}
We now prove the corollary from the introduction.
\begin{taggedcor}{\ref{maincor:length_dual}}
  An additive curve functional satisfies the smoothing
  property~\eqref{eq:smoothing}, and is thus dual to a geodesic
  current, if and only if it is the induced length on curves from a
  length pseudo-metric on~$S$.
\end{taggedcor}
\begin{proof}
Additivity is satisfied by definition of $\ell_d$ on multi-curves. It is clear that $\ell_d$ is a symmetric curve functional, since $d$ is symmetric.
The key point is to check that $\ell_d$ satisfies smoothing. 
Let $C$ be a multi-curve with an essential crossing $x$, and let $C'$ be obtained by smoothing $C$ at $x$.

Fix $\epsilon>0$. Since $\ell_d$ is a curve functional, $\ell_d(C)<\infty$ for every multi-curve $C$. By definition of $\ell_d(C)$, there exists a concrete representative $\gamma$ of $C$ with an essential crossing $x'$ of the same homotopy type as $x$ such that
\[
\ell_d(\gamma) \le \ell_d(C) + \epsilon.
\]
In particular, $\gamma$ is $d$-rectifiable.
Let $\gamma' \in C'$ be the concrete multi-curve obtained by smoothing $\gamma$ at $x'$.

By abuse of notation, let $\gamma \colon I \to X$ be a continuous $d$-rectifiable concrete curve parametrizing the concrete multi-curve $\gamma$ where $I$ is a finite union of intervals. Let $t' \neq t''$ be points in (possibly different components of) $I$ such that $\gamma(t') = \gamma(t'') = x'$ correspond to the essential crossing. Let $\gamma' \colon I' \to X$ be a continuous parametrization of the $d$-rectifiable concrete curve $\gamma'$ such that
\[
\gamma(I) = \gamma'(I')
\]
and $I'$ is another finite union of intervals.

By Lemma~\ref{lem:essential-cross}, the chosen representative $\gamma$ has a crossing
$x'$ corresponding to $x$, whose prescribed smoothing represents
$C'$. Cut the parameter domain of $\gamma$ at the two preimages of
$x'$. The smoothed curve $\gamma'$ is obtained by concatenating
exactly the same finitely many restricted paths, only with different
pairings at their common endpoints. Since length in a pseudometric
space is additive under concatenation of rectifiable paths,
\[
\ell_d(\gamma')=\ell_d(\gamma).
\]

By definition of $\ell_d(C')$ as an infimum over representatives of $C'$, we have
\[
\ell_d(C') \le \ell_d(\gamma') = \ell_d(\gamma) \le \ell_d(C) + \epsilon.
\]
Since $\epsilon>0$ was arbitrary, it follows that
\[
\ell_d(C') \le \ell_d(C),
\]
so $\ell_d$ satisfies smoothing.
Thus, by Theorem~\ref{thm:intersect}, there exists a geodesic current $\mu_d$ dual to $\ell_d$.

Finally, we prove the other implication. Given a geodesic current $\mu$ and a choice of hyperbolic metric $\Sigma$ on $S$, there is an induced invariant pseudometric on $\wt{S}$ defined via

\begin{equation}\label{eq:d_mu}
\wt{d^\Sigma_\mu}(x,y) = \frac{1}{2}\left(\mu(G[x,y)) + \mu(G(x,y])\right),
\end{equation}
inducing a metric $d^\Sigma_\mu$ on~$S$
so that 
\[
\ell_{d^\Sigma_\mu}(C)=i(\mu, C).
\]
See~\cite[Section~4]{BIPP21:Systoles} and~\cite{dRMG22:Duals}.
\end{proof}

One corollary is the following folklore result.

\begin{corollary}\label{cor:length-stable}
  Any length pseudometric $(S,d)$ on a surface~$S$ is stable.
\end{corollary}

\begin{proof}
  This follows from Corollary~\ref{maincor:length_dual}. More
  precisely, this is a consequence of
  Proposition~\ref{prop:power_disconnected_imply_stable}.
\end{proof}

We continue with some remarks.

\begin{remark}\label{rmk:pseudo_metrics}
 For any geodesic current~$\mu$, the pseudometric $\wt{d_\mu}$ on $\wt{X}$ is
 \emph{straight} (\cite[Proposition~4.1]{BIPP21:Systoles}) in the
 sense that if $x,y,z \in l$ for any $l \in G(\Sigma)$, and $x<y<z$ in
 the natural order, then
\[
\wt{d_\mu^\Sigma}(x,z)=\wt{d_\mu^\Sigma}(x,y)+\wt{d_\mu^\Sigma}(y,z).
\]
(The $\wt{d_\mu^\Sigma}$ pseudometric need not be continuous with respect to the topology of $\wt{\Sigma}$.)
If $d$ is a length pseudometric on $S$ and
$\widetilde d$ is its lifted length pseudometric on $\widetilde S$,
then Corollary~\ref{maincor:length_dual} produces a geodesic current $\mu$ such that the
associated straight invariant pseudometric $d^\Sigma_\mu$ has the
same marked length spectrum:
\[
\ell_{d^\Sigma_\mu}(C)=\ell_d(C)
\]
for every closed curve $C$.
\end{remark}

\begin{remark}
Suppose $(S,d)$ is a length pseudometric structure on $S$. Let
$\wt{d}$ be the lift of the metric $d$ to $\wt{S}$, where $\pi_1(X)$
acts by isometries. Then $\ell_d$ is the same
as translation length on $\wt S$
\cite[8.3.13]{BBI01:MetricGeometry}:
\[\ell_d([g])=\inf_{x \in \wt{X}} \wt{d}(x, g x) \eqqcolon \tau_{\wt{d}}(g),\]
This translation length is stable by Corollary~\ref{cor:length-stable}.
More broadly,
Corollary~\ref{cor:length-stable} holds for \emph{convex} Gromov
hyperbolic spaces~\cite[Section~2.6]{Gromov1981HyperbolicManifolds}
but not for more general spaces like geodesic Gromov hyperbolic
spaces \cite[Proposition~6.4]{CDP90:Lectures}; see also Example~\ref{examp:qf}.
\end{remark}

\begin{remark}
  The pseudometric $d_\mu$ constructed from a geodesic current~$\mu$
  by Eq.~\eqref{eq:d_mu} is generally not continuous or positive. In
  fact, $d_\mu$ is continuous iff $\mu$ has no atoms
  \cite[Propositions~4.7, 4.12]{dRMG22:Duals}.
\end{remark}

\begin{question}\label{quest:pseudometric-continuous}
  Can the length pseudometrics in Corollary~\ref{maincor:length_dual}
  be improved so that $d(\cdot,\cdot)$ is continuous on $S \times S$?
  If furthermore the systole of the curve functional is positive, can
  $d(\cdot,\cdot)$ be made into an ordinary length metric? (See
  Corollary~\ref{cor:systole}.)
\end{question}

As motivation for Question~\ref{quest:pseudometric-continuous},
consider the easiest case where continuity and positivity both fail
for the construction above,
when $\mu = \delta_C$ is supported on a multi-curve~$C$. In this case
we can make the metric continuous, as we now sketch. Pick a minimal position
for~$C$, with only transverse self-intersections and no triple
intersections, and pick a small tubular thickening $N(C)$. Define the
length pseudometric on small scales as follows.
\begin{enumerate}[start=0]
\item Within a single connected component of the complement of
  $N(C)$, distances are identically zero.
\item Within a thickening of a single strand of~$C$, measure distances
  by the transverse measure of~$C$.
\item Within a thickening of a crossing, topologically a square, use
  the Manhattan metric.
\end{enumerate}
These pieces are then assembled into a global length pseudometric as
usual using Eq.~\eqref{eq:length-metric}.
Several variations are possible. In fact, if $C$ is filling, we can
make the pseudometric into an honest metric as follows.
\begin{enumerate}[start=0]
\item The complementary regions of $N(C)$ are all disks, by definition
  of ``filling''. Collapse those disks to points; this does not change
  the topology of~$S$.
\item Inside thickenings of a single strand, add a longitudinal
  component to make a Manhattan metric. This does not change lengths
  of curves: shortest paths will cut straight across these regions,
  and the longitudinal component will not matter.
\item No change inside thickenings of crossings.
\end{enumerate}

\subsection{Domination}\label{sec:domination}
We take a brief detour from the special classes of currents to
consider questions related to domination.
Let $\mu_1,\mu_2 \in \Curr(S)$. Say $\mu_1 \leq \mu_2$, or $\mu_2$ \emph{dominates} $\mu_1$, if $i(\mu_1,C) \leq i(\mu_2,C)$ for all $C \in \Curves^+(S)$.
As an application of Corollary~\ref{maincor:length_dual}, we answer positively the following question at the end of \cite{NC01:ShortestGeodesics}.

\begin{question}[Neumann-Coto, 2001]
Let $C_1,C_2$ be two multi-curves. If $C_2$ dominates $C_1$, does it follow that $\ell_g(C_1) \leq \ell_g(C_2)$ for all Riemannian metrics $g$ on $S$?
\end{question}

In fact, we give a stronger result.
Recall that given any length pseudo-metric $d$ on $S$, its
induced curve length $\ell_d \colon \mathcal{C}(S) \to \mathbb{R}$
extends continuously to geodesic currents.
(See~\cite[Theorem~1.5]{EPS20:CountingCurves} for the metric case
and \cite[Corollary~6.13]{CMGR26:GreenMetrics} or
Corollary~\ref{maincor:length_dual} for the pseudo-metric case.)
\begin{corollary}\label{cor:domination}
Let $\mu_1, \mu_2$ be two geodesic currents. If $\mu_2$ dominates $\mu_1$, then
    $\ell_d(\mu_1) \leq \ell_d(\mu_2)$ for all length pseudometrics $d$ on $S$.
\end{corollary}

\begin{proof}
By Corollary~\ref{maincor:length_dual}, for every length pseudometric $d$ on $S$, there exist a dual geodesic current $\mu_d$ for the induced curve functional $\ell_d$.
Approximating the current $\mu_d$ by weighted multi-curves and using continuity of intersection number on currents the result follows.
\end{proof}

\subsection{Measured laminations}
\label{subsec:laminations}

We will now characterize measured laminations among curve functionals.

We will need a relation about the measure of sets intersecting pairs of geodesic segments.

We introduce the following notation. Given two half-closed half-open
geodesics segments $[x,w)$ and $[y,z)$, let $G([x,w), [y,z))$ be the set of bi-infinite geodesics in $G(S)$ intersecting $[x,w)$ and $[y,z)$ transversely. We will refer to sets of this type as \emph{double transversals}. Given a geodesic segment $[x,y)$, a geodesic $\gamma$ intersects $x$ transversely if $x\in\gamma$ and $[x,y) \nsubseteq \gamma$.

Let $x,y,z,w \in \wt{\Sigma}$ appear counter-clockwise as vertices of an
embedded convex geodesic quadrilateral on $\wt{\Sigma}$.

\begin{lemma}
\label{lem:otalclaim}
Given the setup described above, we have 
\begin{multline*}
\mu(G([x,w), [y,z))) + \mu(G((x,w], (y,z])) = \mu(G[x,z)) + \mu(G[w,y))- \mu(G([w,z)) - \mu(G([x,y)).
\end{multline*}
\end{lemma}
\begin{proof}
This is~\cite[Lemma~6.2]{dRMG22:Duals}.
\end{proof}
We also need the following hyperbolic geometry lemma, about \emph{hyperbolic parallelograms} (Appendix~\ref{sec:hyp-par}). It can be seen as a refinement of~\cite[Lemma~5.6]{NCS23:Property}.

Consider Figure~\ref{fig:smoothing_1}.

\begin{proposition}
Consider two elements $a, b \in \pi_1(S)$ with hyperbolic axes $\ora{a}$ and $\ora{b}$ that intersect at a point $O$, and translation lengths $\ell(a),\ell(b)$.
We define the following four points:

\begin{itemize} \item $P$ is the point on $\ora{a}$ obtained by moving forward from $O$ by distance $\ell(a)/2$;
\item $Q$ is the point on $\ora{b}$ obtained by moving backward from $O$ by distance $\ell(b)/2$;
\item $Q'$ is the point on $\ora{b}$ obtained by moving forward from $O$ by distance $\ell(b)/2$;
\item $P'$ is the point on $\ora{a}$ obtained by moving backward from $O$ by distance $\ell(a)/2$.
\end{itemize}
Then:
\begin{enumerate}
\item The segment $(Q,P]$ is on the line $\ora{ab}$ and has length $\ell(ba)/2$.
\item The segment $(P',Q']$ is on the line $\ora{ba}$ and has length $\ell(ba)/2$.
\item The segment $(Q',P]$ is on the line $\ora{bA}$ and has length $\ell(bA)/2$.
\item The segment $(P',Q]$ is on the line $\ora{Ab}$ and has length $\ell(Ab)/2$.
\end{enumerate}
Furthermore, $G((Q,P])\cup a\cdot
G((P',Q'])$ is a Borel fundamental domain for the
action of $\langle ab \rangle$ on $G(\ora{ab})$.
Similarly, $G((Q',P]) \cup b \cdot G((P',Q])$ is a Borel fundamental domain for
 the action of $\langle bA\rangle$ on $G(\ora{bA})$.
\label{prop:keygeometry}
\end{proposition}

  \begin{figure}
\centering{
\fontsize{9pt}{9pt}\selectfont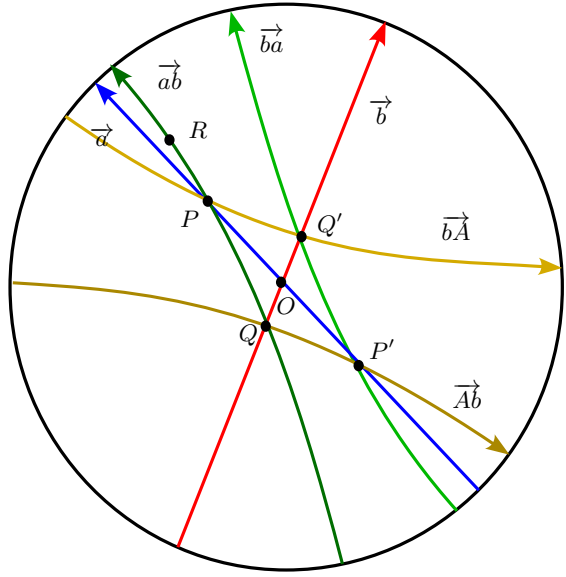}
\caption{Hyperbolic geometry of the smoothing between two distinct curves.}
\label{fig:smoothing_1}
\end{figure}

\begin{proof}
  We prove the first case.
  For $z$ any of the five points in the statement, let $r_z$ be
  the hyperbolic rotation by angle $\pi$ fixing~$z$. Then $a$, as
  acting on the hyperbolic plane, is $r_P r_O$, while $b = r_O r_Q$.
  (Recall that a composition $r_zr_w$ of two rotations by $\pi$ is a
  translation on the line $\overleftrightarrow{zw}$ by twice the
  distance between $z$ and~$w$.)
  It follows that $ab = r_P (r_O)^2 r_Q =
  r_Pr_Q$, from which the statement about $(Q,P]$
  follows. The other cases are similar.

For the assertion on fundamental domains, we note that the segment
$(P',Q']$ on the axis of $\ora{ba}$, is sent by $a$ to a segment
$(P,R]$, where $R=a(Q')$ appears forward of $P$ according to the
orientation of $\ora{ab}$, and the segment $(Q,R]$ is a fundamental
domain for the action of $ab$ on its axis $\ora{ab}$. Thus $\{
\ora{ab} \} \times G((Q,R])= \{ \ora{ab} \} \times  (G((Q,P]) \cup a
\cdot G((P',Q']))$ is a fundamental domain for the action of $ab$ on
$\DG(\Sigma)$, as desired.
The other claim is similar.
\end{proof}

Using it, we characterize measured laminations among curve
functionals.
We start with a  lemma that first requires defining some sets of interest.

Given elements $a, b \in \pi_1(S)$ which R-cross,
consider the hyperbolic parallelogram~$H$ as in Proposition~\ref{prop:keygeometry}, bounded by the vertices $Q,P,P',Q'$.
Consider the sets
\begin{itemize} \item $R_{a,b}^+ \coloneqq G([Q,P), [P', Q'))$, and $(R_{a,b}^{\perp})^+ \coloneqq G([P',Q) \times [Q', P))$, 
\item $R_{a,b}^- \coloneqq G((Q,P], (P', Q'])$, and $(R_{a,b}^{\perp})^- \coloneqq G((P',Q] \times (Q', P])$
\end{itemize}
consisting, respectively, of geodesics crossing from the segment of $H$ contained in $\ora{ab}$ to the open segment of $H$ contained in $\ora{ba}$, and from the segment of $H$ contained in $\ora{a^{-1}b}$ to the one in $\ora{ba^{-1}}$.
Consider also the boxes of geodesics:
\begin{itemize} \item $B_{a,b}^+ = (a^+, b^-] \times [a^-, b^+)$ and  $(B_{a,b}^{\perp})^+ = (b^-, a^-] \times [b^+, a^+)$, 
\item $B_{a,b}^- = [a^+, b^-) \times (a^-, b^+]$ and   $(B_{a,b}^{\perp})^- = [b^-, a^-) \times (b^+, a^+]$.
\end{itemize}
We have $R_{a,b}^{\pm} \subset B_{a,b}^{\pm}$ and $(R_{a,b}^{\perp})^{\pm}  \subset (B_{a,b}^{\perp})^{\pm} $.
Let $B_{a,b}$ and $B_{a,b}^\perp$ be the interior of $B_{a,b}^{\pm}$ and $(B_{a,b}^{\perp})^{\pm} $ respectively.
We will resort to this setup several times in the proofs below.

\begin{lemma}\label{lem:onepoint}
Let $\mu$ be any geodesic current and let $a,b\in\pi_1(S)$ cross. Then
\[
\mu\bigl((B_{a,b}^+\cup B_{a,b}^-)\setminus B_{a,b}\bigr)=0,
\]
and likewise with $B_{a,b}^\perp$ in place of $B_{a,b}$.
\end{lemma}
\begin{proof}
Each piece of the difference is of the form $\{x\}\times J$, where $x$ is
one of $a^\pm,b^\pm$ and $J$ is a half-open interval in
$\partial_\infty S\setminus\{x\}$ whose other endpoint is the other fixed
point of the element $g\in\{a,b\}$ fixing $x$. In the coordinate on the
pencil $P(x)$ in which $g$ acts by translation, $J$ is a half-line, hence
a disjoint union $\bigsqcup_{k\ge0} g^{-k}J_0$ of bounded pieces. Since
$g$ fixes $x$, the sets $\{x\}\times g^{-k}J_0$ are translates of one
another under $g$, so by $\pi_1(S)$-invariance they all carry the same
$\mu$-mass $m$. Then $\mu(\{x\}\times J)$ is $0$ if $m=0$ and $+\infty$
otherwise; but it is bounded above by
$\mu(B_{a,b}^+)+\mu(B_{a,b}^-)$, which is finite. Hence $m=0$.
\end{proof}

The following lemma  says that when taking large powers of $a$ and $b$, the corresponding rectangles approximate boxes of geodesics.

  \begin{figure}
\centering{
\fontsize{9pt}{9pt}\selectfont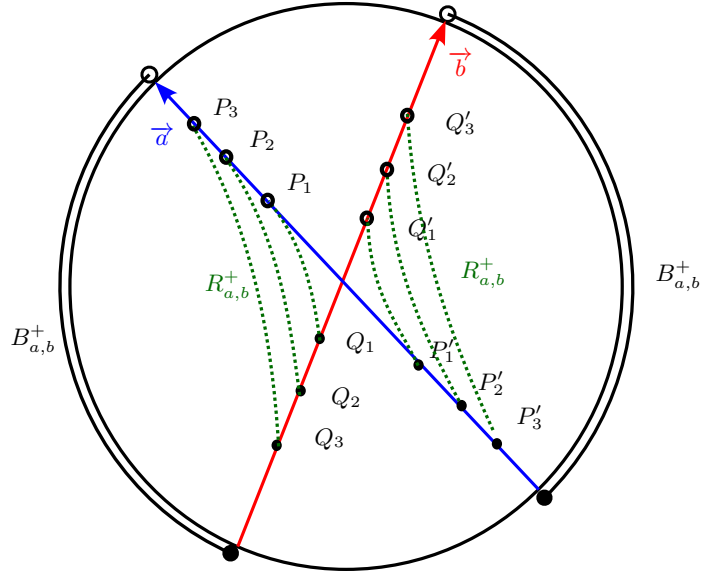}
\caption{Rectangle $R_3^+$ and box $B_{a,b}^+$.}
\label{fig:rectangles}
\end{figure}

\begin{lemma} 
Given elements $a, b \in \pi_1(S)$ which R-cross,
let $B_{a,b}^{\pm}, (B_{a,b}^{\pm})^{\perp}, R_{a,b}^{\pm}, (R_{a,b}^{\pm})^\perp$ as above. Suppose $\mu$ is an arbitrary geodesic current.

For every $n > 0$, let $R_n^{\pm}  \coloneqq (R_{a^n,b^n})^{\pm} $ and $(R_n^{\perp})^{\pm}  \coloneqq (R_{a^n,b^n}^{\perp})^{\pm}$.
Then:
\begin{enumerate} 
\item $R_n^{\pm}  \subset R_{n+1}^{\pm}$, and $\bigcup_{n \geq 1} R_n^{\pm}  = B_{a,b}^{\pm}$ and $(R_n^{\perp})^{\pm}  \subset (R_{n+1}^{\perp})^{\pm}$ and $\bigcup_{n \geq 1} (R_n^{\perp})^{\pm} = (B_{a,b}^{\perp})^{\pm}$.
\item We have $\lim_{n \to \infty} \mu(R_n^{\pm}) \to \mu(B_{a,b}^{\pm})$ and $\lim_n \mu((R_n^{\perp})^{\pm})=\mu((B_{a,b}^{\perp})^{\pm})$.
\item  If we let $c_n=a^n b^{-n}$ and $d_n=a^n b^n$, we have 
\begin{equation}\label{eq:Rn-meas}
\mu(R_n^+)+\mu(R_n^-)=f(a^n) + f(b^n)-f(c_n),
\end{equation}
and
\begin{equation}\label{eq:Rn-perp-meas}
\mu((R_n^\perp)^{+}) + \mu((R_n^\perp)^{-})=f(a^n) + f(b^n) -f(d_n).
\end{equation}
\item We have $\mu(\partial B_{a,b}^{\pm})=0$ and
$\mu(\partial (B_{a,b}^{\perp})^{\pm})=0$, provided $\mu$ has no atom at a
lift of $[a]$ or of $[b]$. In general the difference
$(B_{a,b}^+\cup B_{a,b}^-)\setminus B_{a,b}$, and likewise for
$B_{a,b}^{\perp}$, is $\mu$-null (Lemma~\ref{lem:onepoint}).
\end{enumerate}
\label{lem:measures_parallelogram}
\end{lemma}
\begin{proof}
Item (1) is by construction of $R_n^{\pm}$.
Item (2) follows by Item (1) and continuity of measures from below.
Item (3) follows by Lemma~\ref{lem:otalclaim} and
Proposition~\ref{prop:keygeometry}. Finally, Item (4): by
Lemma~\ref{prop:spikes} the only possible atom in the pencil at
$a^\pm$ is $\ora{a}$ (and at $b^\pm$, $\ora{b}$), and these are
excluded from the half-open boxes; the general bound on the boundary
is Lemma~\ref{lem:onepoint}.
\end{proof}

Recall that a geodesic current is a measured lamination if and only if for every open box of geodesics $B$, we have $\min \{\mu(B),\mu(B^{\perp})\}=0$ (see, e.g.,~\cite[Proposition~2.1]{BIPP24:PositiveCrossratios}).
We now prove  Theorem~\ref{thm:lam-char}, whose statement we recall.

\begin{taggedthm}{\ref{thm:lam-char}}
Let $f \colon \pi_1(S) \to \mathbb{R}$ be a functional on~$\pi_1(S)$.
Then there exists a measured lamination~$\mu_{f}$ on $S$ such that
\[
\text{for every } g\in \pi_1(S),\quad
i(\mu_{f},[g])=f(g)
\]
if and only if $f$ satisfies the conditions of symmetry, class invariance,
stability, oriented smoothing,
and, when $(\ora{g}, \ora{h})$ are crossing,
\[
f(g)+f(h)=\max\{f(gh),f(gH)\}.
\]
\end{taggedthm}

\begin{proof}
Suppose first $f$ satisfies the conditions in the hypotheses including
disconnected max-smoothing.
By Theorem~\ref{thm:intersections_group}, there exists a geodesic current $\mu_f$ so that $f(C)=i(\mu_f,C)$ for all closed curves $C$.

It remains to show that $\mu_f$ is a measured lamination.
Given elements $a, b \in \pi_1(S)$ which R-cross,
consider the hyperbolic parallelogram~$H$ at the center of
Figure~\ref{fig:smoothing_1}, and the notation and setup of Lemma~\ref{lem:measures_parallelogram}.
By the hypothesis of max-smoothing, for every $n$, we have $\mu((R_n^\perp)^+)=\mu((R_n^\perp)^-)=0$ or $\mu(R_n^{+})=\mu(R_n^{-})=0$, by Eq.~\eqref{eq:Rn-meas} and Eq.~\eqref{eq:Rn-perp-meas}.
Since each sequence of sets $R_n^{\pm}$ and $(R_n^{\perp})^{\pm}$ is
nested, either $\mu(R_n^{\pm}) = 0$ for every $n$, in which case
Lemma~\ref{lem:measures_parallelogram}(1) gives
$\mu(B_{a,b}^{\pm})=0$; or there is some $N$ with $\mu(R_N^{\pm}) \neq
0$, and then $\mu((R_n^{\perp})^{\pm})=0$ for all $n \geq N$, whence
$\mu((B_{a,b}^{\perp})^{\pm})=0$.

By density of endpoints of closed geodesics, these boxes have vertices on a dense set of points of $\partial_{\infty} S$.
The pairs of the diagonal endpoints of the family of boxes $B_{a,b}^{\pm}$ (as $a,b$ runs over crossing pairs in $\pi_1(S)$) are dense in the space of geodesics. Hence, we can find sequences of crossing pairs $a_n, b_n \in \pi_1(S)$ so that if $B_n^{\pm} \coloneqq B_{a_n, b_n}^{\pm}$, we have $B_n^{\pm} \subset B_{n+1}^{\pm}$, $(B_{n+1}^{\perp})^{\pm} \subset (B_{n}^{\perp})^{\pm}$, $B=\bigcup_n B_n^{\pm}$, and $B^\perp \subseteq \bigcap_n (B_n^{\perp})^{\pm}$.
Choosing the periodic endpoints successively from the appropriate
sides, we may arrange that
\[
B_n^\pm\subset B_{n+1}^\pm,\qquad
(B_{n+1}^\perp)^\pm\subset(B_n^\perp)^\pm,
\]
with
\[
\bigcup_n B_n^\pm=B,\qquad
\bigcap_n(B_n^\perp)^\pm\supseteq B^\perp.
\] For every $n$,
\[
\min\{\mu(B_n),\mu(B_n^\perp)\}=0.
\]
If $\mu(B_n)=0$ for every $n$, continuity from below gives
$\mu(B)=0$. Otherwise, choose $N$ with $\mu(B_N)>0$. Then
$\mu(B_n^\perp)=0$ for every $n\ge N$, and continuity from above
gives $\mu(B^\perp)=0$, and thus $\mu$ is a measured lamination.

Now we prove the other implication, by proving the contra-positive.
If disconnected max-smoothing is \emph{not} satisfied for the elements
$a,b$, then by Equations~\eqref{eq:Rn-meas}
and~\eqref{eq:Rn-perp-meas}, both:

\begin{itemize}
    \item $\mu(R_{a,b}^{+})>0$ or $\mu(R_{a,b}^{-})>0$;
    \item and
$\mu((R_{a,b}^{\perp})^{+})>0$ or $\mu((R_{a,b}^{\perp})^{-})>0$
\end{itemize} 
For every $\gamma \in R_{a,b}^+ \cup R_{a,b}^-$ and every $\gamma' \in (R_{a,b}^\perp)^+\cup (R_{a,b}^\perp)^-$, $\gamma$ and $\gamma'$ cross. Consequently, we have $i(\mu, \mu) >0$,
and hence $\mu$ is not a measured lamination.
\end{proof}

The above result shows that if one knows $f$ is dual to a geodesic
current, then it is only necessary to check disconnected max-smoothing
to conclude $f$ is a measured lamination, i.e., one does not have to
check essential \emph{self}-crossings of connected components of multi-curves. It is however true that intersection
number with a measured lamination also satisfies connected
max-smoothing, along the lines of Theorem~\ref{thm:intersect}.
\begin{theorem}\label{thm:char_laminations_plus}
A curve functional $f \colon \Curves(S) \to \mathbb{R}$ is dual to a measured lamination if and only if it satisfies additivity and the \emph{max-smoothing property}:
    For any essential crossing of a multi-curve $C$, we have
    \begin{equation}\label{eq:MLsmoothing}
      f\left(\mfig{curves-0}\right) = \max \left\{ f\left(\mfig{curves-1}\right),f\left(\mfig{curves-2}\right) \right\}.
    \end{equation}
\end{theorem}

We start by proving some preliminary lemmas. First, we prove hyperbolic lengths satisfy a ``quasi-max-smoothing'' inequality.

\begin{lemma}\label{lem:quasimaxsmoothing}
    Let $\Sigma = (S,g)$ be a hyperbolic surface, and $a,b \in \pi_1(S)$ be any two elements. Then
    \[
    \ell_\Sigma(ab) \le \max \{\ell_\Sigma(a) + \ell_\Sigma(b), \ell_\Sigma(aB)\} + 2 \log(6).
    \]
\end{lemma}
\begin{proof}
Let $\lambda_\Sigma(g)=2\cosh\left(\frac{\ell_\Sigma(g)}{2}\right)$.
Let $\rho \colon \pi_1(S) \to \SL(2,\bbR)$ be an arbitrary lift of the hyperbolic representation into $\PSL(2,\bbR)$ of $\Sigma$ (which exists because the Euler class is $2-2g$, so even).
The trace identity in $\SL(2,\bbR)$ gives
\[
\tr(\rho(a)\rho(b))=\tr(\rho(a))\tr(\rho(b))-\tr(\rho(aB)). 
\]
Note that $\abs{\tr(\rho(g))} = \lambda_\Sigma(g)$. Taking absolute values, we then get
\begin{equation*}
\lambda_\Sigma(ab) \leq \lambda_\Sigma(a)\lambda_\Sigma(b) + \lambda_\Sigma(aB).
\end{equation*}
Since for every $t\geq0$, $e^{t/2}\leq 2\cosh(t/2)\leq 2e^{t/2}$,
we then get
\[
e^{\ell_\Sigma(ab)/2} \leq 4e^{(\ell_\Sigma(a)+\ell_\Sigma(b))/2}+2e^{\ell_\Sigma(aB)/2}\leq 6\exp\left(\frac12 \max \{ \ell_\Sigma(a)+\ell_\Sigma(b),\ell_\Sigma(aB)\}\right).\qedhere
\]
\end{proof}

\begin{lemma}\label{lem:maxsmooth_ineq}
Given a measured lamination $\lambda$, and \emph{any} pair of elements $a,b \in \pi_1(S)$, we have
\begin{equation*}
i(\lambda,ab) \leq \max \{ i(\lambda, a) + i(\lambda, b), i(\lambda, aB) \}.
\end{equation*}
\end{lemma}
\begin{proof}
  Approximate $\lambda$ by hyperbolic representations:
let $c_n \mathcal{L}_{\Sigma_n} \to \lambda$, where $(c_n)_{n
  \in\mathbb N}$ is a sequence of positive real numbers converging to
$0$, and $\mathcal{L}_{\Sigma_n}$ is a sequence of hyperbolic
Liouville currents (i.e., $i(\mathcal{L}_{\Sigma_n},
C)=\ell_{\Sigma_n}(C)$ for every curve $C$) escaping every compact
subset (see, e.g.,~\cite[Page~152]{Bonahon88:GeodesicCurrent}).
By continuity of intersection number,
\begin{equation*} c_n \ell_{\Sigma_n}(g) \to i(\lambda, g) \text{ for every } g \in \pi_1(S).
\end{equation*}
Now apply Lemma~\ref{lem:quasimaxsmoothing} with $\Sigma=\Sigma_n$.
The error $2c_n\log(6)$
tends to zero as $n \to \infty$, yielding the result.
\end{proof}

\begin{proof}[Proof of Theorem~\ref{thm:char_laminations_plus}]
  Suppose $f$ is a symmetric curve functional satisfying max-smoothing and
  additivity. Max-smoothing clearly implies smoothing, as well as
  disconnected max-smoothing. Thus, by
  Theorem~\ref{thm:lam-char}, $f$ is dual to a measured
  lamination. On the other hand, if $f$ is dual to a measured
  lamination, Theorem~\ref{thm:lam-char} implies disconnected
  max-smoothing, which implies disconnected smoothing. Using the fact
  that $f(\cdot) \coloneqq i(\lambda, \cdot)$ is continuous,
  symmetric, additive and stable,
  Proposition~\ref{prop:cont_conn_disc} then implies connected
  smoothing, that is, if $(\ora{a},\ora{b})$ are parallel, it follows
  that $f([ab]) \geq \max \{ f([a])+f([b]), f([aB])\}$. The other
  inequality follows by Lemma~\ref{lem:maxsmooth_ineq}. This proves connected max-smoothing.
  Finally, power smoothing follows from additivity and stability.
\end{proof}

The fact that measured laminations satisfy max-smoothing was proven by
the second author using
different techniques~\cite{Thu00:GeometricIntersection}.

\subsection{Liouville currents}
\label{subsec:liouville}
Bonahon gives an algebraic characterization for a geodesic current $\mu$ to be a Liouville current $\mathcal{L}_\Sigma$ for a marked hyperbolic structure $\Sigma \in \Teich(S)$, the \emph{Teichm\"uller space} of $S$.

\begin{theorem}[{\cite[Theorem~13]{Bonahon88:GeodesicCurrent}}]   \label{thm:bonahonliouville}
   A geodesic current $\mu \in \Curr(S)$ is $\mu=\mathcal{L}_{\Sigma}$ for some $\Sigma$ if and only if, for any four points $x,y,z,w \in \partial_{\infty} S$, we have
   \[
      e^{-\mu([x,y]\times[z,w])} + e^{-\mu([y,z]\times [w,x])}=1.
   \]
\end{theorem}
Whether the endpoints of the intervals are included or excluded in the
above statement is immaterial, since it follows from the hypothesis that the current is non-atomic. 

We use this to prove Theorem~\ref{thm:hyp-char}. Again we recall the statement.

\begin{taggedthm}{\ref{thm:hyp-char}}
Let $f \colon \pi_1(S) \to \mathbb{R}$.
Then there exists a  hyperbolic metric~$d$
on $S$, with associated Liouville current $\mu_d$, such that
\[
\text{for every } g\in \pi_1(S),\quad
f(g) = \ell_d([g]) = i(\mu_d,[g])
\]
if and only if $f$ satisfies the conditions of symmetry, class invariance,
stability, oriented smoothing,
positivity ($f(g) > 0$ for every non-identity element~$g$),
and, when $(\ora{g},\ora{h})$ are crossing,
\[
\lambda(g)\lambda(h)
  = \lambda(gh)+\lambda(gH)
\]
where $\lambda(g)\coloneqq 2\cosh\!\bigl(f(g)/2\bigr)$.
\end{taggedthm}

\begin{proof}
If $\mu$ is hyperbolic Liouville current then smoothing, additivity, and
stability follow from the fact these are necessary conditions for a
geodesic current, as shown in Appendix~\ref{prop:necessary}. The
$\lambda$-smoothing relation~\eqref{eq:lambda-smoothing} follows from
the relation between the hyperbolic lengths of the sides of a
hyperbolic parallelogram (see Lemma~\ref{lem:hyp_parallelogram})
applied to the hyperbolic parallelogram~$H$ at the center of
Figure~\ref{fig:smoothing_1}, and the fact that the hyperbolic Liouville current is the dual current for the hyperbolic length, by~\cite[Proposition~14]{Bonahon88:GeodesicCurrent}.

The $\lambda$-length functional arises naturally in the study of the
skein algebra of a surface; see, for example,
\cite[Secs.~1--2]{Pen12:Decorated} and
\cite[Sec.~3]{Thurston14:PBI}.
For the other direction, by the assumptions on~$f$ and
Theorem~\ref{thm:intersections_group}, there exists a
geodesic current $\mu_f$ so that $f(C)=i(\mu_f,C)$ for every closed
curve~$C$.
It remains to show that $\mu_f$ is a Liouville current, which is where Equation~\eqref{eq:lambda-smoothing} will play a role.
Given elements $a, b \in \pi_1(S)$ which R-cross, consider the hyperbolic parallelogram~$H$ at the center of
Figure~\ref{fig:smoothing_1}, and the notation and setup of Lemma~\ref{lem:measures_parallelogram}.

Let $A_n \coloneqq f(a^n)=nf(a)$, $B_n \coloneqq f(b^n)=nf(b)$, $C_n \coloneqq f(a^nb^n)$ and $D_n \coloneqq f(a^n b^{-n})$.
Now, we compute
\[
e^{(-\mu(R_n^+)-\mu(R_n^-))/2} + e^{(-\mu((R_n^{\perp})^+)-\mu((R_n^{\perp})^-))/2} = e^{-(A_n+B_n-C_n)/2} + e^{-(A_n+B_n-D_n)/2}.
\]
The identity
\[
\lambda(a^n)\lambda(b^n)
=\lambda(a^n b^{-n})+\lambda(a^n b^n)
\]
gives, after expansion and multiplication by $e^{-(A_n+B_n)/2}$,
\begin{equation}
\begin{aligned}
e^{-(A_n+B_n-C_n)/2}+e^{-(A_n+B_n-D_n)/2}
={}&1+e^{-A_n}+e^{-B_n}+e^{-(A_n+B_n)}\\
&-e^{-(A_n+B_n+C_n)/2}
-e^{-(A_n+B_n+D_n)/2}.
\end{aligned}
\label{eq:comp}
\end{equation}
Because $f(a)>0$ and $f(b)>0$ (since $f(C)>0$ for every curve $C$), it follows that $A_n \to \infty$ and $B_n \to \infty$. 
Thus, since $C_n, D_n \geq 0$, every non-constant term on the
right-hand side of Equation~\ref{eq:comp} tends to zero. Hence
\[
1 = \lim_{n\to\infty}
\left(e^{-(A_n+B_n-C_n)/2}+e^{-(A_n+B_n-D_n)/2}\right)
= \lim_{n\to\infty} e^{\frac{-\mu(R_n^+)-\mu(R_n^-)}{2}} + e^{\frac{-\mu((R_n^{\perp})^+)-\mu((R_n^{\perp})^-)}{2}}.
\]
By Lemma~\ref{lem:measures_parallelogram}(2) we also get $\lim_{n \to \infty} \mu(R_n^{\pm})=\mu(B_{a,b}^{\pm})$ and $\lim_{n \to \infty} \mu((R_n^{\perp})^{\pm})=\mu((B_{a,b}^{\perp})^{\pm})$.
 Letting $B_{a,b}$ and $B_{a,b}^\perp$ denote the corresponding open versions of these boxes, we obtain the relation
\[
e^{-\mu(B_{a,b})} + e^{-\mu(B_{a,b}^{\perp})} =1
\]
for any intersecting pair $a,b \in \pi_1(S)$.

We next show that $\mu$ is non-atomic. Suppose, for
contradiction, that an oriented geodesic $\ell=(\xi,\eta)$ is an atom
of mass $w\coloneqq\mu(\{\ell\})>0$.

By Proposition~3.19, $\ell$ is a lift of a closed geodesic, and it is
the unique possible atom in the pencil $P(\xi)$.

Choose a box $B$ such that $\ell$ lies on exactly one boundary side
of $B$ and is disjoint from $\overline{B^\perp}$. Choose the other
three vertices outside the countable set of endpoints of lifts of
closed geodesics. Thus the only possible boundary atom of either box
is $\ell$.

We approximate the same
cyclic quadruple by periodic-endpoint boxes in two ways. In the first
approximation, the limiting first box contains $\ell$; in the second,
it does not. The limiting transverse box is the same in both
approximations and does not contain $\ell$. Writing $u$ and $v$ for
the common masses away from $\ell$, continuity from below and above
gives respectively $e^{-(u+w)}+e^{-v}=1$
and $e^{-u}+e^{-v}=1$. Subtracting these equalities gives $e^{-u}(e^{-w}-1)=0$,
and therefore $w=0$, a contradiction. Hence $\mu$ is non-atomic. Now we use the
density of $B_{a,b}$ in the space of boxes of geodesics as $a,b$ run
over intersecting pairs in $\pi_1(S)$. Hence, given any open box $B$,
we can find a nested sequence. Specifically, set $B_n \coloneqq
B_{a_n,b_n}$. Then $B_n\subset B_{n+1}$ and $\bigcup_n B_n=B$;
continuity from below gives $\lim_n\mu(B_n)=\mu(B)$. Therefore
\[
e^{-\mu(B)} + e^{-\mu(B^{\perp})}=1.
\]
Therefore, Theorem~\ref{thm:bonahonliouville} shows $\mu$ is a Liouville current.
\end{proof}

%%% Local Variables:
%%% mode: LaTeX
%%% TeX-master: "Intersections"
%%% End:

\section{Skora's theorem: \texorpdfstring{$\protect\mathbb{R}$-trees}{R-trees} dual to measured laminations}
\label{sec:skora}

As an application of Theorem~\ref{thm:intersect}, we prove
Theorem~\ref{mainthm:skora},
giving new
characterizations of surface group actions on
$\mathbb{R}$-trees dual to measured laminations and, among others, giving a new streamlined proof of a classical theorem of
 Skora~\cite{Sko90:GeometricAction}.

\subsection{Finitely generated groups acting on real trees}
We start by recalling background material on $\bbR$-tree actions,
following Culler--Morgan \cite{CM87:GroupActions}.

\textbf{Real trees.}
A \emph{real tree} or \emph{$\mathbb{R}$-tree} is a non-empty metric space $T$ such that for
all $x,y \in T$, the intersection of all connected subsets containing
both $x$ and $y$ is isometric to the interval $[0, d(x, y)]$.  This
intersection is denoted
$[x, y]$, and called the \emph{arc} connecting $x$ and $y$.
When an arc is a singleton, we say it is \emph{degenerate}, otherwise we say it is \emph{non-degenerate}.
An equivalent way to characterize real trees is as geodesic metric
space $X$ satisfying the \emph{$0$-hyperbolic Gromov point condition}:
for any four-tuple of points $x,y,z,w \in X$,
the maximum among the following three quantities appears (at least) twice:
\[
d(x,z) + d(y,w), \quad d(x,y) + d(z,w), \quad d(x,w) + d(y,z).
\]
A \emph{ray} in a real tree is the image of an isometric embedding of
the ray $[0,\infty)$. If $p$ is the image of $0$ under this embedding,
the ray \emph{emanates} from $p$. The \emph{space of ends} of a metric
space $T$ is the limit of the inverse system $\pi_0(T-B)$ where $B$
ranges over the closed and bounded subsets of $T$. Fix a point $p$ in
an $\bbR$-tree $T$. For every end of~$T$, there is a unique
ray
emanating from~$p$ that meets every component of the inverse system
corresponding to the end.
Hence, there is a one-to-one correspondence between
rays emanating from $p$ and ends of $T$. 

Given two distinct
ends of $T$, the corresponding rays meet in an interval with $p$ as
endpoint. The union of these two rays minus the interior of their
intersection is a \emph{line} in $T$. Conversely, given any line in $T$, the
two oppositely oriented rays emanating from any point in the line
determine two distinct ends of the tree, and these ends do not depend
on the choice of point. We will say this line \emph{joins} the two
ends. It follows from the definition of real trees that there is exactly
one line connecting
any two distinct ends. We can alternately define a line directly as the image of an isometric map from $\bbR$.

\textbf{Actions on real trees.} If $\Gamma$ is a finitely generated
group, a
\emph{$\Gamma$-tree} is an $\mathbb{R}$-tree~$T$ together with a group
action $\Gamma \acts T$ by isometries.
There are two types of isometries of real trees: those with fixed
points, called \emph{elliptic}, and those without fixed points, called
\emph{hyperbolic}.
Given $g\in \Gamma$, let $\norm{ g}\coloneqq\inf_{x \in T} d(x,gx)$, and
let $T_g \coloneqq \{ x\in T \mid d(x,gx)=\norm{ g} \}$ be the
\emph{characteristic set} of~$g$; this is a closed, non-empty subtree
of $T$ invariant under the action of $g$.
If $g$ is elliptic, then $\norm{ g } =0$, and $T_g$ is the
set of fixed points of~$g$. The converse is also true: if $\norm{g} =
0$, then $g$ has fixed points \cite[1.3]{CM87:GroupActions}.

If $g$ is hyperbolic, then $\norm{ g } >0$, $T_g$ is isometric to the
real line, and $g$ translates along $T_g$ by distance $\norm{ g}$; in this
case we call $T_g$ the \emph{axis} of $g$.
Let $T_g^+$ be the attracting end of $T_g$ and $T_g^-$ the repelling end.

A $\Gamma$-tree with  translation length function
$\norm{ \cdot }$ is of one of four
types~\cite[pp.\ 572 and 583]{CM87:GroupActions}:
\begin{itemize}
\item \textit{Global fixed point}: $\norm{ \cdot }$ is zero on all of
  $\Gamma$; equivalently,  $\Gamma$ globally fixes a point in~$T$.
\item \textit{Fixed end:} $\norm{ \cdot }$ is non-trivial but zero on $[\Gamma,\Gamma]$;
  equivalently $\Gamma$ fixes an end of $T$ ~\cite[Theorem~2.2,
  Corollary~2.3]{CM87:GroupActions}.
\item \textit{Dihedral:} $\norm{ \cdot }$ is non-trivial on $[\Gamma,\Gamma]$, but
  $\norm{[g,h]}=0$ for all $g,h$ hyperbolic in~$T$; equivalently
  $\Gamma$ fixes a line of $T$ inverting
  orientation~\cite[Theorem~2.5]{CM87:GroupActions}.
\item \textit{Irreducible:} $\norm{[g,h]} \neq 0$ for some $g,h$ hyperbolic in $T$.
\end{itemize}
A \emph{reducible} action of $\Gamma$ is one that is not irreducible,
i.e., it is fixed end, dihedral, or has a global fixed point. 
We say that a $\Gamma$-tree $(\Gamma,T)$ is \emph{minimal} if $T$ is
non-trivial and there is no proper subtree of $T$ that is $\Gamma$-invariant.
If $\norm{\cdot}$ is non-zero then by
\cite[Proposition~3.1]{CM87:GroupActions} there is a
unique minimal invariant subtree.
(If $\norm{\cdot}$ is zero then there
is a global fixed point and so the action is never minimal.)
The length function for the
restriction to this minimal invariant subtree is equal to the length
function of the original tree. Hence, without loss of generality, one
can always restrict to the action on the minimal invariant subtree
without changing the length functional. Note, however, that if the
translation lengths of two minimal $\Gamma$-trees are
equal, that does not necessarily imply the trees are equivariantly
isometric, the exceptions are the non-semi-simple actions, all of which have a fixed end \cite[Theorem 3.7 and \S 2]{CM87:GroupActions}. 

We now recall and recast some classical results on $\mathbb{R}$-actions, drawing on~\cite{CM87:GroupActions, Pa89:Paulin1989TheGT}.

Consider two hyperbolic elements $g, h \in \Gamma$ with axes in the $\mathbb{R}$-tree given respectively by $T_g,T_h$. Let $\Delta \coloneqq T_g \cap T_h$. If $\Delta$ is a non-degenerate interval (i.e., non-empty and not a singleton) we distinguish two types of overlap:
\begin{itemize}
\item if $g$ and $h$ translate in the same direction along $\Delta$, we say
that they \emph{overlap coherently};
\item if $g$ and $h$ translate in opposite
directions along $\Delta$, we say that they \emph{overlap incoherently}.
\end{itemize}
Following~\cite[Definition~3.2]{CM87:GroupActions}, we say that $g,h$ form a \emph{good pair} if 
\begin{itemize} 
\item they are both hyperbolic;
\item their axes in the tree intersect and their overlap is a non-degenerate arc;
\item they overlap coherently; and
\item the length of $T_g \cap T_h$ is strictly less than $\min \{ \norm{ g }, \norm{ h } \}$.
\end{itemize}

\begin{lemma}[{\cite[Remark~3.3]{CM87:GroupActions}}]
If $g,h$ are hyperbolic and $T_g,T_h$ meet in an arc of positive but
finite length, then there exist integers $m>0$ and $n$ so that $g^m$
and~$h^n$
form a good pair of isometries.
\label{lem:powers_goodpair}
\end{lemma}

\begin{proof}[Proof sketch]
   The sign of $n$ is chosen so that the overlap is coherent, and the magnitudes of $m,n$ are chosen
to ensure $m \norm{ g }$ and $n \norm{h}$ are both greater than the length  of $T_g \cap T_h$.
\end{proof}

We next recall a useful criterion that implies reducibility.
\begin{lemma}[{\cite[Lemma~2.1]{CM87:GroupActions}}]
    If there exists a hyperbolic element $g \in \Gamma$ so that for all hyperbolic $h \in \Gamma$ we have  $T_g \cap T_h \neq \emptyset$, then the action is reducible.
    \label{lem:CM_red_criterion}
\end{lemma}

Next, we show that good pairs yield pairs of non-intersecting
hyperbolic elements.

\begin{lemma}
  For any $\Gamma$-tree, if $g,h$ form a good pair, then $gH$ and $Hg$ are hyperbolic isometries of $T$ and $T_{gH}, T_{Hg}$ do not intersect.
    \label{lem:oppgoodpair}
\end{lemma}
\begin{proof}
  This is contained in the proof
  of~\cite[Lemma~3.4]{CM87:GroupActions}. See
  Figure~\ref{fig:trees_goodpair}, where the distance $\Delta$ is
  strictly positive and the position of the relevant points follow
  from the identities $Hg \cdot (G \cdot y) = H \cdot y$ and
  $gH \cdot (h \cdot x) = g \cdot x$, where $[x,y] = T_g \cap T_h$.
\end{proof}

\begin{figure}[ht]
     \centering
\resizebox{80mm}{!}{\fontsize{9pt}{9pt}\selectfont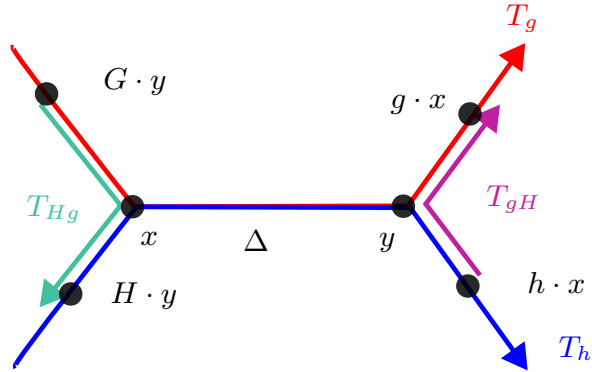}
    \caption{Axes of $gH, Hg$ in $T$ are disjoint when $g,h$ form a good pair}\label{fig:trees_goodpair}
\end{figure}

We now look at the relative positions of the characteristic sets of a pair of elements $g,h \in \Gamma$. 

\begin{lemma}
\label{lem:foundation_trees}
Let $g,h$ be hyperbolic isometries with axes $T_g$ and $T_h$.
We have the following cases, depending on the relative position of $T_g,T_h$.
\begin{enumerate}[(\arabic*)]
\item \label{item:tree-intersect}(Intersecting axes)
$T_g \cap T_h \neq \emptyset$ if and only if $\norm{ g } + \norm{ h } =\max \{ \norm{ gh } , \norm{gH} \}$, as shown in Figure~\ref{fig:trees_intersection}.
Moreover,
\begin{enumerate}
\item $T_g \cap T_h$ is a single point iff $\norm{  g } + \norm{h } = \norm{gh} = \norm{ gH}$;
\item $g,h$ overlap coherently iff 
  $\norm{ g } + \norm{ h } = \norm{ gh } >\norm{gH}$. In this case,
  $T_g \cap
T_h = T_{gh} \cap T_{hg}$, and all four axes intersect coherently there; and
\item $g,h$ overlap incoherently iff 
$\norm{ g } + \norm{ h } = \norm{ gH } >\norm{gh}$. In this case, $T_g \cap T_h = T_{gH} \cap T_{Hg}$.
\end{enumerate}
\item \label{item:tree-disjoint}(Disjoint axes) $T_g \cap T_h =\emptyset$ iff $\norm{gh}=\norm{gH}=\norm{g}+\norm{h}+2d(T_g,T_h)$ with $d(T_g,T_h)>0$,
as shown in Figure~\ref{fig:trees_disjoint}.
In this case, $T_{gh}, T_g$ overlap coherently
in a segment of positive length and
$T_{gh},T_{hg}$ overlap incoherently on a segment of length
$d(T_g, T_h)$.
\end{enumerate}
\end{lemma}

\begin{figure}[ht]
     \centering
\resizebox{80mm}{!}{\fontsize{9pt}{9pt}\selectfont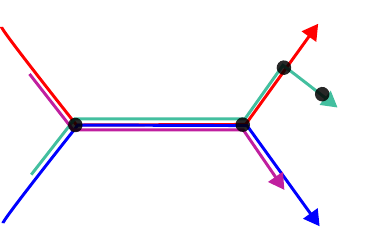}
    \caption{Relative position of intersecting axes in a tree, from
      \cite[Fig.\ 3]{Pa89:Paulin1989TheGT}. See also \cite{CM87:GroupActions}.}
    \label{fig:trees_intersection}
\end{figure}
\begin{figure}[ht]
     \centering
\resizebox{80mm}{!}{\fontsize{9pt}{9pt}\selectfont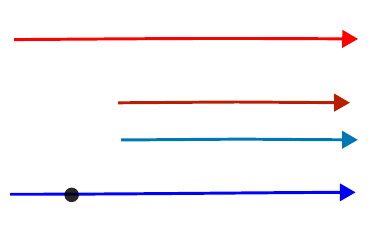}
    \caption{Relative position of disjoint axes in a tree, from
      \cite[Fig.\ 2]{Pa89:Paulin1989TheGT}. See also \cite{CM87:GroupActions}.}\label{fig:trees_disjoint}
\end{figure}

\begin{proof}
The proof of the translation length equalities and inequalities of (1)
and~(2),
as well as the equality $T_g \cap T_h = T_{gh} \cap T_{hg}$
in case~(1b), is in~\cite[Sections~1.5, 1.8 and the proof of Lemma~3.4]{CM87:GroupActions};
case~(1c) is parallel.
The result of part (2) about $T_{gh}, T_g$ overlapping coherently in a segment of positive length can be found in the proof of~\cite[Theorem~2.7]{CM87:GroupActions}.
The statement about the axes $T_{gh},T_{hg}$ overlapping incoherently in the parallel axes case is~\cite[Remark~3.5]{CM87:GroupActions}.
\end{proof}

It will also be useful to determine when a product of hyperbolic elements is hyperbolic.

\begin{lemma}\label{lem:product_hyperbolic}
    Let $T$ be an arbitrary $\Gamma$-tree.
    If $g,h$ are hyperbolic isometries of $T$ so that $\norm{g } \neq \norm{h }$. Then $gh$ is also hyperbolic.
\end{lemma}
\begin{proof}
It follows directly from~\cite[Proposition~1.6]{Pa89:Paulin1989TheGT}.
\end{proof}

The following result will allow us to convert certain incoherent overlaps into coherent ones.

\begin{lemma}[{\cite[Fait, page 854]{Gui05:Coeur}}]
Let $T$ be a $\Gamma$-tree.
Let $g,h$ be two hyperbolic elements of~$\Gamma$ so that
$g \cdot T_h^+ \neq T_h^-$, and let $I$ be a non-degenerate arc contained
in $T_h$. Then for large enough $k \geq 0$, $h^k g h^k$ is hyperbolic
and its axis contains $I$, translating along the same direction as $h$.
\label{lem:make_coherent}
\end{lemma}

We will need analogous statements for pairs $g,h \in \Gamma$ where one
or both group elements are allowed to be elliptic. These have been
analyzed in~\cite{Pa89:Paulin1989TheGT}. We state them here in the
form that will be most convenient for us. 

The last two bullet points are claimed in~\cite[Page~853]{Gui05:Coeur}, but we give a proof here using~\cite[Proposition~1.7]{Pa89:Paulin1989TheGT} for completeness.

\begin{lemma}\label{lem:pau_ell_hyp}
Let $T$ be a $\Gamma$-tree. Suppose $g$ is elliptic and $h$ is
hyperbolic. Then exactly one of the following holds depending on the
relative position of the axes $T_h$ and $T_{ghG}$:
\[ \begin{array}{c@{\qquad}c@{\qquad}c} \text{\emph{(1) coherent overlap or point;}} & \text{\emph{(2) incoherent overlap;}} & \text{\emph{(3) disjoint}} \\[1.2ex] \norm{gh}=\norm{h}=\norm{gH}; & \begin{gathered}\norm{h}\geq \max \{ \norm{gh}, \norm{gH}\}, \\ \norm{h} > \min \{ \norm{gh}, \norm{gH}\}\end{gathered} & \norm{gh}=\norm{gH}>\norm{h}. \end{array} \]
In terms of $T_g \cap T_h$, we have:
\begin{itemize}
\item Case~(3) happens exactly when $T_g$ and $T_h$ are
disjoint, and implies $T_{gh}$ and $T_{hg}$ overlap incoherently.
\item If $T_g \cap T_h$ is a point and $g \cdot T_h^+ \ne T_h^-$, then
  for some $k > 0$, $h^kg$ is hyperbolic.
\item If $T_g \cap T_h$ is empty, or non-empty and not a point, then $gh$ is hyperbolic.
\end{itemize}
\end{lemma}

\begin{proof}
The statement is a rephrased version of \cite[Proposition~1.7]{Pa89:Paulin1989TheGT}. The statement there is divided into two cases depending on the relative position of $T_g, T_h$.  
If $T_g \cap T_h \neq \emptyset$, then $I\coloneqq T_{h} \cap T_{ghG}$ is a non-empty interval and contains $T_g \cap T_h$.
Then, the inequalities between translation lengths follow depending on the overlap between $T_h$ and $T_{ghG}$.
If $T_g \cap T_h = \emptyset$, then the equality of lengths in
case~(3) is stated explicitly
in~\cite[1.5]{CM87:GroupActions}. The fact that
$T_{ghG} \cap T_h = \emptyset$ in this same case follows
from~\cite[1.7]{CM87:GroupActions}.
Since each case is exclusive, the triality depending on the relative position of $T_h$ and $T_{ghG}$ follows.
Now we prove the items in the bullet list.
The first bullet point follows from the analysis in~\cite[Proposition~1.7(2)]{Pa89:Paulin1989TheGT}, in particular, Figure~8 therein. There it is shown that if $z$ denotes the point in $T_g$ closest to $T_h$, then $T_{hg}$ contains the segment $[z,h(z)]$. Since $T_{g}=T_{G}$ for any $g \in \Gamma$, a symmetric argument applied to $G$ and $H$ shows that $T_{HG}$ contains the segment $[z,H(z)]$. In particular, if $y$ denotes the closest point of $T_h$ to $z$, both $T_{hg}$ and $T_{HG}$ contain the segment $[z,y]$, and overlap in that segment coherently. Hence, $hg$ and $gh$ overlap in that segment incoherently.
The second bullet point follows from the fact that the hypotheses imply we are in the setting where $T_g \cap T_h$ is a point in the subcases (ii) and (iii) of~\cite[Proposition~1.7 (1)]{Pa89:Paulin1989TheGT}. In these two settings, $gh$ is hyperbolic provided $\norm{h} >2d(x,z)$
where $x$ is an endpoint of $I$ (which is at most a ray, by assumption), and $\{ z \} = T_g \cap T_h$. Since $T_h = T_{h^k}$, we can guarantee this condition by replacing the pair $(g,h)$ by $(g,h^k)$, where $k$ is taken so that $k\norm{h} > 2d(x,z)$.
For the third bullet point, for $T_g \cap T_h \neq \emptyset$, by~\cite[Proposition~1.7 (1)]{Pa89:Paulin1989TheGT} in all cases where $T_g \cap T_h$ is not a point, we get $\norm{gh}>0$;  for $T_g \cap T_h = \emptyset$, by~\cite[Proposition~1.7 (2)]{Pa89:Paulin1989TheGT}, we get $\norm{gh}>0$, and the claim follows.
\end{proof}

Finally, we need to consider the relative position of the characteristic sets of two elliptic elements.

\begin{lemma}
  Let $T$ be an arbitrary $\Gamma$-tree. If $g,h$ are both elliptic
  isometries, we have $\norm{gh} = \norm{gH}$ and $\norm{gh}>0$ if and only if
  $T_g \cap T_h = \emptyset$. Moreover, if $T_g \cap T_h = \emptyset$,
 then $T_{gh}, T_{hg}$ overlap incoherently  along the shortest connecting arc between $T_g$ and $T_h$.
\label{lem:pau_two_ell}
\end{lemma}
\begin{proof}
The statement is contained
in~\cite[Proposition~1.8]{Pa89:Paulin1989TheGT}, except the equality
$\norm{gh}=\norm{gH}$, and the statement about incoherent overlap, which we outline here. If $T_g \cap T_h \neq
\emptyset$, then $\norm{gh}=\norm{gH}=0$. Otherwise, let $S=[x,y]$ denote
the minimal segment connecting $T_g$ and $T_h$; we already know $S
\subset T_{gh}$ by Paulin's result. Since $T_h = T_H$, we also have $S
\subset T_{gH}$. We therefore have
\[
  \norm{gh} = d(y,gh(y)) = d(y, g(y)) = d(y, gH(y)) = \norm{gH}.
\]
Finally, we prove the statement about incoherent overlap. Let $g(S) \cap S$ be the subsegment $[x,z'] \subset S$ fixed by $g$, and $h(S) \cap S$ the subsegment $[z,y]\subset S$ fixed by $h$. By~\cite[Figure~9]{Pa89:Paulin1989TheGT}, the subarc $[z,g(z)]=[z,gh(z)] \subset T_{gh}$ and the subarc $[z',h(z')]=[z',hg(z')] \subset T_{hg}$ overlap incoherently along the non-degenerate subsegment $[z,z']$.
\end{proof}

\subsection{Surface group actions on real trees}

In this section, we assume $\Gamma = \pi_1(S)$. An \emph{$S$-tree} is
a $\bbR$-tree $T$ with a minimal action of $\pi_1(S)$; in particular
there is no global fixed point.
It is standard that a non-zero measured lamination $\lambda$ defines a
(dual) irreducible $S$-tree (see~\cite{MO84:Valuations}).
It follows from the construction that the translation length of
$g \in\pi_1(S)$  on
$T_\lambda$ is equal to
$i(\lambda,[g])$, where $[g]$ denotes the free homotopy class of a
closed curve associated to $g$.
Given a $S$-tree $T$, if there is a measured lamination $\lambda$ and $\pi_1(S)$-equivariant isometry $\phi \colon T_\lambda \to T$, we say that $T$ is \emph{dual to}~$\lambda$.

Recall that for $\Sigma$ a fixed hyperbolic structure on $S$, the
corresponding hyperbolic axis in $\wt{\Sigma}$ is denoted $\ora{g}$;
be careful to distinguish this from the characteristic set $T_g$ in
the action on~$T$.

In checking whether an $S$-tree is dual to a measured
lamination, recall from the introduction that there are several
relevant conditions on the action:
\begin{itemize}
\item The action is irreducible, not dihedral or fixed end.
\item The action is \emph{small}: the stabilizer of any
  (non-degenerate) arc is cyclic.%
  \footnote{For more general group actions we allow virtually cyclic
    stabilizers, but these are equivalent for surface groups.}
\item The action  \emph{preserves axis intersection}: for any pair
  $g,h \in \pi_1(S)$ acting hyperbolically on $T$, if
  $\ora{g},\ora{h}$ cross then $T_g \cap T_h \neq \emptyset$.
\item The action  \emph{preserves characteristic intersection}: for any pair
  $g,h \in \pi_1(S)$, if
  $\ora{g},\ora{h}$ cross then $T_g \cap T_h \neq \emptyset$. (Here
  $g$ and/or $h$ may be elliptic in their action on $T$.)
\end{itemize}
Morgan--Otal~\cite{MO93:RelativeGrowths} showed that $S$-trees dual to measured laminations
satisfy the first three conditions above. They also showed that if an $S$-tree
is small and preserves axis intersection, it is dual to a measured
lamination, and asked if the ``small'' condition could be dropped.

In a follow-up, Skora noted that dihedral
$S$-trees give counter examples (preserving axis intersection while not
being dual to a measured lamination), but showed that irreducible
$S$-trees that preserve axis intersection are dual to measured
laminations~\cite{Sko90:GeometricAction}.

In this section, we prove Theorem~\ref{mainthm:skora}, in particular showing
that Skora's theorem~\cite{Sko90:GeometricAction} follows readily from
our Theorems~\ref{thm:intersect} and~\ref{thm:lam-char} and
classical considerations
from actions on real trees.

We note that in a later followup,  Skora showed that an $S$-tree with
a small action is also
dual to a measured lamination~\cite{Sk96:Splittings}.
It follows that the action being small is equivalent to it being irreducible and preserving axis intersection.

We start with a preliminary lemma.
\begin{lemma}
Suppose that an $S$-tree $T$  preserves axis intersection and 
suppose that $b$ is hyperbolic. Then there exists $c$ hyperbolic so that $(\ora{c}, \ora{b})$ cross.
\label{lem:hyperbolic_crossing}
\end{lemma}
\begin{proof}
Find $a$ so that $(\ora{a},\ora{b})$ cross, which exists due to density of endpoints.
We follow Guirardel \cite[Fait 1.5 and the proof of Lemme 1.3, sous-cas 1.a, p.~853]{Gui05:Coeur} to show that some
element of the form $c = b^k a^\ell$ with $\ell \ne 0$ is hyperbolic.

If $a$ is hyperbolic, take $c=a$.

Otherwise, consider the relative
position of $T_a$ and~$T_b$.
Suppose first that $T_a \cap T_b$ is empty or larger than a point.
Then by Lemma~\ref{lem:pau_ell_hyp}, $c = ba$ is hyperbolic.

If $T_a \cap T_b$ is a point, consider whether $a \cdot T_b^+
\overset{?}{=} T_b^-$. If this happens, then
necessarily $a^2 \cdot T_b^+ \neq T_b^-$. Set $\ell = 1$ or $\ell = 2$,
whichever one satisfies $a^\ell \cdot T_b^+ \ne T_b^-$.
Then by Lemma~\ref{lem:pau_ell_hyp}, $c = b^ka^\ell$ is hyperbolic for some
$k>0$.

In any of these cases, by
Lemma~\ref{lem:rhs_prod}, $\ora{c}$ crosses $\ora{b}$.
\end{proof}

We also need a geometric criterion for axis-preserving,
irreducible actions.

\begin{proposition}
\label{prop:if_unbounded_then_reducible}
Suppose that an $S$-tree~$T$ preserves axis intersection and is
irreducible. Then for all
hyperbolic $g \in \pi_1(S)$,
there exists a hyperbolic $h \in \pi_1(S)$ so that $\ora{h}$
crosses $\ora{g}$ and $T_g \cap T_h$ is bounded.
\end{proposition}

Before proving this proposition, we make some preliminary lemmas.

We will be supposing there is an element $g$ with unbounded
intersection with every hyperbolic $h$ so that $\ora{h}$ crosses
$\ora{g}$.
In this context, let $E_g \coloneqq \{ h \in \Gamma \mid (\ora{g},\ora{h}) \mbox{ cross and } h \mbox{ is hyperbolic } \}$.
We also pick out subsets of $E_g$ where the intersection is unbounded
on one end or the other:
\begin{align*}
E_g^+
&\coloneqq
  \{ h \in E_g \mid T_h \text{ shares the positive end of } T_g \}
  = \{h \in E_g \mid h \cdot T_g^+ = T_g^+\}\\
E_g^-
&\coloneqq
  \{ h \in E_g \mid T_h \text{ shares the negative end of } T_g \}
  = \{h \in E_g \mid h \cdot T_g^- = T_g^-\}.
\end{align*}

\begin{lemma}
Suppose $a,b$ are hyperbolic.
Suppose $a \in E_g^+ \setminus E_g^-$ and $b \in E_g^- \setminus E_g^+$.
If $m, n \in \mathbb{Z} \setminus \{ 0\}$ is such that $a^m b^n$ is hyperbolic, then $T_{a^m b^n} \cap T_g$ is bounded.
\label{lem:force_bounded}
\end{lemma}
\begin{proof}

  We will show more generally that for every $m, n \in \mathbb{Z} \setminus \{ 0\}$, $a^m b^n$
  fixes neither $T^+_g$ nor $T^-_g$: by assumption, we have
  \begin{align*}
    a^m b^n \cdot T_g^- &= a^m \cdot T_g^- \ne T_g^-\\
    B^n A^m \cdot T_g^+ &= B^n \cdot T_g^+ \ne T_g^+. \qedhere
  \end{align*}
\end{proof}

\begin{lemma}\label{lem:Eg-generates}
  For any irreducible, axis-preserving $S$-tree and every hyperbolic
  $g \in \Gamma$, the subset $E_g$ generates $\Gamma$.
\end{lemma}

Note the irreducibility assumption is necessary: for a dihedral action, every element in $E_g$ fixes both $T_g^+$ and $T_g^-$, but the group as a whole does not.

\begin{proof}
  Fix $h \in \Gamma$. We will write it as a product of elements in $E_g$.
    \begin{case}
\label{case:generate_hyperbolic}
    Suppose first $h$ is hyperbolic.
  \end{case}
    \begin{subcase}
    \label{subcase:generate_hyp_cross}
    $(\ora{g},\ora{h})$  cross.
    \end{subcase} Then $h \in E_g$ and we are done.
    \begin{subcase}
    \label{subcase:generate_hyp_parallel}
    $(\ora{g},\ora{h})$ are parallel.
    \end{subcase}
    By Lemma~\ref{lem:hyperbolic_crossing}, we can take $k \in \Gamma$ so that $\ora{k}$ crosses $\ora{g}$ and $k$ is hyperbolic.
    Thus, for large enough $n>0$, $\ora{k^nh}$ crosses $\ora{g}$ by~Lemma~\ref{lem:convergence_action}, and, if $k$ is hyperbolic, we have $n\norm{k}>\norm{h}$ and by Lemma~\ref{lem:product_hyperbolic}, $k^nh$ is hyperbolic.
    Hence, we can write $h=(k^{-n})(k^n h)$, with $k^{-n}, k^nh \in E_g$.
    \begin{subcase}
      $(\ora{g},\ora{h})$ anti-parallel.
\end{subcase}
      Follow the previous argument, switching $h$ and~$H$.
\begin{case}
\label{case:generate_ell}
    Assume $h$ is elliptic.
    \end{case}
    \begin{subcase}
    \label{subcase:generate_ell_coherent}
    There exists a hyperbolic $b$ so that $b^m h$ is hyperbolic for some $m>0$.
       \end{subcase} 
    Then we can write $h=b^{-m}(b^mh)$, with $b^{-m}$ and $b^mh$ hyperbolic and products of elements of $E_g$ by Case 1.
       \begin{subcase}
       \label{subcase:generate_ell_incoherent}
       No such $b$ exists.
    \end{subcase}
       Then by Lemma~\ref{lem:pau_ell_hyp}, for every hyperbolic element $b$, we have $h \cdot T_b^-=T_b^+$.
    We show this implies the minimal invariant subtree is reducible, reaching a contradiction.
Suppose it is irreducible, so the minimal $\Gamma$-invariant tree is not a line. Since the minimal invariant tree is the union the axes of hyperbolic elements~\cite[the proof of Proposition~3.1]{CM87:GroupActions}, it follows that there must exist hyperbolic elements $b,c$ such that
\[
\{T_b^+,T_b^-,T_c^+,T_c^-\}
\]
contains at least three distinct ends. After relabeling, we may assume that
\[
T_b^\pm=\{\xi,\eta\},
\qquad
T_c^\pm=\{\xi,\zeta\},
\]
where $\eta\neq\zeta$.

Since $h$ swaps the endpoints of every hyperbolic axis, we have
\[
h(\xi)=\eta
\qquad\text{and}\qquad
h(\xi)=\zeta,
\]
a contradiction. Hence all hyperbolic axes have the same pair of endpoints, so the minimal invariant subtree is a line, a contradiction.
\end{proof}

\begin{proof}[Proof of Proposition~\ref{prop:if_unbounded_then_reducible}]
    We will prove a contra-positive. Namely, if the action preserves
    axis intersection and there exists a hyperbolic
    element $g$ so that for every hyperbolic element $h$ with $\ora{h}$ crossing
    $\ora{g}$, the intersection $T_g \cap T_h$ is unbounded, we will deduce
    that the action must be reducible. (Note $T_g \cap T_h$ cannot be
    empty since the action on $T$ is axis-preserving.)
    
    Fix such a~$g$.

We now claim that $E_g = E_g^+$ or $E_g = E_g^-$.
Indeed, suppose there exist $a,b$ hyperbolic so that
$a \in E_g^+ \setminus E_g^-$ and $b \in E_g^- \setminus E_g^+$.
Choose the signs of $m,n$ so that
$\ora{a^m}$ and $\ora{b^n}$ both R-cross $\ora{g}$,
to ensure that $\ora{a^m b^n}$ also R-crosses $\ora{g}$, by Lemma~\ref{lem:rhs_prod}.
Furthermore, choose $\abs{m}$ much larger than $\abs{n}$ so that,
$\norm{a^m} >\norm{b^n}$ and hence, by Lemma~\ref{lem:product_hyperbolic},
$a^m b^n$ is hyperbolic.
Then, by Lemma~\ref{lem:force_bounded}, $T_{a^m b^n} \cap T_g$ is bounded. Thus $\ora{a^m b^n}$ is a hyperbolic element crossing $\ora{g}$ so that $T_{a^m b^n} \cap T_g$ is bounded, a contradiction. 

Thus we may suppose $E_g = E_g^+$. But since $E_g$ generates $\Gamma$
by Lemma~\ref{lem:Eg-generates}, the entire group fixes $T_g^+$ and
the action is fixed-end, hence reducible.
    \end{proof}
    
We are now ready for the main component of Theorem~\ref{mainthm:skora}.

\begin{theorem}
Let $T$ be an irreducible $S$-tree that preserves axis
intersection. Then:
\begin{itemize}
\item $T$ preserves characteristic intersection; and
\item $\norm{ \cdot }$ satisfies max-smoothing. 
\end{itemize}
\label{thm:small}
\end{theorem}

There is one particular fact that we will use multiple times in
different cases, so we pull it out as a lemma in advance. This lemma
is the key place where we invoke irreducibility.

\begin{proposition}\label{prop:parallel-incoherent}
  Let $T$ be an irreducible $S$-tree that preserves axis
  intersection, and pick $g,h \in \pi_1(S)$ with
  $(\ora{g},\ora{h})$ parallel. Then $T_g,T_h$ do not intersect incoherently.
\end{proposition}

\begin{proof}
  Suppose $T_g, T_h$ do intersect incoherently.
    If the overlap is bounded, according to
    Lemma~\ref{lem:powers_goodpair} there exists $m>0,n>0$ so that
    $g^m$ and $h^{-n}$ form a good pair, and hence, by
    Lemma~\ref{lem:oppgoodpair}, $T_{g^mh^n}$ and $T_{h^ng^m}$
    do not intersect. On the other hand, since $\ora{g},\ora{h}$ are
    parallel, it follows by~Lemma~\ref{lem:righthandedfoundation} that
    $\ora{g^mh^n}$ and $\ora{h^ng^m}$ cross, which by preservation of
    axis intersection implies $T_{g^mh^n}$ and $T_{h^ng^m}$  intersect, a
    contradiction.

    Suppose now $T_g,T_h$ have unbounded and incoherent overlap. We will use
    the assumption of irreducibility to construct another pair with
    bounded overlap. Suppose, without essential loss of generality,
    that $(\ora{g},\ora{h})$ are R-parallel,
    and that $T_g \cap T_h$ includes the attracting end of $T_g$ (other configurations follow from the same argument after reversing cyclic orders and interchanging the attracting and repelling ends throughout). By
    Proposition~\ref{prop:if_unbounded_then_reducible}, there exists
    $c_0 \in \pi_1(S)$ acting hyperbolically on~$T$ so that
    $(\ora{c_0},\ora{g})$ R-cross and
    $T_{c_0} \cap T_g$ is bounded, as in Figure~\ref{fig:tree_bounded_incoherent}.
    We may assume that $c_0 \cdot T_g^+ \neq T_g^-$: indeed, otherwise replace $c_0$ by $c_0^2$, since this does not change $T_{c_0}$, it preserves the R-crossing relation and $c_0^2 \cdot T_g^+ = c_0 \cdot T_g^- \neq T_g^-$.
    By
    Lemma~\ref{lem:make_coherent}, if we set $c_1=g^kc_0 g^k$ for
    $k \geq 0$ large enough, the overlap between $T_{c_1}$ and $T_g$
    is coherent and, by Lemma~\ref{lem:rhs_prod}, $\ora{c_1}$ still
    R-crosses $\ora{g}$. Furthermore, this overlap is bounded: $c_1$ fixes an end of $T_g$ if and only if $c_0$ fixes that end, while $c_0$ fixes neither end since $T_{c_0} \cap T_g$ is bounded.
    
    Letting $c_2=g^n c_1 g^{-n}$ for $n$ large
    enough, we can ensure two things. On the one hand, the endpoints of
    $\ora{c_2}$ are near $g^+$ on the boundary of the disk and thus
    $(\ora{c_2},\ora{h})$ are
    parallel. On the other hand, the same conjugation also ensures
    that the (still bounded and coherent) intersection between
    $T_{c_2}$ and $T_g$ is near the attracting end of $T_g$ and thus
    $T_{c_2}$ has bounded, incoherent intersection with $T_h$. Since
    $\ora{c_2}$ and $\ora{h}$ are parallel in the disk, this is a
    contradiction with the first
    part of this proof.
   \end{proof}

\begin{figure}
  \begin{subfigure}[c]{0.45\textwidth}
    \centering
    \fontsize{9pt}{9pt}\selectfont%% Creator: Inkscape 1.3 (0e150ed6c4, 2023-07-21), www.inkscape.org
%% PDF/EPS/PS + LaTeX output extension by Johan Engelen, 2010
%% Accompanies image file 'k_crossing_14.pdf' (pdf, eps, ps)
%%
%% To include the image in your LaTeX document, write
%%   \input{<filename>.pdf_tex}
%%  instead of
%%   \includegraphics{<filename>.pdf}
%% To scale the image, write
%%   \def\svgwidth{<desired width>}
%%   \input{<filename>.pdf_tex}
%%  instead of
%%   \includegraphics[width=<desired width>]{<filename>.pdf}
%%
%% Images with a different path to the parent latex file can
%% be accessed with the `import' package (which may need to be
%% installed) using
%%   \usepackage{import}
%% in the preamble, and then including the image with
%%   \import{<path to file>}{<filename>.pdf_tex}
%% Alternatively, one can specify
%%   \graphicspath{{<path to file>/}}
%% 
%% For more information, please see info/svg-inkscape on CTAN:
%%   http://tug.ctan.org/tex-archive/info/svg-inkscape
%%
\begingroup%
  \makeatletter%
  \providecommand\color[2][]{%
    \errmessage{(Inkscape) Color is used for the text in Inkscape, but the package 'color.sty' is not loaded}%
    \renewcommand\color[2][]{}%
  }%
  \providecommand\transparent[1]{%
    \errmessage{(Inkscape) Transparency is used (non-zero) for the text in Inkscape, but the package 'transparent.sty' is not loaded}%
    \renewcommand\transparent[1]{}%
  }%
  \providecommand\rotatebox[2]{#2}%
  \newcommand*\fsize{\dimexpr\f@size pt\relax}%
  \newcommand*\lineheight[1]{\fontsize{\fsize}{#1\fsize}\selectfont}%
  \ifx\svgwidth\undefined%
    \setlength{\unitlength}{144bp}%
    \ifx\svgscale\undefined%
      \relax%
    \else%
      \setlength{\unitlength}{\unitlength * \real{\svgscale}}%
    \fi%
  \else%
    \setlength{\unitlength}{\svgwidth}%
  \fi%
  \global\let\svgwidth\undefined%
  \global\let\svgscale\undefined%
  \makeatother%
  \begin{picture}(1,1)%
    \lineheight{1}%
    \setlength\tabcolsep{0pt}%
    \put(0,0){\includegraphics[width=\unitlength,page=1]{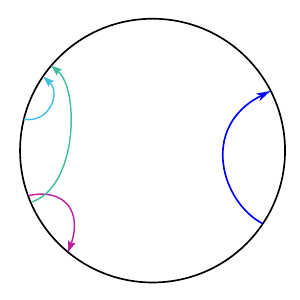}}%
    \put(0.93589672,0.70405374){\color[rgb]{0,0,1}\makebox(0,0)[t]{\lineheight{1.25}\smash{\begin{tabular}[t]{c}$h$\end{tabular}}}}%
    \put(0.30549705,0.28503035){\color[rgb]{0.75294118,0.12156863,0.62352941}\makebox(0,0)[t]{\lineheight{1.25}\smash{\begin{tabular}[t]{c}$c_0$\end{tabular}}}}%
    \put(0.32247058,0.69539393){\color[rgb]{0.2627451,0.74901961,0.62352941}\makebox(0,0)[t]{\lineheight{1.25}\smash{\begin{tabular}[t]{c}$c_1$\end{tabular}}}}%
    \put(0.17397996,0.55179265){\color[rgb]{0.2627451,0.74901961,0.87843137}\makebox(0,0)[t]{\lineheight{1.25}\smash{\begin{tabular}[t]{c}$c_2$\end{tabular}}}}%
    \put(0,0){\includegraphics[width=\unitlength,page=2]{k_crossing_14.pdf}}%
    \put(0.07669038,0.67827698){\color[rgb]{1,0,0}\makebox(0,0)[t]{\lineheight{1.25}\smash{\begin{tabular}[t]{c}$g$\end{tabular}}}}%
  \end{picture}%
\endgroup%

    \caption{Configuration of axes on $\wt{S}$. There are other
      possibilities for the position of
      $\ora{c_0}$ and $\ora{c_1}$ with respect to $\ora{h}$; only
      $\ora{c_2}$ must be as shown.}
    \label{fig:k_crossing_14}
  \end{subfigure}
  \hfill
  \begin{subfigure}[c]{0.45\textwidth}
    \centering
    \fontsize{9pt}{9pt}\selectfont%% Creator: Inkscape 1.3 (0e150ed6c4, 2023-07-21), www.inkscape.org
%% PDF/EPS/PS + LaTeX output extension by Johan Engelen, 2010
%% Accompanies image file 'tree_bounded_incoherent.pdf' (pdf, eps, ps)
%%
%% To include the image in your LaTeX document, write
%%   \input{<filename>.pdf_tex}
%%  instead of
%%   \includegraphics{<filename>.pdf}
%% To scale the image, write
%%   \def\svgwidth{<desired width>}
%%   \input{<filename>.pdf_tex}
%%  instead of
%%   \includegraphics[width=<desired width>]{<filename>.pdf}
%%
%% Images with a different path to the parent latex file can
%% be accessed with the `import' package (which may need to be
%% installed) using
%%   \usepackage{import}
%% in the preamble, and then including the image with
%%   \import{<path to file>}{<filename>.pdf_tex}
%% Alternatively, one can specify
%%   \graphicspath{{<path to file>/}}
%% 
%% For more information, please see info/svg-inkscape on CTAN:
%%   http://tug.ctan.org/tex-archive/info/svg-inkscape
%%
\begingroup%
  \makeatletter%
  \providecommand\color[2][]{%
    \errmessage{(Inkscape) Color is used for the text in Inkscape, but the package 'color.sty' is not loaded}%
    \renewcommand\color[2][]{}%
  }%
  \providecommand\transparent[1]{%
    \errmessage{(Inkscape) Transparency is used (non-zero) for the text in Inkscape, but the package 'transparent.sty' is not loaded}%
    \renewcommand\transparent[1]{}%
  }%
  \providecommand\rotatebox[2]{#2}%
  \newcommand*\fsize{\dimexpr\f@size pt\relax}%
  \newcommand*\lineheight[1]{\fontsize{\fsize}{#1\fsize}\selectfont}%
  \ifx\svgwidth\undefined%
    \setlength{\unitlength}{179.42400742bp}%
    \ifx\svgscale\undefined%
      \relax%
    \else%
      \setlength{\unitlength}{\unitlength * \real{\svgscale}}%
    \fi%
  \else%
    \setlength{\unitlength}{\svgwidth}%
  \fi%
  \global\let\svgwidth\undefined%
  \global\let\svgscale\undefined%
  \makeatother%
  \begin{picture}(1,0.52487958)%
    \lineheight{1}%
    \setlength\tabcolsep{0pt}%
    \put(0.28232595,0.39374918){\color[rgb]{1,0,0}\makebox(0,0)[t]{\lineheight{1.25}\smash{\begin{tabular}[t]{c}$T_g$\end{tabular}}}}%
    \put(0.73471149,0.26237643){\color[rgb]{0.2627451,0.74901961,0.62352941}\makebox(0,0)[t]{\lineheight{1.25}\smash{\begin{tabular}[t]{c}$T_{c_2}$\end{tabular}}}}%
    \put(0.16443441,0.24386482){\color[rgb]{0,0,1}\makebox(0,0)[t]{\lineheight{1.25}\smash{\begin{tabular}[t]{c}$T_h$\end{tabular}}}}%
    \put(0,0){\includegraphics[width=\unitlength,page=1]{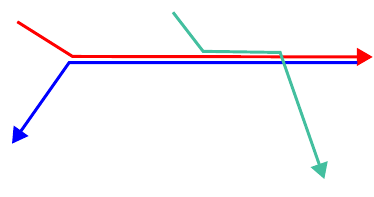}}%
  \end{picture}%
\endgroup%

    \caption{Configuration of axes on $T$}
    \label{fig:tree_k_crossing_14}
  \end{subfigure}
  \caption{Proof of Proposition~\ref{prop:parallel-incoherent} in the
    unbounded intersection case}
  \label{fig:tree_bounded_incoherent}
\end{figure}

\begin{proof}[Proof of Theorem~\ref{thm:small}]
  We consider several cases of the max-smoothing equality, depending
  on whether $g,h$ act hyperbolically or elliptically on the tree,
  and also on the relative position of $\ora{g},\ora{h}$ in
  the disk. We first deal with the case where $g,h$ are both
  hyperbolic, and then reduce the other cases to this one.
  The cases where $\ora{g},\ora{h}$ are anti-parallel immediately
  reduce to the case where they are parallel so we only deal with the
  latter.
\begin{case}
\label{case:hyperbolic}
In the action on~$T$, $g,h$ are both hyperbolic.
\end{case}
\begin{subcase}
$\ora{g}, \ora{h}$ cross.
\end{subcase}
Then $T_g \cap T_h \neq \emptyset$ by axis intersection preservation, and disconnected max-smoothing follows from~Lemma~\ref{lem:foundation_trees}\ref{item:tree-intersect}.
\begin{subcase}
$\ora{g},\ora{h}$ are parallel and $T_g \cap T_h$ is empty.
\end{subcase}
Then $\norm{gh}=\norm{gh^{-1}} > \norm{g}+\norm{h}$
by~Lemma~\ref{lem:foundation_trees}\ref{item:tree-disjoint}.
\begin{subcase}
$\ora{g},\ora{h}$ are parallel and $T_g,T_h$ overlap coherently or in
a single point.
\end{subcase}
    Then $\norm{gh}=\norm{g}+\norm{h}\ge\norm{gh^{-1}}$, by
    Lemma~\ref{lem:foundation_trees}\ref{item:tree-intersect}.
    \begin{subcase}
      $\ora{g},\ora{h}$ are parallel and $T_g,T_h$ overlap incoherently.
    \end{subcase}
    Impossible by Proposition~\ref{prop:parallel-incoherent}.

\begin{case}
  In the action on~$T$, $g,h$ are both elliptic.
\end{case}
\begin{subcase}
$\ora{g},\ora{h}$ cross.
\end{subcase}
Since $\ora{g},\ora{h}$ cross, $\ora{gh},\ora{hg}$ are parallel by Lemma~\ref{lem:righthandedfoundation}.
If $T_g \cap T_h = \emptyset$, by Lemma~\ref{lem:pau_two_ell}, $T_{gh}$ and
$T_{hg}$ overlap incoherently; this
contradicts Proposition~\ref{prop:parallel-incoherent}. Thus
$T_g \cap T_h \ne \emptyset$ and, by Lemma~\ref{lem:pau_two_ell} again,
$\norm{g} + \norm{h} = \norm{gh} = \norm{gH} = 0$.
\begin{subcase}
  $\ora{g}, \ora{h}$ are parallel. \end{subcase}
By Lemma~\ref{lem:pau_two_ell}, $\norm{ gh} = \norm{ gH}$ and
thus we have max\hyp smoothing.

\begin{case}
  In the action on~$T$, $g$ is elliptic and $h$ is hyperbolic.
  In particular, $ghG$ and $h$ act hyperbolically on $T$.
\end{case}
\begin{subcase}
$\ora{g},\ora{h}$ cross.
\end{subcase}
Then $\ora{ghG}$ and $\ora{h}$ are parallel.
By Proposition~\ref{prop:parallel-incoherent}, $T_{ghG}$ and $T_{h}$
cannot intersect incoherently. If they intersect coherently or in a
point then, by Lemma~\ref{lem:pau_ell_hyp}(1), $\norm{ gh } = \norm{h}=
\norm{gH}$ and
$T_g \cap T_h \ne \emptyset$, and we
are done. If $T_{ghG} \cap T_{h} = \emptyset$, by
Lemma~\ref{lem:pau_ell_hyp}(3), $T_{gh}$ and $T_{hg}$ overlap
incoherently. But, since $(\ora{gh},\ora{hg})$ are parallel, this
contradicts Proposition~\ref{prop:parallel-incoherent}.
\begin{subcase} $\ora{g},\ora{h}$ are parallel. \end{subcase}
If $T_{ghG}$ and $T_{h}$ intersect coherently or in a
point then, by Lemma~\ref{lem:pau_ell_hyp}(1), $\norm{ gh } = \norm{h} = \norm{gH}$.
If $T_{ghG}$ and $T_{h}$ do not intersect, then
$\norm{gh}=\norm{gH} > \norm{h}$. In either of these cases, we have
max-smoothing as desired.

If $T_{ghG}$ and $T_{h}$ overlap incoherently, by 
Lemma~\ref{lem:pau_ell_hyp}(2), $\norm{h} \geq \max \{ \norm{gh}, \norm{gH} \} > \min \{ \norm{gh}, \norm{gH} \}$ and we must
derive a contradiction.
If $\norm{h} \geq \norm{gh}>\norm{gH}$, we need to rule out $\norm{h} > \norm{gh} > \norm{gH}$; if $\norm{h} \geq \norm{gH}>\norm{gh}$, we need to rule this case out.
Both cases imply, $\norm{h}>\norm{gh}$.
From
Lemma~\ref{lem:foundation_trees}(1c),
\[
\norm{(ghG)H}=2\norm{h} > \norm{(ghG)h}.
\]
Now consider $gh$ and $GH$. Since one is conjugate to the inverse of
the other, $\norm{gh} = \norm{GH} < \norm{h}$, while $\ora{gh}$ and $\ora{GH}$ cross by Lemma~\ref{lem:righthandedfoundation}\ref{item:ab-parallel}.
We can apply Subcase 1.1 or Subcase 2.1
(depending whether they are hyperbolic or elliptic) to conclude
$2\|h\|>\norm{gh} + \norm{GH} \ge \norm{ghGH}=2\|h\|$, a contradiction.

\medskip
Finally, note that in all the subcases where
$(\ora{g},\ora{h})$ cross,
we concluded that $T_g \cap T_h \neq \emptyset$.
Hence, we have also shown that if an irreducible $S$-tree preserves
axis intersection, it also preserves characteristic intersection.
\end{proof}

\begin{remark}
One could hope to bypass Cases 2 and 3 by approximating action on $T$ by actions consisting solely of hyperbolic isometries (case 1), i.e., free actions, but this would require some extra work.
Indeed, it is not true for a general finitely generated group $\Gamma$
that free $(\Gamma,T)$ actions are dense in the space of small minimal
$(\Gamma,T)$ actions. For free groups, by a result of
Bestvina--Feighn~\cite{BF11:OuterLimits} it is known that the closure
of the space of free actions is equal to
$\operatorname{VS}(\Gamma,T)$, the space of ``very small'' actions, a
strictly proper subset of the subset of small actions. For surface
groups, one can see that free actions are dense in the space of small
actions \emph{a posteriori} once one knows all small actions are dual
to measured laminations: filling measured laminations (those
intersecting all simple closed curves) are dense in the space of
measured laminations, and their dual tree actions are free.
\end{remark}

We follow up with some consequences of this proof.

\begin{proposition}
 \label{prop:preserves_axis}
  For an $S$-tree~$T$,
  let $\mathcal{I}$ denote the set of $(g,h)$ so that
  $(\ora{g},\ora{h})$ are parallel and $T_g,T_h$ overlap incoherently.
 Then the $S$-tree preserves axis intersection if and only if either $\mathcal{I}$ is empty or all $(g,h) \in \mathcal{I}$ have unbounded overlap.
\end{proposition}
\begin{proof}
If $T$ preserves axis intersection, either
the action is irreducible and by Proposition~\ref{prop:parallel-incoherent}
$\mathcal{I}$ is
empty, or it is reducible and hence all intersections, including those
in~$\mathcal{I}$, have unbounded overlap.

Conversely, if $T$ does not preserve axis intersection, by assumption
there exists a pair of hyperbolic elements
$(g,h)$ so that $(\ora{g},\ora{h})$ cross and $T_g \cap T_h =
\emptyset$. Then by Lemma~\ref{lem:foundation_trees}(2)
$T_{gh},T_{hg}$ overlap incoherently with overlap of length
$d(T_g,T_h)$. Since $(gh,hg)$ are parallel and hyperbolic, we
found a pair of parallel hyperbolic elements with bounded incoherent
overlap.
\end{proof}

For the full statement of Theorem~\ref{mainthm:skora},
we prove some related facts about reducible smoothings.

\begin{lemma}
\label{lem:red_not_smoothing}
Let $T$ be an $S$-tree.
If the $\pi_1(S)$-action is reducible,
then there exists a pair of incoherently overlapping elements $(g,h)$ so that
$(\ora{g},\ora{h})$ are parallel and $\norm{\cdot} \colon \Curves^+(S) \to \mathbb{R}$ does not
satisfy connected smoothing.
If the
action is fixed-end, $\norm{\cdot}$
satisfies disconnected smoothing but not connected smoothing.
\end{lemma}

We start with subsidiary facts.

\begin{claim}\label{claim:reducible-overlap}
  Suppose $g,h \in \Gamma$ are hyperbolic elements of a reducible action. Then
 either $T_g,T_h$ overlap coherently, in which case
\begin{equation}
\norm{gh }=\norm{ g } + \norm{ h } > \norm{gH } = |\,\norm{g }- \norm{ h }\, |,
\label{eq:reducible1}
\end{equation}
or they overlap incoherently, in which case
\begin{equation}
\norm{gH } =\norm{ g} + \norm{ h } > \norm{ gh } = |\, \norm{ g } - \norm{ h }\,|.
\label{eq:reducible2}
\end{equation}
\end{claim}
\begin{proof}
See~\cite[Lemma~6.4]{CM87:GroupActions}.
\end{proof}

\begin{claim}\label{claim:incoherent}
  Suppose $a,x \in \Gamma$ are hyperbolic elements of a
  reducible action so that $\norm{ a }>\norm{ x }$ and $(T_a,T_x)$ overlap
  coherently. Then the pair $(T_{ax},T_{xA})$ has incoherent overlap.
\end{claim}
\begin{proof}
  Using Claim~\ref{claim:reducible-overlap}, we see that
  $\norm{(xA)(ax)} = 2\norm{x} < 2\norm{a} = \norm{xA} + \norm{ax}$,
  implying $T_{ax}$ and $T_{xA}$ overlap incoherently.
\end{proof}

\begin{proof}[Proof of Lemma~\ref{lem:red_not_smoothing}]
  We begin by finding a pair of elements so that $(\ora{g},\ora{h})$
  are parallel and $T_g,T_h$ overlap incoherently.

  First, find any pair $(g,h)$ of hyperbolic elements so $(\ora{g},\ora{h})$ are
parallel. Start with any two hyperbolic elements that
are not common powers. (For instance, take any hyperbolic element $g$
and a conjugating element $h$ outside of the normalizer of $\langle g \rangle$.)
If the pair is parallel, we are
done. If it is crossing, consider $(\ora{gh},\ora{hg})$, which is parallel by
Lemma~\ref{lem:righthandedfoundation}. If it is anti-parallel,
consider the pair $(\ora{g},\ora{H})$ instead.

Now we consider their overlaps in the tree. If $(T_g,T_h)$
overlap
incoherently, we have established this claim. Otherwise, following
Example~\ref{ex:ab_parallel_aB_crossing} consider $a \coloneqq hghgh$
and $b \coloneqq hgh$. Then $T_a,T_b$ are still parallel with coherent
overlap, so that
$\norm{a}>\norm{b}$, while
$(\ora{aB},\ora{Ba})$ is crossing. By Lemma~\ref{lem:a_parallel_Ba},
it follows that either $(\ora{ab},\ora{bA})$ is parallel or $(\ora{ba},\ora{aB})$ is parallel.
By Claim~\ref{claim:incoherent}, we have found a pair $(\ora{k},\ora{\ell})$, with
$k = ab, \ell = Ba$ or $k = ba, \ell = bA$, of parallel hyperbolic elements with
incoherent overlap. This proves the first claim.

To see that connected smoothing fails, note
this tells us $\norm{ k\ell } < \norm{k\ell^{-1}}=\norm{k} + \norm{\ell}$, which violates both
oriented and unoriented connected smoothing.

We now turn to proving disconnected smoothing in the fixed-end case. (In
fact most cases of disconnected smoothing hold in the dihedral case
as well.) Fix $g,h\in\Gamma$; we will show that in any case, $\norm{g}
+ \norm{h} = \max\{\norm{gh},\norm{gH}\}$, which in the case
$\ora{g},\ora{h}$ cross implies disconnected smoothing.

If $g,h$ are both hyperbolic,
Claim~\ref{claim:reducible-overlap} gives the desired disconnected
smoothing.

Next, suppose $g$ is elliptic and $h$ is hyperbolic.
Then, by~\cite[Lemma~6.5]{CM87:GroupActions}, either:
\begin{itemize}
\item $\norm{gh}=\norm{gH}=\norm{h}$; or
\item $\norm{gh}=\norm{gH}=0$.
\end{itemize}
In the first case, we have disconnected smoothing. In the second case, we have $\norm{h} > \norm{gh}=\norm{gH}$, so disconnected smoothing is satisfied, as desired.

Finally, if $g$ and $h$ are elliptic. Here we use the assumption the
action has a fixed end,
so $T_g \cap
T_h \ne \emptyset$ (since the fixed sets intersect on a ray to the
fixed end) and hence, by Lemma~\ref{lem:pau_two_ell},
$\norm{gh}=0$.
Thus, both smoothings are satisfied in this case.
\end{proof}

Hence a reducible action cannot be dual to a geodesic current, since
smoothing (both connected and disconnected) is a necessary condition.

\begin{remark}
  In Lemma~\ref{lem:red_not_smoothing}, if the action is dihedral,
  there are elliptic elements $g$ and $h$ in $\Gamma$ with distinct
  fixed points on the invariant line (see, e.g., proof
  of~\cite[Theorem~2.2]{CM87:GroupActions}). Hence,
  $T_g \cap T_h = \emptyset$, and by Lemma~\ref{lem:pau_two_ell},
  $\norm{gh}>0$. Thus, if $(\ora{g},\ora{h})$ cross, disconnected
  smoothing is not satisfied.
\end{remark}

Putting this all together we can give another characterization of small surface actions on real trees.

\begin{theorem}
An $S$-tree is irreducible and preserves axis intersection if and only
if no pair of hyperbolic elements $(g,h)$ so that $(\ora{g},\ora{h})$ are
parallel have incoherent
overlap.
\label{thm:axis_incoherent_characterization}
\end{theorem}
\begin{proof}
One implication is Proposition~\ref{prop:parallel-incoherent}.
In the other direction, it follows by
Proposition~\ref{prop:preserves_axis} that the $S$-tree
preserves axis intersection. Finally, it cannot be reducible since
otherwise it would have hyperbolic pairs with (unbounded) incoherent
overlap, by Lemma~\ref{lem:red_not_smoothing}.
\end{proof}

Now we prove Theorem~\ref{mainthm:skora}, whose statement we recall
for convenience.

\begin{taggedthm}{\ref{mainthm:skora}}
Let $T$ be an $S$-tree. The following are equivalent:
\begin{enumerate}[(a)]
\item There exists a measured lamination $\lambda$ on $S$ such that
$T$ is $\pi_1(S)$-equivariantly isometric to $T_\lambda$;
\item The $S$-tree is irreducible and preserves axis intersection;
\item The $S$-tree is irreducible and preserves characteristic intersection;
\item The translation length function
$\norm{\cdot}\colon \pi_1(S)\to\mathbb{R}$ satisfies max-smoothing;
\item   There are no hyperbolic elements $(a,b)$ whose axes $(\ora{a},\ora{b})$ are
parallel in $\wt{S}$ and overlap incoherently in $T$.
\end{enumerate}
\end{taggedthm}

\begin{proof}[Proof of Theorem~\ref{mainthm:skora}]
The equivalence $\text{(a)} \Leftrightarrow \text{(d)}$ is
Theorem~\ref{thm:char_laminations_plus}.
Morgan--Otal~\cite{MO93:RelativeGrowths} show $\text{(a)} \Rightarrow \text{(b)}$.
Theorem~\ref{thm:small} shows $\text{(b)} \Rightarrow \text{(c)}$, (d).
Theorem~\ref{thm:axis_incoherent_characterization} shows $\text{(b)}
\Leftrightarrow \text{(e)}$. The implication $\text{(c)} \Rightarrow \text{(b)}$ is
immediate.
\end{proof}

%%% Local Variables:
%%% mode: latex
%%% TeX-master: "Intersections"
%%% End:

\section{Cross-ratios}
\label{sec:crossratios}

In this section we will see that every generalized cross-ratio induces
a geodesic current. More precisely, we will prove that generalized
cross-ratios induce functions on $\pi_1(S)$ (periods) satisfying the
hypotheses of Theorem~\ref{thm:intersections_group}, and therefore
induce a geodesic current $\mu$ dual to the period. This unifies the
work of Martone--Zhang~\cite{MZ19:PositivelyRatioed} and
Burger--Iozzi--Parreau--Pozzetti~\cite{BIPP24:PositiveCrossratios}
under the framework of Theorem~\ref{thm:intersect}. Note that
in~\cite{BIPP24:PositiveCrossratios}, geodesic currents are also
constructed from generalized cross-ratios for finite-type surfaces; we
recover their result when restricted to closed surfaces, with the
added benefit that the construction of the extension only involves the
periods, which can be directly accessed from the eigenvalues of the
representation~\cite[Theorem~1.1]{MZ19:PositivelyRatioed}.

\begin{definition}
A \emph{generalized cross-ratio} on $S$ is a function of type
\[
[\cdot, \cdot, \cdot, \cdot] \colon (\partial_{\infty} S)^{(4)} \to \mathbb{R}
\]
where $(\partial_{\infty} S)^{(4)}$ denotes the space of ccw oriented distinct 4-tuples in $(\partial_{\infty} S)^{4}$, satisfying the following four properties:
\begin{enumerate}
\item \emph{flip-invariance:} $[a,b,c,d]=[c,d,a, b].$
\item \emph{additivity:} For any ccw oriented and distinct $a,b,c,d,e$, we have
$[a,b, d, e]=[a,b,c,e] + [a,c,d, e]$.
\item \emph{$\pi_1(S)$-invariance:} For any tuple $(a,b,c,d) \in (\partial_{\infty} S)^{4}$ and every $g \in \pi_1(S)$, $[g \cdot a, g \cdot b, g \cdot c, g \cdot d] = [a,b,c,d]$.
\item \emph{positivity:}  For any tuple $(a,b,c,d) \in (\partial_{\infty} S)^{4}$, we have
\[
[a,b,c,d] \geq 0.
\]
\end{enumerate}
\label{def:general_crossratio}
\end{definition}

There are many definitions of cross ratios in the literature,
including by Ledrappier~\cite{L95:BordDesVar},
Hamenst\"adt~\cite{Ham99:Cocycles}, Otal~\cite{Ot92:Symplectique},
Martone--Zhang~\cite{MZ19:PositivelyRatioed}, and
Burger--Iozzi--Parreau--Pozzetti \cite{BIPP24:PositiveCrossratios}.
In all of these references (except the last one) the authors consider the cross-ratios to be continuous or even H\"older continuous, but we will not assume this here. Furthermore, in all of the above definitions, flip-invariance is assumed.
Labourie~\cite[Definition~3.1]{L07:Crossratio} has considered a more general definition of cross-ratios where flip-invariance is not assumed.  We will always assume our cross-ratios are flip-invariant.

As alluded above, generalized cross-ratios as in our definition above have been recently considered in Burger--Iozzi--Parreau--Pozzetti~\cite{BIPP24:PositiveCrossratios}.
In this setting, there are (at least) two natural generalized cross-ratios associated to a geodesic current (see Question~\ref{que:howmany_crossratios}).
\begin{equation}
[a,b,c,d]^+ \coloneqq \mu([d,a) \times [b,c))
\quad\text{and}\quad
[a,b,c,d]^- \coloneqq \mu((d,a] \times (b,c]).
\label{eq:current_to_crossratio}
\end{equation}
It is easily checked that these satisfy the properties of generalized cross-ratios,
and these cross-ratios are different in general if $\mu$ has atoms (see~\cite[Example~3.4]{BIPP24:PositiveCrossratios} when $\mu$ is a closed curve).
If $\mu$ has no atoms, the two generalized cross-ratios $[\cdot,\cdot,\cdot,\cdot]^{\pm}$ are equal, since the measure of the boundary of any box of geodesics is zero by Proposition~  \ref{prop:spikes}

Given a generalized cross-ratio $[ \cdot, \cdot, \cdot, \cdot ]$, let
the \emph{period} function of a curve $C=[g]$ be, for $x$ an arbitrary
point in $(g^-,g^+)$,
\[
  \ell_{[]}(g) = \ell(g) \coloneqq [g^-, x , g \cdot x, g^+].
\]
The $\pi_1(S)$-invariance and additivity of the cross-ratio show this is independent of~$x$.

Given a geodesic current $\mu$, the two (possibly distinct)
generalized cross-ratios $[ \cdot, \cdot, \cdot, \cdot ]^{\pm}$ induce
the same period. Thus, in general, generalized cross-ratios are not
uniquely determined by their periods. However, if the cross-ratios are
assumed to be H\"older continuous, cross-ratios are indeed determined
by their periods \cite[Theorem 1.f]{L95:BordDesVar} (c.f.\
\cite{Ot92:Symplectique}).

Going the other direction, it is immediate that a cross-ratio gives a
finitely additive function $\bigl[[d,a) \times [b,c)\bigr]\coloneqq [a,b,c,d]$, on the semiring of boxes $[d,a) \times [b,c)$, where $a,b,c,d$ appear in ccw
order on $\partial_\infty S$; however, this function need not be countably additive and will
not in general agree with the geodesic current~$\mu$ we construct. For
instance, if we start with a geodesic current $\mu$ with atoms,
construct the cross-ratio $[a,b,c,d]^-$ as above, and look at the
associated function on boxes $\bigl[ \bigr]^-$, we get a function on boxes that
does not agree with $\mu$ and so cannot be countably additive. Indeed,
 $\bigl[[g^+,g^-)\times [x,g \cdot x)\bigr]=\ell(g)$, which is equal to $i(\mu,g)$ and geodesic
currents are determined by their length functions~\cite{Otal90:SpectreMarqueNegative}. Nevertheless, for
convenience in the proofs, we will use the function on boxes and its finite additivity.

\begin{question}
  For a given geodesic current $\mu$, what are all the cross-ratios
  $[a,b,c,d]$ so that the period function with respect to the
  cross-ratio agrees with the length function from~$\mu$?
  \label{que:howmany_crossratios}
\end{question}

\begin{lemma}
\label{lem:crossratio_stability}
For every $g \in \pi_1(S)$, and every $n \in \mathbb{N}$ we have
\[
\ell(g^n) = n\ell(g).
\]
\end{lemma}
\begin{proof}

This follows readily by additivity and $\pi_1(S)$-invariance of the
cross-ratio, as follows. (See also Equation~\eqref{eq:intersection-stability}.)
\begin{equation}\label{eq:crossratio-stability}
 \ell(g^n) =[g^-, x , g^n \cdot x, g^+]= \sum_{i=1}^{n} [g^-, g^{i-1}x , g^i \cdot x, g^+] = n[g^-, x , g \cdot x, g^+] = n \ell(g).\qedhere
\end{equation}
\end{proof}

The fact that the period is a symmetric curve functional will use flip-invariance. 

\begin{lemma}    \label{lem:crinv}
    For every $g \in \pi_1(S)$, we have
    \[
    \ell(g)=\ell(g^{-1}).
    \]
\end{lemma}
\begin{proof}
The geodesic $\ora{g}$ with endpoints $g^-, g^+$ divides $\partial_{\infty}S$ into two open intervals:
one on the left in the forward direction of $\gamma$ (i.e.,
$(g^+,g^-)$), the other on the right.

Let $y \in  \partial_{\infty}S$ be on the left interval, and $x \in  \partial_{\infty}S$ on the right interval.
We then have
\begin{align*}
\ell(g^{-1}) &= \lim_{n \to \infty} \frac{[g^{-n}y, g^-, g^+, g^{n} y]}{2n} \\
  &= \lim_{n \to \infty}
     \frac{[g^{-n}y, g^-, g^{-n}x, g^{n} y]}{2n} +
     \frac{[g^{-n}y, g^{-n}x, g^{n}x, g^{n} y]}{2n} +
     \frac{[g^{-n}y, g^{n}x, g^+, g^{n} y]}{2n}\\
  &= \ell(g).
\end{align*}
For the first equality, we used Lemma~\ref{lem:crossratio_stability}.
To see the last equality, note that by $\pi_1(S)$-invariance, the
numerator of the first term is $[y,g^-,x,g^{2n}y]$ which is bounded by
$[y,g^-,x,g^+]$, and thus the entire first term has limit zero
as $n \to \infty$. Similarly for the last term, and the middle term is
symmetric under switching $g$ and $g^{-1}$.
\end{proof}

We will need another lemma on equality of cross-ratios.
\begin{lemma}\label{lem:cr-switcheroo}
  Fix $g \in \pi_1(S)$ and $x,y \in \bdy_\infty S$.
  
  \begin{enumerate} 
      \item If $g^-,y,g\cdot y, G \cdot x, x,g^+$
  appear in ccw order as in Fig.~\ref{fig:cr-geo1}, then
  \[
    \bigl[[x,g^+)\times[y,g \cdot y)\bigr]
      = \bigl[[G \cdot x,x)\times[g^-,y)\bigr].
  \]
  \item If $g^-,y,g\cdot y, g^+, x, G \cdot x$
  appear in ccw order as in Fig.~\ref{fig:cr-geo2}, then
  \[
    \bigl[[g^+,x)\times[y,g \cdot y)\bigr]
      = \bigl[[x,G \cdot x)\times[g^-,y)\bigr].
  \]
  \end{enumerate}
\end{lemma}

\begin{figure}
\centering

\begin{subfigure}{0.48\textwidth}
\centering
\resizebox{\linewidth}{!}{\fontsize{9pt}{9pt}\selectfont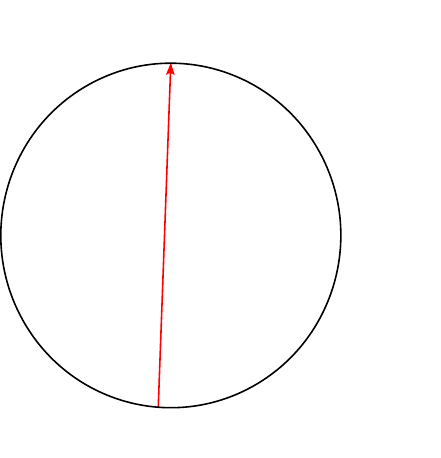}
\caption{Lemma~\ref{lem:cr-switcheroo}(1)}
\label{fig:cr-geo1}
\end{subfigure}
\hfill
\begin{subfigure}{0.48\textwidth}
\centering
\resizebox{\linewidth}{!}{\fontsize{9pt}{9pt}\selectfont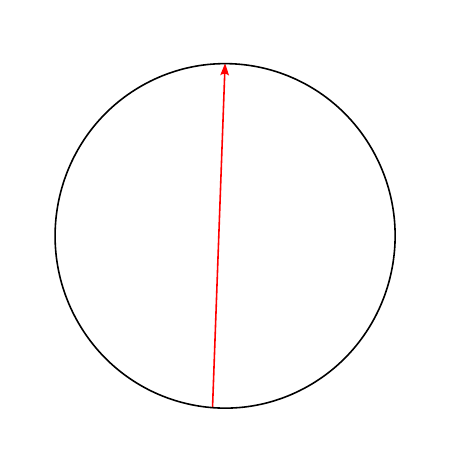}
\caption{Lemma~\ref{lem:cr-switcheroo}(2)}
\label{fig:cr-geo2}
\end{subfigure}

\caption{The statement and proof of Lemma~\ref{lem:cr-switcheroo}.}
\label{fig:cr-geo}
\end{figure}

Other variations of this lemma, with different position of $x,y$, hold
too. The geometric significance of
the common value is unclear to us.

\begin{proof}
  We will prove the first part.
  In what follows, compare Figure~\ref{fig:cr-geo1}.
  Set $L = [x,g^+)\times[y,g \cdot y)$ and
  $R = [G \cdot x,x)\times[g^-,y)$. To show $[L] = [R]$, we will
  translate $R$ by~$g$ and cut off a piece that is a subset of $L$.
  The remaining piece after this excision is of the same general form
  as $R$, so we can repeat.

  To that end, define, for $i \ge 0$,
  \begin{align*}
    R_i &\coloneqq [g^{i-1}\cdot x, g^i \cdot x) \times [g^-, y)\\
    L_i &\coloneqq [g^i \cdot x,g^+) \times [y, g \cdot y)\\
    \intertext{and for $i \ge 1$,}
    V_i &\coloneqq [g^{i-1} \cdot x, g^i \cdot x) \times [y, g \cdot y).
  \end{align*}
  Then $L_0 = L$, $R_0 = R$, and
  $g \cdot R_i = [g^i \cdot x, g^{i+1} \cdot x) \times [g^-, g \cdot
  y) = R_{i+1} \sqcup V_{i+1}$.
  Furthermore, the $V_i$ are all contained in~$L$. More precisely, for
  $n\ge 0$, we have $L = L_n \sqcup \bigsqcup_{i=1}^n V_i$.
  We thus have
  \[
    [R] = [R_n] + \sum_{i=1}^n [V_i] = [R_n] - [L_n] + [L].
  \]

  To complete the proof, we must show
  $\lim_{n^+} [R_n] = \lim_{n^+}[L_n] = 0$. To that end, note that the
  $R_i$ are disjoint, with bounded union:
  \[
    \bigsqcup_{i=0}^n R_i \subset [g^-,y) \times [G \cdot x, g^+).
  \]
  It follows that $\sum_{i=0}^n [R_i]$ is uniformly bounded and thus
  that $\lim_{i^+} [R_i] = 0$.

  The $L_i$ are not disjoint, but become disjoint after translation by
  $G^i$:
  \[
    L_i' \coloneqq G^i \cdot L_i = [x, g^+) \times [G^i \cdot y, G^{i-1} \cdot y).
  \]
  The $L_i'$ are disjoint, with union contained in the same set
  $[g^-,y) \times [G \cdot x, g^+)$, so as before $\lim_{i^+} [L_i'] = 0$.
  The proof of the second part is similar, although the setup is slightly different (see~Figure~\ref{fig:cr-geo2}).
\end{proof}

Now we will prove that $\ell$ satisfies the smoothing hypotheses in Theorem~\ref{thm:intersections_group}.

We divide the proof into two cases, depending on the type of smoothing (connected vs disconnected).

\subsection{Disconnected smoothing}

We will verify Equation~\eqref{eq:disconn_oriented} in Theorem~\ref{thm:intersections_group}.

\begin{lemma}\label{lem:crossratio_disconnected_smoothing}
Let $a, b \in \pi_1(S)$ so that $(\ora{b},\ora{a})$ are R-crossing.
Then
\[
\ell(a) + \ell(b) \geq \ell(ab).
\]
\end{lemma}

\begin{proof}
See
Figure~\ref{fig:rcross_cr}. We start by giving fundamental domains $F(ab)$, $F(a)$, and $F(b)$ for $\langle ab \rangle$,
$\langle a \rangle$, and~$\langle b \rangle$, respectively, with the property that all
boxes share the same right-hand factor $[Ab^+, bA^+)$.

Fundamental domains $F(a)$, $F(b)$ with this right-hand factor are immediate:
\begin{align*}
  F(a) &= [a^+, a^-) \times [Ab^+, bA^+)\\
  F(b) &= [b^+, b^-) \times [Ab^+, bA^+).
\end{align*}

Consider next the fundamental domain $F_0(ab) = [ab^+,ab^-) \times [Ab^+, a
\cdot bA^+)$ for $\langle ab \rangle$, and write it as the union
$F_0(ab) = F_1 \cup F_2$ where
\begin{align*}
  F_1 &= [ab^+,ab^-) \times [Ab^+, bA^+) \\
  F_2 &= [ab^+,ab^-) \times [bA^+, a \cdot bA^+).
  \\\intertext{Translate $F_2$ by $A$ to get}
  F_2 \equiv F_2' &= [ba^+,ba^-) \times [Ab^+, bA^+)
  \\\intertext{so that}
  F_0(ab) \equiv F_1 + F_2' &= \bigl([ab^+,ab^-) + [ba^+,ba^-)\bigr)
                            \times [Ab^+, bA^+)
\end{align*}
where $F \equiv F'$ means $F, F'$ are equivalent up to translation, and in the last step we switch to multi-set notation.

Then we have
\begin{align*}
  F(a) + F(b) - F_0(ab)
    &\equiv \bigl([a^+,a^-) + [b^+, b^-) - [ab^+,ab^-) - [ba^+,ba^-)\bigr)
      \times [Ab^+, bA^+)\\
    &= \bigl([ba^-,a^-) - [b^-,ab^-) - [ab^+,a^+) + [b^+,ba^+)\bigr)
      \times [Ab^+, bA^+)\\
    &= R_1 - R_2 - R_3 + R_4
\end{align*}
where the $R_i$ are the terms appearing in the previous equation in
order, and we have switched to virtual multi-sets. 
\begin{figure}[ht]
     \centering
\resizebox{80mm}{!}{\fontsize{9pt}{9pt}\selectfont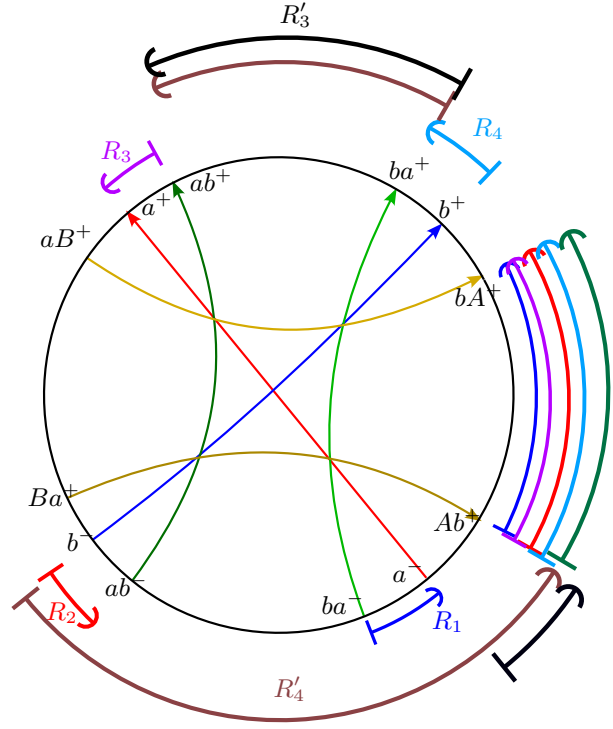}
    \caption{Checking $\ell(a) + \ell(b) \geq \ell(ab)$, part 1.
      The difference $F(a) + F(b) - F_0(ab)$  decomposes into $R_1$ (blue) + $R_2$
      (red) + 
      $R_3$ (purple) + $R_4$ (cyan).}\label{fig:rcross_cr}
\end{figure}

We next show that $[R_3] \le [R_4]$. By Lemma~\ref{lem:cr-switcheroo}(1)
with $g=a$, $y=Ab^+$, and $x=ab^+$, we have
\[
  [R_3] = \bigl[[ba^+,ab^+) \times [a^-,Ab^+)\bigr] = [R_3'].
\]
By the Lemma~\ref{lem:cr-switcheroo}(2), with $g=b$, $y=Ab^+$, and $x=ba^+$, we have
\[
  [R_4] = \bigl[[ba^+,ab^+) \times [b^-,Ab^+)\bigr] = [R_4'].
\]
Since $R_3' \subset R_4'$, we have the desired inequality.
The symmetric argument shows that $[R_2] \le [R_1]$.
\end{proof}

\begin{remark}
  If one traces through the last inequalities, the argument shows that
  \[
    \ell(a) + \ell(b) - \ell(ab) = [R_1] - [R_2] - [R_3] + [R_4]
    = \bigl[[ba^+, ab^+) \times [b^-,a^-)\bigr]
      + \bigl[[ab^-,ba^-) \times [b^+,a^+)\bigr].
  \]
  This is reminiscent of Lemma~\ref{lem:otalclaim}.
\end{remark}

\subsection{Connected smoothing}

Here we verify that Equation~\eqref{eq:conn_oriented} in Theorem~\ref{thm:intersections_group}, oriented connected smoothing, is satisfied.

\begin{lemma}\label{lem:crossratio_connected_smoothing}
Let $a,b \in \pi_1(S)$ with $(\ora{a},\ora{b})$ R-parallel.
Then, we have
\[
\ell(ab) \geq \ell(a) + \ell(b).
\]
\end{lemma}

\begin{proof}

    \begin{figure}[ht!]
     \centering
\resizebox{80mm}{!}{\fontsize{9pt}{9pt}\selectfont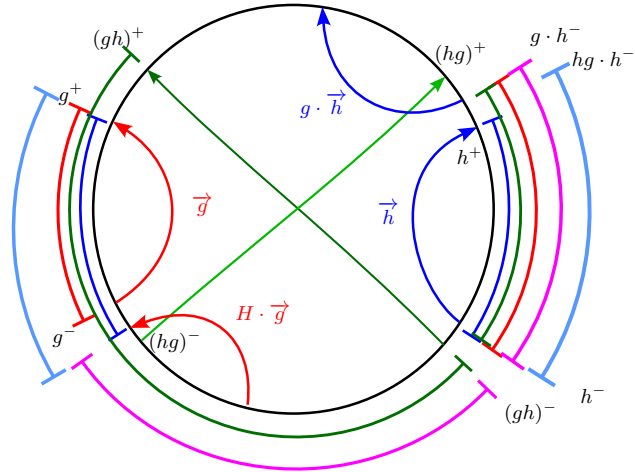}
    \caption{Checking $\ell(ab) \geq \ell(a) + \ell(b)$}\label{fig:c8}
\end{figure} 
Suppose $(\ora{a},\ora{b})$ are R-parallel. By Lemma~\ref{lem:crossratio_disconnected_smoothing}, $\ell$ satisfies disconnected smoothing. Hence, by Lemma~\ref{lem:fekete_upper},
\[
\ell(ab) \geq \lim_n \frac{\ell(a^nb^n)}{n}.
\]
We will now prove that, for sufficiently large $n$, we have
\[
\ell(a^nb^n) \geq \ell(a^n) + \ell(b^n).
\]
Indeed, take $n_0$ large enough so that for all $n \geq n_0$, neither $(\ora{a}, B^n \cdot \ora{a})$ nor $(\ora{b}, a^n \cdot \ora{b})$ cross. 
This can be achieved by NS dynamics of $a$ and $B$. To clarify
notation, we will let $g \coloneqq a^n$ and $h \coloneqq b^n$. Compare
Figure~\ref{fig:c8}, where we consider the fundamental domain $F(gh) =
(gh^+, gh^-) \times [h^-, g\cdot h^-)$ (dark green).
This contains the fundamental domain of $F(H)$ given by $(h^-, h^+) \times [g^+, H \cdot g^+)$ (dark blue). $F(gh)$ also contains the box $R=(H \cdot g^+, gh^-) \times [h^-, g \cdot h^-)$. Note that
\[
h \cdot R = (g^+, hg^-) \times [h^-, hg \cdot h^-)
\]
(cyan), which contains $F(g)$ given by $(g^+, g^-) \times [h^-, g \cdot h^-)$.
Hence, by invariance and finite additivity, we get
\[
[F(g)] + [F(h)] \leq [F(gh)],
\]
which proves the claim.
Therefore, 
\[
\lim_n \frac{\ell(a^nb^n)}{n} \geq \ell(a) + \ell(b),
\]
and thus
\[
\ell(ab) \geq \ell(a) + \ell(b),
\]
as desired.
\end{proof}

\begin{theorem}
There is a one-to-one correspondence between \emph{periods} of generalized cross-ratios and geodesic currents.
\end{theorem}
\begin{proof}
On the one hand, by the Lemmas~\ref{lem:crossratio_stability},~\ref{lem:crinv},~\ref{lem:crossratio_disconnected_smoothing} and~\ref{lem:crossratio_connected_smoothing} above, every period $\ell$ induced by a generalized cross-ratio satisfies the hypotheses of Theorem~\ref{thm:intersections_group} and, since $\ell$ is symmetric, there exists a unique dual geodesic current $\mu_{\ell}$ such that $\ell(C)=i(\mu_{\ell},C)$ for every closed curve. On the other hand, every geodesic current induces a generalized cross-ratio via Equation~\ref{eq:current_to_crossratio}, and hence a period function.
\end{proof}

%%% Local Variables:
%%% mode: latex
%%% TeX-master: "Intersections"
%%% End:

\section{Further questions}
\label{sec:questions}

We finish by discussing some natural follow-up questions and conjectures.

\subsection{Open surfaces}
A natural question is whether our results extend to \emph{open surfaces}, i.e.,\ compact surfaces with boundary or finite-type surfaces with punctures.

For such surfaces, geodesic currents admit at least two natural definitions, depending on how one treats the ends. 
If the ends are realized as cusps, one obtains a space $\Curr_{\mathrm{cusp}}(S)$, defined analogously to Definition~\ref{def:currents} as the space of $\pi_1(S)$-invariant 
Radon measures on the space of geodesics in the universal cover (see, e.g.,~\cite[Definition~2.2]{Sas22:Currents}). In this setting, the space of geodesics includes geodesic arcs connecting cusps. 
If instead one restricts to geodesics projecting to the convex core,
one obtains a space denoted $\Curr_{\mathrm{open}}(S)$ (\cite[Section~2.6]{DLR10:DegenerationFlatMetrics}). See
Figure~\ref{fig:open_surfaces}.

\begin{figure}
\centering
\begin{subfigure}{0.48\textwidth}
\centering
\resizebox{60mm}{!}{
\fontsize{14pt}{14pt}\selectfont%% Creator: Inkscape 1.3 (0e150ed6c4, 2023-07-21), www.inkscape.org
%% PDF/EPS/PS + LaTeX output extension by Johan Engelen, 2010
%% Accompanies image file 'cusp.pdf' (pdf, eps, ps)
%%
%% To include the image in your LaTeX document, write
%%   \input{<filename>.pdf_tex}
%%  instead of
%%   \includegraphics{<filename>.pdf}
%% To scale the image, write
%%   \def\svgwidth{<desired width>}
%%   \input{<filename>.pdf_tex}
%%  instead of
%%   \includegraphics[width=<desired width>]{<filename>.pdf}
%%
%% Images with a different path to the parent latex file can
%% be accessed with the `import' package (which may need to be
%% installed) using
%%   \usepackage{import}
%% in the preamble, and then including the image with
%%   \import{<path to file>}{<filename>.pdf_tex}
%% Alternatively, one can specify
%%   \graphicspath{{<path to file>/}}
%% 
%% For more information, please see info/svg-inkscape on CTAN:
%%   http://tug.ctan.org/tex-archive/info/svg-inkscape
%%
\begingroup%
  \makeatletter%
  \providecommand\color[2][]{%
    \errmessage{(Inkscape) Color is used for the text in Inkscape, but the package 'color.sty' is not loaded}%
    \renewcommand\color[2][]{}%
  }%
  \providecommand\transparent[1]{%
    \errmessage{(Inkscape) Transparency is used (non-zero) for the text in Inkscape, but the package 'transparent.sty' is not loaded}%
    \renewcommand\transparent[1]{}%
  }%
  \providecommand\rotatebox[2]{#2}%
  \newcommand*\fsize{\dimexpr\f@size pt\relax}%
  \newcommand*\lineheight[1]{\fontsize{\fsize}{#1\fsize}\selectfont}%
  \ifx\svgwidth\undefined%
    \setlength{\unitlength}{306.13422784bp}%
    \ifx\svgscale\undefined%
      \relax%
    \else%
      \setlength{\unitlength}{\unitlength * \real{\svgscale}}%
    \fi%
  \else%
    \setlength{\unitlength}{\svgwidth}%
  \fi%
  \global\let\svgwidth\undefined%
  \global\let\svgscale\undefined%
  \makeatother%
  \begin{picture}(1,1.16710076)%
    \lineheight{1}%
    \setlength\tabcolsep{0pt}%
    \put(0,0){\includegraphics[width=\unitlength,page=1]{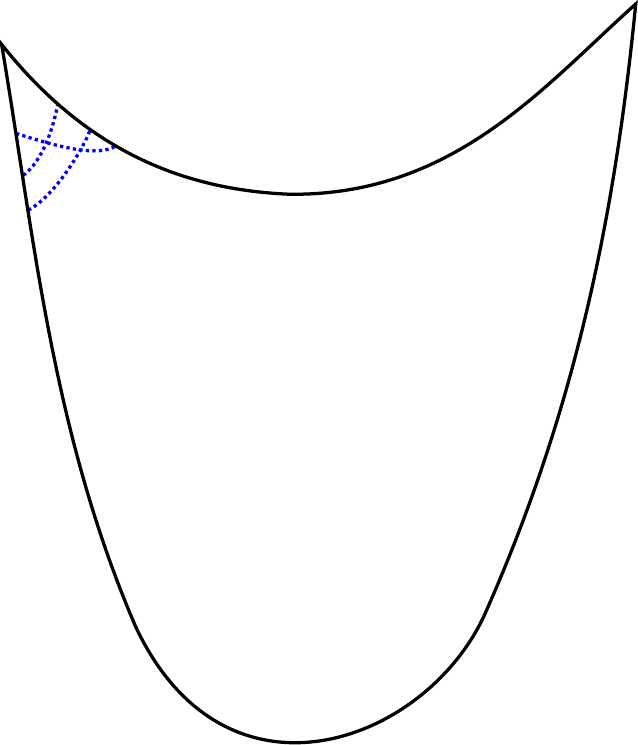}}%
    \put(0.14344737,1.08373156){\color[rgb]{0.6745098,0,0.40784314}\makebox(0,0)[t]{\lineheight{1.25}\smash{\begin{tabular}[t]{c}$A$\end{tabular}}}}%
    \put(0.74139703,0.76309383){\color[rgb]{0,0,1}\makebox(0,0)[t]{\lineheight{1.25}\smash{\begin{tabular}[t]{c}$C_n$\end{tabular}}}}%
    \put(0,0){\includegraphics[width=\unitlength,page=2]{cusp.pdf}}%
  \end{picture}%
\endgroup%
}
\caption{Cusped surface with an arc $A$ and a closed geodesic $C_n$
  wrapping around the arc. We have $C_n \to 2A$.}
\end{subfigure}
\hfill
\begin{subfigure}{0.48\textwidth}
\centering
\resizebox{60mm}{!}{
\fontsize{14pt}{14pt}\selectfont%% Creator: Inkscape 1.3 (0e150ed6c4, 2023-07-21), www.inkscape.org
%% PDF/EPS/PS + LaTeX output extension by Johan Engelen, 2010
%% Accompanies image file 'open.pdf' (pdf, eps, ps)
%%
%% To include the image in your LaTeX document, write
%%   \input{<filename>.pdf_tex}
%%  instead of
%%   \includegraphics{<filename>.pdf}
%% To scale the image, write
%%   \def\svgwidth{<desired width>}
%%   \input{<filename>.pdf_tex}
%%  instead of
%%   \includegraphics[width=<desired width>]{<filename>.pdf}
%%
%% Images with a different path to the parent latex file can
%% be accessed with the `import' package (which may need to be
%% installed) using
%%   \usepackage{import}
%% in the preamble, and then including the image with
%%   \import{<path to file>}{<filename>.pdf_tex}
%% Alternatively, one can specify
%%   \graphicspath{{<path to file>/}}
%% 
%% For more information, please see info/svg-inkscape on CTAN:
%%   http://tug.ctan.org/tex-archive/info/svg-inkscape
%%
\begingroup%
  \makeatletter%
  \providecommand\color[2][]{%
    \errmessage{(Inkscape) Color is used for the text in Inkscape, but the package 'color.sty' is not loaded}%
    \renewcommand\color[2][]{}%
  }%
  \providecommand\transparent[1]{%
    \errmessage{(Inkscape) Transparency is used (non-zero) for the text in Inkscape, but the package 'transparent.sty' is not loaded}%
    \renewcommand\transparent[1]{}%
  }%
  \providecommand\rotatebox[2]{#2}%
  \newcommand*\fsize{\dimexpr\f@size pt\relax}%
  \newcommand*\lineheight[1]{\fontsize{\fsize}{#1\fsize}\selectfont}%
  \ifx\svgwidth\undefined%
    \setlength{\unitlength}{303.16938986bp}%
    \ifx\svgscale\undefined%
      \relax%
    \else%
      \setlength{\unitlength}{\unitlength * \real{\svgscale}}%
    \fi%
  \else%
    \setlength{\unitlength}{\svgwidth}%
  \fi%
  \global\let\svgwidth\undefined%
  \global\let\svgscale\undefined%
  \makeatother%
  \begin{picture}(1,1.10516433)%
    \lineheight{1}%
    \setlength\tabcolsep{0pt}%
    \put(0,0){\includegraphics[width=\unitlength,page=1]{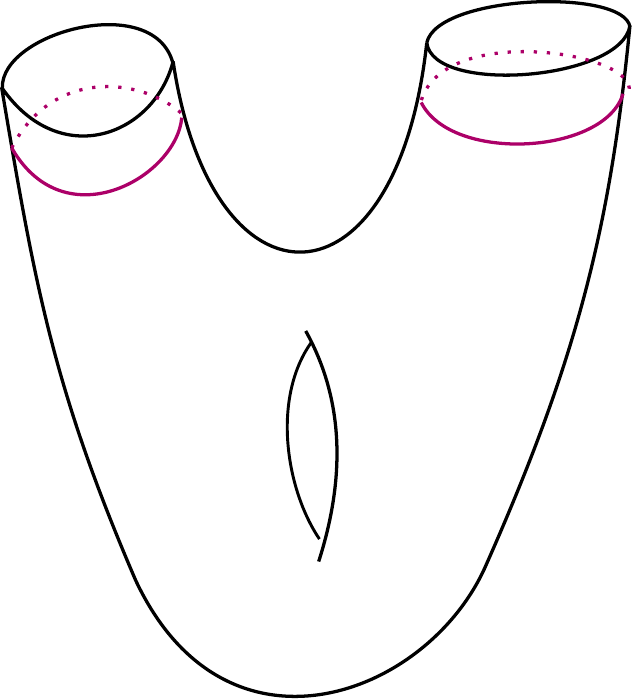}}%
    \put(0.23130501,0.74545599){\color[rgb]{0.6745098,0,0.40784314}\makebox(0,0)[t]{\lineheight{1.25}\smash{\begin{tabular}[t]{c}$B_1$\end{tabular}}}}%
    \put(0.84610005,0.8175049){\color[rgb]{0.6745098,0,0.40784314}\makebox(0,0)[t]{\lineheight{1.25}\smash{\begin{tabular}[t]{c}$B_2$\end{tabular}}}}%
    \put(0,0){\includegraphics[width=\unitlength,page=2]{open.pdf}}%
  \end{picture}%
\endgroup%
}
\caption{Open surface with two boundary parallel curves $B_1$ and $B_2$. The sequence of curves $C_n$ also sits in this surface, and we have $\frac{1}{n}C_n \to B_1 + B_2$.}
\end{subfigure}
\caption{Examples of cusped and open surfaces, and currents on them.}
\label{fig:open_surfaces}
\end{figure}

In $\Curr_{\mathrm{cusp}}(S)$, the intersection pairing fails to be
continuous; see~\cite[Remark~2.26]{MGT21:Smoothings} and
also~\cite{Sas22:Currents, T22:ThurstonCompactification}. It is thus
hard to get started working entirely in this setting.

The case of $\Curr_{\mathrm{open}}(S)$ is different. Even though intersection number is continuous in this case, it is not true that our defining properties of $\AConvex(S)$ guarantee a dual geodesic current.
The key issue is that the intersection number is degenerate: any
geodesic current $\mu$ supported on boundary curves will have zero
intersection number with all geodesic currents $\nu \in \Curr(S)$.

To better understand the relation between the cusped and open
situation, consider Figure~\ref{fig:open_surfaces} in a surface with
two ends for simplicity.
In the cusped surface, let $A$ be the simple arc going from cusp to cusp, which is in $\Curr_{\mathrm{cusp}}(S)$.
Let $C_n$ be a sequence of closed geodesics obtained by following $A$
from near one cusp to the other, wrapping around the second cusp $n$ times,
going back along $A$, and wrapping around the first cusp $n$ times
before closing up. In the weak$^*$-topology of
$\Curr_{\mathrm{cusp}}(S)$, we have $\lim_{n\to\infty}C_n = 2A$. (See, for example,
Sasaki~\cite[Section~6]{Sas22:Currents}.)

In the $\Curr_{\mathrm{open}}$ case, let $B_1$ and $B_2$ be the
boundary parallel curves homotopic to each respective boundary
component. The sequence of curves $C_n$ can also be considered in
$\Curr_{\mathrm{open}}(S)$. Without scaling, $C_n$ does not converge
to any current, but it does projectively. In fact, $\lim_{n \to
  \infty} \frac{1}{n} C_n = B_1 + B_2$.

Consider the geometric intersection number between arcs and curves in the cusped surface and curves in the open surface, i.e., $i \colon \Curves_{\mathrm{cusp}}(S) \times \Curves_{\mathrm{open}}(S)$.
We claim
\[
i(C_n, \tfrac{1}{n} C_n) \to i(2A, B_1 + B_2).
\]
Indeed, $i(C_n, C_n) = 4n$, and $i(2A, B_1+B_2)=4$, from which the result follows.
Following this example, we conjecture there is a well-defined and
continuous geometric intersection number between
$\Curr_{\mathrm{open}}(S)$ and $\Curr_{\mathrm{cusp}}(S)$.

\begin{conjecture}
There exists a continuous function 
\[
I \colon \Curr_{\mathrm{cusp}}(S) \times \Curr_{\mathrm{open}}(S) \to \mathbb{R}
\]
extending the natural geometric intersection pairing between $\Curves_{\mathrm{cusp}}(S)$ and $\Curves_{\mathrm{open}}(S)$.
\end{conjecture}

It is known that if one restricts the right-hand-side
factor to currents supported on a preimage of a compact subset on the
surface, the intersection number is continuous
(See \cite[Theorem~2.4]{BIPP21:Systoles}
and~\cite[Section~6]{Sas22:Currents}).

We can use this intersection number to propose a duality. Say that a
function $f \colon \Curves_{\mathrm{cusp}}(S) \to \mathbb{R}$
satisfies \emph{smoothing} if it satisfies smoothing between any
closed multi-curves, as well as smoothing between arcs and curves and
between arcs, in the sense that after resolving any essential crossing
(any intersection that cannot be homotoped away), the function does
not increase.

Similarly, say that a function
$f \colon \Curves_{\mathrm{open}}(S) \to \mathbb{R}$ satisfies
smoothing if it does not increase after resolving any essential
crossing.

For examples of such functionals, consider an embedded graph $\mathcal{G} \subset S'$.
Equipping $\mathcal{G}$ with a metric, we get a notion of length of
$\ell_\Sigma$ curves in $\Curves_{\mathrm{open}}(S)$ with respect to
this embedded metric graph. (See also
Section~\ref{subsec:embedded_graphs}.) When $\mathcal{G}$ is a rose-graph
(embedded and
with one vertex), Erlandsson~\cite{Erlandsson:WordLength} has
constructed examples of geodesic currents $\mu_{\mathcal{G}}$ in
$\Curr_{\mathrm{open}}(S)$ (in fact, half-integer weighted arcs and
multi-curves) so that
\[
\ell_{\Sigma}(C)=I(\mu_{\mathcal{G}}, C).
\]
This motivates the following question.

\begin{question}\label{quest:open-to-cusp}
Given a function $f \colon \Curves_{\mathrm{open}}(S) \to \mathbb{R}$
which is additive and satisfies smoothing, are there additional
assumptions guaranteeing the existence of a unique geodesic current $\mu \in \Curr_{\mathrm{cusp}}(S)$ so that
\[
f(C)=i(C, \mu)
\]
for every $C \in \Curves_{\mathrm{open}}$?
\end{question}

Unfortunately, additivity and smoothing alone are not enough. For
example, take the function
$f \colon \Curves_{\mathrm{open}}(S) \to \mathbb{R}$ assigning
hyperbolic length to non-boundary parallel curves and $0$ to
boundary-parallel curves. This functional still satisfies smoothing
(since no other curve can cross a boundary-parallel curve essentially)
but does not have a dual geodesic current in
$\Curr_{\mathrm{cusp}}(S)$, since by continuity of $I$ it would have
to assign positive length to limits of curves wrapping around the
boundary. See~\cite[Remark~2.26]{MGT21:Smoothings} for details.

In the other direction, we have a parallel conjecture.

\begin{conjecture}\label{conj:cusp-to-open}
Given a function $f \colon \Curves_{\mathrm{cusp}}(S) \to \mathbb{R}$ which is additive and satisfies smoothing, there exists a unique geodesic current $\mu \in \Curr_{\mathrm{open}}(S)$ so that
\[
f(C)=i(C, \mu)
\]
for every $C \in \Curves_{\mathrm{cusp}}(S)$.
\end{conjecture}

We have two motivating examples of functionals for
Conjecture~\ref{conj:cusp-to-open}.

The first one is related to the example on graphs above. Take again an
embedded metric graph $\mathcal{G} \subset S$, and define the
intersection $i(\mathcal{G},C)$ to be the minimal transversal
intersection between homotopy classes of embedded representatives of
$\mathcal{G}$ into $S$ and a homotopic representative of $C$, for
every closed curve in $\Curves_{\mathrm{cusp}}(S)$. This defines a
function $f \colon \Curves_{\mathrm{cusp}}(S) \to \mathbb{R}$. Then
our conjecture would assert the existence of a geodesic current
$\mu \in \Curr_{\mathrm{open}}(S)$. Presumably, as in Erlandsson's
result, $\mu$ in this case consists of a half-integral combination of
curves.

As a second example, consider a flat metric on the completion of $S'$ so that the missing points become conical singular points. Now the length of an arc between cusps is the length of the arc from cone point to cone point, which is a finite number.
This gives a well-defined notion of length $\ell \colon \Curves_{\mathrm{cusp}}(S) \to \mathbb{R}$.
On the other hand,  $\Curr_{\mathrm{open}}(S)$ are geodesic currents
whose support is contained in the closure of the subspace of geodesics
that do not go through cone points. The natural Liouville metric for
singular flat metrics is an example of such a
current~\cite[Section~3.2]{BL17:RigidityFlat},
and seems to be dual
to~$\ell$.

\subsection{Hyperbolic groups}
In recent work of Cantrell and Reyes and, independently, Kapovich and the first author, the following property has played a key role in extending stable lengths coming from group actions of hyperbolic groups $\Gamma$ acting on uniformly quasi-geodesic hyperbolic spaces $X$~\cite{KMG:BBP,CR25:Manhattan, CMGR26:GreenMetrics}.
Roughly, a Gromov hyperbolic space is a metric space whose triangles are thin. A group $\Gamma$ is said to be Gromov hyperbolic if its Cayley graph is.

A metric space $X$ is uniformly quasi-geodesic if there exists
constants $\lambda >0$ and $\epsilon \geq 0$, so that
every two points $x, y \in X$ are connected by a
$(\lambda,\epsilon)$-quasi-geodesic. We say that an action
$\Gamma \acts X$ of a non-elementary hyperbolic group~$\Gamma$ on a uniformly
quasi-geodesic hyperbolic metric space~$(X,d)$ has the \emph{bounded
  backtracking property} (see~\cite[Section~6.1]{CR25:Manhattan}) if given an orbit map
$\iota \colon \Cay(\Gamma) \to X$, there exists a constant $C>0$ so that
for every two points $x,y \in \Cay(\Gamma)$, and geodesic $[x,y]$,
there exists some quasi-geodesic $[\iota(x),\iota(y)]$ in $X$ so that
$\iota([x,y])$ and $[\iota(x),\iota(y)])$ are $C$-Hausdorff close in $d$.

There are examples of stable lengths satisfying bounded backtracking
which do not satisfy quasi-smoothing
(see Example~\ref{examp:exotic-lengths}). On the other hand, one
can ask if all functions on $\pi_1(S)$ satisfying quasi-smoothing
arise as the stable length of some action of $\pi_1(S)$ satisfying
bounded backtracking. 

\begin{question}
Let $\Gamma$ be a surface group, and let $f \colon \Gamma \to \mathbb{R}$ be a symmetric and stable functional satisfying quasi-smoothing (and hence, extending continuously to geodesic currents).
Does there exist a uniformly quasi-geodesic Gromov hyperbolic space $X$ on which $\Gamma$ acts with the bounded backtracking property and so that $\ell_X = f$, where $\ell_X$ denotes the stable length of the action $\Gamma \acts X$?
\end{question}

Motivated by this, it would be interesting to generalize the quasi-smoothing and smoothing condition to the general setting of hyperbolic groups.

\begin{question}
What is a generalization of the properties of smoothing and
quasi-smoothing for functionals on non-elementary hyperbolic groups
that makes the answer to the
above question work in that setting?
\end{question}

Related to Theorems~\ref{thm:intersections_group}
and~\ref{mainthm:skora}, it would be interesting to give a list of
axioms, in the spirit of oriented smoothing, characterizing
small actions of a non-elementary hyperbolic group $\Gamma$.

\begin{question}
\label{qu:generalize_intersections}
Let $\Gamma$ be a non-elementary Gromov hyperbolic group acting on an
$\mathbb{R}$-tree $T$. Give axioms  characterizing when the action is small, i.e., has
virtually cyclic edge stabilizers.
\end{question}

\subsection{Quasi-morphisms} 
\label{sub:quasimorphism} 

The hypotheses of \emph{oriented quasi-smoothing} are reminiscent of, but weaker than, the conditions for a map 
$f \colon \pi_1(S) \to \mathbb{R}$ to be a \emph{quasi-morphism}, that is, there exists $D > 0$ such that 
\[
\abs{f(ab) - f(a) - f(b)} < D \quad \text{for all } a,b \in \pi_1(S).
\]
A quasi-morphism is called \emph{homogeneous} if $f(a^n) = n f(a)$ for every $a \in \pi_1(S)$ and $n \ge 1$.

Quasi-morphisms have played a prominent role in the study of stable commutator length and bounded cohomology of groups~\cite{Cal09:scl,Pic97:BoundedCohomology_currents}.

It would be interesting to understand whether one can obtain homogeneous quasi-morphisms from curve functionals satisfying oriented quasi-smoothing and stability.

\begin{question}
Let $f \colon \pi_1(S) \to \mathbb{R}$ satisfy oriented quasi-smoothing and stability.  
Are there additional conditions on $f$ ensuring that the function
\[
g(a) \coloneqq f(a) - f(a^{-1})
\]
is a homogeneous quasi-morphism?
\end{question} 

To get a non-trivial function~$g$, the functional~$f$ must be
asymmetric. See Appendix~\ref{sec:asymmetric} for examples.

\subsection{Green metrics}
\label{subsec:green}
Let $\Gamma=\pi_1(\Sigma) \subset \Isom^+(\mathbb{H})=\PSL_2(\mathbb{R})$ be a countable, discrete cocompact subgroup, for some hyperbolic metric $\Sigma$ on $S$.
Let $\mu$ be a finitely supported probability measure on~$\Gamma$
whose support contains a symmetric generating set of $\Gamma$.
Let $w_n \coloneqq g_1 \cdots g_n$ be the random walk with respect to $\mu$, where each $(g_i)$ is identically and independently distributed with distribution $\mu$. Let $(\Gamma^{\mathbb{N}}, \mu^{\mathbb{N}})$ be the step space, and the map $\pi \colon \Gamma^{\mathbb{N}} \to \Gamma^{\mathbb{N}}$ be defined as $\pi((g_n)_{n \in \mathbb{N}}) \coloneqq (w_n)_{n \in \mathbb{N}}$. The image of $\pi$ is equipped with the measure $(\mathbb{P}_{\mu})_*(\mu^{\mathbb{N}})$.
Define the metric
\[ d_{\mu}(x,y) \coloneqq - \log \mathbb{P}_{\mu}(\exists n : w_n x =y).\]
In words, we are considering the probability that a random walk
departing from $x$ will reach~$y$ after some number of steps.
This metric has an associated stable length
function
\[
  \ell_{\mu}(g) \coloneqq \lim_{n\to\infty} d_{\mu}(x, g^n x)/n
\]
for any hyperbolic element $g \in \Gamma$, which can then be thought of as a curve functional $\ell_{\mu} \colon \Curves^+(S) \to \mathbb{R}$. The following question emerged from conversations of the first author with Eduardo Reyes.

\begin{question}
Does $\ell_{\mu}$ satisfy the hypotheses of
Theorem~\ref{thm:intersect} for some admissible probability measure $\mu$ (and thus come from a geodesic current)?
\label{q:green-smoothing}
\end{question}

While we were writing the second version of this manuscript, the following more specific question, proposed in the first version of our manuscript, has been answered (negatively) by Kosenko--Tiozzo~\cite{KS26:Singularity} with the help of AI.

\begin{question}
Does $\ell_\mu$ further satisfy the
hyperbolic metric hypotheses in Theorem~\ref{thm:hyp-char} for some admissible
probability measure~$\mu$ (and thus come
from a hyperbolic metric)?
\label{q:Green}
\end{question}

As we mentioned in the first version, this implies a negative answer to the following well-known question.

\begin{question}
Given $\Sigma$, is there a finitely-supported probability measure
$\mu$ on $\Gamma=\pi_1(\Sigma)$ so that the induced translation length
$\ell_{\mu}$ is dual to the Liouville current $\mathcal{L}_{\Sigma}$?
\label{q:singularity}
\end{question}

By work of Cantrell--Tanaka~\cite[Theorem~1.2]{CT25} and
Gouëzel--Mathéus--Maucourant \cite[Theorem~1.2]{GMM18:Entropy}, a
negative answer to this question settles the so-called
\emph{singularity conjecture} for $\SL_2(\mathbb{R})$
(see~\cite[Conjecture, page~~259]{KLP11:Singularity}). In short, by
these results Question~\ref{q:singularity} is equivalent to asking if
there is a probability measure $\mu$ on $\Gamma$ so that the induced
`hitting measure' $\nu_{\mu}$ on $\partial \Gamma$ is absolutely
continuous with respect to the Lebesgue measure (coming from the
hyperbolic metric $\Sigma$) on $\partial \Gamma$.
It would still be interesting to see what kinds of geodesic currents (if any) can arise as duals to stable lengths of Green metrics, addressing Question~\ref{q:Green}.

\subsection{Smoothing vs multi-smoothing}
\label{subsec:multismoothing}
The hypotheses of smoothing in Equation~\eqref{eq:smoothing} of Theorem~\ref{thm:intersect} apply to one essential crossing at a time. There is a stronger condition than smoothing, where multiple smoothings at essential crossings are performed at once, and the function is required to be non-increasing under such operation.
We will refer to this condition as \emph{multi-smoothing}.
It is clear multi-smoothing implies smoothing.
It is easily checked that intersection numbers satisfy the multi-smoothing property (see, for example, \cite{AH21:EffectiveCounting3}).
What is not so clear is that functionals of $\AConvex(S)$ (that is, satisfying stability, homogeneity, smoothing and additivity) must satisfy multi-smoothing. This is, of course, a consequence of Theorem~\ref{thm:intersect}.
Other curve functionals that are not in $\AConvex(S)$ also satisfy
multi-smoothing. For example, Arana-Herrera also shows
that extremal length satisfies
multi-smoothing \cite{AH21:EffectiveCounting3}.
This prompts the following question.

\begin{question}\label{quest:multi-smoothing}
If a curve functional satisfies homogeneity, stability, convex union,
and smoothing, does it necessarily satisfy multi-smoothing?
\end{question}

In this context, Arenas and Neumann-Coto
recently studied the related concept of \emph{taut smoothing}  \cite{ANC24:TautSmoothings}.

\subsection{Domination vs Smoothing}
\label{subsec:domination}
In another direction from Section~\ref{subsec:multismoothing},
recall from Section~\ref{sec:domination} that a curve $C$ dominates another curve $C'$ if $i(C,D) \geq
i(C',D)$ for every $D \in \Curves(S)$. Accordingly, we say a functional $f$ satisfies
\emph{domination} if for
every pair of curves $(C,C')$ with $C$ dominating $C'$, we have
$f(C) \geq f(C')$.
It is easy to see that if $C'$ is obtained from $C$ by
multi-smoothing, then $C$ dominates $C'$. There are however other
pairs $(C,C')$ with $C$ dominating $C'$ but $C'$ not obtainable from
$C$ via multi-smoothing
(see~\cite[Example]{NC01:ShortestGeodesics}).
It is clear, by definition, that intersection numbers satisfy domination.
All in all, we have the following implications on a functional~$f$:
\[
  \text{Domination} \Rightarrow \text{Multi-smoothing} \Rightarrow \text{Smoothing}.
\]

From Theorem~\ref{thm:intersect} we obtain the following weaker
theorem which is nevertheless conceptually interesting.

\begin{theorem}
A curve functional satisfies domination and additivity if and only if it is dual to a geodesic current.
\end{theorem}

This can be interpreted as saying that linear maps that preserve the poset structure of the space of multi-curves correspond to intersection numbers.

For similar reasons as for the smoothing (or multi-smoothing)
condition, the additivity hypothesis is necessary, since
the square
root extremal length can be easily seen to satisfy domination,
homogeneity and stability, but not additivity, and it is not dual to
an intersection number.
This suggests the following strengthening of
Question~\ref{quest:multi-smoothing}.

\begin{question}
If a curve functional satisfies homogeneity, stability,
convex union, and smoothing, does it necessarily satisfy domination?
\end{question}

\subsection{Convexity on train-track charts}
\label{sec:convexity_charts}

One of the reasons multi-smoothing is a useful property is that,
coupled with convex union, homogeneity, and stability, it
yields convex functions when restricted to train-track coordinates.
Precisely, let $f$ be a function in $\Convex(S)$ that furthermore
satisfies multi-smoothing. Let $\tau$ be a train-track on~$S$, let
$V(\tau)$ be the set of admissible weights on~$\tau$, $\ML(\tau)$ be the subspace of
measured laminations carried by $\tau$, and
$\phi_{\tau} \colon V(\tau) \to \ML(S)$ be the induced
map to measured laminations. If $\tau$ is maximal (i.e., all complementary
components are trigons), then the image of $\phi_\tau$ is an open
chart in $\ML(S)$.
Since $f$ extends continuously to geodesic currents
\cite[Theorem~A]{MGT21:Smoothings}, we can consider the composition
$f_{\tau} \coloneqq f \circ \phi_{\tau} \colon V(\tau) \to
\mathbb{R}$.

\begin{proposition}\label{prop:convexity-charts}
  If $f \in \Convex(S)$ also satisfies multi-smoothing, the function
  $f_\tau \colon V(\tau) \to \bbR$ is convex and homogeneous.
\end{proposition}

\begin{proof}
  The same proof as in
  \cite[Appendix~A]{Mir04:SimpleGeodesics},
  \cite[Theorem~2]{Thu16:Research},
  \cite[Theorem~4.7]{AH21:EffectiveCounting3} applies.
\end{proof}

\subsection{Analogy to Poincar\'e duality}
\label{subsec:poincare_duality}
Theorem~\ref{thm:intersect} gives a new definition of geodesic currents in terms of functionals on curves, by giving a bijection between geodesic currents and a subset of the space of curve functionals.
As mentioned in the introduction, one can view this theorem in analogy to Poincaré duality for closed surfaces.  Given a 1-cohomology class $\alpha^*$ of an $n$-dimensional manifold (represented by homomorphisms from $\pi_1(S)$ to $\mathbb{R}$), Poincaré duality associates an $(n-1)$-dimensional homology class $\alpha$ (represented by real linear combinations of $(n-1)$-cycles). Since $n=2$, these cycles can be expressed as linear combinations of simple closed curves.  

In our setting, the role of a 1-cohomology class is played by a curve
functional that satisfies smoothing, additivity, and
stability. Although it is not a homomorphism from
$\pi_1(S)$ to anything, it preserves a certain combinatorial structure
(see Section~\ref{subsec:domination}). Likewise, the role of the homology
class is taken by a geodesic current.

Our theorem can be thought of in analogy to $\alpha \cap \beta =
  \alpha^*(\beta)$, where $\cap$ denotes the cap product in homology (algebraic intersection), and thus relating cap product to the homology-cohomology pairing.

  This aligns with our theorem, which establishes a duality via the \emph{geometric} intersection number $i(\mu_f, C) = f(C)$, suggesting that the geometric intersection number plays the role of the cap product.

In Poincaré duality, one also has the standard relation $\alpha^* \cup \beta^* = \alpha \cap \beta$, where $\cup$ denotes the cup product in cohomology.  
A natural question that arises then is the following.

\begin{question} What is the appropriate analog of the cup product?
  More precisely, can one define the intersection $i(\cdot,\cdot)$ of
  two geodesic
  currents directly via their dual curve functionals, and prove its
  continuity?
\end{question}

\subsection{Local compactness for variations on the spaces}
\label{sec:local-compactness}

We next look at which spaces of curve functionals in
Definition~\ref{def:functional-variations}
are reasonable spaces. We always
use the product topology for spaces of functionals.
Start with the smallest space $\AConvex(S)$, which by
Theorem~\ref{thm:convex} is homeomorphic to $\Curr(S)$; in particular
it is locally compact (Proposition~\ref{prop:local_compactness}).
% This is not obvious \textit{a priori} by looking directly at the functionals.
We can ask about local compactness more generally.

\begin{question}
  Which of the spaces $\Convex(S)$, $\AConvex^+(S)$, and
  $\Convex^+(S)$ are locally compact?
\end{question}

Let us start with $\Convex(S)$. 
We expect the smoothing relation and convex union properties of~$f$ to
give, for every curve $C$, an upper bound
\[
f(C) \leq \sum_{i=1}^k n_i f(D_i),
\]
where $n_i \geq 0$ are non-negative integers depending on $C$, and $D_1,\cdots, D_k$ is a fixed finite collection of curves (independent of $C$). An inequality like that would give a positive answer to the following conjecture.

\begin{conjecture}\label{conj:convex-compact}
The set $\Convex(S)$ with the product topology is locally compact.
\end{conjecture}

Some examples of interest for this conjecture would be extremal
lengths associated to conformal structures or to elastic graphs
\cite[Sections 4.8 and 4.9]{MGT21:Smoothings}.

Turning to asymmetric functionals, Appendix~\ref{sec:asymmetric} can be summarized as saying that
$\AConvex^+(S)$ is not very related to the space $\Curr^+(S)$ of
oriented geodesic currents: the natural map $\mu \mapsto
\vec\imath(\mu,\cdot)$ is neither injective nor surjective, with what
appears to be a large difference.

Nevertheless the space may be of independent
interest, particularly the subspace $\AConvex_{\ge 0}^+(S)$ of
functionals with non-negative values.
There are many examples of such functionals. For instance,
you can take an asymmetric metric on $S$ and obtain a functional
on curves as in Section~\ref{subsec:length-metrics}. As in the
symmetric case, if there are rectifiable curves in each homotopy class
the result is automatically positive and satisfies oriented smoothing.

For a more concrete example, you can take an asymmetric generating set
for $\pi_1(S)$ as a monoid whose (directed) graph is simple in~$S$, as
in the example in Appendix~\ref{sec:asymmetric}.

\begin{question}\label{quest:AConvex+-compact}
  Is the space $\AConvex^+_{\ge 0}(S)$ locally compact?
\end{question}
\begin{question}\label{quest:asym-pseudo-metric}
  Does every element of $\AConvex^+_{\ge 0}(S)$ come from an
  asymmetric length pseudo-metric on $S$ in the style of
  Corollary~\ref{maincor:length_dual}?
\end{question}

  Note that the restriction to considering length pseudo-metrics on $S$ is crucial. There are many other pseudo-metrics induced by pulling back geometric actions of $\pi_1(S)$ on Gromov hyperbolic spaces, and most of those are not induced by geodesic currents (e.g., quasi-Fuchsian representations, see~\cite{FF22:QF}).

In this oriented context, the restriction that the functional is
positive is non-trivial (unlike for unoriented functionals): for any
functional $f$ and element $\omega \in H^1(S; \bbZ)$, for a sufficiently
large constant $\lambda>0$ the combination $f + \lambda\omega$ will take negative values.

We next turn to the various varieties of quasi-convex functionals from
Definition~\ref{def:functional-variations}.

\begin{question}
  For fixed $K > 0$ and surface $S$, which of the spaces $\AQConvex_K$,
  $\QConvex^+_K$, $\AQConvex^+_K$, $\QConvex^+_K$ are
  locally compact? What about the subspaces $\AQConvex^+_{K,\ge 0}$
  and $\QConvex^+_{K,\ge 0}$ of positive functionals?
\end{question}

Without the restriction to a fixed value of~$K$, these spaces are
unlikely to be locally compact. We give two examples.

Inside $\QConvex^+(S)$ (with no positivity assumption), we have the
large space of quasi-morphisms (Section~\ref{sub:quasimorphism}),
which is an infinite dimensional Banach space and hence not locally
compact (see, e.g., \cite{CRH26:Quasimorphisms}).

On the other hand, if we fix~$K$, the quasi-morphisms in
$\QConvex_K^+(S)$ are the homogeneous $K$-quasi-morphisms. The space
of $K$-quasi-morphisms is easily seen to be locally compact by an
inequality
\[
f(\vec C) \leq K' \ell_F(\vec C)
\]
where $\ell_F$ is the word-length induced by any finite generating
set, and $K'>0$ is a constant depending only on a finite subset of
curves. So it remains plausible that $\QConvex_K^+(S)$ is locally
compact.

%%% Local Variables:
%%% mode: latex
%%% TeX-master: "Intersections"
%%% End:

% \nocite{*}
\appendix

\section{Hyperbolic parallelograms}
\label{sec:hyp-par}

\begin{definition}
  A \emph{hyperbolic parallelogram} is a hyperbolic quadrilateral
  where opposite sides have equal length, as shown in
  Figure~\ref{fig:hyp-par}. Equivalently, it is a quadrilateral where
  the diagonals bisect each other.
\end{definition}

\begin{figure}
  \[
    \includegraphics{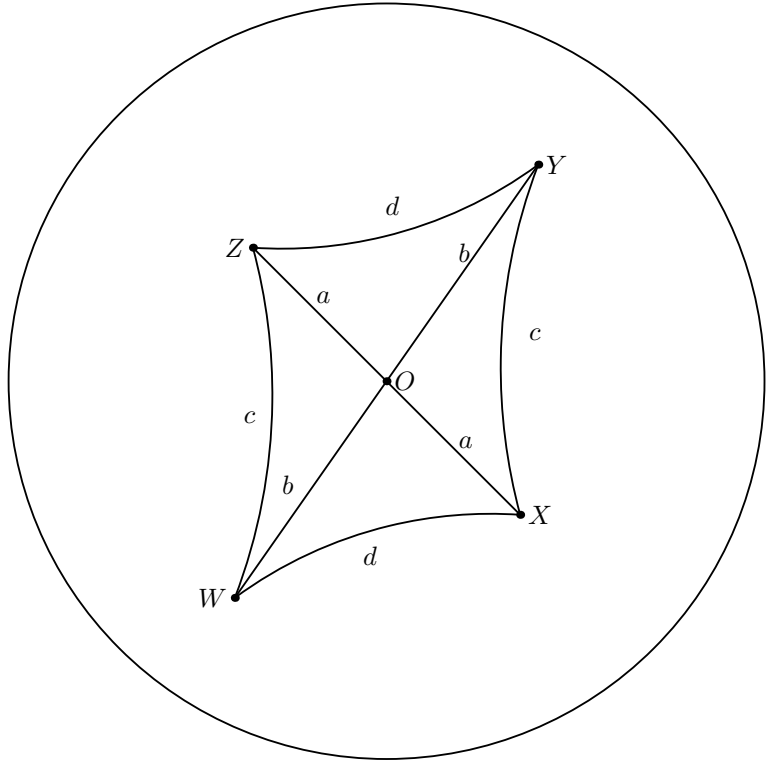}
  \]
  \caption{A hyperbolic parallelogram}
  \label{fig:hyp-par}
\end{figure}

The following geometric lemma is presumably known, though we did not
find a reference. It is the
hyperbolic version of the parallelogram law that, in a parallelogram
with sides $c,d$ and diagonals $2a, 2b$, we have $2(a^2+b^2) =
 c^2+d^2$. It can also be thought of as a formula giving the length of
 the median of
 a hyperbolic triangle in terms of its side lengths.
 
\begin{lemma}[Hyperbolic parallelogram law]\label{lem:hyp_parallelogram}
  Given a hyperbolic parallelogram $XYZW$ with center~$O$, let
  $c = \abs{X-Y} = \abs{Z - W}$ and $d = \abs{Y - Z} = \abs{X - W}$ be
  the lengths of the sides, and let $a = \abs{X - O} = \abs{Z
    - O}$ and $b = \abs{Y - O} = \abs{W - O}$ be the half-lengths of
  the diagonals. Then
  \[
    2 \cosh a \cosh b = \cosh c + \cosh d.
  \]
\end{lemma}

\begin{proof}
  Let $\gamma$ be the measure of $\angle XOY$ and $\delta=\pi-\gamma$ be the
  measure of $\angle YOZ$. The hyperbolic law of cosines gives
  \begin{align*}
    \cosh c &= \cosh a \cosh b - \sinh a \sinh b \cos \gamma\\
    \cosh d &= \cosh a \cosh b - \sinh a \sinh b \cos \delta.
  \end{align*}
  Adding these gives the desired result.
\end{proof}

%%% Local Variables:
%%% mode: latex
%%% TeX-master: "Intersections"
%%% End:

\section{Asymmetric intersection number and non-examples}
\label{sec:asymmetric}
One might wonder whether Theorem~\ref{thm:intersect} applies in the more general setting of
not-necessarily-symmetric curve functionals~$f$. One could hope that
asymmetric intersections with $\mu_f$ (Definition~\ref{def:asymmetric}) recover $f$ in the more
general context. In this section, we dash this hope: there are
asymmetric curve functionals satisfying oriented smoothing and additivity that
are not realized as asymmetric intersection numbers with any oriented
geodesic currents.

  Recall that in~Definition~\ref{def:asymmetric} we introduced an asymmetric
intersection number $\vec i (\cdot, \cdot)$ between oriented geodesic
currents and showed how to evaluate the intersection between a
geodesic current and an oriented curve in
Lemma~\ref{lem:asymmetric_intersection_current}. In particular, $\vec
i(\vec C, \vec D)$ recovers the right-handed intersections between
oriented curves $\vec C$ and~$\vec D$.
We will emphasize curves are oriented by writing
$\vec C$ or $\vec{[g]}$ instead of $C$ or $[g]$, and will write $\vec
C^{-1}$ or $\vec{[G]}$ for the oriented
curve with the opposite orientation.

\begin{definition}
Recall that $h(\vec C)\in H_1(\Sigma)$ denotes the homology class
  associated to~$\vec C$. Let
  $\langle h(\vec C), h(\vec D)\rangle$ or
  $\langle \vec C, \vec D\rangle$ be the algebraic intersection
  number, the number of right-handed crossings minus the number of
  left-handed crossings in any generic representatives of $\vec C$
  and~$\vec D$.
\end{definition}

\begin{lemma}\label{lem:asym-intersect}
  The asymmetric intersection number on curves satisfies
  \begin{equation}\label{eq:asym-intersect}
    \vec{\imath}(\vec C, \vec D) = {\textstyle\frac12}\,i(C, D) +
    {\textstyle\frac12}\langle h(\vec C), h(\vec D)\rangle.
  \end{equation}
  It extends continuously to oriented currents.
\end{lemma}

\begin{proof}
  The equality is 
  immediate from the definitions.

  The continuous extension to currents follows from the fact that the map
  $h \colon \Curves^+(S) \to H_1(S)$ sending an oriented curve to a
  homology class extends continuously to
  oriented currents.
  (The map~$h$ vanishes on oriented currents that are flip-invariant.) Indeed, given $\mu \in \Curr(S)$, we can define a linear map  $h_\mu \colon H^1(S,\mathbb{R})\to \mathbb{R}$ on de Rham cohomology
  \[
  [\omega] \mapsto \int_{T^1 S} \omega_x(v)d\hat{\mu}(x,v)
  \]
  where $\hat{\mu}$ is the geodesic flow invariant measure on
  $T^1 S$
  induced by $\mu$. The map $h$ is then defined by sending $\mu
  \mapsto h_\mu \in
  H_1(S;\mathbb{R})$; see
  \cite{Schwartzman:AsymptoticCycles,RuelleSullivan:CurrentsFlowsDiffeomorphisms}
  for further
  details.
  Using this fact, both terms on the RHS of \eqref{eq:asym-intersect}
  are continuous on currents.

  Alternately, the standard proof of continuity of $i(\cdot,\cdot)$
  reduces to carefully upper
bounding the measure of a current $\alpha \times \beta$ near the
diagonal of the locus of transversely intersecting geodesics
$\mathcal{D}(S)$ (see,
e.g.,~\cite[Theorem~8.2.21]{Martelli16:IntroGeoTop}).
Since by definition $\vec\imath(\alpha,\beta) \leq i(\alpha,\beta)$,
the same
proof gives that $\vec\imath(\cdot,\cdot) \colon \Curr^+(S) \times
\Curr^+(S) \to \mathbb{R}$ is continuous.
\end{proof}

As a result, we have strong constraints on which asymmetric curve
functionals can come from an asymmetric intersection number.

For any curve $\vec C$, let
$r(\vec C) = \OneHalf(\vec C - \vec C^{-1})$ be the anti-symmetrized
virtual copy of $\vec C$. Here a virtual curve is a weighted
curve where the coefficients are allowed to be negative. Similarly
for an oriented multi-curve $\vec C = \sum_i a_i \vec C_i$, define
$r(\vec C) = \sum_i a_i r(\vec C_i)$.
\begin{proposition}\label{prop:asym-constraint}
  For an oriented multi-curve $\vec C$ so that $h(\vec C) = 0$ and for
  any curve~$\vec D$, we have
  \[
    \vec{\imath}(r(\vec{C}), \vec D) = 0.
  \]
\end{proposition}
\begin{proof}
  Immediate from Lemma~\ref{lem:asym-intersect}: both terms vanish.
\end{proof}

This shows that an oriented geodesic current is \emph{not} determined
by its asymmetric intersection numbers with all curves, and
hence the space of oriented geodesic currents cannot be realized as
subspace of the space of (asymmetric) curve functionals. In contrast,
flip-invariant geodesic currents are determined by their geometric
intersection numbers with all closed
curves~\cite[Th\'eor\`eme~2]{Otal90:SpectreMarqueNegative}.

In the other direction, we can construct asymmetric functionals
satisfying smoothing, stability and additivity that cannot come
from asymmetric intersection numbers, from looking at directed graphs
on a surface; compare \cite[Section~4.6]{MGT21:Smoothings}

\begin{example}\label{examp:flat-torus}
  We start with an example on the unpunctured torus~$T$,  not quite in the
  framework of this paper.
  On~$T$, consider the oriented graph $\Delta$ that divides the torus into two
  triangles (with opposite orientation):
  \[
    \Delta = \mfig{hex-torus-0}
  \]
  Label the edges $\vec a, \vec b, \vec c$, as shown, and let
  $\ell_\Delta \colon \Curves^+(T) \to \bbZ$ be
  word-length with respect to this monoid generating set. (In terms of
  standard generators  $a,b$, we have $c = (ab)^{-1}$.) We have the
following behavior:
\begin{equation}\label{eq:l_Delta}
  \ell_\Delta(a^k b^\ell) =
  \mfig{hex-torus-10}.
\end{equation}
Here we have drawn the $(k,\ell)$ grid slanted, to reflect the triangular
symmetry of the generating set.
  This word length satisfies oriented smoothing: we are looking at
  length with respect to the embedded, oriented graph~$\Delta$.
As
such, given a concrete representative for an oriented multi-curve in
a neighborhood of~$\Delta$, we can smooth any essential
crossing and get a new oriented multi-curve still in the neighborhood
by the same smoothing argument as in Corollary~\ref{maincor:length_dual}.
Concretely, in this setting, smoothing
amounts to requiring that the function $f(a^k b^\ell)$ is a convex
function of $k,\ell$, which can readily be verified directly.

Now consider the multi-curve
$\vec C = \vec{[a]} + \vec{[b]} + \vec{[c]}$. We have $h(\vec C) = 0$,
so, by Proposition~\ref{prop:asym-constraint},
$\vec\imath(\vec C, \vec D)$ for any curve $\vec D$. But the
values of $\ell_\Delta$ in Eq.~\eqref{eq:l_Delta} do not satisfy this
constraint: the alternating sum of the values around the central
hexagon is not zero.
Contrast this with
asymmetric intersection numbers: here is $\vec\imath(\vec{[\gamma]},
\cdot)$ where $[\gamma]$ ranges over the shortest curves on the
hexagonal torus,
up to symmetries:
\begin{align*}
    \vec\imath(\vec{[a]}, \cdot) &= \mfig{hex-torus-11} &
    \vec\imath(\ora{[ac^{-1}]}, \cdot) &= \mfig{hex-torus-12}\\
    \vec\imath(\ora{[a^2c^{-1}]}, \cdot) &= \mfig{hex-torus-13} &
    \vec\imath(\ora{[a^3c^{-1}]}, \cdot) &= \mfig{hex-torus-14}.
\end{align*}
These functions all satisfy
Proposition~\ref{prop:asym-constraint}.
\end{example}

\begin{example}
  We can upgrade Example~\ref{examp:flat-torus} to one living on a
  genus~$2$ surface (so hyperbolic and in the context of this paper.)
  Divide the genus~$2$ surface into two once-punctured tori, consider
  monoid generators $a,b,c$ on the right-hand torus as above,
  considered as edges with very long length in an embedded directed
  graph, and supplement them with short monoid generators on the
  left-hand torus and the dividing meridian so lengths are very close
  to the above lengths in the unpunctured torus. Let $\Gamma$ be the
  resulting metric graph; arrange for $\Gamma$ to be embedded. Then
  the argument above shows $\ell_\Gamma$ satisfies oriented smoothing,
  while similar arguments show that the constraint of
  Proposition~\ref{prop:asym-constraint}
  is not satisfied so $\ell_\Gamma$ cannot be realized as an asymmetric
  intersection number.
\end{example}

%%% Local Variables:
%%% mode: latex
%%% TeX-master: "Intersections"
%%% End:

\section{Intersection numbers satisfy smoothing}
\label{app:necessary}

In this appendix we give a detailed proof that intersection numbers
satisfy the hypotheses of our theorem. The only non-trivial condition
to check is smoothing.
In a previous paper, we gave one proof, first proving it
topologically for closed curves and then invoking density of weighted
curves and continuity \cite[Example~4.2]{MGT21:Smoothings}. Yet another proof of the results below could be given considering any of the generalized cross-ratios associated to $\mu$, such as $[a,b,c,d]^+ \coloneqq \mu([d,a) \times [b,c))$, and applying Lemmas~\ref{lem:crossratio_stability},~\ref{lem:crinv},~\ref{lem:crossratio_disconnected_smoothing} and~\ref{lem:crossratio_connected_smoothing}. Here we
give a third, different proof, instead relying
on the hyperbolic geometry of the universal cover and the results of Martone and Zhang
\cite{MZ19:PositivelyRatioed}.

\begin{proposition}
Given a geodesic current $\mu$, the curve functional $i(\mu, \cdot)$ satisfies the hypotheses of Theorem~\ref{thm:intersect}.
\label{prop:necessary}
\end{proposition}
\begin{proof}
We will check each hypothesis:
 By definition (see \cite{Bonahon86:EndsHyperbolicManifolds}), intersection number is bilinear in both factors. In particular, additivity holds.

 Observe that given a curve $C$ with hyperbolic lift
 $\gamma$, and $x \in \gamma$, the set of geodesics $G(x, \gamma^n x]$
 can be partitioned into $\bigcup_{i=1}^{n-1} G(x, \gamma^{i} \cdot
 x]$.  Thus, by Equation~\eqref{eq:intersection_curve},
\begin{equation}\label{eq:intersection-stability}
i(\mu, \gamma^n) = \mu(G(x, \gamma^n x]) = \sum_{i=0}^{n-1} \mu(G(\gamma^{i-1} \cdot x, \gamma^{i} \cdot x]) = n\mu(G(x, \gamma x]) = n i(\mu, \gamma),
\end{equation}
where in the second equality we used finite additivity of $\mu$, and
in the third one, $\Gamma$-invariance. This proves stability and, in
particular, power smoothing.

 The other cases of smoothing are covered by
 Proposition~\ref{prop:smoothing_nec} below.
\end{proof}

The following result can be proven applying similar arguments as in~\cite[Lemma~4.4]{MZ19:PositivelyRatioed} or~\cite[Theorem~2.1]{FF22:QF}.

\begin{proposition}\label{prop:smoothing_nec}
Suppose $a,b \in \pi_1(S)$ are non-trivial elements with distinct fixed points on 
$\partial_{\infty} S$, and let $\mu$ be any geodesic current. Then:
\begin{enumerate}
\item If $\ora{a}$ and $\ora{b}$ cross, then
\[
\begin{aligned}
i(\mu,a)+i(\mu,b) &\ge i(\mu,ab),\\
i(\mu,a)+i(\mu,b) &\ge i(\mu,aB).
\end{aligned}
\]
\item If $\ora{a}$ and $\ora{b}$ are parallel, then
\[
\begin{aligned}
i(\mu,ab) &\ge i(\mu,a)+i(\mu,b),\\
i(\mu,ab) &\ge i(\mu,aB).
\end{aligned}
\]
\end{enumerate}
\end{proposition}

\begin{proof}
\begin{enumerate}
\item Let $p$ be the intersection point between the hyperbolic axes $\ora{a}$ and~$\ora{b}$.
Since $p \in \ora{a}$, we have $x=B \cdot p \in \ora{b}$.
We have $p$ and $y= a \cdot p$ are in $\ora{a}$. Thus
\begin{align*}
  i(\mu, ab) &\leq \mu(G(x, ab x]) &&\text{by \cite[Lemma 4.4(2)]{MZ19:PositivelyRatioed}}\\
  &\leq \mu(G(x, b x]) + \mu(G(bx, ab x]) &&\text{since $G(x, ab x] \subset G(x, b x] \cup G(bx, ab x]$}\\
             &= \mu(G(B p, p]) + \mu(G(p, ap])\\
               &= i(\mu, b) + i(\mu, a).
\end{align*}
In the last equality we use the fact that $p$ is in the
axis of $a$ and of $b$, and~\cite[Lemma 4.4(1)]{MZ19:PositivelyRatioed}.
A similar argument with $aB$ gives the other inequality.
\item
  The inequalities can be checked directly using a similar analysis as in the previous part (in fact, these have already been proven in~\cite[Proposition~4.5]{MZ19:PositivelyRatioed}).
However, we can give a much more direct proof relying on the continuity of intersection number: since $i(\mu, \cdot)$ is continuous by~\cite[Proposition~4.5]{Bonahon86:EndsHyperbolicManifolds} and homogeneous, it follows from Proposition~\ref{prop:cont_conn_disc} that $i(\mu, \cdot)$ satisfies connected smoothing, as we wanted to show.
\qedhere
\end{enumerate}
\end{proof}

%%% Local Variables:
%%% mode: latex
%%% TeX-master: "Intersections"
%%% End:

\section{Flip-invariance}

In this appendix we show the geodesic current $\overline \mu_f$ is flip-invariant on RH boxes.
Recall that in Remark~\ref{rmk:flip_rh_box} we observed that flipping an RH box $B(Q_2,Q_1,P_2,P_1)$ based at $x$ we obtain a new RH box $B(p_1,p_2,q_1,q_2)$, with the slight inconvenience that it is based at $x^{-1}$, instead of $x$. Hence, its value cannot be extracted by evaluating the measure $\overline \mu_f$ on the nose, since it has been defined on RH boxes based \emph{at $x$}. Instead, to analyze its value, we approximate the RH box based at $x^{-1}$ by a suitable sequence of RH boxes based at $x$.

\begin{figure}
\centering{
\resizebox{100mm}{!}{\Huge{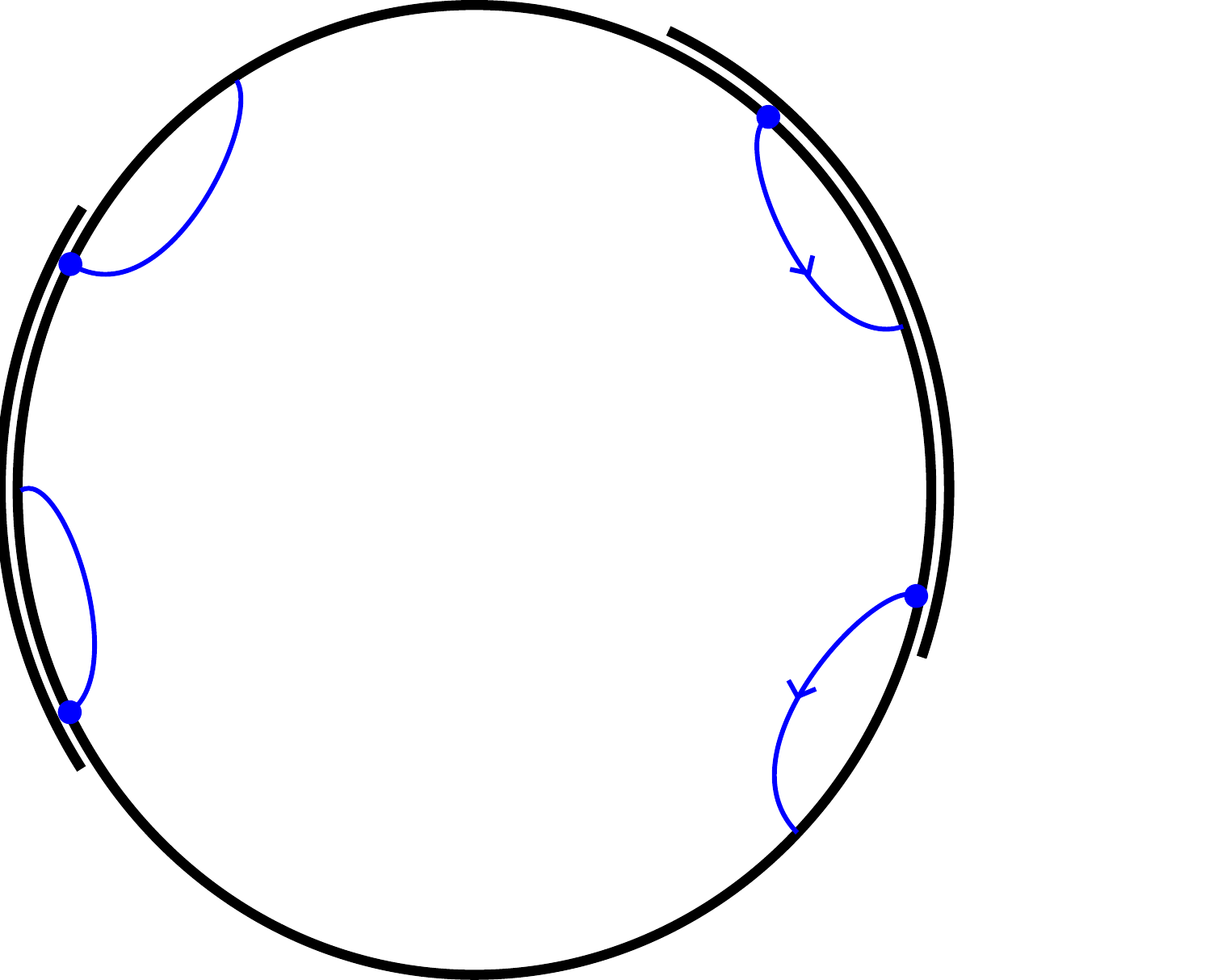}}
\caption{Approximating the flip of an RH box based at $x$ (which becomes an RH box based at $X$), via a collection of nested RH boxes based at $x$.}
\label{fig:flip}
}
\end{figure}

\begin{lemma}
\label{lem:RH_flip_invariance}
Let $B_L(p_1,p_2,q_1,q_2)$ be the RH box based at $X$ showed in Figure~\ref{fig:flip}, obtained by flipping by $\sigma$ an RH box $B(Q_2,Q_1,P_2,P_1)$ based at $x$. We claim
\begin{equation}
\begin{aligned}
4\overline \mu_f(B_L) &=\lim_{n \to \infty}
-f([Q_2P_2 x^n]) + f([Q_2P_1x^n])+f([Q_1P_2x^n])-f([Q_1P_1x^n])\\&=
\lim_{n \to -\infty} -f([q_1p_1x^n]) + f([q_1p_2x^n]) + f([q_2p_1x^n]) - f([q_2p_2x^n]).
\end{aligned}
\label{eq:LH}
\end{equation}
\end{lemma}
\begin{proof}[Proof sketch]
The second equality is obvious since $f$ is symmetric. Let us establish the first equality.
To find $\mu(B_L)$ we will approximate $B_L$ by RH boxes based at $x$, as follows.
Pick $r_i =r_i^{(0)}$, $s_i=s_i^{(0)}$ placed as in Figure~\ref{fig:flip}.
Note $P_i \cdot \ora{x}=\ora{P_i x p_i}$, $q_i \cdot \ora{x} = \ora{q_i x Q_i}$.
Set
\begin{equation}
\begin{matrix}
r_i^{(n)}=r_i P_i x^{-n} p_i & \quad & R_i^{(n)} \cdot \ora{x} =(P_i x p_i)^{n} R_i \cdot \ora{x} \\
s_i^{(n)}=q_i x^{-n} Q_i s_i & \quad & s_i^{(n)} \cdot \ora{x} =(q_i x Q_i)^{-n} s_i \cdot \ora{x}.
\end{matrix}
\end{equation}
Set $B_n =B(r_1^{(n)},r_2^{(n)}, s_1^{(n)}, s_2^{(n)})$, which, by construction, is an RH box based at $x$ for every $n \in \mathbb{N}$, and satisfies
\[
B_{n} \supset B_{n+1}
\]
for every $n \in \mathbb{N}$.
Furthermore,
\[
B_L^\circ=\cup_{n=1}^{\infty} B_n.
\]
By lower semi-continuity of measures and the fact that  $\overline \mu_f(\partial B_L)=0$, it follows that
\[
\overline \mu_f(B_L) = \lim_{n \to \infty} \overline \mu_f(B_n).
\]
We now evaluate this limit. We have
\begin{equation}
\begin{aligned}
4\lim_{n \to \infty} \overline \mu_f(B_n) =
\lim_{n \to \infty} \lim_{m \to \infty}\\ -&f([r_1P_1x^{n}p_1 q_1x^{-n}Q_1s_1x^m])\\
& +f([r_1P_1x^{n}p_1 q_2x^{-n}Q_2s_2x^m])\\
& +f([r_2P_2x^{n}p_2 q_1x^{-n}Q_1s_1x^m])\\
& -f([r_2P_2x^{n}p_2 q_2x^{-n}Q_2s_2x^m]).
\end{aligned}
\end{equation}
Consider the first summand.
Let $c=r_1P_1$, $a=p_1q_1$, $b=Q_1s_1$, $-n$ in place of~$n$, $n$ in
place of~$m$ in
Remark~\ref{rmk:triple_limits}~Equations~\eqref{eq:cross-join-triple}
and~\eqref{eq:cross-split-triple}, if necessary, adjusting
$r_1,r_2,s_1,s_2$ so the hypotheses of parallelism and crossing
conditions in Remark~\ref{rmk:triple_limits} are met; further use
Proposition~\ref{prop:join-split-geom}~\eqref{eq:cross-split} with
$a=r_1P_1$ and $b=p_1 q_1$. Together with
Lemma~\ref{lem:additivity-limit},
it follows that the limits
\begin{equation}
\begin{aligned}
\lim_{n,m\to \infty} &f([cx^{n}ax^{-n}bx^m])-f([x^{-n}])-f([x^{n}])-f([x^m]) \\
&= \lim_{n\to \infty}  \lim_{m\to \infty}f([cx^{n}])+f([ax^{-n}])+f([bx^m])-f([x^{-n}])-f([x^{n}])-f([x^m])
\end{aligned}
\end{equation}
exist and are equal.
Proceeding in the same fashion for the other three terms, common terms cancel out and the limit becomes
\begin{equation}
\begin{aligned}
4\lim_{n \to \infty} \mu(B_n) =
\lim_{n \to \infty} -f([p_1 q_1x^{-n}])
 +f([p_1 q_2x^{-n}])
 +f([p_2 q_1x^{-n}])
 -f([p_2 q_2x^{-n}]).
\end{aligned}
\end{equation}
which, by Eq.~\eqref{eq:LH}, is the desired quantity.
\end{proof}

%%% Local Variables:
%%% mode: latex
%%% TeX-master: "Intersections"
%%% End:

\bibliographystyle{hamsalpha}
\bibliography{Intersections}

\end{document}